\documentclass[11pt, a4paper,reqno]{amsart}
\allowdisplaybreaks[1]
\usepackage[latin1]{inputenc}
\usepackage{amsfonts,amsmath,amssymb, amscd,  graphicx, pinlabel, latexsym, color, exscale, enumerate, pstricks, mathtools}
\DeclareGraphicsExtensions{.jpg,.pdf,.mps,.eps,.png}

\usepackage[all]{xy}
\usepackage{psfrag}
\newtheorem{thm}{Theorem}[section]
\newtheorem{lem}[thm]{Lemma}

\newtheorem{prop}[thm]{Proposition}
\newtheorem{cor}[thm]{Corollary}
\newtheorem{conj}[thm]{Conjecture}

\theoremstyle{definition}
\newtheorem{rem}[thm]{Remark}

\newtheorem{defn}[thm]{Definition}
\newtheorem{nota}[thm]{Notation}

\newcommand{\N}{\mathbb{N}}
\newcommand{\Z}{\mathbb{Z}}
\newcommand{\Q}{\mathbb{Q}}
\newcommand{\R}{\mathbb{R}}
\newcommand{\C}{\mathbb{C}}
\newcommand{\ot}{\otimes}
\DeclareMathOperator{\B}{\mathcal{B}}

\DeclareMathOperator{\W}{\mathcal{W}}
\DeclareMathOperator{\He}{\mathcal{H}}

\DeclareMathOperator{\bim}{\mathcal{B}\it{im}}
\DeclareMathOperator{\kbim}{\mathrm{Kar}{\mathcal{B}\it{im}}}

\DeclareMathOperator{\dbim}{\mathcal{DB}\it{im}}

\DeclareMathOperator{\ebim}{\mathcal{EB}\it{im}}
\DeclareMathOperator{\kebim}{\mathrm{Kar}{\mathcal{EB}\it{im}}}
\DeclareMathOperator{\debim}{\mathcal{DEB}\it{im}}
\DeclareMathOperator{\kdebim}{\mathrm{Kar}{\mathcal{DEB}\it{im}}}
\DeclareMathOperator{\esbim}{\mathcal{ESB}\it{im}}

\DeclareMathOperator{\Hom}{\mbox{Hom}}
\DeclareMathOperator{\F}{\mathcal{F}}

\DeclareMathOperator{\br}{br}
\DeclareMathOperator{\rb}{rb}

\newcommand{\Ucataff}{\cal{U}(\widehat{\mathfrak sl}_n)}
\newcommand{\Uglcataff}{\cal{U}(\widehat{\mathfrak gl}_n)}
\newcommand{\Ucataffy}{\mathcal{U}(\widehat{\mathfrak sl}_n)_{[y]}}
\newcommand{\Uglcataffy}{\mathcal{U}(\widehat{\mathfrak gl}_n)_{[y]}}

\newcommand{\sseq}{{\rm SSeq}}

\newcommand{\onel}{{\mathbf 1}_{\lambda}}
\newcommand{\onelp}{{\mathbf 1}_{\lambda'}}

\newcommand{\refequal}[1]{\xy {\ar@{=}^{#1}
(-1,0)*{};(1,0)*{}};
\endxy}

\newcommand{\Uaff}{\dot{{\bf U}}(\hat{\mathfrak{sl}}_n)}
\newcommand{\Uglaff}{\dot{{\bf U}}(\hat{\mathfrak{gl}}_n)}
\newcommand{\Uglaffext}{\hat{\dot{{\bf U}}}(\hat{\mathfrak{gl}}_n)}
\newcommand{\SD}{{\bf S}}

\newcommand{\Ucat}{\cal{U}({\mathfrak sl}_n)}
\newcommand{\Scat}{\hat{\mathcal{S}}}

\newcommand{\xsum}[2]{
  \vcenter{\xy
  (0,.4)*{\sum};
  (0,3.8)*{\scs #2};
  (0,-3.2)*{\scs #1};
  \endxy}
}

\newcommand{\Uup}{\xy {\ar (0,-3)*{};(0,3)*{} };(0,0)*{\bullet};(2,0)*{};(-2,0)*{};\endxy}
\newcommand{\Udown}{\xy {\ar (0,3)*{};(0,-3)*{} };(0,0)*{\bullet};(2,0)*{};(-2,0)*{};\endxy}

\newcommand{\Ucupri}{\xy (0,-1)*{\dblue\xybox{(-2,1)*{}; (2,1)*{} **\crv{(-2,-3) & (2,-3)} ?(1)*\dir{>};}} \endxy}

\newcommand{\Ucupli}{\xy (0,-1)*{\dblue\xybox{(2,1)*{}; (-2,1)*{} **\crv{(2,-3) & (-2,-3)}?(1)*\dir{>};}}\endxy}

\newcommand{\Ucapri}{\xy (0,1)*{\dblue\xybox{(-2,-1)*{}; (2,-1)*{} **\crv{(-2,3) & (2,3)}?(1)*\dir{>};}}\endxy\;\;}

\newcommand{\Ucapli}{\xy (0,1)*{\dblue\xybox{(2,-1)*{}; (-2,-1)*{} **\crv{(2,3) &(-2,3) }?(1)*\dir{>};}}\endxy\;}

\newcommand{\Ucrossij}{\xy (0,0)*{\dgreen\xybox{\ar (2.5,-2.5)*{};(-2.5,2.5)*{}}}; (0,0)*{\dblue\xybox{\ar (-2.5,-2.5)*{};(2.5,2.5)*{} }};(4,0)*{};(-4,0)*{};\endxy}

\newcommand{\Ucrossdij}{\xy (0,0)*{\dblue\xybox{\ar (2.5,2.5)*{};(-2.5,-2.5)*{}}}; (0,0)*{\dgreen\xybox{\ar (-2.5,2.5)*{};(2.5,-2.5)*{} }};
(4,0)*{};(-4,0)*{};\endxy}

\newcommand{\BOX}{\hbox {$\sqcap$ \kern -1em $\sqcup$}}

\renewcommand{\to}{\rightarrow}
\newcommand{\maps}{\colon}

\newcommand{\scs}{\scriptstyle}
        \newcommand{\be}{\begin{equation}}
        \newcommand{\ee}{\end{equation}}
        \newcommand{\ba}{\begin{eqnarray}}
        \newcommand{\ea}{\end{eqnarray}}
        \newcommand{\ban}{\begin{eqnarray*}}
        \newcommand{\ean}{\end{eqnarray*}}
        \newcommand{\barr}{\begin{array}}
        \newcommand{\earr}{\end{array}}

\def\emph#1{{\sl #1\/}}

\let\hat=\widehat

\let\phi=\varphi
\let\epsilon=\varepsilon

\usepackage{bbm}
\def\cal#1{\mathcal{#1}}%
\def\1{\mathbbm{1}}%
\def\nn{\notag}

\def\shuffle{\,\raise 1pt\hbox{$\scriptscriptstyle\cup{\mskip
               -4mu}\cup$}\,}

\newcommand{\lineu}[1]{\xybox{%
  (-2,0)*{};
  (2,0)*{};
  (0,0)*{}; (0,-18)*{} **\dir{-}; ?(.5)*\dir{<}+(1.7,-7)*{\scs #1};
}}
\newcommand{\lined}[1]{\xybox{%
  (-2,0)*{};
  (2,0)*{};
  (0,0)*{}; (0,-18)*{} **\dir{-}; ?(.5)*\dir{>}+(1.7,-7)*{\scs #1};
}}

\newcommand{\lowrru}[1]{\xybox{%
  (-8,0)*{};
  (8,0)*{};
  (-6,-18)*{};(6,-9)*{} **\crv{(-6,-13) & (6,-15)} ?(1)*\dir{>};
  (6,-9)*{};(6,0)*{}  **\dir{-} ?(.3)*\dir{ }+(2,0)*{\scs {\bf j}};
}}

\newcommand{\lowllu}[1]{\xybox{%
  (-8,0)*{};
  (8,0)*{};
  (6,-18)*{};(-6,-9)*{} **\crv{(6,-13) & (-6,-15)} ?(1)*\dir{>};
  (-6,-9)*{};(-6,0)*{}  **\dir{-} ?(.3)*\dir{ }+(-2,0)*{\scs {\bf j}};
}}

\newcommand{\bbe}[1]{\xybox{%
  (-2,0)*{};
  (2,0)*{};
  (0,0);(0,-18) **\dir{-}; ?(.5)*\dir{<}+(2.3,0)*{\scriptstyle{#1}};
}}

\newcommand{\bbsid}{\xybox{%
  (-2,0)*{};
  (2,0)*{};
  (0,10);(0,4) **\dir{-};
}}
\newcommand{\bbpef}[1]{\xybox{%
  (-6,0)*{};
  (6,0)*{};
  (-4,0)*{}="t1";
  (4,0)*{}="t2";
  "t1";"t2" **\crv{(-4,-6) & (4,-6)}; ?(.15)*\dir{>} ?(.9)*\dir{>}
   ?(.5)*\dir{}+(0,-2)*{\scriptstyle{#1}};
}}
\newcommand{\bbpfe}[1]{\xybox{%
  (-6,0)*{};
  (6,0)*{};
  (-4,0)*{}="t1";
  (4,0)*{}="t2";
  "t2";"t1" **\crv{(4,-6) & (-4,-6)}; ?(.15)*\dir{>} ?(.9)*\dir{>}
  ?(.5)*\dir{}+(0,-2)*{\scriptstyle{#1}};
}}

\newcommand{\bbcfe}[1]{\xybox{%
  (-6,0)*{};
  (6,0)*{};
  (-4,0)*{}="t1";
  (4,0)*{}="t2";
  "t1";"t2" **\crv{(-4,6) & (4,6)}; ?(.15)*\dir{>} ?(.9)*\dir{>}
  ?(.5)*\dir{}+(0,2)*{\scriptstyle{#1}};
}}
\newcommand{\bbcef}[1]{\xybox{%
  (-6,0)*{};
  (6,0)*{};
  (-4,0)*{}="t1";
  (4,0)*{}="t2";
  "t2";"t1" **\crv{(4,6) & (-4,6)}; ?(.15)*\dir{>}
  ?(.9)*\dir{>} ?(.5)*\dir{}+(0,2)*{\scriptstyle{#1}};
}}

\newcommand{\ccbub}[2]{
\xybox{%
 (-6,0)*{};
  (6,0)*{};
  (-4,0)*{}="t1";
  (4,0)*{}="t2";
  "t2";"t1" **\crv{(4,6) & (-4,6)}; ?(.7)*\dir{}+(-2,0)*{\scs #2}
  ?(.05)*\dir{>} ?(1)*\dir{>};
  "t2";"t1" **\crv{(4,-6) & (-4,-6)};
   ?(.3)*\dir{}+(0,0)*{\bullet}+(0,-3)*{\scs {#1}};
}}
\newcommand{\cbub}[2]{
\xybox{%
 (-6,0)*{};
  (6,0)*{};
  (-4,0)*{}="t1";
  (4,0)*{}="t2";
  "t2";"t1" **\crv{(4,6) & (-4,6)};?(.7)*\dir{}+(-2,0)*{\scs #2};
   ?(0)*\dir{<} ?(.95)*\dir{<};
  "t2";"t1" **\crv{(4,-6) & (-4,-6)};
   ?(.3)*\dir{}+(0,0)*{\bullet}+(0,-3)*{\scs {#1}};
}}

\newcommand{\ncbub}{
\xybox{%
 (-6,0)*{};
  (6,0)*{};
  (-4,0)*{}="t1";
  (4,0)*{}="t2";
  "t2";"t1" **\crv{(4,6) & (-4,6)}; ?(0)*\dir{<} ?(.95)*\dir{<};
  "t2";"t1" **\crv{(4,-6) & (-4,-6)}; ?(.3)*\dir{};
}}

\newcommand{\bbdl}[1]{\xybox{%
  (2,0);(0,-8) **\crv{(2,-2)&(0,-6)}; ?(.5)*\dir{>}
}}
\newcommand{\bbdlu}[1]{\xybox{%
  (2,0);(0,-8) **\crv{(2,-2)&(0,-6)}; ?(.5)*\dir{<}
}}
\newcommand{\bbdr}[1]{\xybox{%
  (-2,0);(0,-8) **\crv{(-2,-2)&(0,-6)}; ?(.5)*\dir{>}
}}
\newcommand{\bbdru}[1]{\xybox{%
  (-2,0);(0,-8) **\crv{(-2,-2)&(0,-6)}; ?(.5)*\dir{<}
}}

\newrgbcolor{dblue}{0 0 0.5}
\newrgbcolor{dred}{0.9 0 0}
\newrgbcolor{dgreen}{0 0.6 0}
\newrgbcolor{dpink}{1 0.2 1}
\newrgbcolor{doran}{0.6 0.2 0}
\newrgbcolor{dturq}{0 0.6 0.4}
\newrgbcolor{dyellow}{0.6 0.6 0}
\newrgbcolor{dbrun}{0.3 0.3 0}
\newrgbcolor{dgrey}{0.4 0.4 0.4}
\newrgbcolor{dpurple}{0.6 0 0.4}
\newrgbcolor{dblack}{0 0 0}
\newrgbcolor{dduck}{0 0.3 0.3}

\newcommand{\laii}{\lambda_{i+1}}

\newcommand{\bbox}[1]{\framebox{$\scs #1$}}

\newcommand{\bscs}{\black\scs}
\newcommand{\llambda}{\overline{\lambda}}

\newcommand{\ii}{\underline{\textbf{\textit{i}}}}
\newcommand{\jj}{\underline{\textbf{\textit{j}}}}

\newcommand{\n}{\noindent}

\newcommand{\glcat}{\mathcal{U}(\hat{\mathfrak{gl}}_n)}

\newcommand{\Ugla}{\dot{\bf U}(\hat{\mathfrak{gl}}_n)}

\newcommand{\figins}[3] 
{\raisebox{#1pt}{\includegraphics[height=#2 in]{#3}}}
\newcommand{\bZ}{\mathbb{Z}}

\numberwithin{equation}{section}

\title{Categorifications of the extended affine Hecke algebra 
and the affine $q$-Schur algebra 
$\hat{\SD}(n,r)$ for $3\leq r < n$}
\author{Marco Mackaay and Anne-Laure Thiel}

\thanks{The two authors were supported by the FCT - Funda\c c\~{a}o para a 
Ci\^{e}ncia e a Tecnologia, through project number PTDC/MAT/101503/2008, 
New Geometry and Topology.}
\date{}

\begin{document}

\begin{abstract}
We categorify the extended affine Hecke algebra and the 
affine quantum Schur algebra 
$\hat{\SD}(n,r)$ for $3\leq r < n$, using results on diagrammatic 
categorification in affine type A by Elias-Williamson (extension of 
Elias-Khovanov for finite type A) and Khovanov-Lauda respectively. 
We also define $2$-representations of these categorifications on 
an extension of the $2$-category of affine (singular) Soergel bimodules. 
These results are the affine analogue of the results in~\cite{MSVschur}. 
\end{abstract}

\maketitle

\paragraph*{Acknowledgements}
We thank Jie Du and Qiang Fu for helpful exchanges of emails on the affine 
quantum Schur algebras, Mikhail Khovanov for helpful comments on 
the Krull-Schmidt property of the Karoubi envelope of 
$\mathcal{U}(\hat{\mathfrak{sl}}_n)_{[y]}$, and Ben Elias and 
Geordie Williamson for giving us an early version of their paper~\cite{EW} 
and some helpful comments on the Soergel categories in affine type $A$.

\tableofcontents
%\noindent
%{\sc Keywords:} 

\section*{Introduction}
Khovanov and Lauda~\cite{KL3}, and Rouquier~\cite{R2} following a 
slightly different approach, defined a graded additive $2$-category 
$\mathcal{U}(\mathfrak{g})$ with ``nice properties'' for any root datum. 
The $2$-morphisms are defined by string diagrams with regions labeled by 
$\mathfrak{g}$--weights. They are generated by a finite 
set of elementary diagrams which obey a finite set of relations. 
The split Grothendieck group of the Karoubi envelope of 
$\mathcal{U}(\mathfrak{g})$ is isomorphic 
to the idempotented version of the corresponding quantum group 
$\dot{\mathbf{U}}(\mathfrak{g})$. In Crane and Frenkel's~\cite{CrF} 
terminology, we say that 
$\mathcal{U}(\mathfrak{g})$ {\em categorifies} 
$\dot{\mathbf{U}}(\mathfrak{g})$.

Khovanov and Lauda only proved this {\em categorification theorem} 
for $\mathfrak{g}=\mathfrak{sl}_n$. A key ingredient of that proof was a 
$2$-representation of $\mathcal{U}(\hat{\mathfrak{sl}}_n)$ 
on a $2$-category build out of the 
cohomology rings of partial flag varieties. The equivariant cohomology 
rings of these varieties, which also give rise to a $2$-representation, 
are equivalent to the singular Soergel bimodules of type $A$, introduced and 
studied by Williamson in his PhD thesis in 2008 and published in~\cite{Wi}. 
The general categorification 
theorem was proved by Webster~\cite{Web}, by geometric techniques 
well beyond our understanding. 

In~\cite{MSVschur},  Mackaay, Sto\v si\'{c} and Vaz 
defined a quotient of $\Ucat$, 
denoted $\mathcal{S}(n,r)$, and proved that it 
categorifies the quantum Schur algebra $\SD(n,r)$, 
for any $r\in\mathbb{Z}_{>0}$. If $n\geq r$, then $\mathcal{S}(n,r)$ contains a 
full sub-$2$-category which categorifies the Hecke algebra 
$\He_{A_{r-1}}$. This sub-$2$-category is equivalent to the 
$2$-category of (ordinary) Soergel bimodules 
of type $A_{r-1}$, as was proved in~\cite{MSVschur} using Elias and Khovanov's 
diagrammatic presentation of the Soergel $2$-category~\cite{EKh}. 

In the same paper, Mackaay, Sto\v si\'{c} and Vaz also 
showed that Khovanov and Lauda's $2$-representation of $\Ucat$ 
on the singular Soergel bimodules descends 
to $\SD(n,r)$. Its restriction to the aforementioned sub-$2$-category 
of $\SD(n,r)$ is exactly Elias and Khovanov's $2$-equivalence of their 
diagrammatic $2$-category and the $2$-category of 
(ordinary) Soergel bimodules.     
\vskip0.2cm
Naturally the question arises whether the results in~\cite{MSVschur} 
extend to affine type $A$. In this paper, we show that this is indeed 
the case for $3\leq r < n$: 
\begin{itemize}
\item As Libedinsky explained in~\cite{Li1}, one can define Soergel 
bimodules using the geometric representation of the 
affine Weyl group $\W_{\hat{A}_{r-1}}$ 
or Soergel's extension of that representation~\cite{S3}. 
The geometric representation is not faithful, whereas 
Soergel's representation is reflection faithful. Both 
representations give rise to categories of Soergel bimodules which 
categorify the affine Hecke algebra $\He_{\hat{A}_{r-1}}$, as shown 
in~\cite{Har,Li,Li1,S3}. For more information on this topic, 
see also~\cite{EW}. 

However, the extension of the geometric representation 
to the {\em extended} affine Weyl group $\hat{W}_{\hat{A}_{r-1}}$ is ``too degenerate'' 
and cannot be used to categorify the {\em extended} affine Hecke algebra 
$\hat{\He}_{\hat{A}_{r-1}}$. Soergel's representation 
extends nicely to $\hat{W}_{\hat{A}_{r-1}}$ 
and we show that there is a corresponding 
category of bimodules, denoted $\ebim_{\hat{A}_{r-1}}$, 
which categorifies $\hat{\He}_{\hat{A}_{r-1}}$. 

Since Soergel's representation of $\hat{\W}_{\hat{A}_{r-1}}$ has dimension 
$r+1$, the corresponding bimodules are defined over $\Q[y,x_1,\ldots,x_r]$ 
where $\deg(y)=\deg(x_1)=\ldots=\deg(x_r)=2$. We will show that 
multiplication by $y$ is always a bimodule map of degree two. 
\item We also define a diagrammatic $2$-category $\debim_{\hat{A}_{r-1}}$, 
similar to the ones in~\cite{EKh,EW}, and show that it is $2$--equivalent 
to $\ebim_{\hat{A}_{r-1}}$ (actually they are equivalent 
as monoidal categories, i.e. $2$-categories with one object).
Here we use the corresponding result for the non-extended category 
of affine bimodules and its diagrammatic analogue due to Elias and 
Williamson~\cite{EW}. 
\item We define a $y$-deformation of the level-zero $2$-category 
$\mathcal{U}(\hat{\mathfrak{sl}}_n)$. In order to do that, the homogeneous 
$2$-morphisms are defined over $\Q[y]$ instead of $\Q$, with $y$ a formal 
variable of degree two. We denote this $2$-category by $\Ucataffy$ and 
prove that its Karoubi envelope is Krull-Schmidt. 

We recover $\Ucataff$ when we divide by the ideal generated by $y$. This ideal 
is virtually nilpotent, so the Grothendieck groups of the Karoubi envelopes 
of $\Ucataffy$ and $\Ucataff$ are isomorphic. 
\item We define a quotient of $\Ucataffy$, which we denote 
$\Scat(n,r)_{[y]}$. We prove that $\Scat(n,r)_{[y]}$ categorifies 
the affine quantum Schur algebra $\hat{\SD}(n,r)$. Again, the ideal 
generated by $y$ is virtually nilpotent, so the quotient of $\Scat(n,r)_{[y]}$ 
by this ideal also categorifies $\hat{\SD}(n,r)$. We denote this quotient 
by $\Scat(n,r)$, which can also be obtained as a quotient of 
Khovanov and Lauda's original $\Ucataff$.  
\item We define a $2$-functor 
$$\Sigma_{n,r}\colon \debim^*_{\hat{A}_{r-1}}\to\Scat^*(n,r)_{[y]},$$ 
prove it to be faithful and conjecture it to be full.\footnote{We will explain 
the ${}^*$ notation in Section~\ref{sec:catHeck}. It basically allows us to consider 
$2$--morphisms of arbitrary degree.} 
\item We define the $2$-category of 
extended affine singular Soergel bimodules $\esbim_{\hat{A}_{r-1}}$ and 
give the affine analogue (and $y$-deformation) of Khovanov and Lauda's 
$2$--representation, i.e. a $2$-functor 
$$\mathcal{F}'\colon \Scat(n,r)^*_{[y]}\to \esbim_{\hat{A}_{r-1}}^*.$$
\end{itemize}

\begin{rem}
The case $n=r$ is different, because $\hat{\SD}(n,n)$ is not a quotient of 
$\Uaff$ but only of a strictly larger algebra. Therefore, one has to 
extend the Khovanov-Lauda affine calculus in order to define $\Scat(n,n)$. 
This case is dealt with in a follow-up paper~\cite{MTh2}. 

The case $n<r$ cannot be dealt with at present, because even the 
decategorified story has not been worked out (see Problem 2.4.5 in~\cite{DDF}).
\end{rem}

\begin{rem}
There is a technical detail, which will be fully explained in Section~\ref{sec:AffSchur} 
but should be mentioned here already. Just as in~\cite{MSVschur}, we actually 
define $\Scat(n,r)_{[y]}$ as a quotient of $\Uglcataffy$, 
which is a $2$-category obtained from $\Ucataffy$ 
by switching to degenerate $\hat{\mathfrak{gl}}_n$--weights of 
level zero for the labels of the regions 
in the string diagrams. We conjecture that 
$\Uglcataffy$ categorifies 
the level-zero $\dot{\mathbf U}(\hat{\mathfrak{gl}}_n)$, but do not need 
that fact for the rest of this paper.  
\end{rem}

The results in this paper have several points of interest. 
The categories $\ebim_{\hat{A}_{r-1}}$ and $\debim_{\hat{A}_{r-1}}$ contain the 
new objects $B_{\rho^{\pm}}$ and $\pm$, respectively, and 
the corresponding morphisms. As our results show, these objects and morphisms 
also show up naturally in $\Scat(n,r)_{[y]}$ as $1$ and $2$-morphisms.   

The $y$-deformation $\Ucataff_{[y]}$ and its Schur quotient 
$\Scat(n,r)_{[y]}$ are new. As the results in this paper show, 
they both show up naturally when considering Soergel bimodules over 
$\Q[y,x_1,\ldots,x_r]$, which is the ring of polynomial functions on 
Soergel's reflection faithful representation of the affine Weyl group.  
 
Furthermore, there are interesting (possible) links with 
other categorifications of the (extended) affine Hecke algebra 
and the quantum affine Schur algebra. Lusztig~\cite{LuQG, LuAff} and 
Ginzburg and 
Vasserot~\cite{GV} gave a categorification of $\hat{\SD}(n,r)$ 
using perverse sheaves, 
extending Grojnowski and Lusztig's approach to the categorification 
of $\SD(n,r)$. It would be interesting to find the precise relation with the 
categorification presented here and in our follow-up paper~\cite{MTh2}.   

In this paper we also define an extended version of the 
affine singular Soergel bimodules. Williamson introduced and 
studied the $2$-category of singular 
Soergel bimodules for any Coxeter group in his PhD thesis in 2008, 
the results of which were published in~\cite{Wi}, and proved that 
it categorifies a certain ``new'' algebra, which he called 
the {\em Schur algebroid}. In finite type $A$ the Schur algebroid is 
isomorphic to the quantum Schur algebra. Williamson's 
affine type $A$ Schur algebroid should also be closely related to the 
affine quantum Schur algebra. Whatever the precise relation turns out to 
be, the $2$-representation of $\Scat(n,r)_{[y]}$ on the 
extended affine singular Soergel bimodules 
establishes an interesting relation between Khovanov and Lauda's work and 
Williamson's. 

Another point of interest is related to the possibility of 
categorifying the so called {\em Kirillov-Reshetikhin modules} of $\Uaff$. 
These level zero modules can be defined for any affine quantum group 
and have been 
intensively studied (see~\cite{ChP, DDF, DuFu, Kash} for more information 
and references). 

In affine type $A$ (and only in that type), 
they are special examples of {\em evaluation modules} $V_{\lambda, a}$, 
where $\lambda$ is a dominant weight and $a\in\mathbb{C}^{*}$. 
If $\lambda$ is an $n$-part partition of $r$, then $V_{\lambda, a}$ 
descends to a representation of $\hat{\SD}(n,r)$. More precisely, 
$V_{\lambda, a}$ is defined by pulling back (the technical term is 
{\em inflating}) the action of $\SD(n,r)$ on the irrep $V_{\lambda}$ via 
the so called {\em evaluation map} 
$$\mbox{ev}_a\colon\hat{\SD}(n,r)\to \SD(n,r).$$

If $\lambda=(m^i)$, i.e. $m$ times the $i$-th fundamental 
$\mathfrak{gl}_n$--weight, then 
it is known that $V_{\lambda, q^{i-m+2}}$ is isomorphic to a 
Kirillov-Reshetikhin module and has a canonical basis. 
It seems likely that one can categorify the evaluation map 
$$\mbox{ev}_{q^{i-m+2}}\colon \hat{\SD}(n,mi)\to \SD(n,mi)$$
and therefore $V_{\lambda, q^{i-m+2}}$, but such a categorification is 
beyond the scope of this paper.   
\vskip0.2cm
This paper is organized as follows:
\begin{itemize}
\item In Section~\ref{sec:affine}, we recall the definition of 
and some basic results on affine roots and weights, the (extended) affine 
Weyl group and the (extended) affine Hecke algebra.  
\item In Section~\ref{sec:catHeck}, we define $\ebim_{\hat{A}_{r-1}}$ and recall 
H\"{a}rterich's categorification result. 
We also use Elias-Khovanov type diagrams in order to 
define $\debim_{\hat{A}_{r-1}}$ and show that it gives a diagrammatic 
presentation of $\ebim_{\hat{A}_{r-1}}$, using Elias and Williamson's results 
in~\cite{EW}.
\item In Section~\ref{sec:AffSchur}, we first recall the definition of 
$\hat{\SD}(n,r)$ and its relation to the extended affine Hecke algebra. After 
that, we give the definitions of $\Uglcataffy$, $\Ucataffy$ and 
$\Scat(n,r)_{[y]}$.
\item In Section~\ref{sec:embed}, we give the $2$-functor 
$$\Sigma_{n,r}\colon \debim_{\hat{A}_{r-1}}^*\to \Scat^*(n,r)_{[y]}$$
and prove that it is well-defined.    
\item In Section~\ref{sec:2rep}, we define the affine $2$-representation 
$$\mathcal{F}'\colon \Scat^*(n,r)_{[y]}\to\esbim_{\hat{A}_{r-1}}^*.$$ 
The fact that it is well-defined follows essentially from the well-definedness of the 
analogous $2$-functor for finite type $A$ and a ``conjugation trick'', 
which we will explain. 
\item In Section~\ref{GrAlg}, we use the results in the previous sections 
to prove that $\Scat(n,r)_{[y]}$ categorifies $\hat{\SD}(n,r)$.
\end{itemize}

\section{The affine setting}
\label{sec:affine}
\subsection{Affine roots of level zero}

We use the well-known realization of~$\hat{\mathfrak{sl}}_r$ as the central extension of the loop algebra of~$\mathfrak{sl}_r$ together with a derivation 
(see for example \cite{PS, Kac, Fu}), i.e. the underlying vector space is isomorphic to 
$$\hat{\mathfrak{sl}}_r = \mathcal{L}(\mathfrak{sl}_r) \oplus \Q \langle c \rangle\oplus \Q \langle d \rangle.$$
In order to express its root system, consider the 
Cartan subalgebra 
$$\hat{\mathfrak{h}} = \mathfrak{h} \oplus \Q \langle c \rangle \oplus \Q \langle d \rangle = \Q \langle h_i,c,d | i= 1, \dots,r-1 \rangle$$ 
with~$\mathfrak{h}$ being the Cartan subalgebra of~${\mathfrak{sl}}_r$. 
The roots with respect to~$\hat{\mathfrak{h}}$ are
$$ \alpha = \left( \bar{\alpha}, 0,m \right), \quad \bar{\alpha} \in \Phi \left( \mathfrak{sl}_r \right), \ m \in \Z$$
and
$$ \alpha = \left(0,0,m \right), \quad  m \in \Z \setminus \lbrace 0 \rbrace.$$
The roots of the first family are called the {\em real roots} and the ones of the second family are called the {\em imaginary roots}.

The simple roots are  
$$\alpha_i = \left( \bar{\alpha_i}, 0,0 \right) \quad i=1,\dots,r-1$$
and
$$\alpha_r = \left( -\bar{\theta}, 0,1\right) = \delta - \bar{\theta} $$
where~$\bar{\alpha}_i = \varepsilon_i - \varepsilon_{i+1}$, for~$i=1,\dots,r-1$ are the simple roots of~$\mathfrak{sl}_r$, $\bar{\theta} = \bar{\alpha}_1 + \dots + 
\bar{\alpha}_{r-1} = \varepsilon_1 - \varepsilon_r$ is the highest root and~$\delta$ is the dual element of~$d$. 
The elements~$\varepsilon_i$ for~$i=1,\dots,r$ are the canonical basis vectors in $\mathbb{Z}^r$.

Weights are triples of the form 
$$\kappa=\left( \bar{\kappa},k,m \right),$$
where $\bar{\kappa}$ is an $\mathfrak{sl}_r$--weight and $k$ and $m$ are integers. 
The integer $k$ is called the {\em level} of $\kappa$. The inner product between two 
weights~$\kappa = \left( \bar{\kappa}, k,m \right)$ and~$\kappa' = \left( \bar{\kappa}', k',m' \right)$ is given by
$$\left< \kappa, \kappa' \right> = \left< \bar{\kappa}, \bar{\kappa'} \right>+km'+k'm,$$
where $\left< \bar{\kappa}, \bar{\kappa'} \right>$ is the usual inner product of $\mathfrak{sl}_r$--weights. 
In particular, we have 
$$\left< \alpha_i, \alpha_{i} \right> = 2 \quad \mbox{for all} \ i=1,\dots,r$$
and
$$\left< \alpha_r, \alpha_{1} \right> = \left< \alpha_i, \alpha_{i+1} \right> = -1 \quad \mbox{for all} \ i=1,\dots,r-1.$$
The simple coroots are~$\alpha_i^{\vee} = \alpha_i$ for all~$i=1,\dots,r$.

In the following sections, we will also use $\hat{\mathfrak{gl}_r}$--weights
$$\kappa=(\bar{\kappa},k,m),$$
where $\bar{\kappa}$ denotes a non-affine $\mathfrak{gl}_r$--weight.  

\begin{rem}
In this paper, we will only consider $\hat{\mathfrak{sl}}_r$ and 
$\hat{\mathfrak{gl}}_r$--weights of level zero.
\end{rem}

\subsection{The Weyl group action}

For any root~$\alpha \in \Phi \left( \hat{\mathfrak{sl}}_r \right)$ and $\hat{\mathfrak{sl}}_r$--weight~$\kappa$, the Weyl reflection $\sigma_{\alpha}$ is defined by 
$\sigma_{\alpha}(\kappa) = \kappa - \left< \kappa,\alpha^{\vee} \right> \alpha$. So if~$\alpha = \left( \bar{\alpha}, 0, n \right)$ 
and~$\kappa = \left( \bar{\kappa}, k, m \right)$, we can express $\sigma_{\alpha}(\kappa)$ as follows
\begin{equation}
\sigma_{\alpha}(\kappa) = \left( \bar{\kappa} - \left< \bar{\kappa}, \bar{\alpha}^{\vee} \right> \bar{\alpha} -kn \bar{\alpha}^{\vee}, k, m - \left< \bar{\kappa}, \bar{\alpha}^{\vee} \right> n -\frac{2kn^2}{\left<\alpha, \alpha \right>} \right).
\end{equation}
The {\em affine Weyl group} $\W_{\hat{A}_{r-1}}$ is the group generated by all these reflections. For any simple root $\alpha_i$, we write $\sigma_i:=\sigma_{\alpha_i}$. 

Similarly, one can consider the action of $\sigma_i$ on a level zero~$\hat{\mathfrak{gl}}_r$--weight of the 
form~$\left( \varepsilon_j, 0, m\right)$ with~$j = 1,\dots, r$. If $i \neq r$, one gets
\begin{equation}
\sigma_i(\varepsilon_j,0,m)  = 
\left\lbrace
\begin{array}{ll}
(\varepsilon_{i+1},0,m) &\quad \mbox{if} \ j=i \\
(\varepsilon_{i},0,m)  &\quad \mbox{if} \ j=i+1 \\
(\varepsilon_{j},0,m) &\quad \mbox{otherwise.}  \\
\end{array}
\right.
\end{equation}
The action of~$\sigma_r$ is given by
\begin{equation}
\sigma_r(\varepsilon_j,0,m)  = 
\left\lbrace
\begin{array}{ll}
(\varepsilon_{r},0,m+1) &\quad \mbox{if} \ j=1 \\
(\varepsilon_{1},0,m-1)  &\quad \mbox{if} \ j=r \\
(\varepsilon_{j},0,m) &\quad \mbox{otherwise.}  \\
\end{array}
\right.
\end{equation}

For each simple~$\mathfrak{sl}_r$--root $\bar{\alpha}_i$, $i = 1,\dots,r-1$, there also exists a translation~$t_{\bar{\alpha}_i}$, which acts on the level 
zero~$\hat{\mathfrak{gl}}_r$--weights as follows
\begin{equation}
t_{\bar{\alpha}_i} \left( \bar{\kappa}, 0, m \right) = \left( \bar{\kappa}, 0, m - \left(\kappa_i - \kappa_{i+1} \right) \right)
\end{equation}
where~$\bar{\kappa} =\left( \kappa_1,\dots,\kappa_r \right)$ and the indices are taken to be modulo $r$, e.g. $\kappa_{r+1}=\kappa_1$ by definition.

One can prove that $\W_{\hat{A}_{r-1}}$ is the semidirect product of the finite Weyl group~$\W_{A_{r-1}}$, generated by the 
reflections~$\sigma_i$ for $i=1, \dots, r-1$, and of the abelian group $\left< t_{\bar{\alpha}_1^{\vee}}, \dots,t_{\bar{\alpha}_{r-1}^{\vee}} \right>$ of translations along 
the coroot lattice of~$\mathfrak{sl}_r$.

\subsection{The extended affine Weyl group}

See \cite{Lu}, \cite{DG} or \cite{DDF} for more details about the extended affine Weyl group. For example in \cite{DG}, 
there is a definition of this group different from the following, it is described as a subgroup of permutations of~$\Z$.

Let us now consider the translations~$t_{\varepsilon_i}$ along the simple $\mathfrak{gl}_r$--roots~$\varepsilon_i$,  $i = 1,\dots,r$. Their action on a 
level zero~$\hat{\mathfrak{gl}}_r$--weight~$\kappa$ is given by 
\begin{equation}
t_{\varepsilon_i} \left( \bar{\kappa}, 0, m \right) = \left( \bar{\kappa}, 0, m - \kappa_i \right).
\end{equation}

The {\em extended affine Weyl group} $\hat{\W}_{\hat{A}_{r-1}}$ is defined as the semidirect product of the finite Weyl group~$\W_{A_{r-1}}$ and the abelian 
group~$\left< t_{\varepsilon_1}, \dots,t_{\varepsilon_{r}} \right>$ of translations along the coroot lattice of~$\mathfrak{gl}_r$. It contains the affine Weyl 
group~$\W_{\hat{A}_{r-1}}$ as a normal subgroup. 

The group~$\hat{\W}_{\hat{A}_{r-1}}$ is generated by $\sigma_1, \dots, \sigma_{r-1}$ and~$t_{\varepsilon_1}, \dots, t_{\varepsilon_r}$, which satisfy the following relations:
\begin{equation}
 \sigma_i t_{\varepsilon_j} \sigma_i = t_{\sigma_i(\varepsilon_j)} \quad \mbox{for} \  i=1,\dots, r-1 \ \mbox{and} \ j=1,\dots, r.
\end{equation}
Hence the set of generators is not minimal, e.g. one can obtain any~$t_{\varepsilon_j}$ for $j=2, \dots,r$ by conjugating~$t_{\varepsilon_1}$ by certain reflections.

There is another presentation of~$\hat{\W}_{\hat{A}_{r-1}}$, which is important for this paper. It involves the following specific element 
$$\rho =  t_{\varepsilon_1} \sigma_1 \dots \sigma_{r-1},$$
which acts on a level zero~$\hat{\mathfrak{gl}}_r$--weight~$\kappa = \left(\bar{\kappa}, 0,m \right)$ by 
\begin{equation}
\rho \left( \bar{\kappa}, 0, m \right) = \left(\left(\kappa_r, \kappa_1, \dots, \kappa_{r-1}\right), 0, m - \kappa_r \right).
\end{equation}
The action of its inverse $\rho^{-1} = \sigma^{-1}_{r-1} \dots \sigma^{-1}_{1} t_{-\varepsilon_1} $ is given by
\begin{equation}
\rho^{-1} \left( \bar{\kappa}, 0, m \right) = \left(\left(\kappa_2, \dots, \kappa_r, \kappa_{1}\right), 0, m + \kappa_1 \right).
\end{equation}
One then sees that~$\hat{\W}_{\hat{A}_{r-1}}$ is generated by 
$$\sigma_1, \dots, \sigma_{r},\rho,$$ 
subject to the relations
\begin{align}
\sigma_i^2 & =  1 & & \mbox{for} \ i=1,\dots, r \label{W1}\\
\sigma_i \sigma_ j & =  \sigma_ j \sigma_i & & \mbox{for distant} \ i,j=1,\dots, r \label{W2}\\
\sigma_i \sigma_ {i+1} \sigma_i& =  \sigma_ {i+1} \sigma_i \sigma_ {i+1}& & \mbox{for} \ i=1,\dots, r   \label{W3}\\
\rho \sigma_i \rho^{-1} & =  \sigma_{i+1} & & \mbox{for} \ i=1,\dots, r \label{W4}
\end{align}
where the indices have to be understood modulo $r$, as before. We say that $i$ and $j$ are {\em distant} if 
$j\not\equiv i\pm 1 \mod r$. Using this set of generators, 
any element~$w \in \hat{\W}_{\hat{A}_{r-1}}$ can be written in the following way
\begin{equation}\label{decw}
w = \rho^k w' = \rho^k \sigma_{i_1} \cdots \sigma_{i_l}
\end{equation}
where~$k \in \Z$ is unique and~$\sigma_{i_1} \cdots \sigma_{i_l}$ is a 
reduced expression of the element $w' \in \W_{\hat{A}_{r-1}}$.

Note that the conventions here are opposite to the ones chosen by Doty and Green \cite{DG}.

\subsection{The extended affine braid group and Hecke algebra}

One can form the {\em extended affine braid group}~$\hat{\B}_{\hat{A}_{r-1}}$ associated to~$\hat{\W}_{\hat{A}_{r-1}}$. It admits the same presentation 
as~$\hat{\W}_{\hat{A}_{r-1}}$ except that one omits the involutivity relations of the generators~$\sigma_i$ for~$i=1, \dots,r$. 

One can also define the {\em extended affine Hecke algebra}~$\hat{\He}_{\hat{A}_{r-1}}$, which is the quotient of the $\Q(q)$-group 
algebra of $\hat{\B}_{\hat{A}_{r-1}}$ by the relations
$$T_{\sigma_{i}}^{2} = (q^2 -1)T_{\sigma_i} +q^2 \quad \text{for all } i=1, \dots, r $$
with $q$ being a formal parameter. For more details about this algebra, 
see~\cite{Op1, Op2,OpD,OpS,DDF}.

A $\Q(q)$--basis of $\hat{\He}_{\hat{A}_{r-1}}$ is given 
by the set
$$\lbrace T_w, w \in  \hat{\W}_{\hat{A}_{r-1}} \rbrace$$ 
where 
$$T_w =  T_{\rho}^k T_{w'} = T_{\rho}^k T_{\sigma_{i_1}} \cdots T_{\sigma_{i_l}}$$ 
with~$w$, $k$, $w'$ and $\sigma_{i_j}$ as in \eqref{decw}. 
See \cite{Gr} and \cite{DG}.

The above shows that $\hat{\He}_{\hat{A}_{r-1}}$ is generated by 
$$\lbrace T_{\rho},T_{\rho^{-1}}, T_{\sigma_i}, i=1,\ldots,r\rbrace$$ 
as an algebra. An alternative set of algebraic generators 
of~$\hat{\He}_{\hat{A}_{r-1}}$ is given by 
\begin{equation}
 \lbrace  T_{\rho} , T_{\rho}^{-1}, b_i, i=1, \dots r \rbrace
\end{equation}
where the $b_i := C_{\sigma_i}' = q^{-1}(1+T_{\sigma_i})$ are the 
Kazhdan-Lusztig generators. 
The relations satisfied by these generators are the following:
\begin{align}
 b_i^2 & =  (q + q^{-1})b_i & & \mbox{for} \ i=1,\dots, r \label{H1}\\
b_i b_ j & =  b_ j b_i & & \mbox{for distant} \ i,j=1,\dots, r \label{H2}\\
b_i b_ {i+1} b_i + b_ {i+1}& =  b_ {i+1} b_i b_ {i+1} +  b_i& & \mbox{for} \ i=1,\dots, r   \label{H3}\\
 T_{\rho} b_i T_{\rho}^{-1} & = b_{i+1} & & \mbox{for} \ i=1,\dots, r \label{H4}.
\end{align}

In \cite{GH}, Grojnowski and Haiman show that $\hat{\He}_{\hat{A}_{r-1}}$ has 
the following Kazhdan-Lusztig basis 
\begin{equation}
\label{eq:GH}
\{ T_{\rho}^{k} C_w', k \in \Z \ \mbox{and} \ w \in \W_{\hat{A}_{r-1}}\},
\end{equation}
with the usual positive integrality property.

\section{A categorification of the extended affine Hecke algebra}
\label{sec:catHeck}
\begin{nota}
Let $\mathcal{C}$ be a $\Q$--linear $\mathbb{Z}$--graded additive category (resp. $2$--category) with translation (see Section 5.1 in~\cite{Lau} 
for the technical definitions). 
In all examples in this paper, the vector space of 
morphisms (resp. $2$--morphisms) of any fixed degree is finite-dimensional.  

The Karoubi envelope of $\mathcal{C}$ is denoted by 
$\mathrm{Kar}{\mathcal{C}}$.

By $\mathcal{C}^*$ we denote the category (resp. $2$--category) with the 
same objects (resp. same objects and $1$--morphisms) as $\mathcal{C}$, but 
whose hom-spaces (resp. $2$--hom-spaces) are defined by 
\begin{equation*}
\mathcal{C}^*(x,y)=\oplus_{t\in\Z}\Hom_{\mathcal{C}}(x\{t\},y).
\end{equation*}

A degree preserving functor $\F : \mathcal{C} \rightarrow \mathcal{D}$ between two such categories $\mathcal{C}$ and $\mathcal{D}$ lifts to a functor 
between the enriched categories $\mathcal{C}^* \rightarrow \mathcal{D}^*$ and 
to a functor between the 
Karoubi envelopes $\mathrm{Kar}{\mathcal{C}} \rightarrow 
\mathrm{Kar}{\mathcal{D}}$,  
which are both also denoted $\F$.
\end{nota}

\subsection{An extension of Soergel's categorification}

\subsubsection{Action on polynomial rings}\label{action}

Consider the polynomial ring~$R = \Q[y][x_1, \dots, x_r]$. The extended affine Weyl group~$\hat{\W}_{\hat{A}_{r-1}}$ acts faithfully on~$R$ as follows:
\begin{eqnarray*}
\rho (x_i) & = &
\begin{cases}
 x_{i+1} & \mbox{for} \ i=1,\dots, r-1 \\ 
 x_1 - y & \text{for} \ i = r
\end{cases}
\\
\rho^{-1} (x_i) & = &
\begin{cases}
 x_{i-1} & \mbox{for} \ i=2,\dots, r \\ 
 x_r + y & \text{for} \ i = 1
\end{cases}
\\
t_{\varepsilon_j} (x_i) & = &
\begin{cases}
x_j - y & \text{for} \ i = j \\
x_{i} & \mbox{otherwise}  
\end{cases}
\\
\sigma_j (x_i) & = &
\begin{cases}
x_{j+1} & \text{for} \ i = j \\
x_{j} & \text{for} \ i = j+1 \\
x_{i} & \mbox{otherwise}  
\end{cases} 
\qquad \mbox{for} \ j=1,\dots, r-1
\\
\sigma_r (x_i) & = &
\begin{cases}
x_{r}+y & \text{for} \ i = 1 \\
x_{1}-y & \text{for} \ i = r \\
x_{i} & \mbox{otherwise} 
\end{cases}
\\
\end{eqnarray*}

\begin{rem}\label{hart}
The action above naturally extends the action of the (non-extended) affine 
Weyl group $\W_{\hat{A}_{r-1}}$. We are working here with Soergel's original reflection faithful realization of the affine type $A$ Coxeter system $\W_{\hat{A}_{r-1}}$ considered in \cite{S3} and \cite{Har}. Indeed, when generalized to the extended affine Weyl group~$\hat{\W}_{\hat{A}_{r-1}}$, this representation remains faithful which explains why we choose to use this precise realization to achieve a categorification of the extended affine Hecke algebra $\hat{\He}_{\hat{A}_{r-1}}$ in the present paper.

Let us set $X_i = x_{i+1} - x_i$ for $i = 1, \dots,  r-1$ and $X_r = x_1 - x_r - y$.
\end{rem}

\subsubsection{Extended Soergel bimodules}\label{bim}

For any~$i=1, \dots ,r$, we define the~$R$--bimodule 
\begin{equation}
B_i = R \ot_{R^{\sigma_i}} R 
\end{equation}
where~$R^{\sigma_i}$ is the subalgebra of elements of~$R$ fixed by the reflection~$\sigma_i \in \hat{\W}_{\hat{A}_{r-1}}$:
\begin{eqnarray*}
 R^{\sigma_i} & = & \Q[y][x_1, \dots, x_i +x_{i+1},x_i x_{i+1}, \dots, x_r] \quad \mbox{for} \ i=1, \dots, r-1\\
 R^{\sigma_r} & = & \Q[y]\left[x_2, \dots, x_{r-1}, x_r +x_{1},\left(x_r + y/2 \right) \left(x_{1} - y/2\right)\right] 
\end{eqnarray*}
We also define the {\em twisted~$R$--bimodule}~$B_{\rho}$ (resp.~$B_{\rho^{-1}}$), which coincides with~$R$ as a left~$R$--module but is twisted by $\rho$ 
(resp. by $\rho^{-1}$) as a right $R$--module, i.e. any~$a \in R$ acts on~$B_{\rho}$ on the right by multiplication by~$\rho(a)$ (resp.~$\rho^{-1}(a)$). 

Arkhipov introduced twisted bimodules associated to simple reflections 
and the twisted functors obtained by tensoring with them. These were 
used by several people in their work on category $\mathcal{O}$; 
for some history and references see~\cite{Maz}. Here we (only) consider 
twisted bimodules associated to powers of $\rho$. 

We introduce a grading on $R$, $R^{\sigma_i}$, $B_i$ and~$B_{\rho^{\pm 1}}$ by setting
$$\deg(y) = \deg (x_k)=2$$ 
for all~$k=1,\dots,r$. Curly brackets will indicate a shift of the grading: 
if $M=\underset{i\in \Z}{\bigoplus}M_i$ is a $\Z$-graded bimodule and~$p$ an integer, then 
the $\Z$--graded bimodule~$M\{p\}$ is defined by $M\{p\}_i = M_{i-p}$ for all~$i \in \Z$.

Form now the category monoidally generated by the graded $R$--bimodules defined above. Then allow direct sums and grading shifts of these objects and 
consider only morphisms which are degree-preserving morphisms of $R$--bimodules, and denote this category~$\ebim_{\hat{A}_{r-1}}$. Its Karoubi envelope~$\kebim_{\hat{A}_{r-1}}$ is 
a $\Q$--linear graded additive monoidal category with translation, which we call {\em the category of extended Soergel bimodules of type}~$\hat{A}_{r-1}$. 

As mentioned in Remark~\ref{hart}, we are precisely using Soergel's original realization, so the bimodules $B_i$ considered here are the ones constructed by H\"{a}rterich \cite{Har} and Soergel \cite{S3}. Therefore Soergel's category~$\kbim_{\hat{A}_{r-1}}$ of affine type~$A$ is equivalent to the full subcategory of~$\kebim_{\hat{A}_{r-1}}$ generated by the $B_i$, for $i=1,\ldots,r$. Let us recall H\"{a}rterich's categorification result~\cite{Har}:
\begin{thm}[H\"{a}rterich]\label{thm:hart}
We have 
$$\He_{\hat{A}_{r-1}}\cong K_0^{\Q(q)}(\kbim_{\hat{A}_{r-1}})$$
where 
$$K_0^{\Q(q)}(\kbim_{\hat{A}_{r-1}}):=K_0(\kbim_{\hat{A}_{r-1}})\otimes_{\Z[q,q^{-1}]}\Q(q).$$

For each $w\in \W_{\hat{A}_{r-1}}$, there exists a unique 
indecomposable bimodule $B_w$ in $\kbim_{\hat{A}_{r-1}}$. Conversely, 
any indecomposable bimodule in this category is isomorphic 
to $B_w\{t\}$, for a certain $w\in \W_{\hat{A}_{r-1}}$ and a certain 
grading shift $t\in \Z$. 

Under the isomorphism above the 
Kazhdan-Lusztig basis element $C'_w$ is mapped to $[B_w\{-1\}]$, 
for any $w\in \W_{\hat{A}_{r-1}}$. 
\end{thm} 

In the rest of the paper, we keep the convention that subscripts are considered to be modulo $r$, e.g. the bimodule~$B_{r+1}$ is by definition equal to $B_{1}$. 
We will also use the notation~$B_{\rho}^{\ot k}$ for any~$k \in \Z$, where this bimodule is defined to be the tensor product of~$|k|$ copies of~$B_{\rho}$ if $k \geq 0$ 
and of~$B_{\rho^{-1}}$ if $k \leq 0$. In both cases $B_{\rho}^{\ot k}$ is isomorphic to~$B_{\rho^k}$.

\subsubsection{Categorification of~$\hat{\He}_{\hat{A}_{r-1}}$}

\begin{lem}\label{isobim}
For any $i=1,\ldots, r$, there exists an $R$--bimodule isomorphism
\begin{equation}\label{tiso1}
 B_{\rho} \ot_R B_i \cong B_{i+1} \ot_R B_{\rho}. 
\end{equation}
Applying these isomorphisms $r$ times gives an isomorphism
\begin{equation}\label{tiso2}
 B_{\rho}^{\ot r} \ot_R B_i \cong B_{i} \ot_R B_{\rho}^{\ot r},
\end{equation}
for any $i=1,\ldots,r$.
\end{lem}

\begin{proof} \hfill \\
\noindent $\bullet$ \textit{Isomorphism~\eqref{tiso1}}

First note that there exist natural isomorphisms of~$R$--bimodules:
$$ B_{\rho} \ot_R B_i \cong B_{\rho} \ot_{R^{\sigma_i}} R \quad \text{and} \quad  B_{i+1} \ot_R B_{\rho} \cong R \ot_{R^{\sigma_{i+1}}} B_{\rho} .$$
Define the isomorphism of~$R$--bimodules~$\psi : B_{\rho} \ot_{R^{\sigma_i}} R \rightarrow R \ot_{R^{\sigma_{i+1}}} B_{\rho}$ by 
$$\psi (a \ot b) = a \ot \rho(b).$$
This isomorphism is well--defined, because $\rho$ defines an isomorphism between $R^{\sigma_i}$ and $R^{\sigma_{i+1}}$. 

\noindent $\bullet$ \textit{Isomorphism~\eqref{tiso2}}

Note that $B_{\rho}^{\ot r} \cong B_{\rho^r}$ and that~$\rho^r$ leaves the ring~$R^{\sigma_i}$ invariant.
\end{proof}

\begin{thm}\label{catexthe}
The category~$\kebim_{\hat{A}_{r-1}}$ categorifies the extended affine Hecke algebra, i.e. 
$$\hat{\He}_{\hat{A}_{r-1}} \cong K_0^{\Q(q)}(\kebim_{\hat{A}_{r-1}}),$$
where 
$$K_0^{\Q(q)}(\kebim_{\hat{A}_{r-1}}):=K_0(\kebim_{\hat{A}_{r-1}})\otimes_{\Z[q,q^{-1}]}\Q(q).$$
Under this isomorphism, the indecomposables in~$\kebim_{\hat{A}_{r-1}}$ correspond exactly to the Kazhdan-Lusztig basis of~$\hat{\He}_{\hat{A}_{r-1}}$. In particular, 
$[B_i\{-1\}]$ corresponds to~$b_i$ and~$[B_{\rho^{\pm 1}}]$ corresponds to~$T_{\rho}^{\pm 1}$.  
\end{thm}

\begin{proof}
Recall that 
$$\He_{\hat{A}_{r-1}}\cong K_0^{\Q(q)}(\kbim_{\hat{A}_{r-1}}),$$ 
by Theorem~\ref{thm:hart}. 
The indecomposables in~$\kbim_{\hat{A}_{r-1}}$, which are 
denoted by~$B_w\{-1\}$ for~$w\in\ \W_{\hat{A}_{r-1}}$, 
correspond exactly to the Kazhdan-Lusztig basis 
elements~$\{C_w', w \in \W_{\hat{A}_{r-1}}\}$ of~$\He_{\hat{A}_{r-1}}$. 
Moreover, $B_w$ appears as a direct summand of the tensor 
product~$B_{i_1} \ot_R \dots \ot_R B_{i_k}$ where~$\sigma_{i_1} \dots \sigma_{i_k}$ 
is a reduced expression 
of~$w$.

We define the homomorphism of algebras 
$$\hat{\He}_{\hat{A}_{r-1}} \rightarrow K_0^{\Q(q)}(\kebim_{\hat{A}_{r-1}})$$
by
$$b_i\mapsto[B_i\{-1\}]\quad\text{and}\quad T_{\rho}^{\pm 1}\mapsto [B_{\rho^{\pm 1}}].$$ 
The homomorphism is well-defined, as follows from the following isomorphisms in~$\ebim_{\hat{A}_{r-1}}$:
\begin{align}
B_i \otimes_R B_i & \cong   B_i \oplus B_i\{2\}   \label{S1}\\
B_i \otimes_R B_j & \cong B_j \otimes_R B_i \ \mbox{for distant}\; i,j\label{S2}\\
B_i \otimes_R B_{i+1} \otimes_R B_i\oplus B_{i+1} \{2\}& \cong B_{i+1} \otimes_R B_{i} \otimes_R B_{i+1} \oplus B_i \{2\}  \label{S3}\\
B_{\rho} \ot_R B_i &\cong B_{i+1} \ot_R B_{\rho} \label{S4}, 
\end{align}
for~$i,j=1,\dots, r$. 

Note that any tensor product of~$B_{\rho^{\pm 1}}$'s and $B_i$'s can be rewritten in the following way $$B_{\rho}^{\ot k} \ot_R B_{i_1} \ot_R \dots \ot_R B_{i_l},$$ 
by sliding all the~$B_{\rho^{\pm 1}}$'s to the left using the isomophism~\eqref{S4}.  

Let us now look at the indecomposables of the category~$\kebim_{\hat{A}_{r-1}}$. First observe that if the bimodule~$M$ is indecomposable in~$\kebim_{\hat{A}_{r-1}}$ then, for any~$k \in \Z$, 
the tensor product~$B_{\rho}^{\ot k} \ot_R M$ is indecomposable as well. Indeed assume that 
$$B_{\rho}^{\ot k} \ot_R M \cong P \oplus Q,$$ then tensoring on the left by~$B_{\rho}^{\ot -k}$ gives 
$$M \cong B_{\rho}^{\ot -k} \ot_R P \oplus B_{\rho}^{\ot -k} \ot_R Q,$$ 
which contradicts the fact that~$M$ is supposed to be indecomposable.

So let~$M$ be an indecomposable of~$\kebim_{\hat{A}_{r-1}}$. It is a direct summand of some tensor 
product~$B_{\rho}^{\ot k} \ot_R B_{i_1} \ot_R \dots \ot_R B_{i_l}$. Then the indecomposable bimodule~$B_{\rho}^{\ot -k} \ot_R M$ is a direct summand 
of~$B_{i_1} \ot_R \dots \ot_R B_{i_l}$. The latter tensor product belongs to the subcategory~$\kbim_{\hat{A}_{r-1}}$ of~$\kebim_{\hat{A}_{r-1}}$. Thus~$B_{\rho}^{\ot -k} \ot_R M$ is of the form~$B_w$ 
for some~$w\in\ \W_{\hat{A}_{r-1}}$. We can conclude that the indecomposables of the category~$\kebim_{\hat{A}_{r-1}}$ are all of the form 
$$ B_{\rho}^{\ot k} \ot_R B_w \quad \mbox{for} \ k \in \Z \ \mbox{and} \ w\in\ \W_{\hat{A}_{r-1}}.$$
Their Grothendieck classes correspond bijectively to the elements of 
$$\{T_{\rho}^{ k} C_w', k \in \Z \ \mbox{and} \ w \in \W_{\hat{A}_{r-1}}\},$$
which is precisely the Kazhdan-Lusztig basis of~$\hat{\He}_{\hat{A}_{r-1}}$ 
in~\eqref{eq:GH}. 
\end{proof}

\begin{rem}\label{homsp}
In the category $\ebim_{\hat{A}_{r-1}}$, the $B_i$ are self-adjoint and $B_{\rho}$ and $B_{\rho}^{-1}$ form a biadjoint pair. Therefore, there exist isomorphisms 
\begin{eqnarray*}
\Hom_{\ebim_{\hat{A}_{r-1}}}(B_{\rho}^{\ot k} \ot_R B_{i_1} \ot_R \dots \ot_R B_{i_m}, B_{\rho}^{\ot l} \ot_R B_{j_1} \ot_R \dots \ot_R B_{j_n})&\cong\\ 
\Hom_{\ebim_{\hat{A}_{r-1}}} (R, B_{\rho}^{\ot l-k} \ot_R B_{j_1} \ot_R \dots \ot_R B_{j_n} \ot_R B_{i_m} \ot_R \dots \ot_R B_{i_1}).
\end{eqnarray*}
The latter hom-space is equal to zero except when~$k=l$, in which case it is isomorphic to the corresponding hom-space in~$\bim_{\hat{A}_{r-1}}$. 

Indeed any morphism 
$$f \in \Hom_{\ebim_{\hat{A}_{r-1}}} ( R, B_{\rho}^{\ot l-k} \ot_R B_{j_1} \ot_R \dots \ot_R B_{j_n} \ot_R B_{i_m} \ot_R \dots \ot_R B_{i_1})$$ is completely 
determined by the image~$p$ of~$1$. 
Since~$f$ is a morphism of $R$--bimodules, we have $ pa=ap$ for any~$a \in R$. In particular, for~$a = \sum_{i=1}^{r} x_i$ we have 
$$p \left( \sum_{i=1}^{r} x_i \right) = \left( \sum_{i=1}^{r} x_i \right) p.$$
Since~$ \sum_{i=1}^{r} x_i $ is invariant under all the reflections~$\sigma_j$ for~$j=1, \dots, r$, we also have
\begin{eqnarray*}
p \left( \sum_{i=1}^{r} x_i \right)& = & \rho^{l-k}\left( \sum_{i=1}^{r} x_i \right) p \\
& = & \left( \sum_{i=1}^{r} x_i - (l-k)y\right) p.
\end{eqnarray*}
This implies that~$p$, and therefore the morphism~$f$, has to be zero unless~$k=l$.
\end{rem}

\subsubsection{Representation of $\hat{\B}_{\hat{A}_{r-1}}$ in $\mathcal{K}(\kebim_{\hat{A}_{r-1}})$}

Let $\mathcal{K}(\kebim_{\hat{A}_{r-1}})$ be the homotopy category of bounded complexes in $\kebim_{\hat{A}_{r-1}}$.

Rouquier~\cite{R} obtained a representation of~$\B_{\hat{A}_{r-1}}$ in $\mathcal{K}(\kbim_{\hat{A}_{r-1}})$. 
We generalize his result in the extended affine setting. 
\begin{rem}
In what follows, we use by matter of convenience the notation $X_i=x_{i+1}-x_{i}$, for $i=1,\ldots, r-1$, and $X_r=x_1-x_r-y$, 
as defined in Remark~\ref{hart}. 
\end{rem}

To each braid generator $\sigma_i \in \hat{\B}_{\hat{A}_{r-1}}$ we assign 
the cochain complex~$F(\sigma_i)$ of graded $R$--bimodules
\begin{equation}\label{CX sigma}
F(\sigma_i):  0  \longrightarrow  R\{2\}  \xrightarrow{\rb_i}  B_i \longrightarrow  0
\end{equation}
where~$B_i$ sits in cohomological degree 0, and the $R$--bimodule 
morphism~$\rb_i$ maps~$1$ to $\frac{1}{2} \left( X_i\ot 1 + 1 \ot X_i \right)$, for $i=1,\ldots,r$. 

To $\sigma_i^{-1}$ we assign 
the cochain complex~$F(\sigma_i^{-1})$ of graded~$R$--bimodules
\begin{equation}\label{CX sigma-1}
F(\sigma_i^{-1}) :  0  \longrightarrow  B_i\{-2\}  \xrightarrow{\br_i} R\{-2\}  \longrightarrow  0
\end{equation}
where~$B_i\{-2\}$ sits in cohomological degree 0 and the $R$--bimodule 
morphism~$\br_i$ is just the multiplication. 

To~$\rho$ (resp. $\rho^{-1}$) we assign the cochain complex of 
graded $R$--bimodules
\begin{equation}\label{CX rho}
F(\rho)  :  0  \longrightarrow  B_{\rho}  \longrightarrow  0
\end{equation}
\begin{equation}\label{CX rho-1}
\left( \text{resp.} \quad F(\rho^{-1})  :  0  \longrightarrow  B_{\rho^{-1}}  \longrightarrow  0 \right)
\end{equation}
where~$B_{\rho}$ (resp. $B_{\rho^{-1}}$) sits in cohomological degree~$0$. 

To the unit element~$1 \in \hat{\B}_{\hat{A}_{r-1}}$ we assign the complex of 
graded $R$-bimodules
\begin{equation}\label{CX one}
F(1)  :  0  \longrightarrow  R  \longrightarrow  0
\end{equation}
where~$R$ sits in cohomological degree 0; the complex~$F(1)$ is a 
unit for the tensor product of complexes. Finally, to any extended 
affine braid word, we assign the tensor product over~$R$ of the 
complexes associated to the generators appearing in the expression of the 
word.

\begin{prop}
The above defines a categorical representation 
of~$\hat{\B}_{\hat{A}_{r-1}}$ in $\mathcal{K}(\kebim_{\hat{A}_{r-1}})$.
\end{prop}

\subsection{The diagrammatic version}

The category of Soergel bimodules of finite type~$A$ is described via planar diagrams 
by Elias and Khovanov in~\cite{EKh}. They associate planar diagrams to certain generating bimodule maps and give a complete set of relations on 
them. Elias and Williamson \cite{EW} worked out the generalization of the diagrammatic approach to Soergel bimodules 
for any Coxeter group which does not contain a standard parabolic subgroup isomorphic to $H_3$.

Our aim is to define and study an extension of Elias and 
Williamson's diagrammatic category for extended affine type~$A$~\cite{EW}. 
In affine type $A$ Elias and Williamson's diagrammatic category is a 
straightforward generalization of Elias and Khovanov's original category, 
which was defined for finite type $A$. In addition to Elias and Williamson's 
diagrams, we also have to introduce a new type of strand. These new strands 
are oriented and their endpoints are labeled $+$ or $-$ depending on 
orientations. The Karoubi envelope of the diagrammatic 
category~$\debim_{\hat{A}_{r-1}}$ obtained in this way is equivalent to 
the category of 
extended Soergel bimodules~$\kebim_{\hat{A}_{r-1}}$ of affine type~$A$, as we will 
show.

\subsubsection{Definition of~$\debim_{\hat{A}_{r-1}}$}\label{cat:debim}

First start with the category whose objects are graded finite sequences of integers belonging to~$\{1, \dots, r\}$ and the symbols~$+$ and~$-$. 
Graphically we represent these sequences by sequences of colored points (read from left to right) of the $x$--axis of the real plane~$\R^2$. 
The morphisms are then equivalence classes of $\Q$--linear combinations of 
graded planar diagrams in~$\R \times [0,1]$ (read from bottom to top) and composition is defined by 
vertically glueing the diagrams and rescaling the vertical coordinate. These morphisms are defined by generators and relations listed below. This category 
possesses a monoidal structure given by stacking sequences and diagrams next to each other.

Let $\debim_{\hat{A}_{r-1}}$ be the category containing all direct sums and grading shifts of these objects and let its morphisms be the degree-preserving diagrams. 
The diagrammatic extended Soergel category is by definition its Karoubi envelope~$\kdebim_{\hat{A}_{r-1}}$.

In the diagrams, the strands whose endpoints are $+$ or $-$--signs are oriented and the other strands are non-oriented. The non-oriented strands 
can be colored with integers belonging to $\{1,\ldots,r\}$. Two colors~$i$ and~$j$ are called {\em adjacent} 
(resp. {\em distant}) if~$i\equiv j\pm 1\mod r$ (resp. $i\not\equiv j\pm 1\mod r $). By convention, no label means that the equation holds for 
any color $i \in \{1, \dots,r\}$.

The morphisms of~$\debim_{\hat{A}_{r-1}}$ are built out of the following generating diagrams. The non-oriented diagrams are the affine analogues of Elias and Khovanov's diagrams, 
the ones involving oriented strands are new. 

\begin{itemize}
\item Generators involving only one color:
\begin{equation*}
\xymatrix@R=1.0mm{
&
\figins{-15}{0.5}{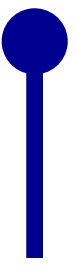}
&
\figins{-9}{0.5}{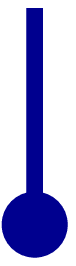}
&
\figins{-15}{0.55}{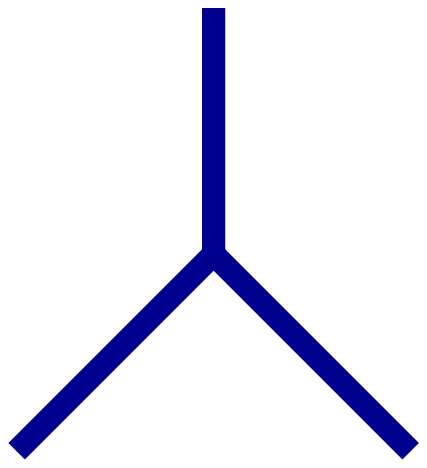}
&
\figins{-15}{0.55}{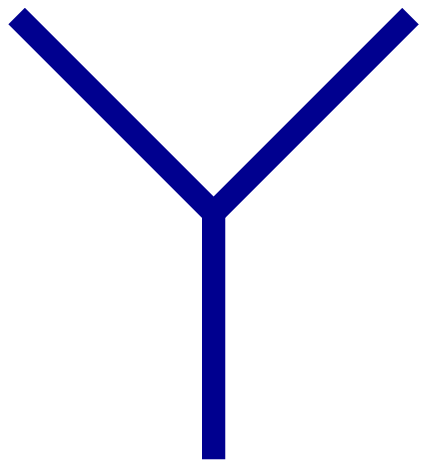}
\\
\text{Degree} & 1 & 1 & -1 & -1
\\
& \text{enddot}_i & \text{startdot}_i & \text{merge}_i & \text{split}_i
}
\end{equation*}

\medskip

It is useful to define the cap and cup as follows
\begin{equation*}
\figins{-17}{0.55}{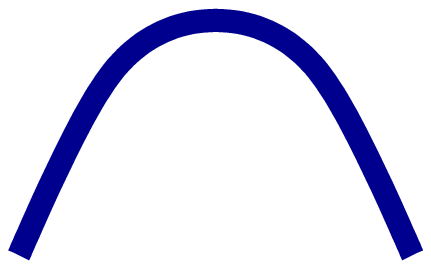}\
:= \
\figins{-17}{0.55}{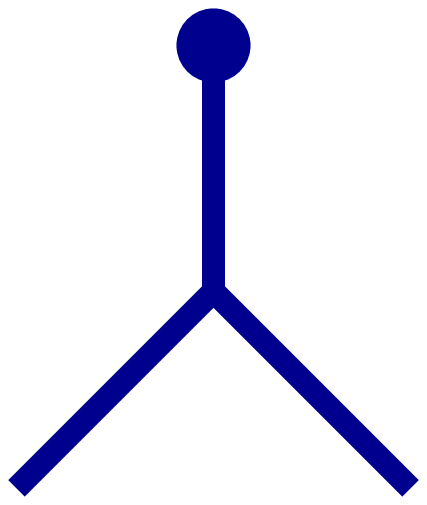}\
\mspace{50mu}
\figins{-17}{0.55}{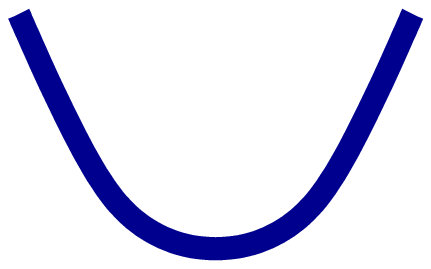} \
:= \
\figins{-17}{0.55}{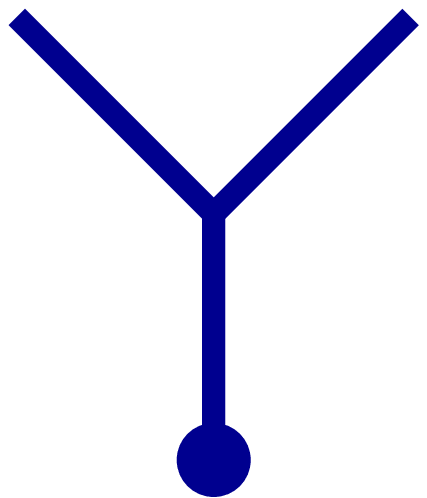}\
\end{equation*}

\medskip

\item Generators involving two colors:
\begin{itemize}
\item the $4$--valent vertex with distant colors, of degree~$0$, denoted~$4\mbox{vert}_{i,j}$
\smallskip
\begin{equation*}
\labellist
\tiny\hair 2pt
\pinlabel $i$ at  -4 -10
\pinlabel $j$ at 134 -10
\pinlabel $i$ at 134 140
\pinlabel $j$ at -4 140
\endlabellist
\figins{-15}{0.55}{4vert}\vspace{1.5ex}
\end{equation*}

\item and the $6$--valent vertices with adjacent colors~$i$ and~$j$, of degree~$0$, denoted~$6\mbox{vert}_{i,j}$ and~$6\mbox{vert}_{j,i}$
\smallskip
\begin{equation*}
\labellist
\tiny\hair 2pt
\pinlabel $i$   at  -4 -10 
\pinlabel $j$ at  66 -12
\pinlabel $i$ at 136 -10
\pinlabel $j$ at -4 140
\pinlabel $i$ at 66 140
\pinlabel $j$ at 136 140
\pinlabel $j$ at 250 -10
\pinlabel $i$ at 320 -10
\pinlabel $j$ at 390 -10
\pinlabel $i$ at 250 140
\pinlabel $j$ at 320 140
\pinlabel $i$ at 390 140
\endlabellist
\figins{-17}{0.55}{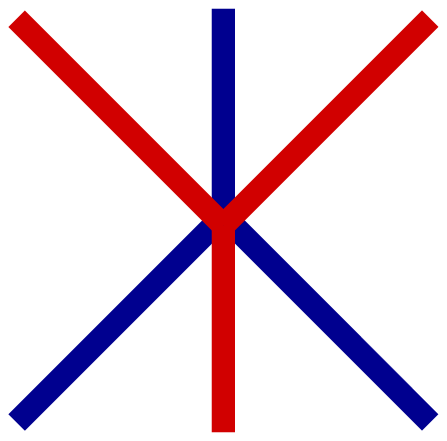}
\mspace{60mu}
\figins{-17}{0.55}{6vertu} 
\vspace{1.5ex}
\end{equation*}
\end{itemize}

\medskip

\item Generators involving only oriented strands, of degree~$0$:
\smallskip
\begin{equation*}
\labellist
\pinlabel $+$strand at 0 -40
\pinlabel $-$strand at 140 -40
\pinlabel $+$cap at 300 -40
\pinlabel $-$cap at 505 -40
\pinlabel $-$cup at 715 -40
\pinlabel $+$cup at 910 -40
\tiny\hair 2pt
\pinlabel $+$ at 12 120
\pinlabel $+$ at 755 120
\pinlabel $+$ at 855 120
\pinlabel $-$ at 140 120
\pinlabel $-$ at 660 120
\pinlabel $-$ at 950 120
\pinlabel $-$ at 140 -10
\pinlabel $+$ at 255 -10
\pinlabel $-$ at  460 -10
\pinlabel $+$ at  12 -10
\pinlabel $-$ at 350 -10
\pinlabel $+$ at  555 -10
%\pinlabel +strand at 0 -40
\endlabellist
\figins{-15}{0.55}{rho}
\mspace{60mu}
\figins{-15}{0.55}{rho-1}
\mspace{60mu}
\figins{-15}{0.55}{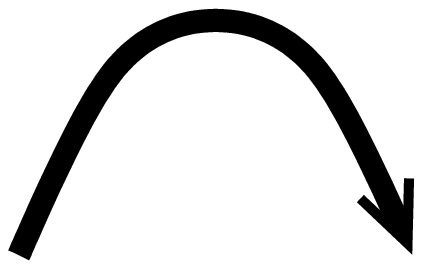}
\mspace{60mu}
\figins{-15}{0.55}{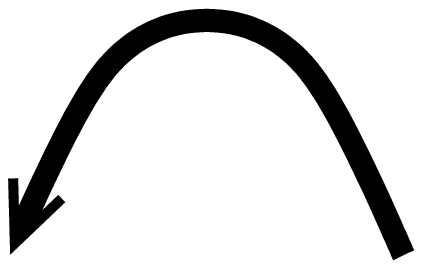}
\mspace{60mu}
\figins{-15}{0.55}{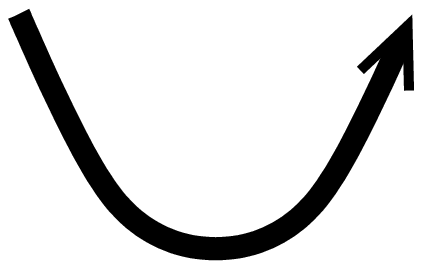}
\mspace{60mu}
\figins{-15}{0.55}{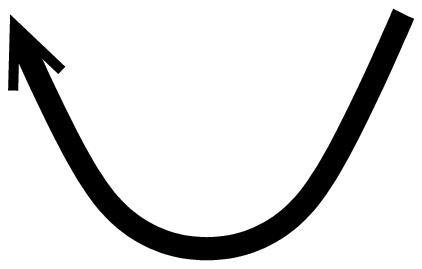} 
\end{equation*}
\bigskip
\bigskip

\item Generators involving oriented strands and adjacent colored strands. The mixed $4$--valent vertex of degree~$0$:
\smallskip
\begin{equation*}
\labellist
\pinlabel $4\mbox{vert}_{+,i}$ at 70 -40
\pinlabel $4\mbox{vert}_{i+1,+}$ at 325 -40
\pinlabel $4\mbox{vert}_{i,-}$ at 570 -40
\pinlabel $4\mbox{vert}_{-,i+1}$ at 825 -40
\tiny\hair 2pt
\pinlabel $i+1$ at -5 140
\pinlabel $i$ at 380 140
\pinlabel $i+1$ at 630 140
\pinlabel $i$ at 740 140
\pinlabel $+$ at 136 140
\pinlabel $+$ at 250 140
\pinlabel $-$ at 495 140
\pinlabel $-$ at 885 140
\pinlabel $i$ at 136 -10
\pinlabel $i+1$ at 250 -10
\pinlabel $i$ at  495 -10
\pinlabel $i+1$ at 885 -10
\pinlabel $+$ at -5 -10
\pinlabel $+$ at 380 -10
\pinlabel $-$ at  630 -10
\pinlabel $-$ at 740 -10
\endlabellist
\figins{-15}{0.55}{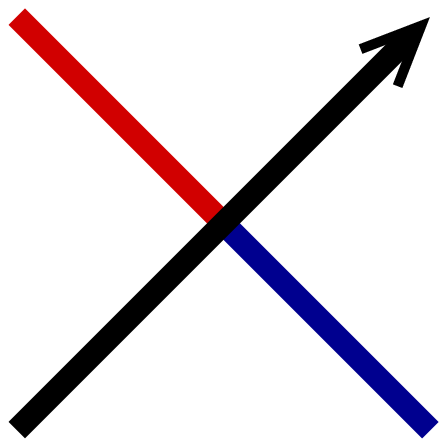}
\mspace{60mu}
\figins{-15}{0.55}{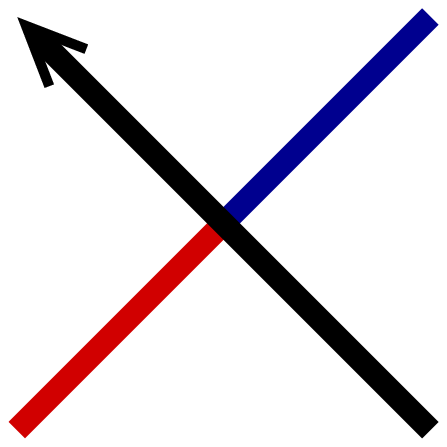}
\mspace{60mu}
\figins{-15}{0.55}{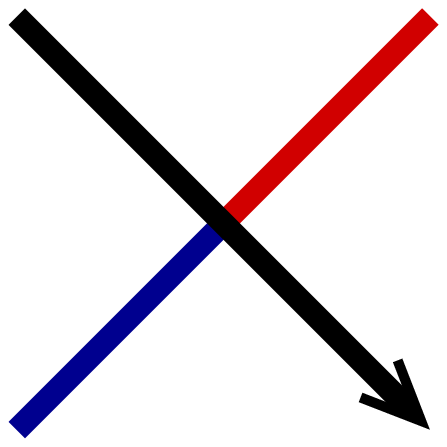}
\mspace{60mu}
\figins{-15}{0.55}{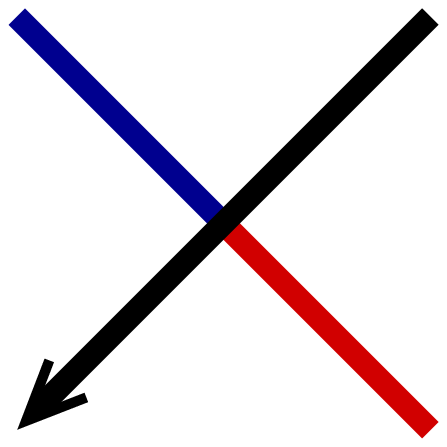}
\end{equation*}

\bigskip
\bigskip

\item Generators involving boxes, of degree~$2$, denoted~$\mbox{box}_i$
\begin{equation*}
{\bbox{i}}
\end{equation*}
for all $i=1,\dots, r$, and denoted~$\mbox{box}_y$
\begin{equation*}
{\bbox{y}}
\end{equation*}
\end{itemize}

\medskip

The generating diagrams are subject to the following relations. The relations involving only non-oriented diagrams are the obvious 
affine analogues of Elias and Khovanov's relations, the ones with diagrams involving oriented strands are new.  
\begin{itemize}
\item Isotopy relations:
\begin{equation}\label{eq:adj}
\figins{-17}{0.55}{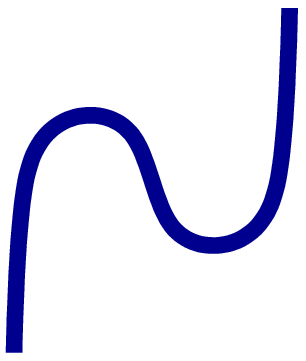}\
=\
\figins{-17}{0.55}{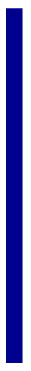}\
=\
\figins{-17}{0.55}{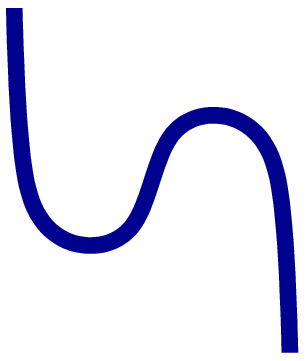}
\end{equation}

\begin{equation}\label{eq:curldot}
\figins{-17}{0.55}{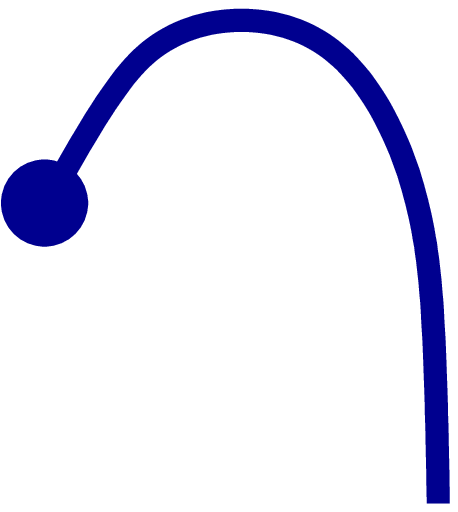}\
=\
\figins{-17}{0.55}{enddot.eps}\
=\
\figins{-17}{0.55}{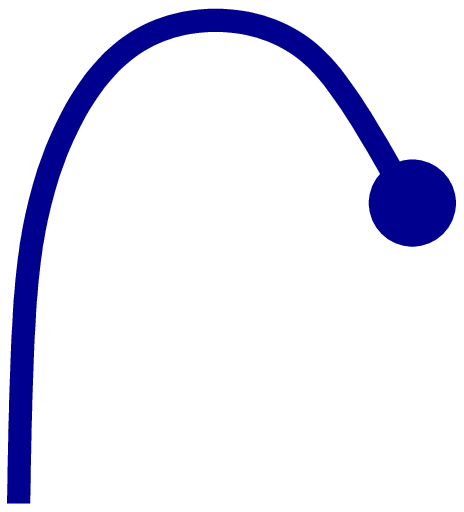}
\end{equation}

\begin{equation}\label{eq:v3rot}
\figins{-17}{0.55}{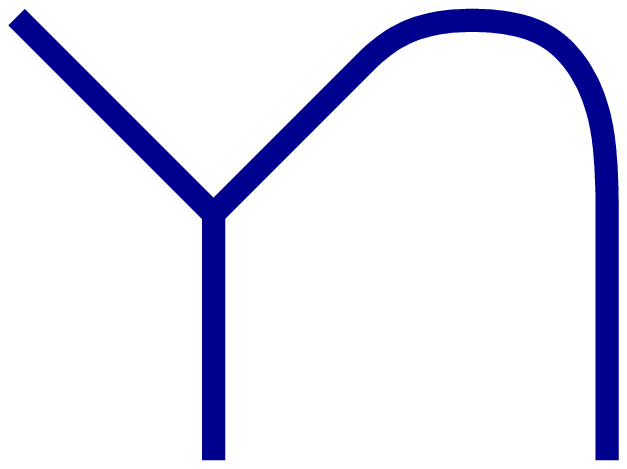}\
=\
\figins{-17}{0.55}{merge.eps}\
=\
\figins{-17}{0.55}{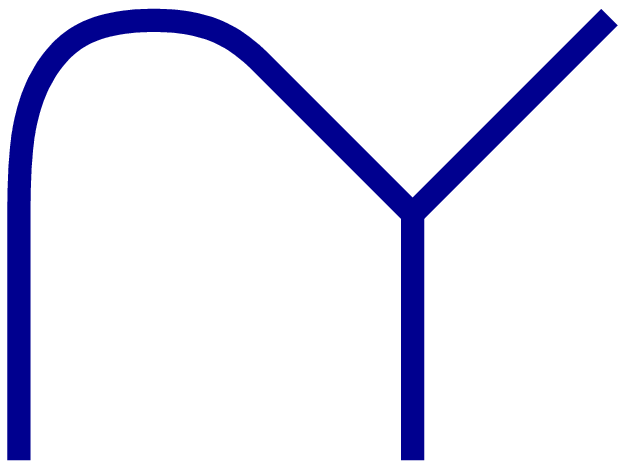}
\end{equation}

\begin{equation}\label{eq:v4rot}
\figins{-17}{0.55}{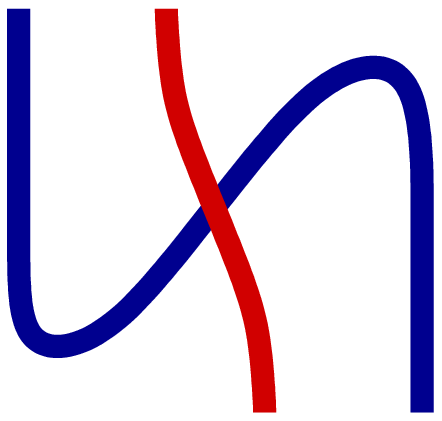}\
=\
\figins{-17}{0.55}{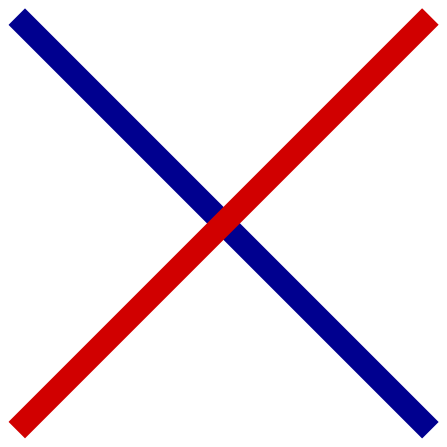}\
=\
\figins{-17}{0.55}{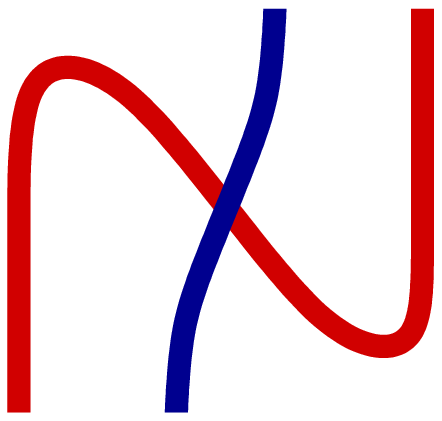}
\end{equation}

\begin{equation}\label{eq:v6rot}
\figins{-17}{0.55}{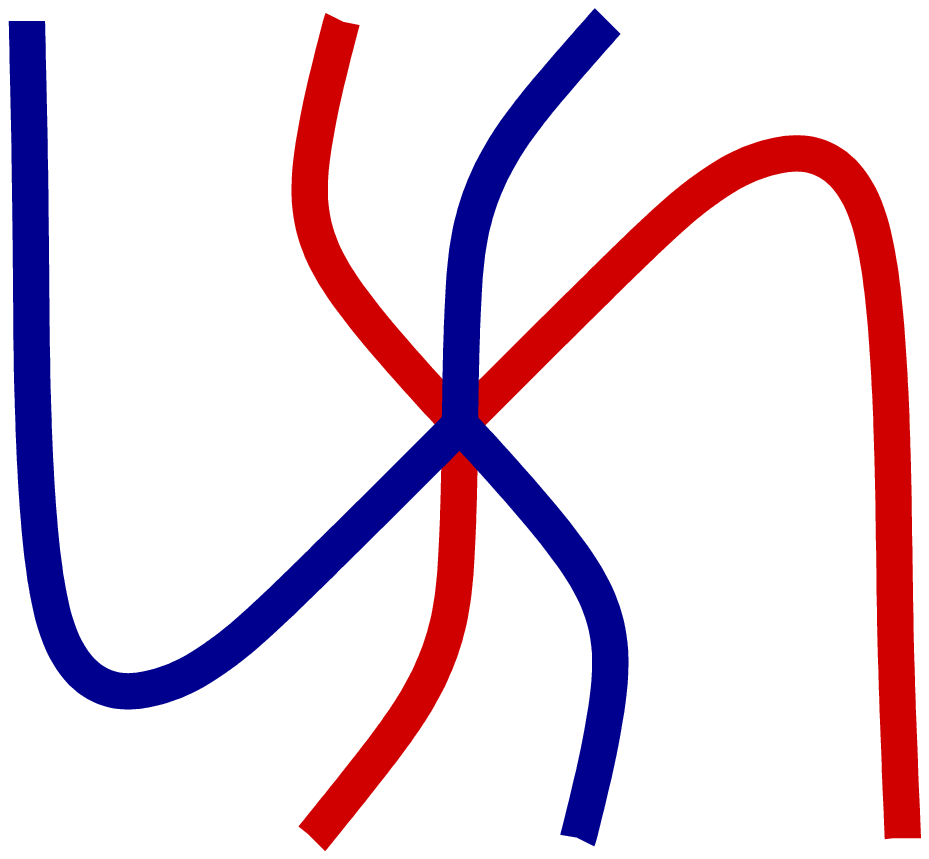}\
=\
\figins{-17}{0.55}{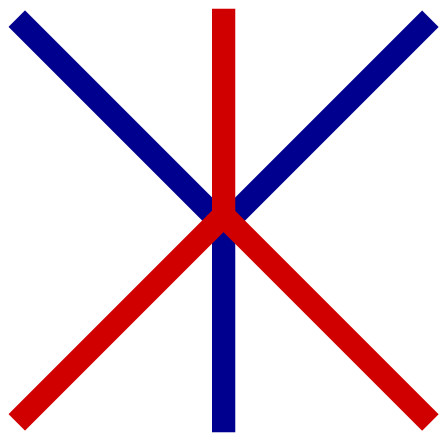}\
=\
\figins{-17}{0.55}{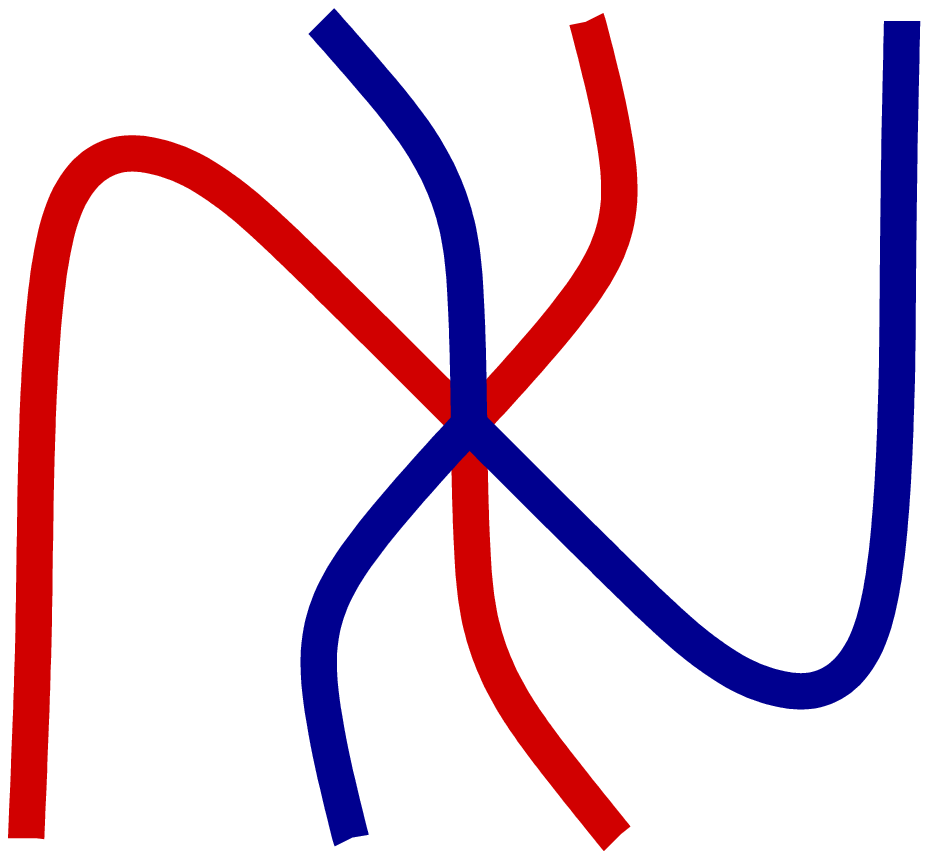}
\end{equation}

\begin{equation}\label{eq:adjmu}
\figins{-17}{0.55}{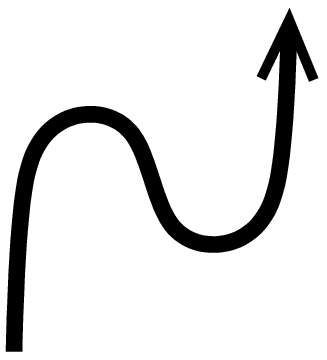}\
=\
\figins{-17}{0.55}{rho}\
=\
\figins{-17}{0.55}{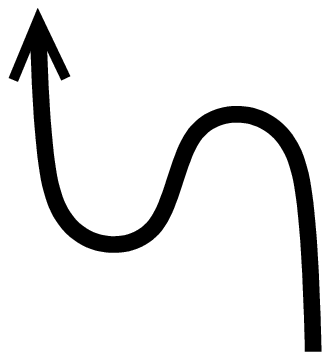}
\end{equation}

\begin{equation}\label{eq:adjmd}
\figins{-17}{0.55}{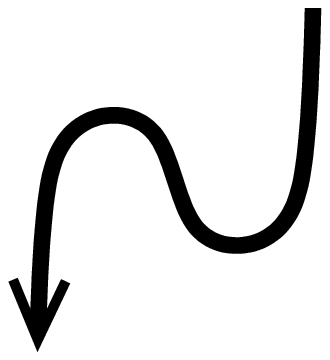}\
=\
\figins{-17}{0.55}{rho-1}\
=\
\figins{-17}{0.55}{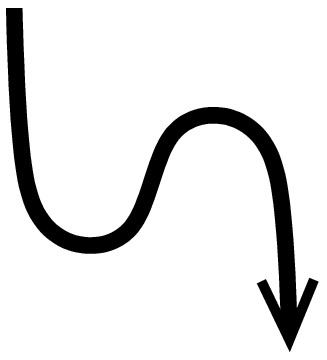}
\end{equation}

\begin{equation}\label{eq:v4mrotu}
\figins{-17}{0.55}{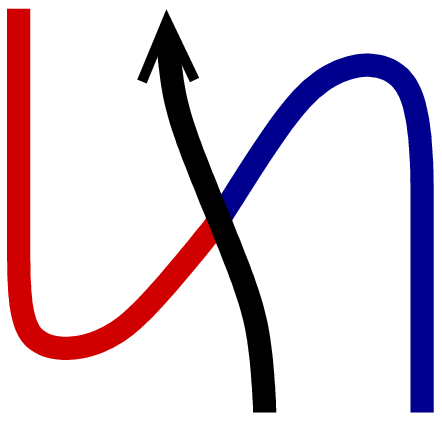}\
=\
\figins{-17}{0.55}{4mvert-ur}\
=\
\figins{-17}{0.55}{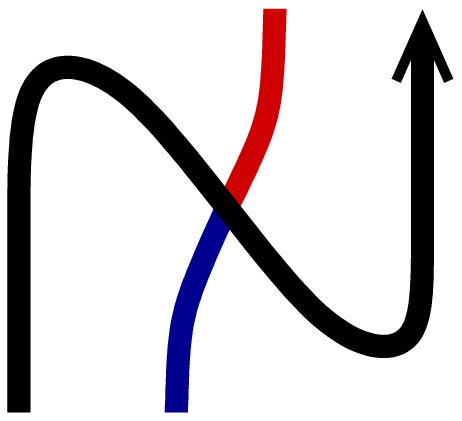}
\end{equation}

\begin{equation}\label{eq:v4mrotd}
\figins{-17}{0.55}{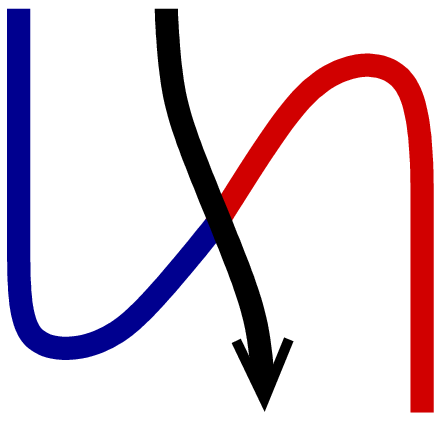}\
=\
\figins{-17}{0.55}{4mvert-dl}\
=\
\figins{-17}{0.55}{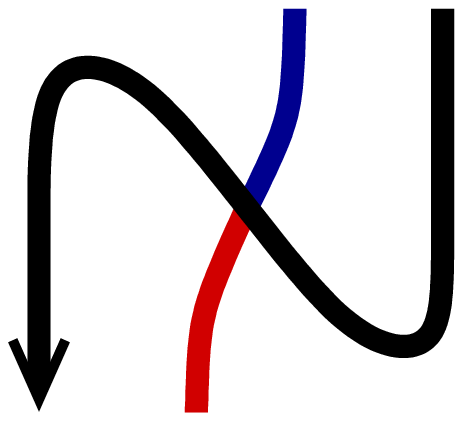}
\end{equation}

%%%%%%%%%%%%%%%%%%%%%%%%%
\bigskip\medskip

\item Relations involving one color:

\begin{equation}\label{eq:dumbrot}
\figins{-16}{0.5}{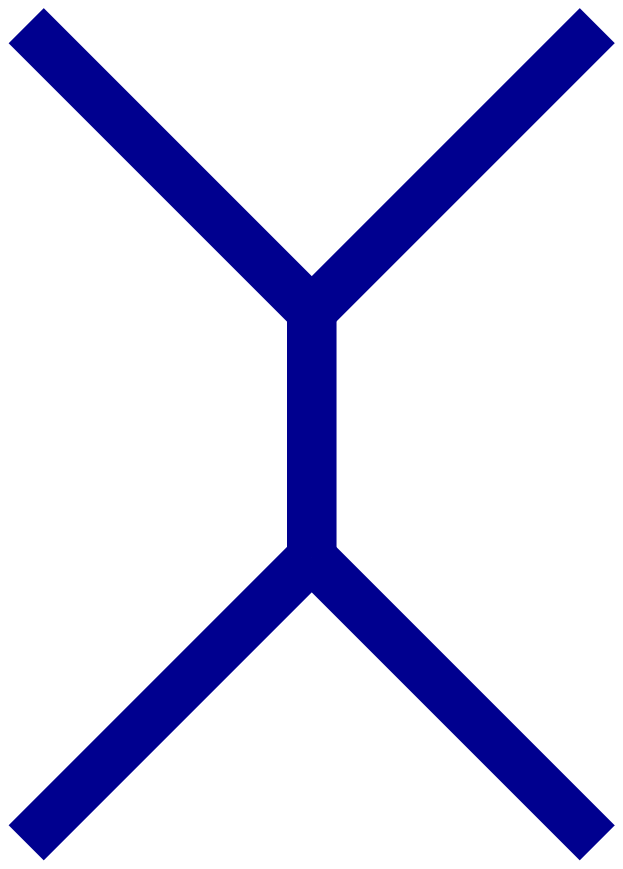}\
=\
\figins{-14}{0.45}{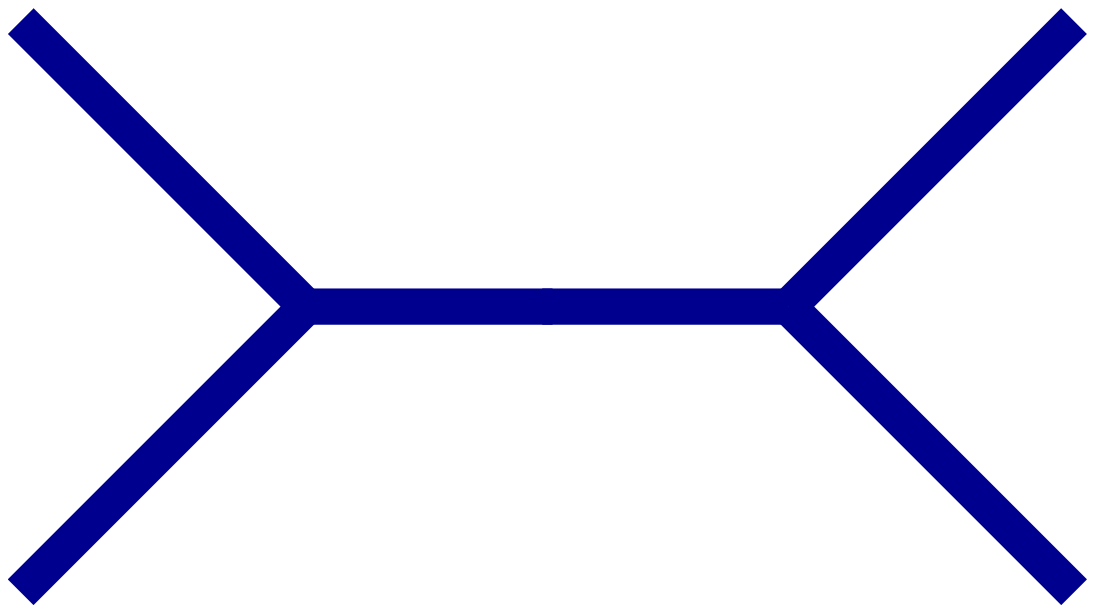}
\end{equation}

\begin{equation}\label{eq:lollipop}
\figins{-17}{0.55}{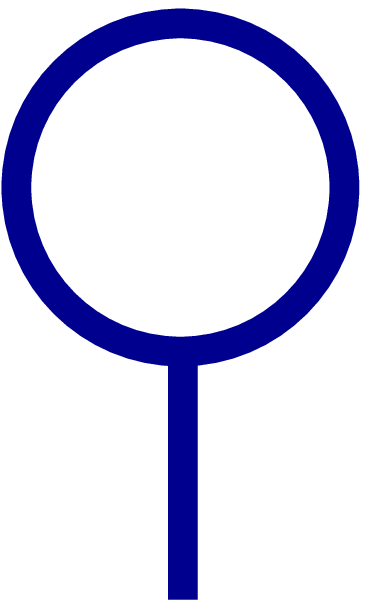}\
=\
0
\end{equation}

\begin{equation}\label{eq:deltam}
\figins{-17}{0.55}{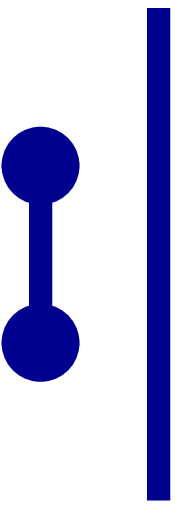}\
+\
\figins{-17}{0.55}{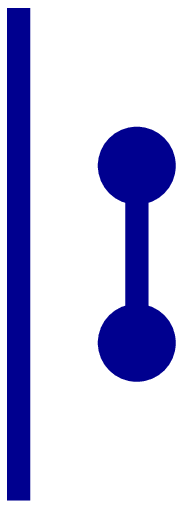}\
=\ 
2 \ \figins{-17}{0.55}{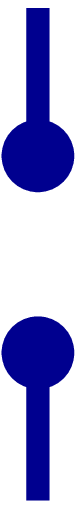}
\end{equation}

%%%%%%%%%%%%%%%%%%%%%%%%
\medskip
\item Relations involving two distant colors:
\begin{equation}\label{eq:reid2dist}
\figins{-32}{0.9}{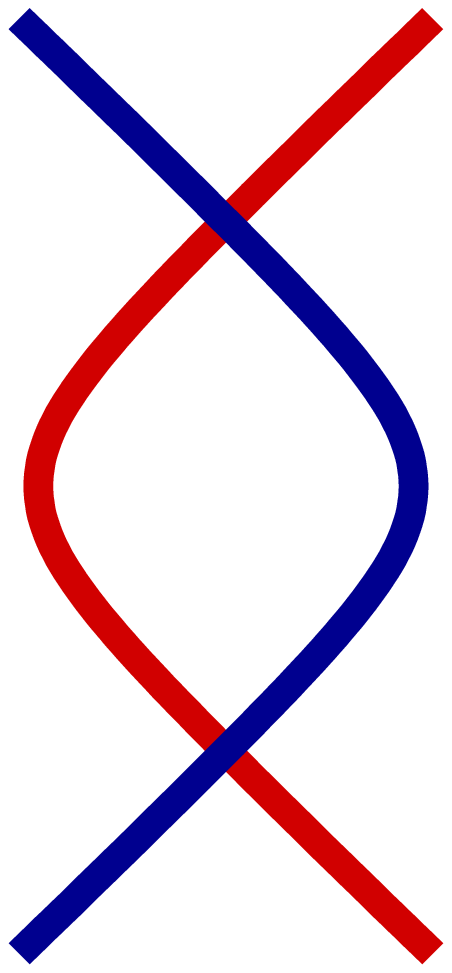}\
=\
\figins{-32}{0.9}{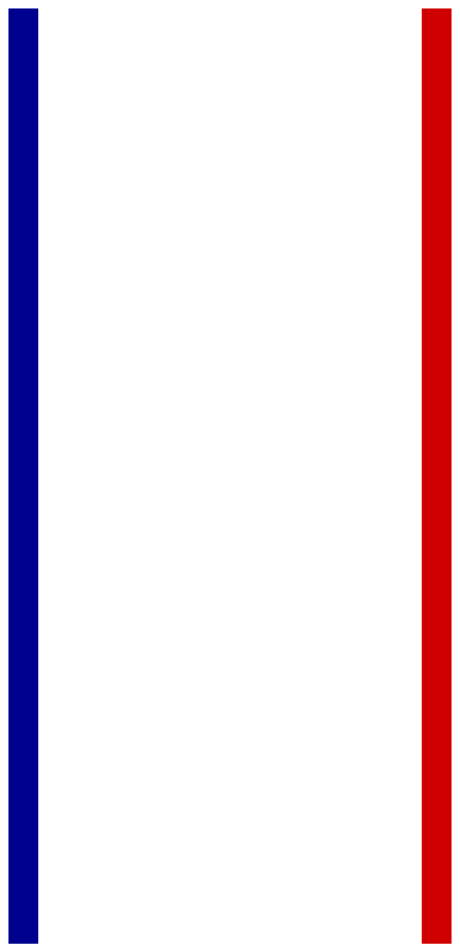}
\end{equation}

\begin{equation}\label{eq:slidedotdist}
\figins{-16}{0.5}{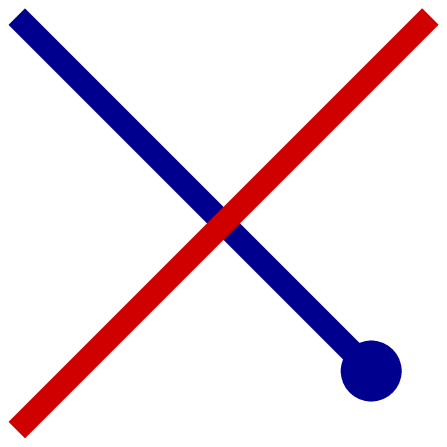}\
=\
\figins{-16}{0.5}{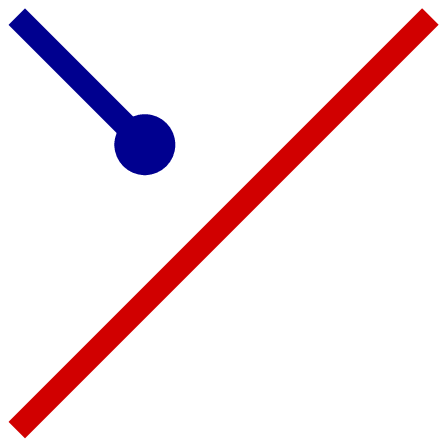}
\end{equation}

\begin{equation}\label{eq:slide3v}
\figins{-17}{0.55}{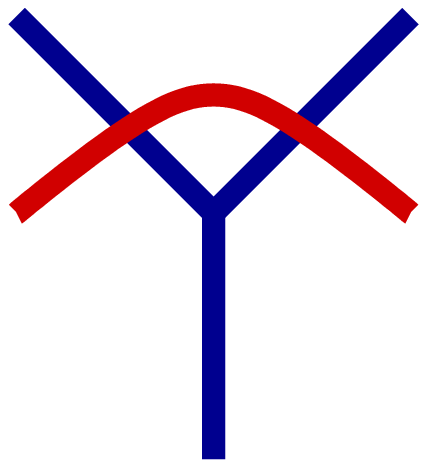}\
=\
\figins{-17}{0.55}{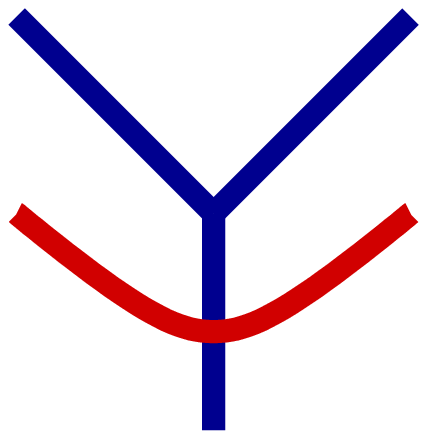}
\end{equation}

%%%%%%%%%%%%%%%%%%%%%
\medskip
\item Relations involving two adjacent colors:
\begin{equation}\label{eq:dot6v}
\figins{-16}{0.5}{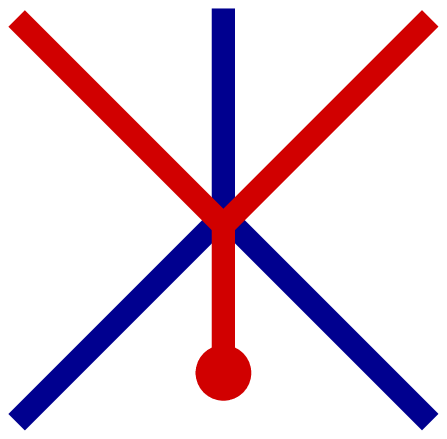}\
=\
\figins{-16}{0.5}{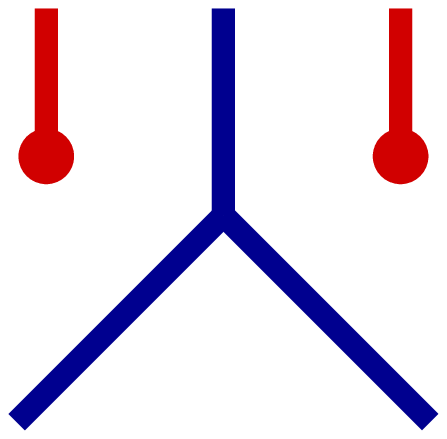}\
+\
\figins{-16}{0.5}{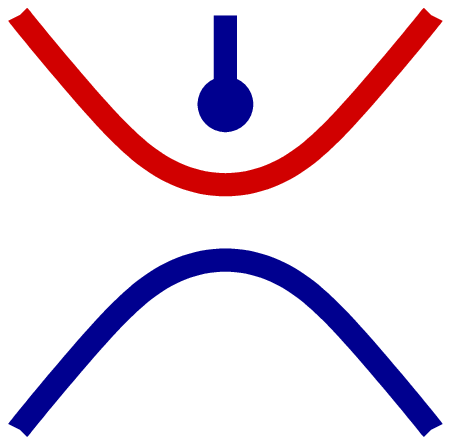}
\end{equation}

\begin{equation}\label{eq:reid3}
\figins{-30}{0.85}{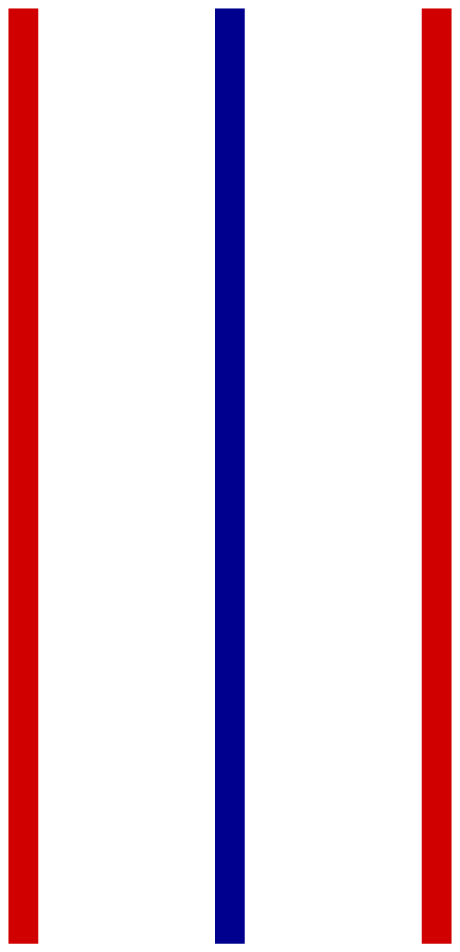}\
=\
\figins{-30}{0.85}{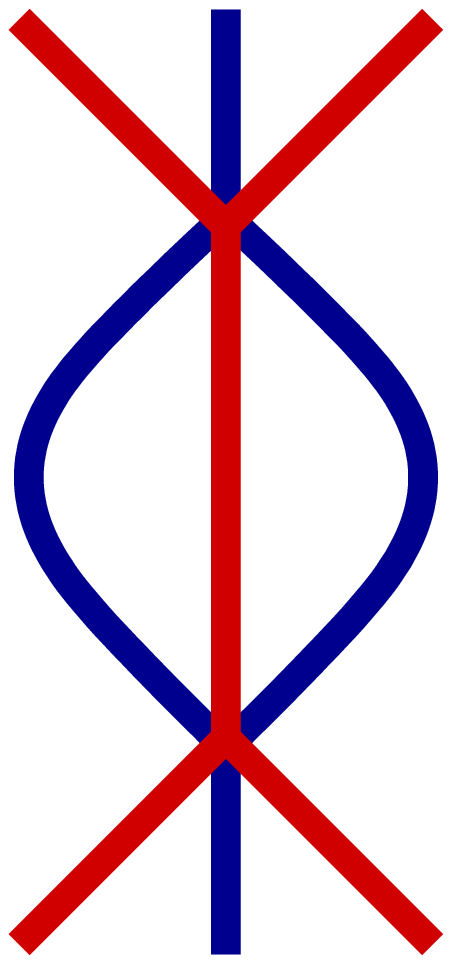}\
-\
\figins{-30}{0.85}{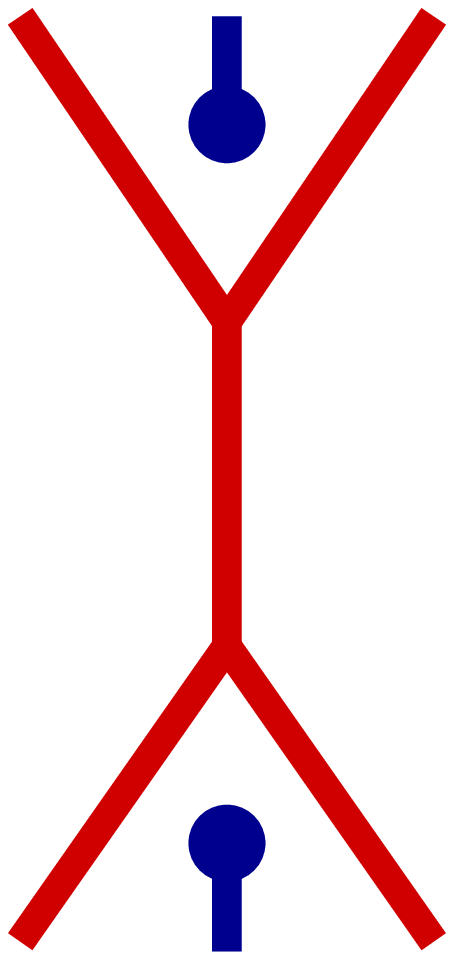}
\end{equation}

\begin{equation}\label{eq:dumbsq}
\figins{-30}{0.85}{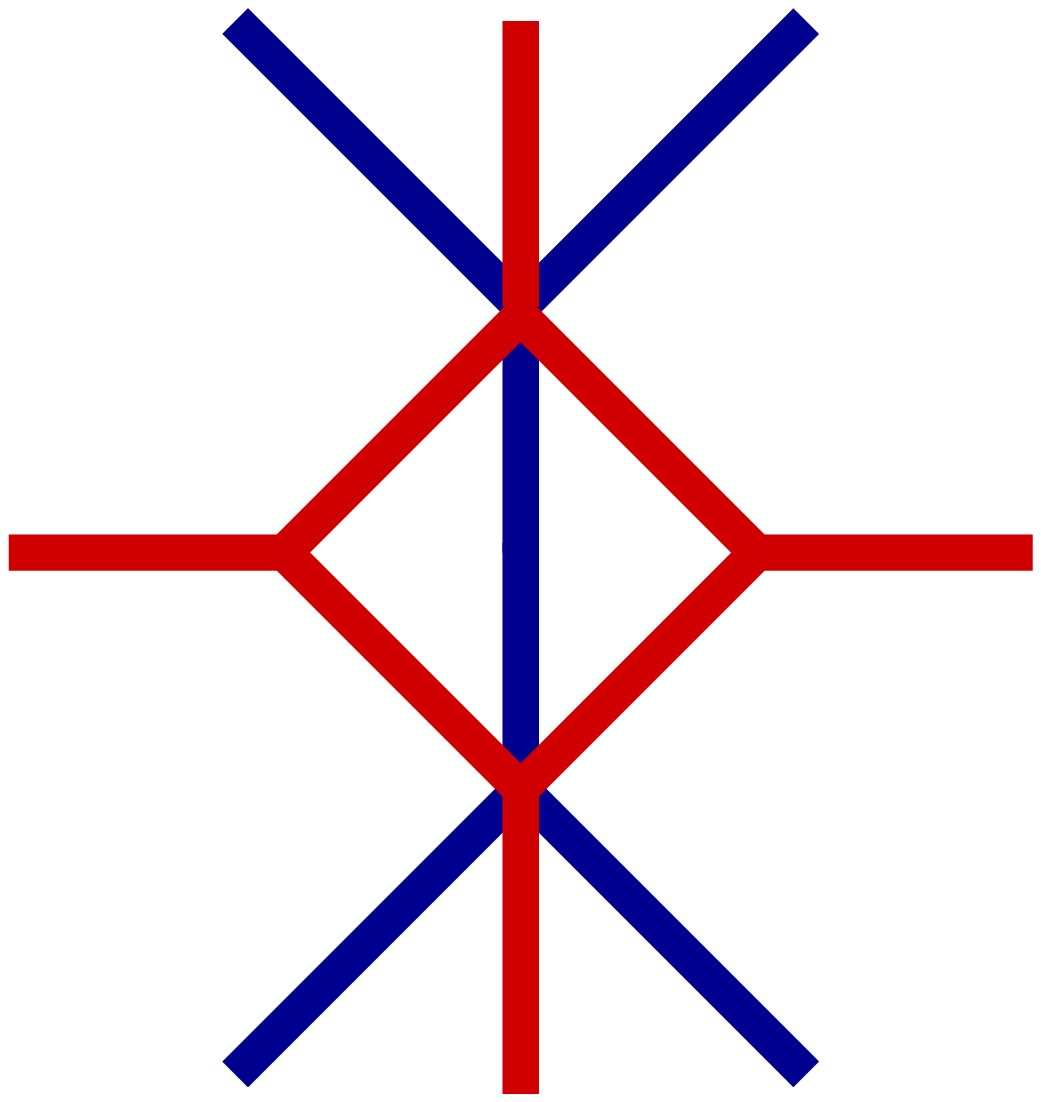}\
=\
\figins{-30}{0.85}{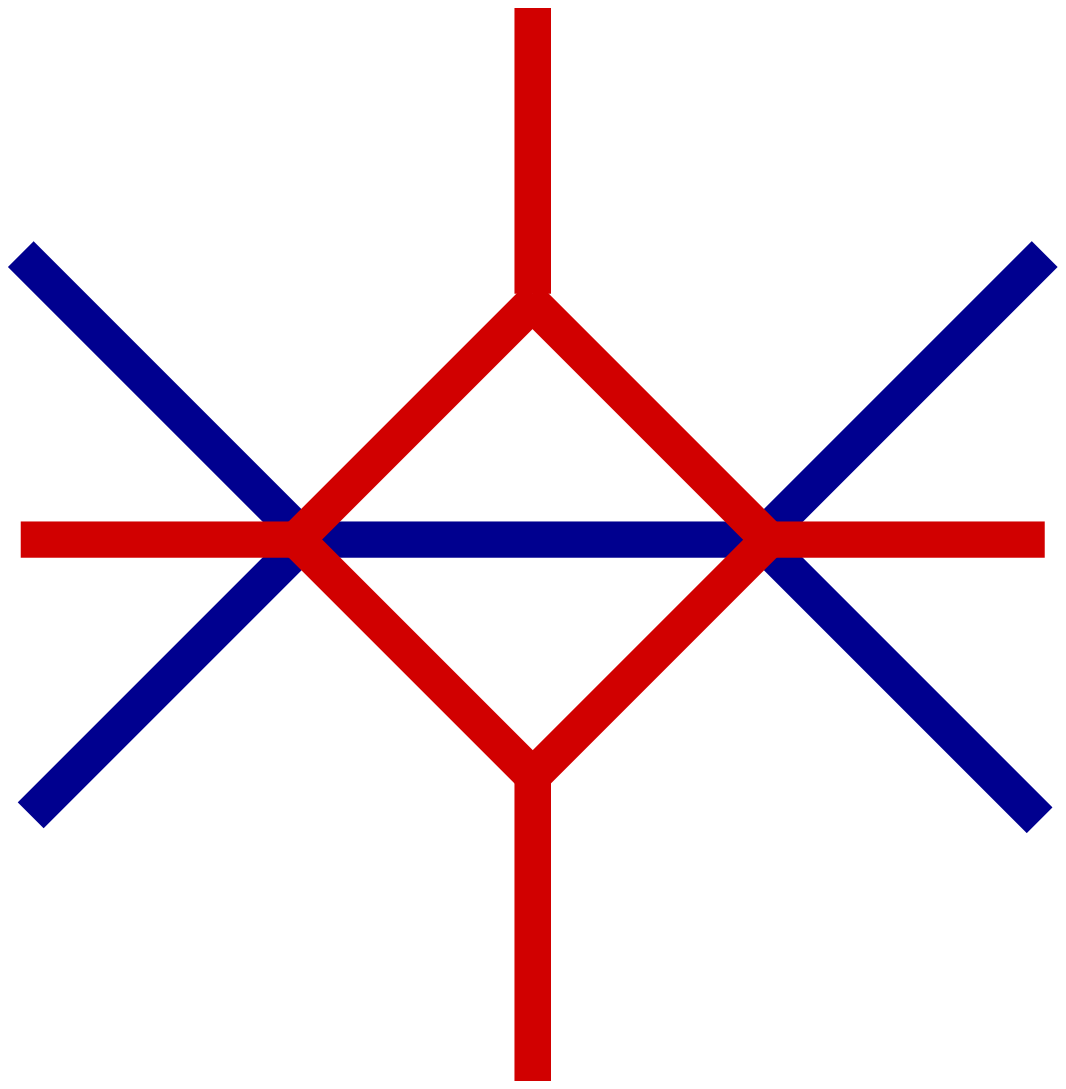}
\end{equation}

\begin{equation}\label{eq:slidenext}
\labellist
\tiny\hair 2pt
\pinlabel $j$ at -15  35
\pinlabel $i$ at  46 -10
\endlabellist
\figins{-17}{0.55}{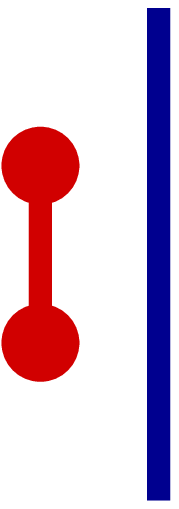}\
-\
\labellist
\tiny\hair 2pt
\pinlabel $j$ at  63  35
\pinlabel $i$ at   5 -10
\endlabellist
\figins{-17}{0.55}{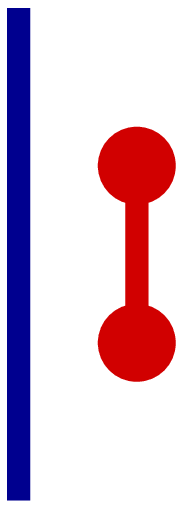}\quad
=\
\frac{1}{2}
\Biggl(\
\labellist
\tiny\hair 2pt
\pinlabel $i$   at  58  35
\pinlabel $i$   at   5 -10
\endlabellist
\figins{-17}{0.55}{edge-startenddot}\
-\
\labellist
\tiny\hair 2pt
\pinlabel $i$   at  -5   35
\pinlabel $i$   at  48 -10
\endlabellist
\figins{-17}{0.55}{startenddot-edge}\
\Biggr)
\end{equation}

%%%%%%%%%%%%%%%%%%%%%%%
\medskip
\item Relation involving three distant colors:
\begin{equation}\label{eq:slide4v}
\figins{-18}{0.6}{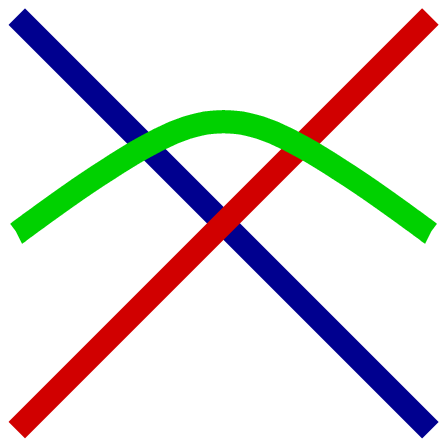}\
=\
\figins{-18}{0.6}{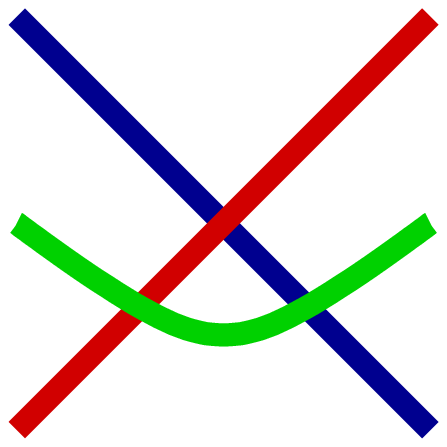}
\end{equation}

\medskip
\item Relation involving two adjacent colors and one distant from the other two:
\begin{equation}\label{eq:slide6v}
\figins{-18}{0.6}{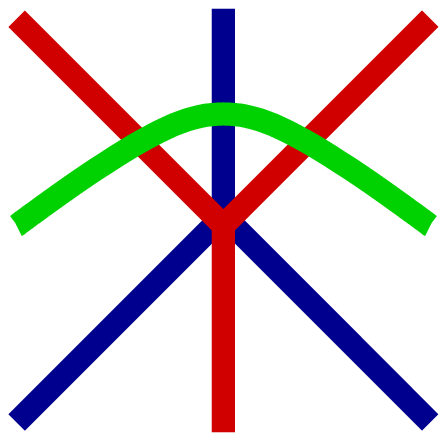}\
=\
\figins{-18}{0.6}{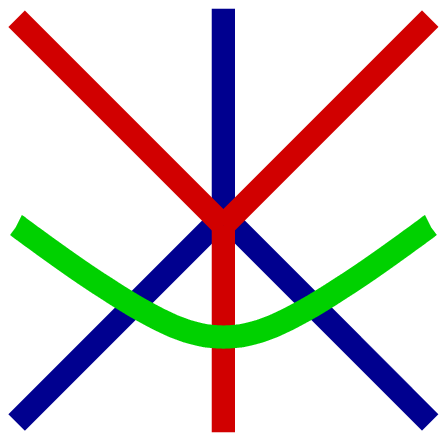}
\end{equation}

\medskip
\item Relation involving three adjacent colors:
\begin{equation}\label{eq:dumbdumbsquare}
\figins{-30}{0.85}{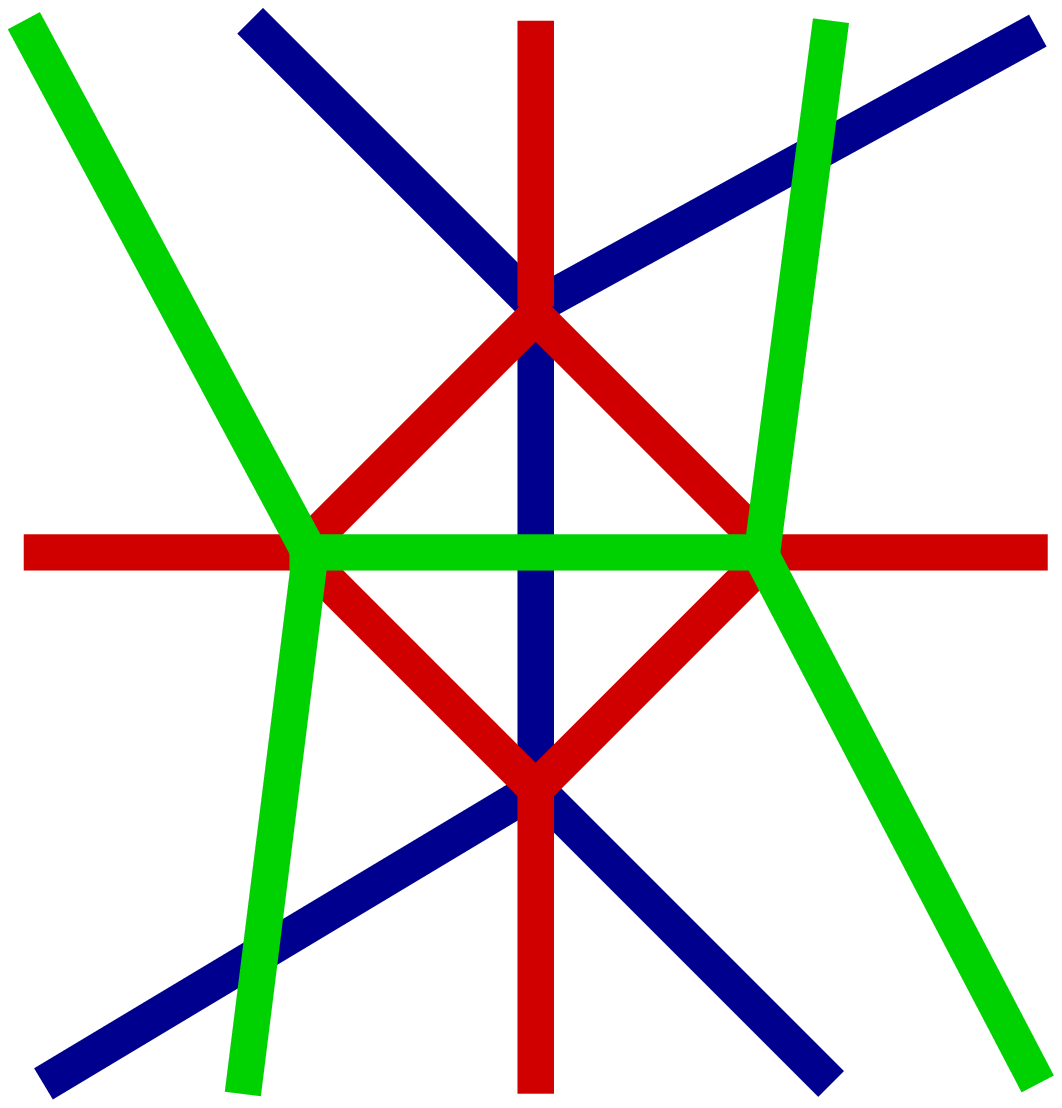}\
=\
\figins{-30}{0.85}{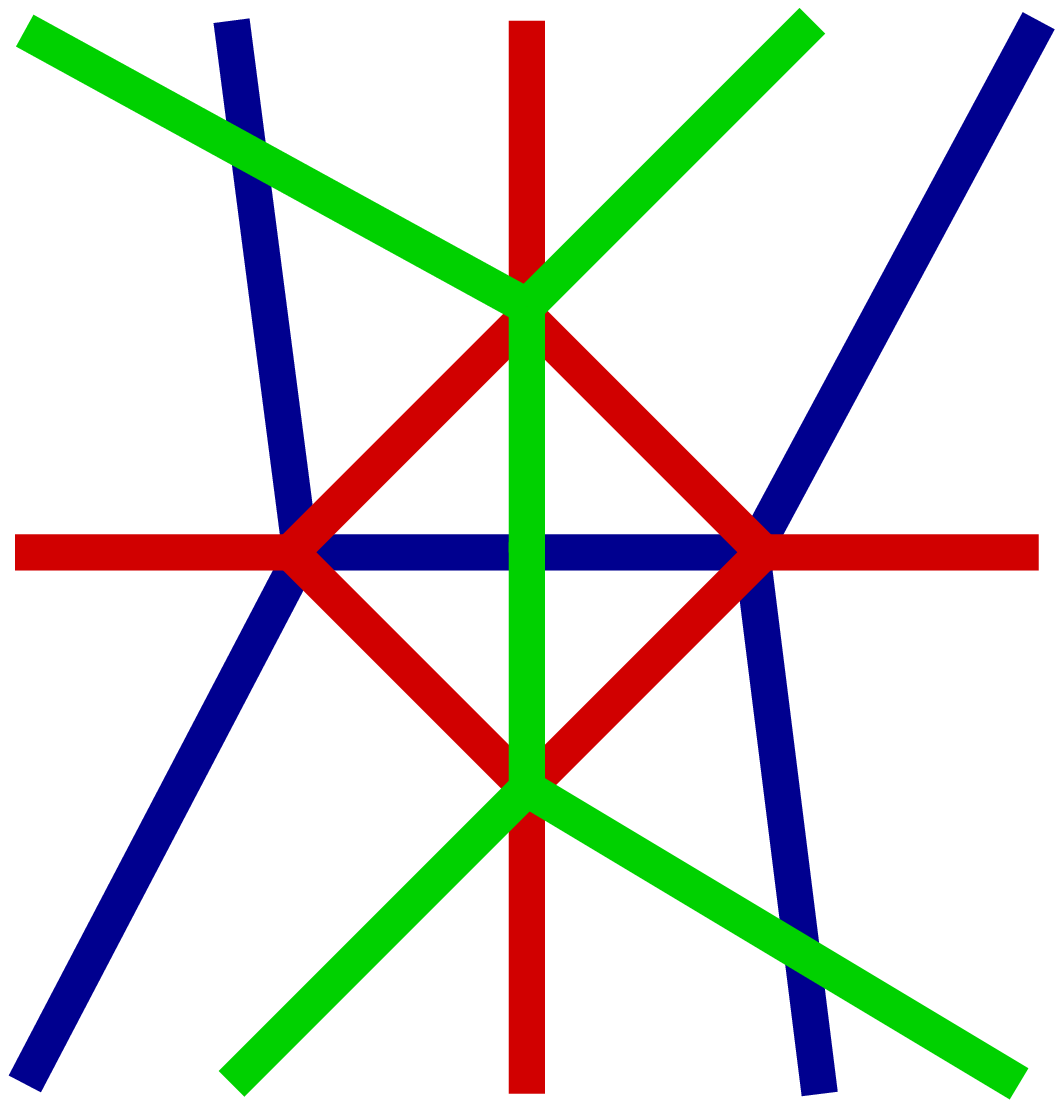}
\end{equation}

\medskip
\item Relations involving only oriented strands:
\begin{equation}\label{eq:orbub}
\figins{-10}{0.35}{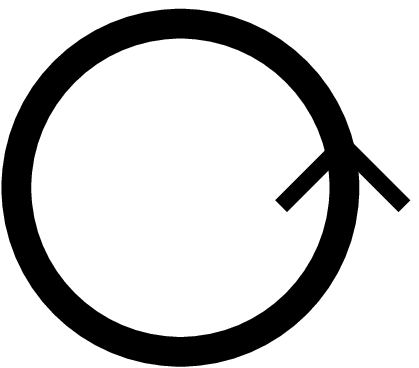}\
=\
1 \
=\
\figins{-10}{0.35}{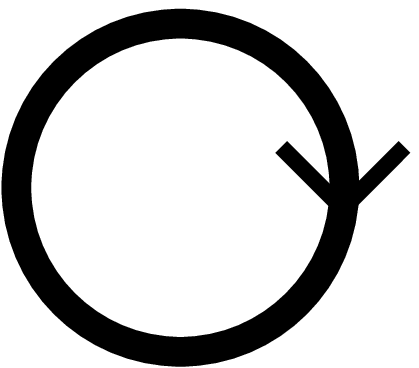}
\end{equation}

\begin{equation}\label{eq:capcupud}
\figins{-20}{0.9}{capcup-ud}\
=\
\figins{-20}{0.65}{rho-ud}
\end{equation}

\begin{equation}\label{eq:capcupdu}
\figins{-20}{0.9}{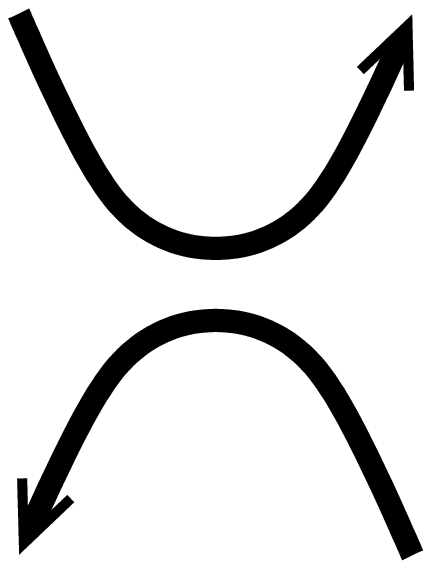}\
=\
\figins{-20}{0.65}{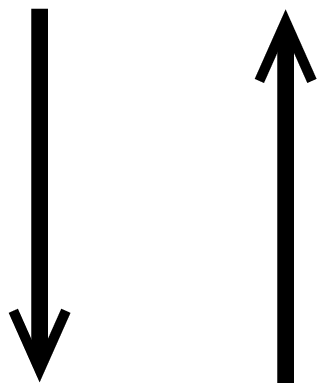}
\end{equation}

\medskip
\item Relations involving oriented strands and distant colored strands:
\begin{equation}\label{eq:slide4mv}
\figins{-18}{0.6}{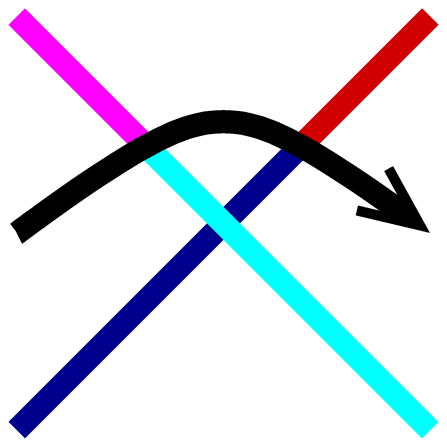}\
=\
\figins{-18}{0.6}{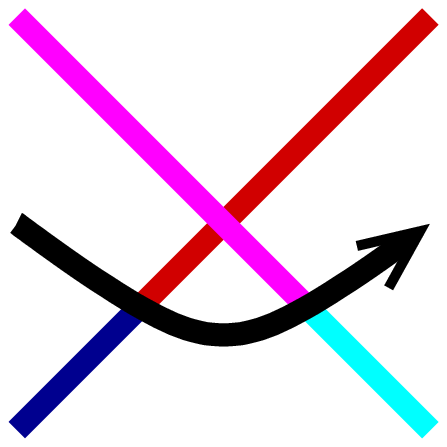}
\end{equation}

\medskip
\item Relations involving oriented strands and two adjacent colored strands:
\begin{equation}\label{eq:reid2ml}
\figins{-32}{0.9}{reid2m-l1}\
=\
\figins{-32}{0.9}{reid2m-l2}
\end{equation}

\begin{equation}\label{eq:reid2mr}
\ \figins{-32}{0.9}{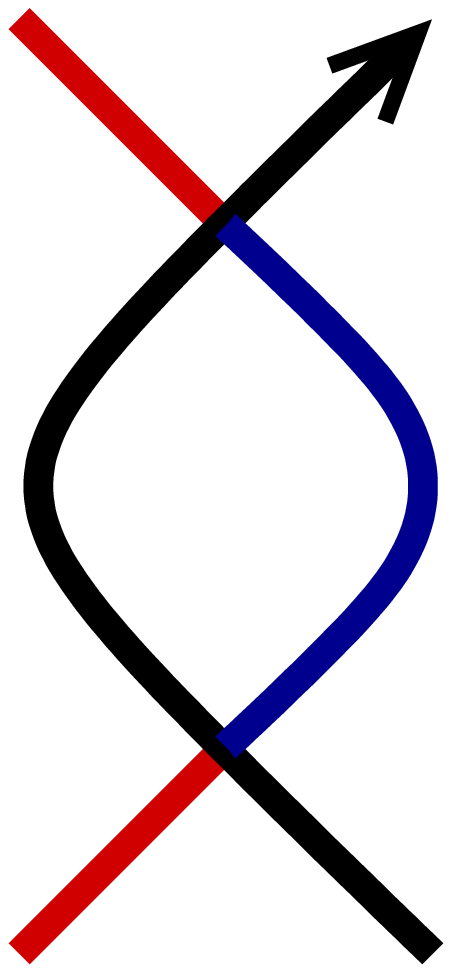}\ 
=\ 
\figins{-32}{0.9}{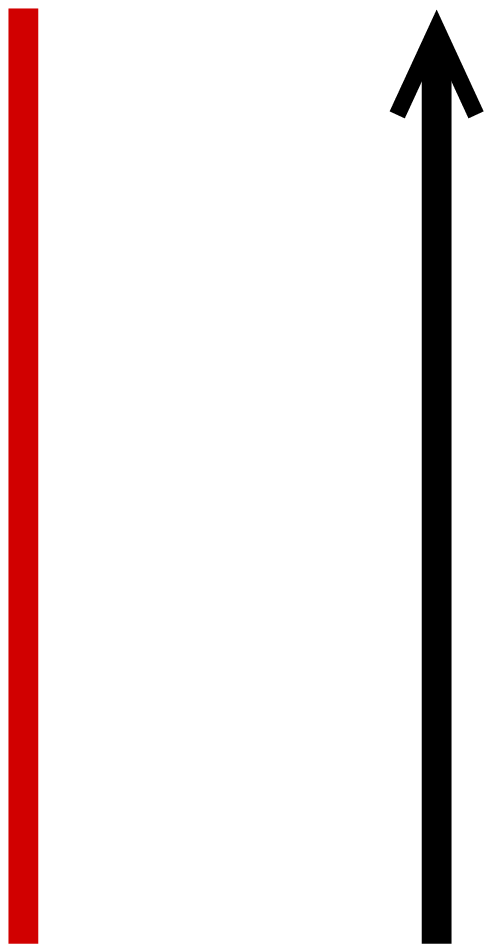}
\end{equation}

\begin{equation}\label{eq:slidedotdist-md}
\figins{-16}{0.5}{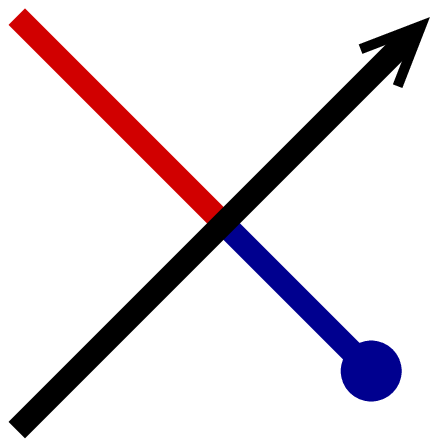}\
=\
\figins{-16}{0.5}{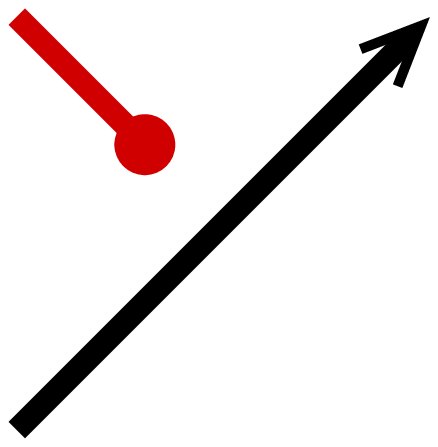}
\end{equation}

\begin{equation}\label{eq:slidedotdist-mu}
\figins{-16}{0.5}{4mvertdot-u}\
=\
\figins{-16}{0.5}{4mvertnodot-u}
\end{equation}

\begin{equation}\label{eq:mslide3v}
\ \figins{-17}{0.55}{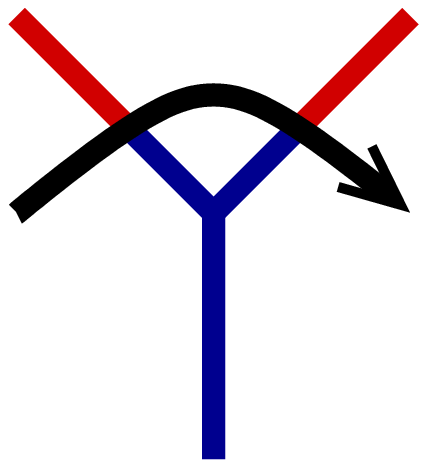}\
=\
\figins{-17}{0.55}{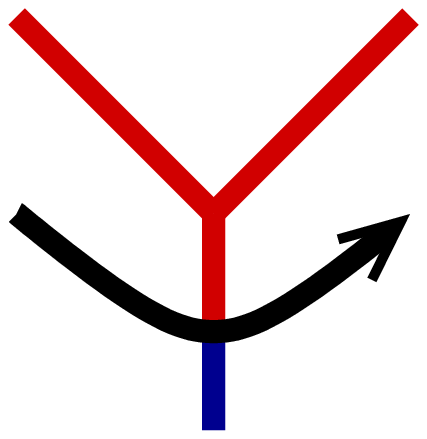}
\end{equation}

\medskip
\item Relations involving oriented strands and three adjacent colored strands:

\begin{equation}\label{eq:slide6mv}
\figins{-18}{0.6}{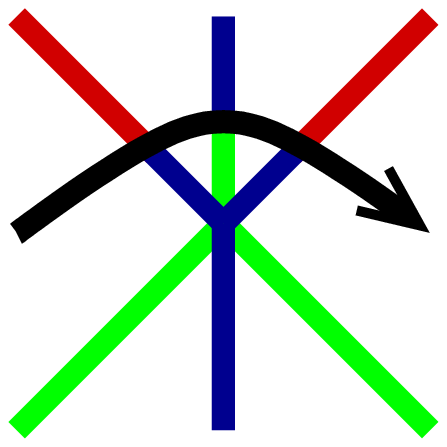}\
=\
\figins{-18}{0.6}{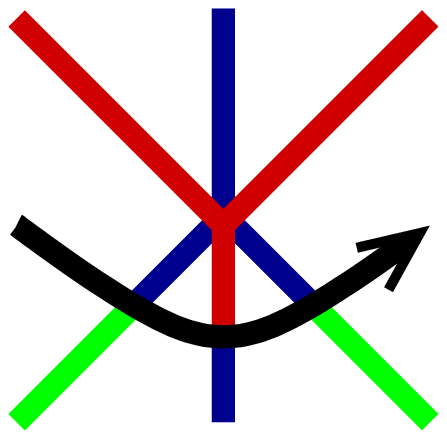}
\end{equation}

\begin{equation}\label{eq:slide6mv2}
\figins{-18}{0.6}{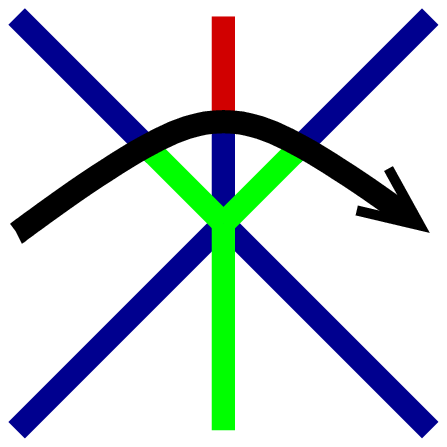}\
=\
\figins{-18}{0.6}{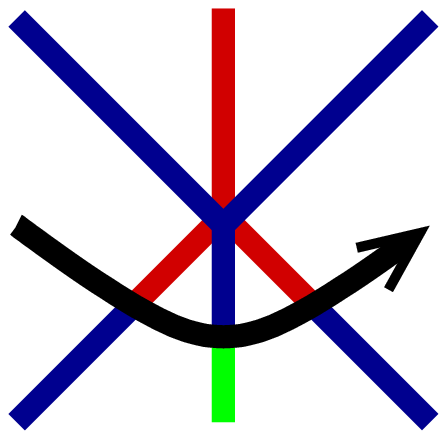}
\end{equation}

\medskip
\item Relations involving boxes:
\begin{align}
\label{eq:box1}
\figins{-6}{0.25}{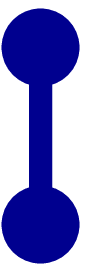}_{\,i}\
&= \
 \bbox{i+1} -\bbox{i}
\rlap{\hspace*{5ex} for $i\neq r$}
\\[1ex] \displaybreak[0]
\label{eq:box12}
\figins{-6}{0.25}{startenddot.eps}_{\,r}\
&= \
 \bbox{1} -\bbox{r} - \bbox{y}
\\[1ex] \displaybreak[0]
\label{eq:box13}
\biggl(
\bbox{i}\ 
 +\
\bbox{i+1} 
\biggr)\
\underset{i}{\figins{-20}{0.65}{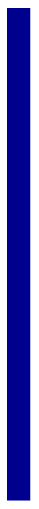}}
&=
\underset{i}{\figins{-20}{0.65}{vedge.eps}}
\biggl(
\bbox{i}\
+\
\bbox{i+1}
\biggr)
\\[1ex] \displaybreak[0]
\label{eq:box14}
\bbox{i}\ \bbox{i+1}\
\underset{i}{\figins{-20}{0.65}{vedge.eps}}
&=
\underset{i}{\figins{-20}{0.65}{vedge.eps}}\ \bbox{i}\ \bbox{i+1}
\rlap{\hspace*{5ex} for $i\neq r$}
\\[1ex] \displaybreak[0]
\label{eq:box15}
\biggl(\bbox{r} + \frac{1}{2} \bbox{y}\biggr) \biggl(\bbox{1} - \frac{1}{2} \bbox{y}\biggr)
\underset{r}{\figins{-20}{0.65}{vedge.eps}}
&=
\underset{r}{\figins{-20}{0.65}{vedge.eps}}\ \biggl(\bbox{r} + \frac{1}{2} \bbox{y}\biggr) \biggl(\bbox{1} - \frac{1}{2} \bbox{y}\biggr)
\\[1ex] \displaybreak[0]
\label{eq:box16}
\bbox{j}\
\underset{i}{\figins{-20}{0.65}{vedge.eps}} 
&=
\underset{i}{\figins{-20}{0.65}{vedge.eps}}\ \bbox{j}
\rlap{\hspace*{5ex} for $j\neq i, i+1$}
\\[1ex] \displaybreak[0]
\bbox{y}\
\underset{i}{\figins{-20}{0.65}{vedge.eps}} 
&=
\underset{i}{\figins{-20}{0.65}{vedge.eps}}\ \bbox{y}
\label{eq:box2}
\\[1ex] \displaybreak[0]
\bbox{y}\
\figins{-20}{0.65}{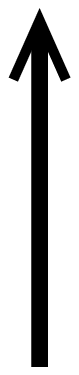} 
&=\figins{-20}{0.65}{rho.eps}\ \bbox{y}
\label{eq:box3}
\\[1ex] \displaybreak[0]
\label{eq:box31}
\bbox{y}\
\figins{-20}{0.65}{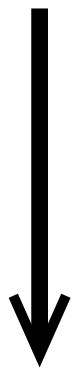} 
&=\figins{-20}{0.65}{rho-1.eps}\ \bbox{y}
\\[1ex] \displaybreak[0]
\label{eq:box32}
\bbox{i+1}\
\figins{-20}{0.65}{rho.eps} 
&=\figins{-20}{0.65}{rho.eps}\ \bbox{i}
\rlap{\hspace*{5ex} for $ i \neq r$}
\\[1ex] \displaybreak[0]
\label{eq:box33}
\biggl(
\bbox{1}\
-\
\bbox{y}
\biggr)\
\figins{-20}{0.65}{rho.eps} 
&=\figins{-20}{0.65}{rho.eps}\ \bbox{r}
\\[1ex] \displaybreak[0]
\label{eq:box34}
\bbox{i-1}\
\figins{-20}{0.65}{rho-1.eps} 
&=\figins{-20}{0.65}{rho-1.eps}\ \bbox{i}
\rlap{\hspace*{5ex} for $ i \neq 1$}
\\[1ex] \displaybreak[0]
\biggl(
\bbox{r}\
+\
\bbox{y}
\biggr)\
\figins{-20}{0.65}{rho-1.eps} 
&=\figins{-20}{0.65}{rho-1.eps}\ \bbox{1}
\label{eq:boxlast}
\end{align}
\end{itemize}

\begin{rem}\label{uselessrels}
Note that the Relations~\eqref{eq:slide4mv}--\eqref{eq:slide6mv2} also hold when the oriented strand has the opposite orientation. This follows from 
Relations~\eqref{eq:slide4mv}--\eqref{eq:slide6mv2} and Relations of isotopy.

Note also that the generators and relations listed above are redundant. This redundancy helps us to simplify some of the proofs later on. 
More specifically, let us list some of these unnecessary generators and relations:

\begin{itemize}
 \item[(i)] Relation \eqref{eq:capcupdu} follows from Relation \eqref{eq:capcupud} and isotopy invariance,
 \item[(ii)] Relation \eqref{eq:box34} follows from Relation \eqref{eq:box32} and isotopy invariance,
 \item[(iii)] Relation \eqref{eq:box33} follows from Relations \eqref{eq:box32} for $i=r-1$, \eqref{eq:box1} for $i=r-1$, 
\eqref{eq:box2}, \eqref{eq:slidedotdist-md} and \eqref{eq:slidedotdist-mu},
 \item[(iv)] Relation \eqref{eq:boxlast} follows from Relations \eqref{eq:box31}, \eqref{eq:box33} and isotopy invariance,
 \item[(v)] If we sum Relation \eqref{eq:box1} for all $i=1,\dots,r-1$ and Relation \eqref{eq:box12}, we obtain that
\begin{equation*}
 \bbox{y}
=
-\sum\limits_{k=1}^{r}
 \figins{-6}{0.25}{startenddot.eps}_{\,k}
\end{equation*}
\item[(vi)] If we sum Relation \eqref{eq:slidenext} for $j=i-1$ and for $j=i+1$, we obtain that
\begin{equation*}
\figins{-7}{0.27}{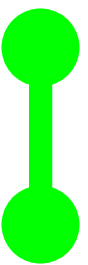}_{\,i-1}\
+\
\figins{-7}{0.27}{startenddot.eps}_{\,i}\
+\qquad
\labellist
\tiny\hair 2pt
\pinlabel $i+1$ at -30  45
\pinlabel $i$ at  46 -10
\endlabellist
\figins{-17}{0.55}{sedot-edge-d}\
=\
\labellist 
\tiny\hair 2pt
\pinlabel $i-1$   at  85  45
\pinlabel $i$   at   5 -10
\endlabellist
\figins{-17}{0.55}{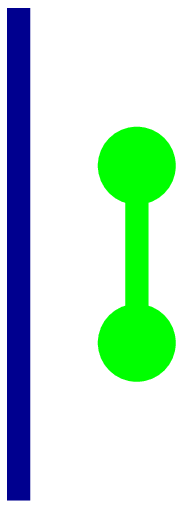}\
\qquad + \
\figins{-7}{0.27}{startenddot.eps}_{\,i}\
+\
\figins{-7}{0.27}{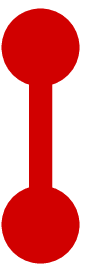}_{\,i+1}
\end{equation*}
 \item[(vii)] Relation \eqref{eq:box2} follows from Relation \eqref{eq:slidedotdist}, the two relations exhibited just above in (v) and (vi) and isotopy invariance,
 \item[(viii)] Relation \eqref{eq:box3} follows from Relations \eqref{eq:slidedotdist-md} and \eqref{eq:slidedotdist-mu} and the relation exhibited in (v),
 \item[(iix)] Relation \eqref{eq:box31} follows from Relation \eqref{eq:box3} and isotopy invariance.
\end{itemize}
\end{rem}

\subsubsection{Functor from $\debim_{\hat{A}_{r-1}}^*$ to $\ebim_{\hat{A}_{r-1}}^*$}

Let us construct a degree preserving functor
$$\F\colon\debim_{\hat{A}_{r-1}}^*\to \ebim_{\hat{A}_{r-1}}^*$$ 
which extends the one by Elias and Khovanov. On objects, it is defined as follows: it maps each integer $i\in \{1, \dots,r \}$ to $B_{i}\{-1\}$, 
and the symbols $+$ and $-$ to~$B_{\rho}$ and~$B_{\rho^{-1}}$ respectively. Sequences of these are mapped to tensor products. The empty sequence is sent to $R$. 

For the morphisms one only needs to specify~$\F$ on the generators. A sequence of vertical strands is mapped to the identity of the corresponding bimodule 
and $\mbox{box}_i$ (resp. $\mbox{box}_y$) is mapped to multiplication by $x_i$ (resp. $y$). Furthermore, we define (recall that 
$X_i=x_{i+1}-x_{i}$, for $i=1,\ldots, r-1$, and $X_r=x_1-x_r-y$) 
\begin{eqnarray*}
\F (\mbox{enddot}_i) & = & \br_i : a \ot b \mapsto ab \\
\F (\mbox{startdot}_i) & = & \rb_i : a \mapsto \frac{a}{2} \left( X_i \ot 1 + 1 \ot X_i \right) \\
\F (\mbox{merge}_i) & = & \mbox{pr}_i : 
\left\lbrace
\begin{array}{clc}
a \ot 1 \ot b  & \mapsto & 0  \\
a \ot X_i \ot b  & \mapsto & 2a \ot b \\
\end{array}
\right.  \\
\F (\mbox{split}_i) & = & \mbox{inj}_i : a \ot b \mapsto a \ot 1 \ot b \\
\F (4\mbox{vert}_{i,j}) & = & \mbox{f}_{i,j} :  a \ot 1 \ot b \mapsto a \ot 1 \ot b \\
\F (6\mbox{vert}_{i,i+1}) & = & \mbox{f}_{i,i+1} :  
\left\lbrace
\begin{array}{l}
a \ot 1 \ot 1 \ot b  \mapsto  a \ot 1 \ot 1 \ot b \\
a \ot (X_{i+1} \ot 1 + 1 \ot X_{i+1}) \ot b  \mapsto  0 \\
\end{array}
\right.  \\
\F (6\mbox{vert}_{i+1,i}) & = & \mbox{f}_{i+1,i} :  
\left\lbrace
\begin{array}{l}
a \ot 1 \ot 1 \ot b  \mapsto  a \ot 1 \ot 1 \ot b  \\
a \ot (X_{i} \ot 1 + 1 \ot X_{i})  \ot b  \mapsto  0 \\
\end{array}
\right.  \\
\F (+\mbox{cap}) & = & _{+,-}\mbox{r} : a \ot b \mapsto a \rho(b) \\
\F (-\mbox{cap}) & = & _{-,+}\mbox{r} : a \ot b \mapsto a \rho^{-1}(b) \\
\F (-\mbox{cup}) & = & \mbox{r}_{-,+} : a  \mapsto a \ot 1  \\
\F (+\mbox{cup}) & = & \mbox{r}_{+,-} : a  \mapsto a \ot 1  \\
\F (4\mbox{vert}_{+,i}) & = & \mbox{fl}_{+,i} : a \ot b \mapsto a\ot \rho(b) \\
\F (4\mbox{vert}_{i+1,+}) & = & \mbox{fl}_{i+1,+} : a \ot b \mapsto a\ot \rho^{-1}(b) \\
\F (4\mbox{vert}_{i,-}) & = & \mbox{fl}_{i,-} : a \ot b \mapsto a\ot \rho(b) \\
\F (4\mbox{vert}_{-,i+1}) & = & \mbox{fl}_{-,i+1} : a \ot b \mapsto a\ot \rho^{-1}(b) \\
\F (\mbox{box}_i) & = & \mbox{m}_i : a \mapsto ax_i \\
\F (\mbox{box}_y) & = & \mbox{m}_y : a \mapsto ay \\
\end{eqnarray*}

\begin{prop}
The functor $\F$ is well-defined, degree preserving and essentially surjective.
\end{prop}

\begin{proof}
The fact that the functor~$\F$ is well-defined and degree preserving amounts to a straightforward 
verification that it preserves the relations~\eqref{eq:adj}--\eqref{eq:boxlast} and the degrees of the morphisms. For the relations 
involving only non-oriented strands, this is completely analogous to Elias and Khovanov's case. For the relations involving oriented strands, the 
calculations are new but easy. 

Furthermore, in view of the definitions of the objects of 
$\debim_{\hat{A}_{r-1}}^*$ and~$\ebim_{\hat{A}_{r-1}}^*$, the functor~$\F$ is clearly essentially surjective.
\end{proof}

\begin{prop}
The functor $\F$ is full.
\end{prop}

\begin{proof}
Libedinsky has proved in \cite{Li} that all the morphisms of the category~$\bim_{\hat{A}_{r-1}}^*$ are generated by the following ones:
\begin{itemize}
 \item $\br_i$, $\rb_i$, $\mbox{pr}_i$ and $\mbox{inj}_i$ for all $i=1, \dots,r$
\item $\mbox{f}_{i,j}$ for all $i,j=1, \dots,r$ with $i \neq j$.
\end{itemize}
In view of Remark~\ref{homsp}, this implies that all the morphisms of the category~$\ebim_{\hat{A}_{r-1}}^*$ are generated by the ones listed by Libedinsky and copied above, 
together with $\mbox{fl}_{+,i}$ and $\mbox{fl}_{i+1,+}$, giving  
$$B_{\rho} \ot_R B_i \cong B_{i+1} \ot_R B_{\rho} \quad \mbox{for} \  i=1, \dots, r,$$ 
and $_{+,-}\mbox{r}$, $_{-,+}\mbox{r}$, $\mbox{r}_{-,+} $ and $\mbox{r}_{+,-}$, giving 
$$B_{\rho} \ot_R B_{\rho^{-1}} \cong R \cong  B_{\rho^{-1}} \ot_R B_{\rho}.$$ 
Thus the functor~$\F$ is full, since all the morphisms generating~$\ebim_{\hat{A}_{r-1}}^*$ are in the image of~$\F$. 
\end{proof}

\begin{thm}
\label{thm:diameqbim}
The categories $\debim_{\hat{A}_{r-1}}^*$ and $\ebim_{\hat{A}_{r-1}}^*$ are equivalent and 
so are their Karoubi envelopes $\kdebim_{\hat{A}_{r-1}}$ and $\kebim_{\hat{A}_{r-1}}$.
\end{thm}

\begin{proof}
Only the faithfulness of $\F$ remains to be proved. 

For a given object $X$ in $\debim_{\hat{A}_{r-1}}^*$, let $k_X$ denote the total sum of 
plus and minus signs in $X$. Given two objects $X,Y$ in $\debim_{\hat{A}_{r-1}}^*$, 
the hom-space between $X$ and $Y$ is non-zero only if $k_X=k_Y$ (see Remark~\ref{homsp}). Any object 
$X$ in $\debim_{\hat{A}_{r-1}}^*$ is isomorphic to the object $(\mathrm{sign}(k_X)^{|k_X|},X')$, 
where $X'$ does not have any signs and is obtained from 
$X$ by applying the commutation isomorphisms $(\pm,i) \cong (i\pm1,\pm)$ and 
the isomorphisms $(\pm,\mp) \cong \emptyset$.  

Let $D$ be a diagram representing a morphism from $X$ to $Y$, such that 
$k_X=k_Y$. Since 
$$\mathrm{End}(\mathrm{sign}(k_X)^{|k_X|})\cong\Q,$$
it follows that 
$$\mathrm{Hom}_{\debim_{\hat{A}_{r-1}}^*}(X,Y)\cong \mathrm{Hom}_{\debim_{\hat{A}_{r-1}}^*}(X',Y').$$
Let us illustrate this by the example in Figure~\ref{fig:exdiag}.
\begin{figure}[htbp]
  \begin{center}
   \figins{-17}{3}{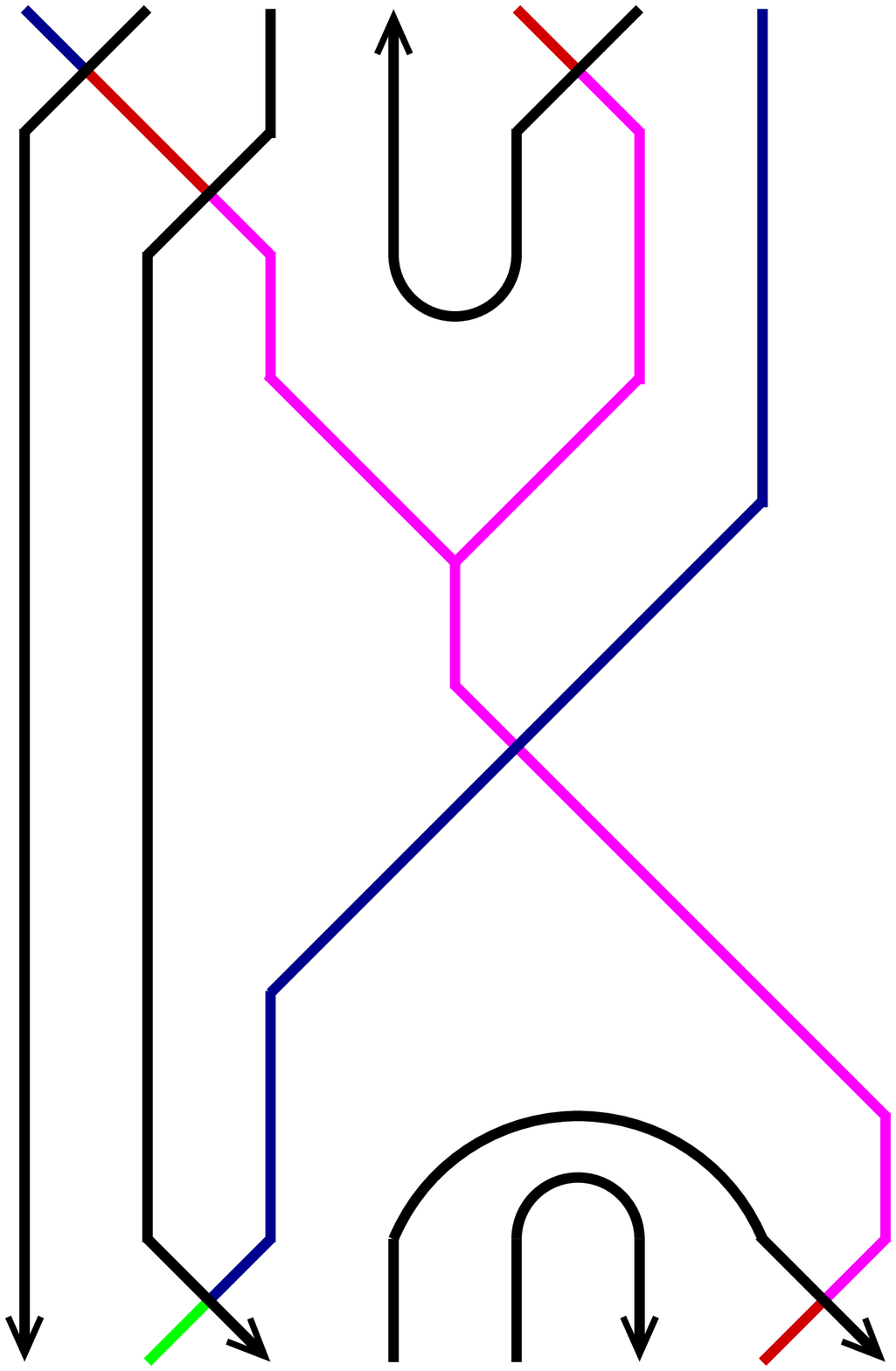} 
   \caption{Example of a decomposition of a diagram $D$}
   \label{fig:exdiag}
  \end{center}
\end{figure}

Since 
$$\mathrm{Hom}_{\debim_{\hat{A}_{r-1}}^*}(X',Y')\cong \mathrm{Hom}_{\dbim_{\hat{A}_{r-1}}^*}(X',Y'),$$ 
the faithfulness of $\F$ follows from the faithfulness of Elias and Williamson's 
analogous functor for non-extended affine type $A$, which they proved in Theorem 6.28 of~\cite{EW}. 
\end{proof}

\section{The affine Schur quotient}
\label{sec:AffSchur}
\subsection{Notations}\label{notSch}

Let~$n,r$ be integers, with~$n>r$ and $r\geq 3$, and let~$q$ be a formal parameter.

A generic $\mathfrak{gl}_n$--weight will be denoted $\lambda = (\lambda_1, \dots, \lambda_n)$, and we set $\bar{\lambda}_i = \lambda_i - \lambda_{i+1}$. Define
$$\Lambda(n,r)=\{\lambda\in \N^n\colon \,\, \sum_{i=1}^{n}\lambda_i=r\}$$ 
and
$$\Lambda^+(n,r)=\{\lambda\in\Lambda(n,r)\colon r\geq \lambda_1\geq\lambda_2\geq\cdots \geq\lambda_n\geq 0\}.$$ 
We will denote by $ (1^r)$ the weight $ \varepsilon_1 + \dots + \varepsilon_r$.

Let us formally denote by~$\bar{\alpha}_n$ the opposite of the highest root of~$\mathfrak{sl}_n$, \textit{i.e.} $\bar{\alpha}_n = - \bar{\theta} =  - \bar{\alpha}_1 - \dots  - \bar{\alpha}_{n-1} = \varepsilon_n - \varepsilon_1 $.

As always in this paper, we use the convention that the indices appearing in the relations are considered modulo~$n$.

\subsection{The (extended) affine algebras}\label{genalg}
In the following definitions we do not need to consider the derivation, which we used above. 
   
\begin{defn} The {\em extended quantum general linear algebra} 
$\hat{\mathbf U}_q(\hat{\mathfrak{gl}}_n)$ is 
the associative unital $\Q(q)$-algebra generated by $R^{\pm 1}$, $K_i^{\pm 1}$ and $E_{\pm i}$, for $i=1,\ldots, n$, subject to the relations
\begin{gather}
K_iK_j=K_jK_i\quad K_iK_i^{-1}=K_i^{-1}K_i=1
\\
E_iE_{-j} - E_{-j}E_i = \delta_{i,j}\dfrac{K_iK_{i+1}^{-1}-K_i^{-1}K_{i+1}}{q-q^{-1}}
\\
K_iE_{\pm j}=q^{\pm \left< \epsilon_i,\bar{\alpha}_j \right> }E_{\pm j}K_i
\\
E_{\pm i}^2E_{\pm (i\pm 1)}-(q+q^{-1})E_{\pm i}E_{\pm (i\pm 1)}E_{\pm i}+E_{\pm (i\pm 1)}E_{\pm i}^2=0
\\
E_{\pm i}E_{\pm j}-E_{\pm j}E_{\pm i}=0\qquad\text{for distant}\;i,j\\
RR^{-1}=R^{-1}R=1\\
RX_iR^{-1}=X_{i+1}\qquad\text{for}\;X_i\in\{E_{\pm i},K_i^{-1}\}.
\end{gather} 
\end{defn}

\begin{defn} 
\label{defn:qsln}
The {\em affine quantum general linear algebra} 
${\mathbf U}_q(\hat{\mathfrak{gl}}_n)\subseteq \hat{\mathbf U}_q(\hat{\mathfrak{gl}}_n)$ is 
the unital $\Q(q)$-subalgebra generated by $E_{\pm i}$ and $K_i^{\pm 1}$, for $i=1,\ldots,n$. 

The {\em affine quantum special linear algebra} ${\mathbf U}_q(\hat{\mathfrak{sl}}_n)\subseteq {\mathbf U}_q(\hat{\mathfrak{gl}}_n)$ is 
the unital $\Q(q)$-subalgebra generated by 
$E_{\pm i}$ and $K_iK^{-1}_{i+1}$, for $i=1,\ldots, n$.
\end{defn}

We will also need the bialgebra structure on 
$\hat{\mathbf U}_q(\hat{\mathfrak{gl}}_n)$.
\begin{defn} $\hat{\mathbf U}_q(\hat{\mathfrak{gl}}_n)$ is a bialgebra with 
counit 
$\epsilon\colon \hat{\mathbf U}_q(\hat{\mathfrak{gl}}_n)\to \Q(q)$
defined by 
$$\epsilon(E_{\pm i})=0,\qquad\epsilon(R^{\pm 1})=\epsilon(K_i^{\pm 1})=1$$
and coproduct 
$\Delta\colon \hat{\mathbf U}_q(\hat{\mathfrak{gl}}_n)\to 
\hat{\mathbf U}_q(\hat{\mathfrak{gl}}_n)\otimes \hat{\mathbf U}_q(\hat{\mathfrak{gl}}_n),$ 
defined by 
\begin{eqnarray}
\Delta(1)&=&1\otimes 1\\
\Delta(E_i)&=&E_i\otimes K_iK_{i+1}^{-1}+1\otimes E_i\\
\Delta(E_{-i})&=&K_i^{-1}K_{i+1}\otimes E_{-i}+E_{-i}\otimes 1\\
\Delta(K_i^{\pm 1})&=&K_i^{\pm 1}\otimes K_i^{\pm 1}\\
\Delta(R^{\pm 1})&=&R^{\pm 1}\otimes R^{\pm 1}.
\end{eqnarray}
\end{defn} 
As a matter of fact, $\hat{\mathbf U}_q(\hat{\mathfrak{gl}}_n)$ is 
even a Hopf algebra, but we do not need the antipode in this paper. Note 
that $\Delta$ and $\epsilon$ can be restricted to 
${\mathbf U}_q(\hat{\mathfrak{gl}}_n)$ and 
${\mathbf U}_q(\hat{\mathfrak{sl}}_n)$, which are bialgebras too. 

At level 0 and forgetting the derivation, we can work with the ${\mathbf U}_q(\mathfrak{sl}_n)$--weight 
lattice, when considering ${\mathbf U}_q(\hat{\mathfrak{sl}}_n)$--weight representations. Omitting the extra entry corresponding to the 
derivation makes this weight-lattice degenerate, because $\alpha_1+\alpha_2+\cdots+\alpha_n=0$, but that does not matter in this section. Similarly, we can work with the ${\mathbf U}_q(\mathfrak{gl}_n)$--weight lattice, when considering $\hat{\mathbf U}_q(\hat{\mathfrak{gl}}_n)$ and ${\mathbf U}_q(\hat{\mathfrak{gl}}_n)$--weight representations.
 
Suppose that $V$ is a 
${\mathbf U}_q(\hat{\mathfrak{gl}}_n)$--weight representation with 
weights $\lambda=(\lambda_1,\ldots,\lambda_n)\in\bZ^n$, i.e. 
$$V\cong \bigoplus_{\lambda}V_{\lambda}$$ 
and $K_i$ acts as multiplication by 
$q^{\lambda_i}$ on $V_{\lambda}$. Then $V$ is also a 
${\mathbf U}_q(\hat{\mathfrak{sl}}_n)$--weight representation with weights 
$\overline{\lambda}=(\overline{\lambda}_1,\ldots,\overline{\lambda}_{n-1})\in
\bZ^{n-1}$ such that 
$\overline{\lambda}_j=\lambda_j-\lambda_{j+1}$ for $j=1,\ldots,n-1$. 

Conversely, there is not a unique choice of ${\mathbf U}_q(\hat{\mathfrak{gl}}_n)$--action
on a given ${\mathbf U}_q(\hat{\mathfrak{sl}}_n)$--weight representation with weights $\mu=(\mu_1,\ldots,\mu_{n-1})$. 
We first have to fix the action of $K_1\cdots K_n$. In terms of weights, this 
corresponds to the observation that, for any given $r\in\bZ$ the equations 
\begin{align}
\label{eq:sl-gl-wts1}
\lambda_i-\lambda_{i+1}&=\mu_i\\
\label{eq:sl-gl-wts2}
\qquad \sum_{i=1}^{n}\lambda_i&=r
\end{align}  
determine $\lambda=(\lambda_1,\ldots,\lambda_n)$ uniquely, 
if there exists a solution to~\eqref{eq:sl-gl-wts1} and~\eqref{eq:sl-gl-wts2} 
at all. We therefore define the map 
$\phi_{n,r}\colon \bZ^{n-1}\to \bZ^{n}\cup \{*\}$ 
by 
$$
\phi_{n,r}(\mu)=\lambda 
$$
if~\eqref{eq:sl-gl-wts1} and \eqref{eq:sl-gl-wts2} have a solution, and 
put $\phi_{n,r}(\mu)=*$ otherwise.   

As far as weight representations are concerned, we can restrict our attention to the 
Beilinson-Lusztig-MacPherson idempotented version of these 
quantum groups, denoted $\Uglaffext$, $\Uglaff$ and $\Uaff$ respectively. 
For each $\lambda\in\bZ^n$ adjoin an idempotent $1_{\lambda}$ to 
$\hat{\mathbf U}_q(\hat{\mathfrak{gl}}_n)$ and add 
the relations
\begin{align*}
1_{\lambda}1_{\mu} &= \delta_{\lambda,\mu}1_{\lambda}   
\\
E_{\pm i}1_{\lambda} &= 1_{\lambda\pm\bar{\alpha}_i}E_{\pm i}
\\
K_i1_{\lambda} &= q^{\lambda_i}1_{\lambda}
\\
R1_{(\lambda_1,\ldots,\lambda_n)}&= 1_{(\lambda_n,\lambda_1,\ldots,\lambda_{n-1})}R.
\end{align*}
\begin{defn} 
\label{defn:Uglndot}
The {\em idempotented extended affine quantum general linear algebra} is defined by 
$$\Uglaffext=\bigoplus_{\lambda,\mu\in\bZ^n}1_{\lambda}\hat{\mathbf U}_q(\hat{\mathfrak{gl}}_n)1_{\mu}.$$
\end{defn}

Of course one defines $\Uglaff\subset\Uglaffext$ as the idempotented subalgebra 
generated by $1_{\lambda}$ and $E_{\pm i}1_{\lambda}$, 
for $i=1,\ldots, n$ and $\lambda\in\mathbb{Z}^n$.  

Similarly for $\mathbf U_q(\hat{\mathfrak{sl}}_n)$, adjoin an idempotent $1_{\lambda}$ 
for each $\lambda\in\bZ^{n-1}$ and add the relations
\begin{align*}
1_{\lambda}1_{\mu} &= \delta_{\lambda,\mu}1_{\lambda}   
\\
E_{\pm i}1_{\lambda} &= 1_{\lambda\pm\overline{\alpha}_i}E_{\pm i}
\\
K_iK^{-1}_{i+1}1_{\lambda} &= q^{\lambda_i}1_{\lambda}.
\end{align*}
\begin{defn} The {\em idempotented quantum special linear algebra} is defined by 
$$\Uaff=\bigoplus_{\lambda,\mu\in\bZ^{n-1}}1_{\lambda}{\mathbf U}_q(\hat{\mathfrak{sl}}_n)1_{\mu}.$$
\end{defn}
Any weight-representation of $\hat{\mathbf U}_q(\hat{\mathfrak{gl}}_n)$, ${\mathbf U}_q(\hat{\mathfrak{gl}}_n)$ or 
${\mathbf U}_q(\hat{\mathfrak{sl}}_n)$ is also a representation of $\Uglaffext$, $\Uglaff$ or $\Uaff$, respectively. This is not true for non-weight representations, of which there are many. There are also other differences of course, e.g. 
$\Uglaffext$, $\Uglaff$ and $\Uaff$ are not unital, because they have 
infinitely many idempotents. For that same reason, they are not bialgebras, 
although their action on tensor products of weight representations is 
well-defined. 
 
\subsection{The affine $q$-Schur algebra}

Let us first recall Green's~\cite{Gr,DG} tensor space and the action of 
$\hat{\mathbf U}_q(\hat{\mathfrak{gl}}_n)$ on it. We will also recall some 
basic results about this action and add some of our own.  
Wherever we omit a proof in this section, the corresponding result was 
taken from~\cite{DG}. When we give a proof, it is because the corresponding result 
cannot be found in the literature and we had to prove it ourselves, e.g. 
the inner product on tensor space is probably known to experts, 
but there seems to be no written reference.  

Let $V$ be the $\Q(q)$-vector space freely generated by 
$\{e_t\mid t\in\mathbb{Z}\}$. 
\begin{defn}
The following defines an action of $\hat{\mathbf U}_q(\hat{\mathfrak{gl}}_n)$ on $V$
\begin{eqnarray}
E_ie_{t+1}=e_t&\text{if}\; i\equiv t\mod n\\
E_ie_{t+1}=0&\text{if}\; i\not\equiv t\mod n\\
E_{-i}e_t=e_{t+1}&\text{if}\; i\equiv t\mod n\\
E_{-i}e_t=0&\text{if}\; i\not\equiv t\mod n\\
K_i^{\pm 1}e_t=q^{\pm 1}e_t&\text{if}\; i\equiv t\mod n\\
K_i^{\pm 1}e_t=e_t&\text{if}\; i\not\equiv t\mod n\\
R^{\pm 1}e_t=e_{t\pm 1}&\text{for all}\; t\in\mathbb{Z}.
\end{eqnarray}
\end{defn} 

Note that $V$ is clearly a weight representation of $\hat{{\mathbf U}}_q(\hat{\mathfrak{gl}}_n)$, with 
$e_t$ having weight $\epsilon_i$, for $i\equiv t \mod n$. Therefore $V$ is also 
a representation of $\Uglaffext$. 

From now on, let $r\in\mathbb{N}_{>0}$ be arbitrary but fixed. 
As usual, one extends the above action to $V^{\otimes r}$, 
using the coproduct in $\hat{\mathbf U}_q(\hat{\mathfrak{gl}}_n)$. Again, this is a weight representation 
and therefore a representation of $\Uglaffext$, which we call Green's {\em tensor space}. 

We also define a $\Q(q)$-bilinear form on $V$ by $\langle e_s,e_t\rangle=\delta_{st}$, which extends to $V^{\otimes r}$ factorwise, i.e.  
$$\langle v_1\otimes \cdots\otimes v_r,w_1\otimes\cdots\otimes w_r\rangle:=\langle v_1,w_1\rangle\cdots\langle v_r,w_r\rangle.$$ 
\begin{lem}
\label{lem:nondeg}
For any $v\in V^{\otimes r}$, we have 
$$\langle v,v\rangle\ne 0.$$
\end{lem}   
\begin{proof}
We can write $v$ uniquely as $\sum_{\underline{t}\in T}v_{\underline{t}}e_{\underline{t}}$, such that   
$T$ is a finite subset of $\mathbb{Z}^r$ and for any 
$\underline{t}=(t_1,\ldots, t_r)\in T$ we have 
$$v_{\underline{t}}\in\Q(q)\quad\text{and}\quad e_{\underline{t}}=e_{t_1}\otimes\cdots\otimes e_{t_r}.$$
Thus we get 
$$\langle v,v\rangle=\sum_{\underline{t}\in T}v_{\underline{t}}^2.$$
For each $\underline{t}\in T$, write $v_{\underline{t}}=f_{\underline{t}}(q)/g_{\underline{t}}(q)$, where $f_{\underline{t}}(q),g_{\underline{t}}(q)
\in\Q[q]$ have g.c.d. equal to 1. Choose $q_0\in\Q$ such 
that $f_{\underline{t}}(q_0)\ne 0, g_{\underline{t}}(q_0)\ne 0$ for all $\underline{t}\in T$ (such a number exists, because $T$ is finite 
and each polynomial has only finitely many roots). Then $v_{\underline{t}}(q_0)=f_{\underline{t}}(q_0)/g_{\underline{t}}(q_0)\in\Q^*$ and 
we have 
$$\sum_{\underline{t}\in T} v_{\underline{t}}(q_0)^2 > 0.$$
This implies that 
$$\sum_{\underline{t}\in T}v_{\underline{t}}^2\ne 0.$$
\end{proof}

There is a right action of $\hat{\He}_{\hat{A}_{r-1}}$ on $V^{\otimes r}$ 
which commutes with the left action of 
$\hat{\mathbf U}_q(\hat{\mathfrak{gl}}_n)$. 
Its precise definition, which can be found 
in~\cite{Gr, DG}, is not relevant here.

\begin{defn} The { \em affine $q$-Schur algebra} $\hat{\SD}(n,r)$ is by definition 
the centralizing algebra $$\mbox{End}_{\hat{\He}_{\hat{A}_{r-1}}}(V^{\otimes r}).$$ 
\end{defn}
By affine Schur-Weyl duality, the image of $\psi_{n,r}\colon \hat{\mathbf U}_q(\hat{\mathfrak{gl}}_n)\to \mbox{End}(V^{\otimes r})$ is 
always isomorphic to $\hat{\SD}(n,r)$. If $n>r$, we can even restrict to ${\mathbf U}_q(\hat{\mathfrak{sl}}_n)\subset 
\hat{\mathbf U}_q(\hat{\mathfrak{gl}}_n)$, i.e.  
$$\psi_{n,r}({\mathbf U}_q(\hat{\mathfrak{sl}}_n))\cong \hat{\SD}(n,r).$$
For $n=r$, this is no longer true.  
 
\begin{defn}\label{defn:rho} Let $\rho\colon \hat{\mathbf U}_q(\hat{\mathfrak{gl}}_n)\to \hat{\mathbf U}_q(\hat{\mathfrak{gl}}_n)$ be the $\Q(q)$-linear 
algebra anti-involution defined by 
$$
\rho(E_i)=qK_iK_{i+1}^{-1}E_{-i}\quad \rho(E_{-i})=qK_i^{-1}K_{i+1}E_i\quad \rho(K_i)=K_i,\quad \rho(R)=R^{-1},
$$
for $1\leq i\leq n$. 
\end{defn}

The proof of the following lemma is a straightforward check, which we leave 
to the reader. 
\begin{lem}\label{lem:rhoDeltacommute} We have 
$$\Delta\rho=(\rho\otimes\rho)\Delta.$$
\end{lem}

\begin{lem}\label{lem:rho} For any $X\in \hat{\mathbf U}_q(\hat{\mathfrak{gl}}_n)$ and any $v,w\in V^{\otimes r}$, we have 
$$\langle Xv,w\rangle=\langle v,\rho(X)w\rangle.$$ 
\end{lem}
\begin{proof} By Lemma~\ref{lem:rhoDeltacommute}, it suffices to check the above for $r=1$ and $v=e_i$ and 
$w=e_j$, for any $i,j\in\mathbb{Z}$. This is straightforward and left to the reader. 
\end{proof}

Note that $\rho$ can also be defined on $\Uglaff$, such that 
$\rho(1_{\lambda})=1_{\lambda}$ for any $\lambda\in\mathbb{Z}^r$, and that 
it descends to $\hat{\SD}(n,r)$.

\subsection{A presentation of $\hat{\SD}(n,r)$ for $n>r$}\label{SchAlg}
In this subsection, let $n>r$. Recall that in this case 
$$\psi_{n,r}\colon \Uglaff\to \mbox{End}(V^{\otimes r})\to \hat{\SD}(n,r)$$ 
is surjective. This gives rise to the following presentation of $\hat{\SD}(n,r)$. The proof can be found 
in~\cite{DG}.

\begin{thm}\cite{DG}\label{thm:presentationnd}
$\hat{\SD}(n,r)$ is isomorphic to the associative unital $\Q(q)$-algebra generated by~$1_{\lambda}$, for~$\lambda\in\Lambda(n,r)$, and~$E_{\pm i}$, for~$i=1,\ldots,n$, 
subject to the relations
\begin{gather}
1_{\lambda}1_{\mu} = \delta_{\lambda,\mu}1_{\lambda} 
\\
\sum_{\lambda\in\Lambda(n,r)}1_{\lambda} = 1
\\
E_{\pm i}1_{\lambda} = 1_{\lambda\pm\bar{\alpha}_i}E_{\pm i}
\\
E_iE_{-j}-E_{-j}E_i = \delta_{ij}\sum\limits_{\lambda\in\Lambda(n,r)} [\lambda_i - \lambda_{i+1}]1_{\lambda}
 \\
E_{\pm i}^2E_{\pm (i\pm 1)}-(q+q^{-1})E_{\pm i}E_{\pm (i\pm 1)}E_{\pm i}+E_{\pm (i\pm 1)}E_{\pm i}^2=0 
\\
E_{\pm i}E_{\pm j}-E_{\pm j}E_{\pm i} =0 \qquad\text{for distant}\;i,j.
\end{gather}
We use the convention that~$1_{\mu}X1_{\lambda}=0$ whenever~$\mu$ or~$\lambda$ is not contained in~$\Lambda(n,r)$. 
Recall that~$[a]$ is the $q$-integer $(q^a-q^{-a})/(q-q^{-1})$.
\end{thm}

We will use signed sequences $\ii=(\mu_1 i_1,\ldots,\mu_m i_m)$, with~$m\in\N$, $\mu_j\in\{\pm 1\}$ and $i_j\in \{1,\ldots,n\}$. 
The set of signed sequences we denote~$\sseq$. For a signed sequence~$\ii=(\mu_1 i_1,\ldots,\mu_m i_m)$ we define 
$$\ii_{\Lambda}:=\mu_1\bar{\alpha}_{i_1} + \dots + \mu_m\bar{\alpha}_{i_m}.$$
We write $E_{\ii}$ for the product~$E_{\mu_1 i_1}\dots E_{\mu_{m} i_{m}}$. For any $\lambda \in \Z^n$ and~$\ii \in \sseq$, we have
$$E_{\ii}1_{\lambda}=1_{\lambda + \ii_{\Lambda}}E_{\ii}.$$

The surjection $\psi_{n,r}\colon \Uaff\to\hat{\SD}(n,r)$ can also be given 
explicitly in terms of the generators in Theorem~\ref{thm:presentationnd}. 
For any $\lambda\in\mathbb{Z}^{n-1}$, we have 
\begin{equation}
\label{eq:psi}
\psi_{n,r}(E_{\pm i}1_{\lambda})=E_{\pm i}1_{\phi_{n,r}(\lambda)},
\end{equation}
where $\phi_{n,r}\colon \mathbb{Z}^{n-1}\to \Lambda(n,r)\cup \{*\}$ is the map defined in~\ref{genalg}. 
By convention, we put $1_{*}=0$. 

Recall the definition of $\rho$ in~\ref{defn:rho}. Since $\rho$ is an algebra anti-homomorphism, 
we get 
$$\rho\colon 1_{\lambda}\hat{\SD}(n,r)1_{\mu}\to 1_{\mu}\hat{\SD}(n,r)1_{\lambda},$$
for any $\lambda,\mu\in\Lambda(n,r)$. 

\begin{lem}
\label{lem:non-deg}
For any $\lambda,\mu\in\Lambda(n,r)$ and any non-zero element 
$1_{\lambda}X1_{\mu}\in \hat{\SD}(n,r)$, we have  
$$1_{\lambda}X\rho(X)1_{\lambda}\ne 0\quad\text{and}\quad 1_{\mu}\rho(X)X1_{\mu}\ne 0.$$
\end{lem}
\begin{proof}
By the surjectivity of $\psi_{n,r}$, we know that there exist 
$e_{t_1},\ldots, e_{t_r}\in V$ such that 
$$X1_{\mu}(e_{t_1}\otimes\cdots\otimes e_{t_r})\ne 0.$$
By Lemmas~\ref{lem:rho} and~\ref{lem:nondeg}, we get 
\begin{eqnarray*}
\langle e_{t_1}\otimes\cdots\otimes e_{t_r},1_{\mu}\rho(X)X1_{\mu}(e_{t_1}\otimes\cdots\otimes e_{t_r})\rangle&=&\\
\langle X1_{\mu}(e_{t_1}\otimes\cdots\otimes e_{t_r}),X1_{\mu}(e_{t_1}\otimes\cdots\otimes e_{t_r})\rangle&\ne& 0.
\end{eqnarray*}
Therefore, we see that 
$$\psi_{n,r}(1_{\mu}\rho(X)X1_{\mu})\ne 0.$$

The other case follows automatically, because $X\rho(X)=\rho^2(X)\rho(X)=\rho(Y)Y$ for 
$Y=\rho(X)$.
\end{proof}

We can also give an explicit formula for the well-known embedding (see~\cite{DG}) 
of $\hat{\He}_{\hat{A}_{r-1}}$ into $\hat{\SD}(n,r)$. Let $1_r=1_{(1^r)}$. 
We define the following map 
$$\sigma_{n,r}\colon\hat{\He}_{\hat{A}_{r-1}}\to 1_{r}\hat{\SD}(n,r)1_{r}$$
by 
$$\sigma_{n,r}(b_i)=1_rE_{-i}E_i1_r=1_rE_iE_{-i}1_r,$$
for $i=1,\ldots,r-1$,
$$\sigma_{n,r}(b_r)=1_rE_{-n}\dots E_{-r} E_{r}\dots E_{n} 1_r,$$
\begin{eqnarray*}
 \sigma_{n,r}(T_{\rho})& = & 1_rE_{-n} \dots E_{-r-1} E_{-1}\dots E_{-r}1_r \\
     & (= & 1_rE_{-n}E_{-1}\dots E_{-r+1} E_{-n+1}\dots E_{-r}1_r)
\end{eqnarray*}
and
\begin{eqnarray*}
 \sigma_{n,r}(T_{\rho^{-1}}) & = & 1_rE_{r} \dots E_{1} E_{r+1}\dots E_{n}1_r \\
  & ( = & 1_rE_{r}\dots E_{n-1} E_{r-1} \dots E_{1}E_{n}1_r).
\end{eqnarray*}
It is easy to check that $\sigma_{n,r}$ is well-defined. It turns out that $\sigma_{n,r}$ is actually an isomorphism, which induces the 
affine $q$-Schur functor 
$$\hat{\SD}(n,r)-\mbox{mod}\to \hat{\He}_{\hat{A}_{r-1}}-\mbox{mod}$$ 
between the categories of finite-dimensional modules of the extended 
affine Hecke algebra and of the affine $q$--Schur algebra. 
This functor is an equivalence (see Theorem 4.1.3 in~\cite{DDF}, for example).

The following result will be needed in Section~\ref{GrAlg}.

\begin{prop}\label{alg} Let $n>r$. Suppose that $A$ is a $\Q(q)$-algebra and 
$$f\colon \hat{\SD}(n,r)\to A$$ 
is a surjective $\Q(q)$-algebra homomorphism which is an embedding 
when restricted to $1_r \hat{\SD}(n,r)1_r\cong \hat{\He}_{\hat{A}_{r-1}}$. 
Then $f$ is a $\Q(q)$-algebra isomorphism
$$A\cong \hat{\SD}(n,r).$$  
\end{prop}

\begin{proof}
We first prove that $f$ is an embedding when restricted to 
$1_r \hat{\SD}(n,r)1_{\lambda}$ and $1_{\lambda}\hat{\SD}(n,r)1_r$, 
for any $\lambda\in\Lambda(n,r)$. Suppose that this is not true in the first 
case, then there exists a non-zero element 
$1_rX1_{\lambda}\in 1_r\hat{\SD}(n,r)1_{\lambda}$ in the kernel of $f$. By 
Lemma~\ref{lem:non-deg}, we have 
$1_rX\rho(X)1_r\ne 0$. However, we have 
$$f(1_rX\rho(X)1_r)=f(1_rX1_{\lambda})f(1_{\lambda}\rho(X)1_r)=0,$$
which leads to a contradiction, because by hypothesis 
$f$ is an embedding when restricted to $1_r\hat{\SD}(n,r)1_r$.
The second case can be proved similarly. 

Now, for any $\lambda,\mu\in\Lambda(n,r)$, let 
$1_{\lambda}X1_{\mu}\in\hat{\SD}(n,r)$ be an arbitrary non-zero element. 
By Schur-Weyl duality, we have 
$$1_{\lambda}\hat{\SD}(n,r)1_{\mu}\cong 
\text{Hom}_{\hat{\He}_{\hat{A}_{r-1}}}(1_{\mu}\hat{\SD}(n,r)1_{r},1_{\lambda}\hat{\SD}(n,r)1_{r}),$$
where the isomorphism is induced by left composition. Therefore, there exists 
an element $1_{\mu}Y1_r$ such that 
$$1_{\lambda}XY1_r\ne 0.$$
By the above, we have 
$$f(1_{\lambda}X1_{\mu})f(1_{\mu}Y1_{1_r})=f(1_{\lambda}XY1_r)\ne 0,$$
so $$f(1_{\lambda}X1_{\mu})\ne 0.$$ 
This shows that $f$ is an embedding when restricted to 
$1_{\lambda}\hat{\SD}(n,r)1_{\mu}$. Since $\lambda$ and $\mu$ were arbitrary, 
this shows that $f$ is an isomorphism. 
\end{proof}

Let us end this section giving an embedding between affine 
$q$--Schur algebras which we will use in Section \ref{2rep}.
\begin{prop} The $\Q(q)$-linear algebra homomorphism 
$$\iota_{n}\colon \hat{\SD}(n,r)\to \hat{\SD}(n+1,r)$$ 
defined by
\begin{eqnarray}
1_{\lambda}&\mapsto& 1_{(\lambda,0)}\\
E_{\pm i}1_{\lambda}&\mapsto& E_{\pm i}1_{(\lambda,0)}\\
E_n1_{\lambda}&\mapsto &E_nE_{n+1}1_{(\lambda,0)}\\
E_{-n}1_{\lambda}&\mapsto &E_{-(n+1)}E_{-n}1_{(\lambda,0)},
\end{eqnarray}
for any $1\leq i\leq n-1$ and $\lambda\in\Lambda(n,r)$, 
is an embedding and gives an isomorphism of algebras
$$\hat{\SD}(n,r)\cong \bigoplus_{\lambda,\mu\in\Lambda(n,r)}1_{(\lambda,0)}\hat{\SD}(n+1,r)1_{(\mu,0)}.$$
\end{prop}

\subsection{The 2-categories $\glcat_{[y]}$ and $\Scat(n,r)_{[y]}$}\label{sec:scat}                                      
 
In this section we define three $2$-categories, 
$\Ucataffy$, $\Uglcataffy$ and~$\Scat(n,r)_{[y]}$, using a graphical calculus 
analogous to Khovanov and Lauda's in~\cite{KL3}. 

\subsubsection{The 2-category $\glcat_{[y]}$}\label{sec:glcat} 
The $2$-category $\Ucataffy$ is defined just as $\Ucataff$ in~\cite{KL3}, 
but the $2$-HOM-spaces are tensored with $\mathbb{Q}[y]$, where $y$ is a formal variable of degree 
two, and some of the KL-relations are $y$-deformed, as shown below. 
In order to define $\Uglcataffy$, change the weights 
in the definition of $\Ucataffy$ into (degenerate) level-zero 
$\hat{\mathfrak{gl}}_n$--weights (i.e. $\mathfrak{gl}_n$--weights).
The $2$-category $\Scat(n,r)_{[y]}$ is then defined 
as a quotient of $\Uglcataffy$. This is 
precisely the affine analogue of what was done in~\cite{MSVschur}. 

\begin{rem} We use the sign conventions from~\cite{MSVschur} in the 
relations on $2$-morphisms, which differ 
from Khovanov and Lauda's sign conventions. For more details on this 
change of convention, see below. 
\end{rem}

\begin{rem}
We do not prove that the 2-category~$\Uglcataffy$ 
provides a categorification of~${\mathbf U}_q(\hat{\mathfrak{gl}}_n)$, 
although we conjecture 
that it does. We will prove that $\Scat(n,r)_{[y]}$ categorifies 
$\hat{\SD}(n,r)$. 
\end{rem}

In order to avoid giving several long definitions which are very similar, we 
only define $\Uglcataffy$ in detail. The $2$-category 
$\Ucataffy$ is defined exactly as $\Uglcataffy$, but using degenerate 
level-zero $\hat{\mathfrak{sl}}_n$-weights. The $2$-category 
$\Scat(n,r)_{[y]}$ is defined as a quotient of $\Uglcataffy$. We will show 
that $\Scat(n,r)_{[y]}$ can also be obtained as a quotient of $\Ucataffy$. 

To be more precise, we first define the $\Q[y]$-linear graded $2$-category 
with translation $\Uglcataffy^*$, whose $2$-morphisms are 
$\Q[y]$-linear combinations of homogeneous 
$2$-morphisms of various degrees. 
The $\Q$-linear $2$-category $\Uglcataffy$ is then obtained by restricting to 
the degree-zero $2$-morphisms. 

\begin{defn} \label{def_glcat} The additive $\Q[y]$-linear graded 2-category 
with translation $\Uglcataffy^*$ consists of

\n $\bullet$ objects:  $\lambda \in \Z^n$.

\n The hom-category $\Uglcataffy^*(\lambda,\lambda')$ between two objects~$\lambda$, $\lambda'$ is an additive $\Q[y]$-linear graded category with 
translation defined by:

\n $\bullet$ objects\footnote{We refer to objects of the category
$\Uglcataffy^*(\lambda,\lambda')$ as 1-morphisms of $\Uglcataffy^*$.  
Likewise, the morphisms of
$\Uglcataffy^*(\lambda,\lambda')$ are called 2-morphisms in $\Uglcataffy^*$. } 
of $\Uglcataffy^*(\lambda,\lambda')$. A 1-morphism in~$\Uglcataffy^*$ 
from~$\lambda$ to~$\lambda'$
is a formal finite direct sum of 1-morphisms
$$ \cal{E}_{\ii} \onel\{t\} = \onelp \cal{E}_{\ii} \onel\{t\}
:= \cal{E}_{\mu_1 i_1}\dotsm\cal{E}_{\mu_m i_m} \onel\{t\}
$$
for any~$t\in \Z$ and signed sequence~$\ii\in\sseq$ such that~$\lambda'=\lambda+\ii_{\Lambda}$ and~$\lambda$, $\lambda'\in\bZ^n$. 

\n The morphisms of~$\Uglcataffy^*(\lambda,\lambda')$ are presented by generators and relations. Multiplication by $y$ is indicated graphically by 
$\bbox{y}$ in the diagrams below. 

\n $\bullet$ generators. For 1-morphisms~$\cal{E}_{\ii} \onel\{t\}$
and~$\cal{E}_{\jj} \onel\{t'\}$ in~$\Uglcataffy^*$, the 
hom-sets~$\Uglcataffy^*(\cal{E}_{\ii} \onel\{t\},\cal{E}_{\jj} \onel\{t'\})$ of~$\Uglcataffy^*(\lambda,\lambda')$ are graded 
$\Q[y]$-vector spaces given by linear combinations of 
diagrams of homogeneous degrees, modulo 
certain relations, built from composites of:
\begin{enumerate}[i)]
   \item  Degree-zero identity 2-morphisms $1_x$ for each 1-morphism $x$ in
$\Uglcataffy^*$; in particular for any $i \in \{1, \dots,n\}$, $t\in \Z$ and $\lambda \in \Z^n$, the identity 2-morphisms $1_{\cal{E}_{i} \onel}\{t\}$ and
$1_{\cal{E}_{-i} \onel}\{t\}$ are represented graphically by
\[
\begin{array}{ccc}
  1_{\cal{E}_{i} \onel\{t\}} &\quad  & 1_{\cal{E}_{-i} \onel\{t\}} \\ \\
   \xy
 (0,0)*{\dblue\xybox{(0,8);(0,-8); **\dir{-} ?(.5)*\dir{>}+(2.3,0)*{\scriptstyle{}};}};
 (-1,-11)*{ i};(-1,11)*{ i};
 (6,2)*{ \lambda};
 (-8,2)*{ \lambda +i_{\Lambda}};
 (-10,0)*{};(10,0)*{};
 \endxy
 & &
 \;\;   
   \xy
 (0,0)*{\dblue\xybox{(0,8);(0,-8); **\dir{-} ?(.5)*\dir{<}+(2.3,0)*{\scriptstyle{}};}};
 (-1,-11)*{ i};(-1,11)*{ i};
 (6,2)*{ \lambda};
 (-8.5,2)*{ \lambda -i_{\Lambda}};
 (-12,0)*{};(12,0)*{};
 \endxy
\\ \\
   \;\;\text{ {\rm deg} 0}\;\;
 & &\;\;\text{ {\rm deg} 0}\;\;
\end{array}
\]

More generally, for a signed sequence~$\ii=(\mu_1 i_1, \mu _2i_2, \dots \mu_m i_m)$, the identity~$1_{\cal{E}_{\ii} \onel\{t\}}$ 2-morphism is
represented as
\begin{equation*}
\begin{array}{ccc}
  \xy
 (-12,0)*{\dblue\xybox{(-12,8);(-12,-8); **\dir{-};}};
 (-4,0)*{\dred\xybox{(-4,8);(-4,-8); **\dir{-};}};
 (4,0)*{\cdots};
 (12,0)*{\dgreen\xybox{(12,8);(12,-8); **\dir{-};}};
 (-12,11)*{i_1}; (-4,11)*{ i_2};(12,11)*{ i_m };
  (-12,-11)*{ i_1}; (-4,-11)*{ i_2};(12,-11)*{ i_m};
 (18,2)*{ \lambda}; (-20,2)*{ \lambda+\ii_{\Lambda}};
 \endxy
\end{array}
\end{equation*}
where the strand labeled~$i_{k}$ is oriented up if~$\mu_{k}=+$ and oriented down if $\mu_{k}=-$. We will often place labels on the side of a strand and omit the labels at the top and bottom. The signed sequence can be recovered from the labels and the orientations on the strands. We might also forget the object on the left of the diagram which can be recovered from the object on the right and the signed sequence corresponding to the diagram.

\item For any $\lambda \in \bZ^n$ the 2-morphisms 
\[
\begin{tabular}{|l|c|c|c|c|}
\hline
 {\bf Notation:} \xy (0,-5)*{};(0,7)*{}; \endxy&
 $\dblue\Uup_{\black i,\lambda}$  &  $\dblue\Udown_{\black i,\lambda}$  
 &$\Ucrossij_{i,j,\lambda}$
 &$\Ucrossdij_{i,j,\lambda}$  \\
 \hline
 {\bf 2-morphism:} &   \xy
 (0,0)*{\dblue\xybox{(0,7);(0,-7); **\dir{-} ?(.75)*\dir{>}+(2.3,0)*{\scriptstyle{}}
 ?(.1)*\dir{ }+(2,0)*{\black \scs i};
 (0,-2)*{\txt\large{$\bullet$}};}};
 (4,4)*{ \lambda};
 (-8,4)*{ \lambda +i_{\Lambda}};
 (-10,0)*{};(10,0)*{};
 \endxy
 &
     \xy
 (0,0)*{\dblue\xybox{(0,7);(0,-7); **\dir{-} ?(.75)*\dir{<}+(2.3,0)*{\scriptstyle{}}
 ?(.1)*\dir{ }+(2,0)*{\black\scs i};
 (0,-2)*{\txt\large{$\bullet$}};}};
 (-6,4)*{ \lambda};
 (8,3.9)*{ \lambda +i_{\Lambda}};
 (-10,0)*{};(10,9)*{};
 \endxy
 &
   \xy
  (0,0)*{\xybox{
    (0,0)*{\dblue\xybox{(-4,-4)*{};(4,4)*{} **\crv{(-4,-1) & (4,1)}?(1)*\dir{>} ;}};
    (0,0)*{\dgreen\xybox{(4,-4)*{};(-4,4)*{} **\crv{(4,-1) & (-4,1)}?(1)*\dir{>};}};
    (-5,-3)*{\scs i};
     (5.1,-3)*{\scs j};
     (8,1)*{ \lambda};
     (-12,0)*{};(12,0)*{};
     }};
  \endxy
 &
   \xy
  (0,0)*{\xybox{
    (0,0)*{\dgreen\xybox{(-4,4)*{};(4,-4)*{} **\crv{(-4,1) & (4,-1)}?(1)*\dir{>} ;}};
    (0,0)*{\dblue\xybox{(4,4)*{};(-4,-4)*{} **\crv{(4,1) & (-4,-1)}?(1)*\dir{>};}};
    (-6,-3)*{\scs i};
     (6,-3)*{\scs j};
     (8,1)*{ \lambda};
     (-12,0)*{};(12,0)*{};
     }};
  \endxy
\\ & & & &\\
\hline
 {\bf Degree:} & \;\;\text{  $i\cdot i$ }\;\;
 &\;\;\text{  $i\cdot i$}\;\;& \;\;\text{  $-i \cdot j$}\;\;
 & \;\;\text{  $-i \cdot j$}\;\; \\
 \hline
\end{tabular}
\]

\[
\begin{tabular}{|l|c|c|c|c|}
\hline
  {\bf Notation:} \xy (0,-5)*{};(0,7)*{}; \endxy
 & \text{$\Ucupri_{i,\lambda}$} 
 & \text{$\Ucupli_{i,\lambda}$} 
 & \text{$\Ucapli_{i,\lambda}$} 
 & \text{$\Ucapri_{i,\lambda}$} \\
 \hline
  {\bf 2-morphism:} & \xy
    (0,-3)*{\dblue\bbpef{\black i}};
    (8,-3)*{ \lambda};
    (-12,0)*{};(12,0)*{};
    \endxy
  & \xy
    (0,-3)*{\dblue\bbpfe{\black i}};
    (8,-3)*{ \lambda};
    (-12,0)*{};(12,0)*{};
    \endxy
  & \xy
    (0,0)*{\dblue\bbcef{\black i}};
    (8,3)*{ \lambda};
    (-12,0)*{};(12,0)*{};
    \endxy 
  & \xy
    (0,0)*{\dblue\bbcfe{\black i}};
    (8,3)*{ \lambda};
    (-12,0)*{};(12,0)*{};(8,8)*{};
    \endxy\\& & &  &\\ \hline
 {\bf Degree:} & \;\;\text{  $1+\llambda_i$}\;\;
 & \;\;\text{ $1-\llambda_i$}\;\;
 & \;\;\text{ $1+\llambda_i$}\;\;
 & \;\;\text{  $1-\llambda_i$}\;\;
 \\
 \hline
\end{tabular}
\]
\end{enumerate}
where the degrees are given by the symmetric $\Z$-valued bilinear form on $\C \{1,\ldots,n\}$  
\begin{equation*}
i\cdot j = 
\begin{cases}
 2 & \mbox{if} \ i=j \\
-1 & \mbox{if} \ i\equiv j\pm 1\mod n \\
 0 & \mbox{if} \ i\not\equiv j\pm 1\mod n .
\end{cases}
\end{equation*}

\n $\bullet$ relations:

\n $\star$ Biadjointness and cyclicity:
\begin{enumerate}[i)]
\item\label{it:sl2i}  $\mathbf{1}_{\lambda+i_{\Lambda}}\cal{E}_{i}\onel$ and
$\onel\cal{E}_{-i}\mathbf{1}_{\lambda+i_{\Lambda}}$ are biadjoint, up to grading shifts:
\begin{equation} \label{eq_biadjoint1}
\text{$
  \xy   0;/r.18pc/:
    (0,0)*{\dblue\xybox{
    (-8,0)*{}="1";
    (0,0)*{}="2";
    (8,0)*{}="3";
    (-8,-10);"1" **\dir{-};
    "1";"2" **\crv{(-8,8) & (0,8)} ?(0)*\dir{>} ?(1)*\dir{>};
    "2";"3" **\crv{(0,-8) & (8,-8)}?(1)*\dir{>};
    "3"; (8,10) **\dir{-};}};
    (12,-9)*{\lambda};
    (-6,9)*{\lambda+i_{\Lambda}};
    (-10,-8)*{\scs i};
    \endxy
    \quad =
    \quad
\xy   0;/r.18pc/:
    (0,0)*{\dblue\xybox{
    (-8,0)*{}="1";
    (0,0)*{}="2";
    (8,0)*{}="3";
    (0,-10);(0,10)**\dir{-} ?(.5)*\dir{>};}};
    (5,8)*{\lambda};
    (-9,8)*{\lambda+i_{\Lambda}};
    (-2,-8)*{\scs i};
    \endxy
    \quad =
    \quad
\xy   0;/r.18pc/:
    (0,0)*{\dblue\xybox{
    (8,0)*{}="1";
    (0,0)*{}="2";
    (-8,0)*{}="3";
    (8,-10);"1" **\dir{-};
    "1";"2" **\crv{(8,8) & (0,8)} ?(0)*\dir{>} ?(1)*\dir{>};
    "2";"3" **\crv{(0,-8) & (-8,-8)}?(1)*\dir{>};
    "3"; (-8,10) **\dir{-};}};
    (12,9)*{\lambda};
    (-5,-9)*{\lambda+i_{\Lambda}};
    (10,-8)*{\scs i};
    \endxy
$}
\end{equation}

\begin{equation}\label{eq_biadjoint2}
\text{$
 \xy   0;/r.18pc/:
    (0,0)*{\dblue\xybox{
    (-8,0)*{}="1";
    (0,0)*{}="2";
    (8,0)*{}="3";
    (-8,-10);"1" **\dir{-};
    "1";"2" **\crv{(-8,8) & (0,8)} ?(0)*\dir{<} ?(1)*\dir{<};
    "2";"3" **\crv{(0,-8) & (8,-8)}?(1)*\dir{<};
    "3"; (8,10) **\dir{-};}};
    (12,-9)*{\lambda+i_{\Lambda}};
    (-6,9)*{ \lambda};
    (-10,-8)*{\scs i};
    \endxy
    \quad =
    \quad
\xy   0;/r.18pc/:
    (0,0)*{\dblue\xybox{
    (-8,0)*{}="1";
    (0,0)*{}="2";
    (8,0)*{}="3";
    (0,-10);(0,10)**\dir{-} ?(.5)*\dir{<};}};
    (9,8)*{\lambda+i_{\Lambda}};
    (-5,8)*{ \lambda};
    (-2,-8)*{\scs i};
    \endxy
    \quad =
    \quad
 \xy  0;/r.18pc/:
    (0,0)*{\dblue\xybox{
    (8,0)*{}="1";
    (0,0)*{}="2";
    (-8,0)*{}="3";
    (8,-10);"1" **\dir{-};
    "1";"2" **\crv{(8,8) & (0,8)} ?(0)*\dir{<} ?(1)*\dir{<};
    "2";"3" **\crv{(0,-8) & (-8,-8)}?(1)*\dir{<};
    "3"; (-8,10) **\dir{-};}};
    (12,9)*{\lambda+i_{\Lambda}};
    (-6,-9)*{ \lambda};
    (10,-8)*{\scs i};
    \endxy
$}
\end{equation}

\item As well for dotted lines:
\begin{equation} \label{eq_cyclic_dot}
\text{$
    \xy
    (0,0)*{\dblue\xybox{
    (-8,5)*{}="1";
    (0,5)*{}="2";
    (0,-5)*{}="2'";
    (8,-5)*{}="3";
    (-8,-10);"1" **\dir{-};
    "2";"2'" **\dir{-} ?(.5)*\dir{<};
    "1";"2" **\crv{(-8,12) & (0,12)} ?(0)*\dir{<};
    "2'";"3" **\crv{(0,-12) & (8,-12)}?(1)*\dir{<};
    "3"; (8,10) **\dir{-};
    (0,4)*{\txt\large{$\bullet$}};}};
    (15,-9)*{ \lambda+i_{\Lambda}};
    (-12,9)*{\lambda};
    (10,8)*{\scs };
    (-10,-8)*{\scs i};
    \endxy
    \quad = \quad
      \xy
 (0,0)*{\dblue\xybox{
 (0,10);(0,-10); **\dir{-} ?(.75)*\dir{<}+(2.3,0)*{\scriptstyle{}}
 ?(.1)*\dir{ }+(2,0)*{\scs };
 (0,0)*{\txt\large{$\bullet$}};}};
 (-6,5)*{ \lambda};
 (8,5)*{ \lambda +i_{\Lambda}};
 (-10,0)*{};(10,0)*{};(-2,-8)*{\scs i};
 \endxy
    \quad = \quad
    \xy
    (0,0)*{\dblue\xybox{
    (8,5)*{}="1";
    (0,5)*{}="2";
    (0,-5)*{}="2'";
    (-8,-5)*{}="3";
    (8,-10);"1" **\dir{-};
    "2";"2'" **\dir{-} ?(.5)*\dir{<};
    "1";"2" **\crv{(8,12) & (0,12)} ?(0)*\dir{<};
    "2'";"3" **\crv{(0,-12) & (-8,-12)}?(1)*\dir{<};
    "3"; (-8,10) **\dir{-};
    (0,4)*{\txt\large{$\bullet$}};}};
    (15,9)*{\lambda+i_{\Lambda}};
    (-12,-9)*{\lambda};
    (-10,8)*{\scs };
    (10,-8)*{\scs i};
    \endxy
$}
\end{equation}
\item All 2-morphisms are cyclic with respect to the above biadjoint
   structure. This is ensured by the relations \eqref{eq_biadjoint1}--\eqref{eq_cyclic_dot}, and the
   relations
\begin{equation} \label{eq_cyclic_cross-gen}
\text{$
\xy 0;/r.19pc/:
(0,0)*{\dred\xybox{
     (-4,-4)*{};(4,4)*{} **\crv{(-4,-1) & (4,1)}?(1)*\dir{>};
     (-4,-4)*{};(18,-4)*{} **\crv{(-4,-16) & (18,-16)} ?(1)*\dir{<}?(0)*\dir{<};
     (18,-4);(18,12) **\dir{-};
     (4,4)*{};(-18,4)*{} **\crv{(4,16) & (-18,16)} ?(1)*\dir{>};  
     (-18,4);(-18,-12) **\dir{-}; 
}};
(0,0)*{\dblue\xybox{
     (4,-4)*{};(-4,4)*{} **\crv{(4,-1) & (-4,1)};
     (12,-4);(12,12) **\dir{-};
     (-12,4);(-12,-12) **\dir{-};
     (4,-4)*{};(12,-4)*{} **\crv{(4,-10) & (12,-10)}?(1)*\dir{<}?(0)*\dir{<};
      (-4,4)*{};(-12,4)*{} **\crv{(-4,10) & (-12,10)}?(1)*\dir{>}?(0)*\dir{>};
}};
     (8,1)*{ \lambda};
     (-10,0)*{};(10,0)*{};
      (20,11)*{\scs j};(10,11)*{\scs i};
      (-20,-11)*{\scs j};(-10,-11)*{\scs i};
   %  }};
  \endxy
\quad =  \quad \xy
(0,0)*{\dred\xybox{
     (-4,-4)*{};(4,4)*{} **\crv{(-4,-1) & (4,1)}?(0)*\dir{<};}};
(0,0)*{\dblue\xybox{
     (4,-4)*{};(-4,4)*{} **\crv{(4,-1) & (-4,1)}?(0)*\dir{<};}};
     (-5,3)*{\scs i};
     (5.1,3)*{\scs j};
     (-8,0)*{ \lambda};
     (-12,0)*{};(12,0)*{};
  \endxy \quad =  \quad
 \xy 0;/r.19pc/:
(0,0)*{\dblue\xybox{
      (4,-4)*{};(-4,4)*{} **\crv{(4,-1) & (-4,1)}?(1)*\dir{>};
      (-4,4)*{};(18,4)*{} **\crv{(-4,16) & (18,16)} ?(1)*\dir{>};
      (4,-4)*{};(-18,-4)*{} **\crv{(4,-16) & (-18,-16)} ?(1)*\dir{<}?(0)*\dir{<};
      (18,4);(18,-12) **\dir{-};
      (-18,-4);(-18,12) **\dir{-};
}};
(0,0)*{\dred\xybox{
      (-4,-4)*{};(4,4)*{} **\crv{(-4,-1) & (4,1)}; 
      (12,4);(12,-12) **\dir{-};
      (-10,0)*{};(10,0)*{};
      (-4,-4)*{};(-12,-4)*{} **\crv{(-4,-10) & (-12,-10)}?(1)*\dir{<}?(0)*\dir{<};
      (4,4)*{};(12,4)*{} **\crv{(4,10) & (12,10)}?(1)*\dir{>}?(0)*\dir{>};
      (-12,-4);(-12,12) **\dir{-};
}};
  (8,1)*{ \lambda};
 (-20,11)*{\scs i};(-10,11)*{\scs j};
  (20,-11)*{\scs i};(10,-11)*{\scs j};
  \endxy
$}
\end{equation}
The cyclic condition on 2-morphisms expressed by \eqref{eq_biadjoint1}--\eqref{eq_cyclic_cross-gen} ensures that diagrams related by isotopy represent
the same 2-morphism in $\Uglcataffy$.

It will be convenient to introduce degree zero 2-morphisms:
\begin{equation} \label{eq_crossl-gen}
\text{$
  \xy
(0,0)*{\dred\xybox{
    (-4,-4)*{};(4,4)*{} **\crv{(-4,-1) & (4,1)}?(1)*\dir{>};}};
(0,0)*{\dblue\xybox{
    (4,-4)*{};(-4,4)*{} **\crv{(4,-1) & (-4,1)}?(0)*\dir{<};}};
    (-5,-3)*{\scs j};
     (-5,3)*{\scs i};
     (8,2)*{ \lambda};
     (-12,0)*{};(12,0)*{};
  \endxy
:=
 \xy 0;/r.19pc/:
(0,0)*{\dred\xybox{
    (4,-4)*{};(-4,4)*{} **\crv{(4,-1) & (-4,1)}?(1)*\dir{>};
    (-4,4);(-4,12) **\dir{-}; 
    (4,-4);(4,-12) **\dir{-};
}};
(0,0)*{\dblue\xybox{
    (-4,-4)*{};(4,4)*{} **\crv{(-4,-1) & (4,1)};
    (-12,-4);(-12,12) **\dir{-};
    (12,4);(12,-12) **\dir{-};
    (-10,0)*{};(10,0)*{};
    (-4,-4)*{};(-12,-4)*{} **\crv{(-4,-10) & (-12,-10)}?(1)*\dir{<}?(0)*\dir{<};
    (4,4)*{};(12,4)*{} **\crv{(4,10) & (12,10)}?(1)*\dir{>}?(0)*\dir{>};
}};
    (16,1)*{\lambda};
    (-14,11)*{\scs i};(-2,11)*{\scs j};
    (14,-11)*{\scs i};(2,-11)*{\scs j};
 \endxy
  \quad = \quad
  \xy 0;/r.19pc/:
(0,0)*{\dblue\xybox{
    (-4,-4)*{};(4,4)*{} **\crv{(-4,-1) & (4,1)}?(1)*\dir{<};
    (4,4);(4,12) **\dir{-};(-4,-4);(-4,-12) **\dir{-};
}};
(0,0)*{\dred\xybox{
    (4,-4)*{};(-4,4)*{} **\crv{(4,-1) & (-4,1)};
     (12,-4);(12,12) **\dir{-};
     (-12,4);(-12,-12) **\dir{-};
     (10,0)*{};(-10,0)*{};
     (4,-4)*{};(12,-4)*{} **\crv{(4,-10) & (12,-10)}?(1)*\dir{>}?(0)*\dir{>};
      (-4,4)*{};(-12,4)*{} **\crv{(-4,10) & (-12,10)}?(1)*\dir{<}?(0)*\dir{<};
}};
     (16,1)*{\lambda};
     (14,11)*{\scs j};(2,11)*{\scs i};
     (-14,-11)*{\scs j};(-2,-11)*{\scs i};
  \endxy
$}
\end{equation}
\begin{equation} \label{eq_crossr-gen}
\text{$
  \xy
(0,0)*{\dred\xybox{
    (-4,-4)*{};(4,4)*{} **\crv{(-4,-1) & (4,1)}?(0)*\dir{<};}};
(0,0)*{\dblue\xybox{
    (4,-4)*{};(-4,4)*{} **\crv{(4,-1) & (-4,1)}?(1)*\dir{>};}};
    (5.1,-3)*{\scs i};
     (5.1,3)*{\scs j};
     (-8,2)*{ \lambda};
     (-12,0)*{};(12,0)*{};
  \endxy
:=
 \xy 0;/r.19pc/:
(0,0)*{\dblue\xybox{
    (-4,-4)*{};(4,4)*{} **\crv{(-4,-1) & (4,1)}?(1)*\dir{>};
    (4,4);(4,12) **\dir{-};
    (-4,-4);(-4,-12) **\dir{-};
}};
(0,0)*{\dred\xybox{
     (4,-4)*{};(-4,4)*{} **\crv{(4,-1) & (-4,1)};
     (12,-4);(12,12) **\dir{-};
     (-12,4);(-12,-12) **\dir{-};
     (10,0)*{};(-10,0)*{};
     (4,-4)*{};(12,-4)*{} **\crv{(4,-10) & (12,-10)}?(1)*\dir{<}?(0)*\dir{<};
     (-4,4)*{};(-12,4)*{} **\crv{(-4,10) & (-12,10)}?(1)*\dir{>}?(0)*\dir{>};
}};
     (-16,1)*{\lambda};
     (14,11)*{\scs j};(2,11)*{\scs i};
     (-14,-11)*{\scs j};(-2,-11)*{\scs i};
\endxy
  \quad = \quad
  \xy 0;/r.19pc/:
(0,0)*{\dred\xybox{
     (4,-4)*{};(-4,4)*{} **\crv{(4,-1) & (-4,1)}?(1)*\dir{<};
     (-4,4);(-4,12) **\dir{-};     (4,-4);(4,-12) **\dir{-};
}};
(0,0)*{\dblue\xybox{
     (-4,-4)*{};(4,4)*{} **\crv{(-4,-1) & (4,1)};
     (-12,-4);(-12,12) **\dir{-};
     (12,4);(12,-12) **\dir{-};
     (-10,0)*{};(10,0)*{};
     (-4,-4)*{};(-12,-4)*{} **\crv{(-4,-10) & (-12,-10)}?(1)*\dir{>}?(0)*\dir{>};
     (4,4)*{};(12,4)*{} **\crv{(4,10) & (12,10)}?(1)*\dir{<}?(0)*\dir{<};
}};
     (-16,1)*{\lambda};
     (-14,11)*{\scs i};(-2,11)*{\scs j};
     (14,-11)*{\scs i};(2,-11)*{\scs j};
  \endxy
$}
\end{equation}
where the second equality in \eqref{eq_crossl-gen} and \eqref{eq_crossr-gen}
follow from \eqref{eq_cyclic_cross-gen}. 
\end{enumerate}

\n $\star$ Bubble relations:
\begin{enumerate}[i)]
\item All dotted bubbles of negative degree are zero. That is,
\begin{equation} \label{eq_positivity_bubbles}
 \xy
 (-12,0)*{\dblue\cbub{\black m}{\black i}};
 (-8,8)*{\lambda};
 \endxy
  = 0
 \qquad
  \text{if $m<\llambda_i-1$} \qquad
 \xy
 (-12,0)*{\dblue\ccbub{\black m}{\black i}};
 (-8,8)*{\lambda};
 \endxy = 0\quad
  \text{if $m< -\llambda_i-1$}
\end{equation}
for all $m \in \Z_+$, where a dot carrying a label $m$ denotes the
$m$-fold iterated vertical composite of $\Uup_{i,\lambda}$ or
$\Udown_{i,\lambda}$ depending on the orientation.  

\item A dotted bubble of degree zero equals $\pm 1$:
\begin{equation}\label{eq:bubb_deg0}
\xy 0;/r.18pc/:
 (0,-1)*{\dblue\cbub{\black\llambda_i-1}{\black i}};
  (4,8)*{\lambda};
 \endxy
  = (-1)^{\laii} \quad \text{for $\llambda_i \geq 1$,}
  \qquad \quad
  \xy 0;/r.18pc/:
 (0,-1)*{\dblue\ccbub{\black -\llambda_i-1}{\black i}};
  (4,8)*{\lambda};
 \endxy
  = (-1)^{\laii-1} \quad \text{for $\llambda_i \leq -1$.}
\end{equation}

\item For the following relations we employ the convention that all summations are increasing, so that a summation of the form $\sum_{f=0}^{m}$ is zero 
if $m < 0$.
\begin{eqnarray}
\label{eq:redtobubbles}
  \text{$\xy 0;/r.18pc/:
  (10,8)*{\lambda};
  %(0,0)*{\twoIu{i}};
  (0,-3)*{\dblue\xybox{
  (-3,-8)*{};(3,8)*{} **\crv{(-3,-1) & (3,1)}?(1)*\dir{>};?(0)*\dir{>};
    (3,-8)*{};(-3,8)*{} **\crv{(3,-1) & (-3,1)}?(1)*\dir{>};
  (-3,-12)*{\bbsid};
  (-3,8)*{\bbsid};
  (3,8)*{}="t1";
  (9,8)*{}="t2";
  (3,-8)*{}="t1'";
  (9,-8)*{}="t2'";
   "t1";"t2" **\crv{(3,14) & (9, 14)};
   "t1'";"t2'" **\crv{(3,-14) & (9, -14)};
   "t2'";"t2" **\dir{-} ?(.5)*\dir{<};}};
   (9,0)*{}; (-7.5,-12)*{\scs i};
 \endxy$} \; = \; -\sum_{f=0}^{-\llambda_i}
   \xy
  (19,4)*{\lambda};
  (0,0)*{\dblue\bbe{}};(-2,-8)*{\scs i};
  (12,-2)*{\dblue\cbub{\black\llambda_i-1+f}{\black i}};
  (0,6)*{\dblue\bullet}+(6,1)*{\scs -\llambda_i-f};
 \endxy
\qquad \quad
  \text{$ \xy 0;/r.18pc/:
  (-12,8)*{\lambda};
   (0,-2)*{\dblue\xybox{
   (-3,-8)*{};(3,8)*{} **\crv{(-3,-1) & (3,1)}?(1)*\dir{>};?(0)*\dir{>};
    (3,-8)*{};(-3,8)*{} **\crv{(3,-1) & (-3,1)}?(1)*\dir{>};
  (3,-12)*{\bbsid};
  (3,8)*{\bbsid}; %(6,-8)*{\scs i};
  (-9,8)*{}="t1";
  (-3,8)*{}="t2";
  (-9,-8)*{}="t1'";
  (-3,-8)*{}="t2'";
   "t1";"t2" **\crv{(-9,14) & (-3, 14)};
   "t1'";"t2'" **\crv{(-9,-14) & (-3, -14)};
  "t1'";"t1" **\dir{-} ?(.5)*\dir{<};}};(7.5,-11)*{\scs i};
 \endxy$} \; = \;
 \sum_{g=0}^{\llambda_i}
   \xy
  (-12,8)*{\lambda};
  (0,0)*{\dblue\bbe{}};(2,-8)*{\scs i};
  (-12,-2)*{\dblue\ccbub{\black -\llambda_i-1+g}{\black i}};
  (0,6)*{\dblue\bullet}+(8,-1)*{\scs \llambda_i-g};
 \endxy
\end{eqnarray}
%
% square relations
%
\begin{eqnarray}
\label{eq:EF}
 \vcenter{\xy 0;/r.18pc/:
  (0,0)*{\dblue\xybox{
  (-8,0)*{};
  (8,0)*{};
  (-4,10)*{}="t1";
  (4,10)*{}="t2";
  (-4,-10)*{}="b1";
  (4,-10)*{}="b2";
  "t1";"b1" **\dir{-} ?(.5)*\dir{<};
  "t2";"b2" **\dir{-} ?(.5)*\dir{>};}};
  (-6,-8)*{\scs i};
  (6,-8)*{\scs i};
  (10,2)*{\lambda};
  (-10,2)*{\lambda};
  \endxy}
&\quad = \quad&
 \vcenter{   \xy 0;/r.18pc/:
    (0,0)*{\dblue\xybox{
    (-4,-4)*{};(4,4)*{} **\crv{(-4,-1) & (4,1)}?(1)*\dir{>};
    (4,-4)*{};(-4,4)*{} **\crv{(4,-1) & (-4,1)}?(1)*\dir{<};?(0)*\dir{<};
    (-4,4)*{};(4,12)*{} **\crv{(-4,7) & (4,9)};
    (4,4)*{};(-4,12)*{} **\crv{(4,7) & (-4,9)}?(1)*\dir{>};}};
    (8,8)*{\lambda};(-6,-7)*{\scs i};(6.8,-7)*{\scs i};
 \endxy}
  \quad - \quad
   \sum_{f=0}^{\llambda_i-1} \sum_{g=0}^{f}
    \vcenter{\xy 0;/r.18pc/:
    (-10,10)*{\lambda};
  (0,0)*{\dblue\xybox{
  (-8,0)*{}; (8,0)*{};
  (-4,-15)*{}="b1";
  (4,-15)*{}="b2";
  "b2";"b1" **\crv{(5,-8) & (-5,-8)}; ?(.05)*\dir{<} ?(.93)*\dir{<}
  ?(.8)*\dir{}+(0,-.1)*{\bullet}+(-5,2)*{\black\scs f-g};
  (-4,15)*{}="t1";
  (4,15)*{}="t2";
  "t2";"t1" **\crv{(5,8) & (-5,8)}; ?(.15)*\dir{>} ?(.95)*\dir{>}
  ?(.4)*\dir{}+(0,-.2)*{\bullet}+(3,-2)*{\black\scs \mspace{38mu}\;\;\; \llambda_i-1-f};
  (0,0)*{\ccbub{\black\scs \quad\;\;\;-\llambda_i-1+g}{i}};}};
  \endxy}
\label{eq_ident_decomp0}
 %\nn
% \\  \; \nn \\
\\[2ex]
 \vcenter{\xy 0;/r.18pc/:
  (0,0)*{\dblue\xybox{
  (-8,0)*{};(8,0)*{};
  (-4,10)*{}="t1";
  (4,10)*{}="t2";
  (-4,-10)*{}="b1";
  (4,-10)*{}="b2";
  "t1";"b1" **\dir{-} ?(.5)*\dir{>};
  "t2";"b2" **\dir{-} ?(.5)*\dir{<};}};
  (-6,-8)*{\scs i};(6,-8)*{\scs i};
  (10,2)*{\lambda};
  (-10,2)*{\lambda};
  \endxy}
&\quad = \quad&
   \vcenter{\xy 0;/r.18pc/:
    (0,0)*{\dblue\xybox{
    (-4,-4)*{};(4,4)*{} **\crv{(-4,-1) & (4,1)}?(1)*\dir{<};?(0)*\dir{<};
    (4,-4)*{};(-4,4)*{} **\crv{(4,-1) & (-4,1)}?(1)*\dir{>};
    (-4,4)*{};(4,12)*{} **\crv{(-4,7) & (4,9)}?(1)*\dir{>};
    (4,4)*{};(-4,12)*{} **\crv{(4,7) & (-4,9)};}};
    (8,8)*{\lambda};(-6.8,-7)*{\scs i};(6,-7)*{\scs i};
 \endxy}
  \quad - \quad
\sum_{f=0}^{-\llambda_i-1} \sum_{g=0}^{f}
    \vcenter{\xy 0;/r.18pc/:
  (0,0)*{\dblue\xybox{ 
  (-8,0)*{}; (8,0)*{};
  (-4,-15)*{}="b1";
  (4,-15)*{}="b2";
  "b2";"b1" **\crv{(5,-8) & (-5,-8)}; ?(.1)*\dir{>} ?(.95)*\dir{>}
  ?(.8)*\dir{}+(0,-.1)*{\bullet}+(-5,2)*{\black\scs f-g};
  (-4,15)*{}="t1";
  (4,15)*{}="t2";
  "t2";"t1" **\crv{(5,8) & (-5,8)}; ?(.15)*\dir{<} ?(.97)*\dir{<}
  ?(.4)*\dir{}+(0,-.2)*{\bullet}+(3,-2)*{\black\scs \mspace{32mu}\;\;-\llambda_i-1-f};
  (0,0)*{\cbub{\black\scs \quad\; \llambda_i-1+g}{i}};}};
  (-10,10)*{\lambda};
  \endxy} \label{eq_ident_decomp}
\end{eqnarray}

\item Fake bubbles:

Notice that for some values of $\lambda$ the dotted bubbles appearing above have negative labels. A composite of $\dblue\Uup_{\black\!\! i,\lambda}$
or $\dblue\Udown_{\black i,\lambda}$ with itself a negative number of times does not make sense. These dotted bubbles with negative labels, called fake bubbles, are
formal symbols inductively defined by the equation
%\begin{center}
\begin{eqnarray}\label{eq_infinite_Grass}
\left(\ \xy 0;/r.15pc/:
 (0,0)*{\dblue\xybox{
 (0,0)*{\ccbub{\black\mspace{-32mu}-\llambda_i-1}{\black i}};}};
  (4,8)*{\lambda};
 \endxy
 +
 \xy 0;/r.15pc/:
 (0,0)*{\dblue\xybox{
 (0,0)*{\ccbub{\black\mspace{-12mu}-\llambda_i-1+1}{\black i}};}};
  (4,8)*{\lambda};
 \endxy t
 + \cdots +
 \xy 0;/r.15pc/:
 (0,0)*{\dblue\xybox{
 (0,0)*{\ccbub{\black\mspace{-12mu}-\llambda_i-1+r}{\black i}};}};
  (4,8)*{\lambda};
 \endxy t^{r}
 + \cdots
\right)& \times & \nn
\\
\left( \xy 0;/r.15pc/:
(0,0)*{\dblue\xybox{
 (0,0)*{\cbub{\black\mspace{-22mu}\llambda_i-1}{\black i}};}};
  (4,8)*{\lambda};
 \endxy
 +
 \xy 0;/r.15pc/:
 (0,0)*{\dblue\xybox{
 (0,0)*{\cbub{\black\mspace{-8mu}\llambda_i-1+1}{\black i}};}};
  (4,8)*{\lambda};
 \endxy t
 + \cdots +
 \xy 0;/r.15pc/:
 (0,0)*{\dblue\xybox{
 (0,0)*{\cbub{\black\mspace{-8mu}\llambda_i-1+r}{\black i}};}};
 (4,8)*{\lambda};
 \endxy t^{r}
 + \cdots
\right) & = & -1  %\nn \\ 
\end{eqnarray}
%\end{center}
and the additional condition
\[
\xy 0;/r.18pc/:
 (0,0)*{\dblue\cbub{\black -1}{\black i}};
  (4,8)*{\lambda};
 \endxy
 \quad = (-1)^{\laii},
 \qquad
  \xy 0;/r.18pc/:
 (0,0)*{\dblue\ccbub{\black -1}{\black i}};
  (4,8)*{\lambda};
 \endxy
  \quad = (-1)^{\laii-1} \qquad \text{if $\llambda_i =0$.}
\]
Although the labels are negative for fake bubbles, one can check that the overall
degree of each fake bubble is still positive, so that these fake bubbles do not
violate the positivity of dotted bubble axiom. The above equation, called the
infinite Grassmannian relation, remains valid even in high degree when most of
the bubbles involved are not fake bubbles.

\end{enumerate}

\n $\star$ NilHecke relations:
 \begin{equation} \label{eq_nil_rels}
  \vcenter{\xy 0;/r.18pc/:
    (0,0)*{\dblue\xybox{
    (-4,-4)*{};(4,4)*{} **\crv{(-4,-1) & (4,1)}?(1)*\dir{>};
    (4,-4)*{};(-4,4)*{} **\crv{(4,-1) & (-4,1)}?(1)*\dir{>};
    (-4,4)*{};(4,12)*{} **\crv{(-4,7) & (4,9)}?(1)*\dir{>};
    (4,4)*{};(-4,12)*{} **\crv{(4,7) & (-4,9)}?(1)*\dir{>};}};
  (8,8)*{\lambda};(-5,-7)*{\scs i};(5.1,-7)*{\scs i};
 \endxy}
 =0, \qquad \quad
 \vcenter{
 \xy 0;/r.18pc/:
 (0,0)*{\dblue\xybox{
    (-4,-4)*{};(4,4)*{} **\crv{(-4,-1) & (4,1)}?(1)*\dir{>};
    (4,-4)*{};(-4,4)*{} **\crv{(4,-1) & (-4,1)}?(1)*\dir{>};
    (4,4)*{};(12,12)*{} **\crv{(4,7) & (12,9)}?(1)*\dir{>};
    (12,4)*{};(4,12)*{} **\crv{(12,7) & (4,9)}?(1)*\dir{>};
    (-4,12)*{};(4,20)*{} **\crv{(-4,15) & (4,17)}?(1)*\dir{>};
    (4,12)*{};(-4,20)*{} **\crv{(4,15) & (-4,17)}?(1)*\dir{>};
    (-4,4)*{}; (-4,12) **\dir{-};
    (12,-4)*{}; (12,4) **\dir{-};
    (12,12)*{}; (12,20) **\dir{-};}}; 
     (-9.5,-11)*{\scs i};(1.5,-11)*{\scs i};(9.5,-11)*{\scs i};
  (12,0)*{\lambda};
\endxy}
 \;\; =\;\;
 \vcenter{
 \xy 0;/r.18pc/:
    (0,0)*{\dblue\xybox{
    (4,-4)*{};(-4,4)*{} **\crv{(4,-1) & (-4,1)}?(1)*\dir{>};
    (-4,-4)*{};(4,4)*{} **\crv{(-4,-1) & (4,1)}?(1)*\dir{>};
    (-4,4)*{};(-12,12)*{} **\crv{(-4,7) & (-12,9)}?(1)*\dir{>};
    (-12,4)*{};(-4,12)*{} **\crv{(-12,7) & (-4,9)}?(1)*\dir{>};
    (4,12)*{};(-4,20)*{} **\crv{(4,15) & (-4,17)}?(1)*\dir{>};
    (-4,12)*{};(4,20)*{} **\crv{(-4,15) & (4,17)}?(1)*\dir{>};
    (4,4)*{}; (4,12) **\dir{-};
    (-12,-4)*{}; (-12,4) **\dir{-};
    (-12,12)*{}; (-12,20) **\dir{-};}};
  (-1.5,-11)*{\scs i};(9.5,-11)*{\scs i};(-9.5,-11)*{\scs i};
  (12,0)*{\lambda};
\endxy} 
  \end{equation}
\begin{equation} \label{eq_nil_dotslide}
  \xy
  (0,1)*{\dblue\xybox{
  (4,4);(4,-4) **\dir{-}?(0)*\dir{<}+(2.3,0)*{};
  (-4,4);(-4,-4) **\dir{-}?(0)*\dir{<}+(2.3,0)*{};}};
  (6,2)*{\lambda};(-6.2,-2)*{\scs i}; (4,-2)*{\scs i};
 \endxy
\
  =
\xy
  (0,0)*{\dblue\xybox{
    (-4,-4)*{};(4,4)*{} **\crv{(-4,-1) & (4,1)}?(1)*\dir{>}?(.25)*{\bullet};
    (4,-4)*{};(-4,4)*{} **\crv{(4,-1) & (-4,1)}?(1)*\dir{>};
    (-5,-3)*{\black\scs i};
     (5.1,-3)*{\black\scs i};
     (8,1)*{\black\lambda};
     (-10,0)*{};(10,0)*{};
     }};
  \endxy
  -
 \xy
  (0,0)*{\dblue\xybox{
    (-4,-4)*{};(4,4)*{} **\crv{(-4,-1) & (4,1)}?(1)*\dir{>}?(.75)*{\bullet};
    (4,-4)*{};(-4,4)*{} **\crv{(4,-1) & (-4,1)}?(1)*\dir{>};
    (-5,-3)*{\black\scs i};
     (5.1,-3)*{\black\scs i};
     (8,1)*{\black\lambda};
     (-10,0)*{};(10,0)*{};
     }};
  \endxy
\
  =
\xy
  (0,0)*{\dblue\xybox{
    (-4,-4)*{};(4,4)*{} **\crv{(-4,-1) & (4,1)}?(1)*\dir{>};
    (4,-4)*{};(-4,4)*{} **\crv{(4,-1) & (-4,1)}?(1)*\dir{>}?(.75)*{\bullet};
    (-5,-3)*{\black\scs i};
     (5.1,-3)*{\black\scs i};
     (8,1)*{\black\lambda};
     (-10,0)*{};(10,0)*{};
     }};
  \endxy
  -
  \xy
  (0,0)*{\dblue\xybox{
    (-4,-4)*{};(4,4)*{} **\crv{(-4,-1) & (4,1)}?(1)*\dir{>} ;
    (4,-4)*{};(-4,4)*{} **\crv{(4,-1) & (-4,1)}?(1)*\dir{>}?(.25)*{\bullet};
    (-5,-3)*{\black\scs i};
     (5.1,-3)*{\black\scs i};
     (8,1)*{\black\lambda};
     (-10,0)*{};(10,0)*{};
     }};
  \endxy %\nn %\\
\end{equation}

\n $\star$ For $i \neq j$
\begin{equation} \label{eq_downup_ij-gen}
\text{$
 \vcenter{   \xy 0;/r.18pc/:
(0,0)*{\dblue\xybox{
    (-4,-4)*{};(4,4)*{} **\crv{(-4,-1) & (4,1)}?(1)*\dir{>};
    (4,4)*{};(-4,12)*{} **\crv{(4,7) & (-4,9)}?(1)*\dir{>};
}};
(0,0)*{\dred\xybox{
    (4,-4)*{};(-4,4)*{} **\crv{(4,-1) & (-4,1)}?(1)*\dir{<};?(0)*\dir{<};
    (-4,4)*{};(4,12)*{} **\crv{(-4,7) & (4,9)};
}};
   (7,4)*{\lambda};(-6,-7)*{\scs i};(6.5,-7)*{\scs j};
 \endxy}
 \;\;= \;\;
\xy 0;/r.18pc/:
(3.5,0)*{\dred\xybox{
    (3,9);(3,-9) **\dir{-}?(.55)*\dir{>}+(2.3,0)*{};}};
(-3.5,0)*{\dblue\xybox{
    (-3,9);(-3,-9) **\dir{-}?(.5)*\dir{<}+(2.3,0)*{};}};
    (7,2)*{\lambda};(-6,-6)*{\scs i};(3.8,-6)*{\scs j};
 \endxy
 \qquad
    \vcenter{\xy 0;/r.18pc/:
(0,0)*{\dblue\xybox{
    (-4,-4)*{};(4,4)*{} **\crv{(-4,-1) & (4,1)}?(1)*\dir{<};?(0)*\dir{<};
    (4,4)*{};(-4,12)*{} **\crv{(4,7) & (-4,9)};
}}; 
(0,0)*{\dred\xybox{
    (4,-4)*{};(-4,4)*{} **\crv{(4,-1) & (-4,1)}?(1)*\dir{>};
    (-4,4)*{};(4,12)*{} **\crv{(-4,7) & (4,9)}?(1)*\dir{>};
}};
    (7,4)*{\lambda};(-6.5,-7)*{\scs i};(6,-7)*{\scs j};
 \endxy}
 \;\;=\;\;
\xy 0;/r.18pc/:
(3.5,0)*{\dred\xybox{
    (3,9);(3,-9) **\dir{-}?(.5)*\dir{<}+(2.3,0)*{};
}};
(-3.5,0)*{\dblue\xybox{
    (-3,9);(-3,-9) **\dir{-}?(.55)*\dir{>}+(2.3,0)*{};
}};
    (7,2)*{\lambda};(-6.5,-6)*{\scs i};(3.6,-6)*{\scs j};
 \endxy
$}
\end{equation}

\n $\star$ The analogue of the $R(\nu)$-relations:
\begin{enumerate}[i)]
\item For $i \neq j$
\begin{equation}
\label{eq_r2_ij-gen}
  \vcenter{\xy 0;/r.18pc/:
(0,0)*{\dblue\xybox{
    (-4,-4)*{};(4,4)*{} **\crv{(-4,-1) & (4,1)}?(1)*\dir{>};
    (4,4)*{};(-4,12)*{} **\crv{(4,7) & (-4,9)}?(1)*\dir{>};
}};
(0,0)*{\dred\xybox{
    (4,-4)*{};(-4,4)*{} **\crv{(4,-1) & (-4,1)}?(1)*\dir{>};
    (-4,4)*{};(4,12)*{} **\crv{(-4,7) & (4,9)}?(1)*\dir{>};
}};
    (8,8)*{\lambda};(-5,-6)*{\scs i};
    (5.3,-6)*{\scs j};
 \endxy} =
\end{equation}
\begin{equation*}
\begin{cases}
\vcenter{\xy 0;/r.18pc/:
(4,0)*{\dred\xybox{
  (3,9);(3,-9) **\dir{-}?(.5)*\dir{<}+(2.3,0)*{};
}};
(-2.5,0)*{\dblue\xybox{
  (-3,9);(-3,-9) **\dir{-}?(.5)*\dir{<}+(2.3,0)*{};
}};
  (8,2)*{\lambda};(-5,-6)*{\scs i};(5.1,-6)*{\scs j};
 \endxy} &  \text{if $i \cdot j=0$}\\ \\
  \epsilon(i,j)\left(
  \vcenter{\xy 0;/r.18pc/:
(4.5,0)*{\dred\xybox{
  (3,9);(3,-9) **\dir{-}?(.5)*\dir{<}+(2.3,0)*{};
}};
(-2.5,0)*{\dblue\xybox{
  (-3,9);(-3,-9) **\dir{-}?(.5)*\dir{<}+(2.3,0)*{};
  (-3,4)*{\bullet};
}};
  (8,2)*{\black\lambda}; 
  (-5,-6)*{\bscs i};     
(5.1,-6)*{\bscs j};
 \endxy} 
\quad - \quad
 \vcenter{\xy 0;/r.18pc/:
(4,0)*{\dred\xybox{
  (3,9);(3,-9) **\dir{-}?(.5)*\dir{<}+(2.3,0)*{}; (3,4)*{\bullet};
}};
(-2,0)*{\dblue\xybox{
  (-3,9);(-3,-9) **\dir{-}?(.5)*\dir{<}+(2.3,0)*{};
}};
  (9,2)*{\black\lambda};
  (-5,-6)*{\bscs i};     (5.1,-6)*{\bscs j};
 \endxy}\right)
   & \begin{array}{l}
\text{if $i \cdot j =-1$}\\ 
\text{and $\{i,j\}\ne \{1,n\}$}
\end{array}\\ \\

\epsilon(i,j)\left(
  \vcenter{\xy 0;/r.18pc/:
(4.5,0)*{\dred\xybox{
  (3,9);(3,-9) **\dir{-}?(.5)*\dir{<}+(2.3,0)*{};
}};
(-2.5,0)*{\dblue\xybox{
  (-3,9);(-3,-9) **\dir{-}?(.5)*\dir{<}+(2.3,0)*{};
  (-3,4)*{\bullet};
}};
  (8,2)*{\black\lambda}; 
  (-5,-6)*{\bscs i};     (5.1,-6)*{\bscs j};
 \endxy} \quad
 - \quad
 \vcenter{\xy 0;/r.18pc/:
(4,0)*{\dred\xybox{
  (3,9);(3,-9) **\dir{-}?(.5)*\dir{<}+(2.3,0)*{}; (3,4)*{\bullet};
}};
(-2,0)*{\dblue\xybox{
  (-3,9);(-3,-9) **\dir{-}?(.5)*\dir{<}+(2.3,0)*{};
}};
  (9,2)*{\black\lambda};
  (-5,-6)*{\bscs i};     (5.1,-6)*{\bscs j};
 \endxy}\right)\quad - \quad
 \bbox{y} \vcenter{\xy 0;/r.18pc/:
(4,0)*{\dred\xybox{
  (3,9);(3,-9) **\dir{-}?(.5)*\dir{<}+(2.3,0)*{}; 
}};
(-2,0)*{\dblue\xybox{
  (-3,9);(-3,-9) **\dir{-}?(.5)*\dir{<}+(2.3,0)*{};
}};
  (9,2)*{\black\lambda};
  (-5,-6)*{\bscs i};     (5.1,-6)*{\bscs j};
 \endxy} & \text{if $\{ i,j\} = \{1,n\}$}
\end{cases}
\end{equation*}
\vskip0.5cm
\n For $i\cdot j=-1$, we define  
$$
\epsilon(i,j):=
\begin{cases}
1& \text{if}\; i\equiv j+1\mod n\\
-1& \text{if}\; i\equiv j-1\mod n
\end{cases}.
$$
Note that this sign takes into account the standard 
orientation of the Dynkin diagram. 

\begin{eqnarray} \label{eq_dot_slide_ij-gen}
\xy
(0,0)*{\dblue\xybox{
    (-4,-4)*{};(4,4)*{} **\crv{(-4,-1) & (4,1)}?(1)*\dir{>}?(.75)*{\bullet};
}};
(0,0)*{\dred\xybox{
    (4,-4)*{};(-4,4)*{} **\crv{(4,-1) & (-4,1)}?(1)*\dir{>};
}};
    (-5,-3)*{\scs i};
    (5.1,-3)*{\scs j};
    (8,1)*{ \lambda};
    (-10,0)*{};(10,0)*{};
\endxy
 \;\; =
\xy
(0,0)*{\dblue\xybox{
    (-4,-4)*{};(4,4)*{} **\crv{(-4,-1) & (4,1)}?(1)*\dir{>}?(.25)*{\bullet};
}};
(0,0)*{\dred\xybox{
    (4,-4)*{};(-4,4)*{} **\crv{(4,-1) & (-4,1)}?(1)*\dir{>};
}}; 
     (-5,-3)*{\scs i};
     (5.1,-3)*{\scs j};
     (8,1)*{ \lambda};
     (-10,0)*{};(10,0)*{};
\endxy
\qquad  \xy
(0,0)*{\dblue\xybox{
    (-4,-4)*{};(4,4)*{} **\crv{(-4,-1) & (4,1)}?(1)*\dir{>};
}};
(0,0)*{\dred\xybox{
    (4,-4)*{};(-4,4)*{} **\crv{(4,-1) & (-4,1)}?(1)*\dir{>}?(.75)*{\bullet};
}};
    (-5,-3)*{\scs i};
    (5.1,-3)*{\scs j};
    (8,1)*{ \lambda};
    (-10,0)*{};(10,0)*{};
  \endxy
\;\;  =
  \xy
(0,0)*{\dblue\xybox{
    (-4,-4)*{};(4,4)*{} **\crv{(-4,-1) & (4,1)}?(1)*\dir{>} ;
}};
(0,0)*{\dred\xybox{
    (4,-4)*{};(-4,4)*{} **\crv{(4,-1) & (-4,1)}?(1)*\dir{>}?(.25)*{\bullet};
}};
    (-5,-3)*{\scs i};
     (5.1,-3)*{\scs j};
     (8,1)*{ \lambda};
     (-10,0)*{};(12,0)*{};
   \endxy
\end{eqnarray}

\item Unless $i = k$ and $i \cdot j=-1$
\begin{equation}
\text{$
 \vcenter{
 \xy 0;/r.18pc/:
(0,0)*{\dblue\xybox{
    (-4,-4)*{};(4,4)*{} **\crv{(-4,-1) & (4,1)}?(1)*\dir{>};
    (4,4)*{};(12,12)*{} **\crv{(4,7) & (12,9)}?(1)*\dir{>};
    (12,12)*{}; (12,20) **\dir{-};
}};
(-4,0)*{\dred\xybox{
    (4,-4)*{};(-4,4)*{} **\crv{(4,-1) & (-4,1)}?(1)*\dir{>};
    (-4,12)*{};(4,20)*{} **\crv{(-4,15) & (4,17)}?(1)*\dir{>};
    (-4,4)*{}; (-4,12) **\dir{-};
}};
(0,0)*{\dgreen\xybox{
    (12,4)*{};(4,12)*{} **\crv{(12,7) & (4,9)}?(1)*\dir{>};
    (4,12)*{};(-4,20)*{} **\crv{(4,15) & (-4,17)}?(1)*\dir{>};
    (12,-4)*{}; (12,4) **\dir{-};
}};
  (12,0)*{\lambda};
  (-10,-11)*{\scs i};
  (  2,-11)*{\scs j};
  (10.5,-11)*{\scs k};
\endxy}
 \;\; =\;\;
 \vcenter{
 \xy 0;/r.18pc/:
(0,0)*{\dgreen\xybox{
    (4,-4)*{};(-4,4)*{} **\crv{(4,-1) & (-4,1)}?(1)*\dir{>};
    (-4,4)*{};(-12,12)*{} **\crv{(-4,7) & (-12,9)}?(1)*\dir{>};
    (-12,12)*{}; (-12,20) **\dir{-};
}};
(4,0)*{\dred\xybox{
    (-4,-4)*{};(4,4)*{} **\crv{(-4,-1) & (4,1)}?(1)*\dir{>};
    (4,12)*{};(-4,20)*{} **\crv{(4,15) & (-4,17)}?(1)*\dir{>};
    (4,4)*{}; (4,12) **\dir{-};
}};
(0,0)*{\dblue\xybox{
    (-12,4)*{};(-4,12)*{} **\crv{(-12,7) & (-4,9)}?(1)*\dir{>};
    (-4,12)*{};(4,20)*{} **\crv{(-4,15) & (4,17)}?(1)*\dir{>};
    (-12,-4)*{}; (-12,4) **\dir{-};
}};
  (12,0)*{\lambda};
  (10,-11)*{\scs k};
  (-1.5,-11)*{\scs j};
  (-9.5,-11)*{\scs i};
\endxy} \label{eq_r3_easy-gen}
$}
\end{equation}

For $i \cdot j =-1$
\begin{equation}
\text{$
 \vcenter{
 \xy 0;/r.18pc/:
(0,0)*{\dblue\xybox{
    (-4,-4)*{};(4,4)*{} **\crv{(-4,-1) & (4,1)}?(1)*\dir{>};
    (4,4)*{};(12,12)*{} **\crv{(4,7) & (12,9)}?(1)*\dir{>};
    (12,4)*{};(4,12)*{} **\crv{(12,7) & (4,9)}?(1)*\dir{>};
    (4,12)*{};(-4,20)*{} **\crv{(4,15) & (-4,17)}?(1)*\dir{>};
    (12,-4)*{}; (12,4) **\dir{-};
    (12,12)*{}; (12,20) **\dir{-};
}};
(-4,0)*{\dred\xybox{
    (4,-4)*{};(-4,4)*{} **\crv{(4,-1) & (-4,1)}?(1)*\dir{>};
    (-4,12)*{};(4,20)*{} **\crv{(-4,15) & (4,17)}?(1)*\dir{>};
    (-4,4)*{}; (-4,12) **\dir{-};
}};
  (12,0)*{\lambda};
  (-10,-11)*{\scs i};
  (1.5,-11)*{\scs j};
  (9.5,-11)*{\scs i};
\endxy}
\quad - \quad
 \vcenter{
 \xy 0;/r.18pc/:
(0,0)*{\dblue\xybox{
    (4,-4)*{};(-4,4)*{} **\crv{(4,-1) & (-4,1)}?(1)*\dir{>};
    (-4,4)*{};(-12,12)*{} **\crv{(-4,7) & (-12,9)}?(1)*\dir{>};
    (-12,4)*{};(-4,12)*{} **\crv{(-12,7) & (-4,9)}?(1)*\dir{>};
    (-4,12)*{};(4,20)*{} **\crv{(-4,15) & (4,17)}?(1)*\dir{>};
    (-12,-4)*{}; (-12,4) **\dir{-};
    (-12,12)*{}; (-12,20) **\dir{-};
}};
(4,0)*{\dred\xybox{
    (-4,-4)*{};(4,4)*{} **\crv{(-4,-1) & (4,1)}?(1)*\dir{>};
    (4,12)*{};(-4,20)*{} **\crv{(4,15) & (-4,17)}?(1)*\dir{>};
    (4,4)*{}; (4,12) **\dir{-};
}};
  (12,0)*{\lambda};
  (10,-11)*{\scs i};
  (-1.5,-11)*{\scs j};
  (-9.5,-11)*{\scs i};
\endxy}
 \;\; =\;\;\;\; \epsilon(i,j)
\xy 0;/r.18pc/:
(10,0)*{\dblue\xybox{
  (4,12);(4,-12) **\dir{-}?(.5)*\dir{<};
  (22,12);(22,-12) **\dir{-}?(.5)*\dir{<}?(.25)*\dir{}+(0,0)*{}+(10,0)*{\scs};
}};
(3.5,0)*{\dred\xybox{
  (-4,12);(-4,-12) **\dir{-}?(.5)*\dir{<}?(.25)*\dir{}+(0,0)*{}+(-3,0)*{\scs };
}};
  (20,0)*{\lambda}; 
  (-5.6,-11)*{\scs i};
  (3.1,-11)*{\scs j};
  (15.2,-11)*{\scs i};
 \endxy
 \label{eq_r3_hard-gen}
$}.
\end{equation}
\end{enumerate}
%%%%%%%%%%%%%%%%%%%%%%%%%%%%%%%%%%%%%%%%%%%%%%%%%%%%%%%

\n $\bullet$ Composition of $1$-morphisms is defined by multiplication 
and horizontal and vertical composition of $2$-morphisms are defined by 
juxtaposition and glueing, as usual.  
\end{defn}

\begin{defn}
The $2$-category $\Uglcataffy$ is the $\Q$-linear 
$2$-subcategory of $\Uglcataffy^*$ given by the degree-zero $2$-morphisms. 
\end{defn}

\begin{rem}
Note that the $2$-hom-spaces in $\Uglcataffy$ are no longer $\Q[y]$-linear, 
because $\deg(y)=2$. They are finite-dimensional as $\Q$-vector spaces, 
because the original KL $2$-HOM-spaces are finite-dimensional in each 
degree and their grading is bounded below.  
\end{rem}

As already remarked, $\Ucataffy$ is defined similarly. Only some signs in 
the relations involving right cups and caps have to be changed, 
so that all relations 
really depend on $\hat{\mathfrak{sl}}_n$-weights and not on 
$\hat{\mathfrak{gl}}_n$-weights. We use the convention which is 
the affine analogue of the signed-version 
in~\cite{KL3}. For more information on this change of signs, 
see~\eqref{eq:signs}.

\begin{defn}
Let $\Ucataff$ and $\Uglcataff$ denote the $\Q$-linear $2$-categories 
obtained by modding out $\Ucataffy$ and $\Uglcataffy$ by the ideal 
generated by $y$. 
\end{defn}
\n Note that $\Ucataff$ is isomorphic to the original 
KL $2$-category and $\Uglcataff$ isomorphic to its 
$\hat{\mathfrak{gl}}_n$-analogue. 

We do not know if $\Ucataffy$ is isomorphic to $\Ucataff\otimes_{\Q}\Q[y]$, 
but it does not seem so.

\subsubsection{Further relations in $\Uglcataffy$}

The other $\Uglcataffy$-relations expressed below follow from the relations in Definition~\ref{def_glcat} and are going to be used in the sequel.

\n $\star$ Bubble slides: \\ 
\n If $\{i,j\}\ne\{1,n\}$, we have  
\begin{equation}
\label{eq:bub_slides}
\text{$ 
   \xy
  (14,8)*{\lambda};
  (0,0)*{\dgreen\bbe{}};
  (0,-12)*{\scs j};
  (12,-2)*{\dblue\ccbub{\black -\llambda_i-1+m}{\black i}};
  (0,6)*{ }+(7,-1)*{\scs  };
 \endxy
$}
\quad = \quad
 \begin{cases}
 \ \xsum{f=0}{m}(f-m-1)
   \xy
  (0,8)*{\lambda+j_{\Lambda}};
  (12,0)*{\dgreen\bbe{}};
  (12,-12)*{\scs j};
  (0,-2)*{\dblue\ccbub{\black - \overline{(\lambda + j_\Lambda)}_i -1 + f}{\black i}};
  (12,6)*{\dgreen\bullet}+(5,-1)*{\scs m-f};
 \endxy
    &  \text{if $i=j$} 
\\ \\
\ \qquad \qquad  \xy
  (0,8)*{\lambda+j_{\Lambda}};
  (12,0)*{\dgreen\bbe{}};
  (12,-12)*{\scs j};
  (0,-2)*{\dblue\ccbub{\black -\overline{(\lambda + j_\Lambda)_i} -1+m }{\black i}};
 \endxy  &  \text{if $i \cdot j=0$}
 \end{cases}
\end{equation}
%%%%%%%%%%%%%%%%%%
\begin{align}
\label{eq:2ndbubbslide}
\text{$ 
   \xy
  (14,8)*{\lambda};
  (0,0)*{\dred\bbe{}};
  (0,-12)*{\scs i+1};
  (12,-2)*{\dblue\ccbub{\black -\llambda_i-1+m}{\black i}};
  (0,6)*{ }+(7,-1)*{\scs  };
 \endxy
$}
&= \quad
   \xy
  (-4,8)*{\lambda+(i+1)_{\Lambda}};
  (12,0)*{\dred\bbe{}};
  (12,-12)*{\scs i+1};
  (-6,-2)*{\dblue\ccbub{\black -(\overline{\lambda +(i+1)_{\Lambda}})_i-2+m}{\black i}};
  (12,6)*{\dred\bullet}+(5,-1)*{\scs };
 \endxy
 \quad - \quad
  \xy
  (-4,8)*{\lambda+(i+1)_{\Lambda}};
  (12,0)*{\red\bbe{}};
  (11,-12)*{\scs i+1};
  (-6,-2)*{\dblue\ccbub{\black -(\overline{\lambda+(i+1)_{\Lambda}})_i -1+m}{\black i}};
 \endxy
\\ \displaybreak[0]
\label{eq:extrabubble3}
\text{$ 
   \xy
  (6,8)*{\lambda};
  (12,0)*{\dred\bbe{}};
  (12,-12)*{\scs i+1};
  (0,-2)*{\dblue\ccbub{\black -\llambda_i-1+m}{\black i}};
  (-12,6)*{ }+(7,-1)*{\scs  };
 \endxy
$}
&= 
-\sum\limits_{f+g=m}
   \xy
  (18,8)*{\lambda-(i+1)_{\Lambda}};
  (0,0)*{\dred\bbe{}};
  (0,-12)*{\scs i+1};
  (16,-2)*{\dblue\ccbub{\black -(\overline{\lambda -(i+1)_{\Lambda}})_i-1+g}{\black i}};
  (0,6)*{\dred\bullet}+(-3,-1)*{\scs f};
 \endxy
\\\nn\\ 
\label{eq:extrabubble4}
\displaybreak[0]
\text{$ 
   \xy
  (14,8)*{\lambda};
  (0,0)*{\dred\bbe{}};
  (0,-12)*{\scs i+1};
  (12,-2)*{\dblue\cbub{\black \llambda_i-1+m}{\black i}};
  (0,6)*{ }+(7,-1)*{\scs  };
 \endxy
$}
&= -\sum\limits_{f+g=m}\ \
   \xy
  (-4,8)*{\lambda+(i+1)_{\Lambda}};
  (12,0)*{\dred\bbe{}};
  (12,-12)*{\scs i+1};
  (-6,-2)*{\dblue\cbub{\black (\overline{\lambda +(i+1)_{\Lambda}})_i-1+g}{\black i}};
  (12,6)*{\dred\bullet}+(2,-1)*{\scs f};
 \endxy
\\\nn \\ \displaybreak[0]
\label{eq:extrabubblelast}
\text{$ 
   \xy
  (6,8)*{\lambda};
  (12,0)*{\dred\bbe{}};
  (12,-12)*{\scs i+1};
  (0,-2)*{\dblue\cbub{\black \llambda_i-1+m}{\black i}};
  (-12,6)*{ }+(7,-1)*{\scs  };
 \endxy
$}
&= 
\quad
   \xy
  (18,8)*{\lambda-(i+1)_{\Lambda}};
  (0,0)*{\dred\bbe{}};
  (0,-12)*{\scs i+1};
  (16,-2)*{\dblue\cbub{\black (\overline{\lambda -(i+1)_{\Lambda}})_i-2+m}{\black i}};
  (0,6)*{\dred\bullet}+(-3,-1)*{\scs };
 \endxy
\quad-\quad
   \xy
  (18,8)*{\lambda-(i+1)_{\Lambda}};
  (0,0)*{\dred\bbe{}};
  (0,-12)*{\scs i+1};
  (16,-2)*{\dblue\cbub{\black (\overline{\lambda -(i+1)_{\Lambda}})_i-1+m}{\black i}};
 \endxy
\end{align}
If we switch labels $i$ and $i+1$, then the r.h.s. of the above equations gets 
a minus sign. Bubble slides with the vertical strand oriented downwards 
can easily be obtained from the ones above by rotating the diagrams $180$ 
degrees. 
\medskip

If $\{i,j\}=\{1,n\}$, we get 
\begin{align}
\text{$ 
   \xy
  (14,8)*{\lambda};
  (0,0)*{\dred\bbe{}};
  (0,-12)*{\scs 1};
  (12,-2)*{\dblue\ccbub{\black -\llambda_n-1+m}{\black n}};
  (0,6)*{ }+(7,-1)*{\scs  };
 \endxy
$}
&= \quad
   \xy
  (-4,8)*{\lambda+(1)_{\Lambda}};
  (12,0)*{\dred\bbe{}};
  (12,-12)*{\scs 1};
  (-6,-2)*{\dblue\ccbub{\black -(\overline{\lambda +(1)_{\Lambda}})_n-2+m}{\black n}};
  (12,6)*{\dred\bullet}+(5,-1)*{\scs };
 \endxy
 \quad - \quad
  \xy
  (-4,8)*{\lambda+(1)_{\Lambda}};
  (12,0)*{\red\bbe{}};
  (11,-12)*{\scs 1};
  (-6,-2)*{\dblue\ccbub{\black -(\overline{\lambda+(1)_{\Lambda}})_n -1+m}{\black n}};
 \endxy \nn
\\
& - \quad \bbox{y}
\xy
  (-4,8)*{\lambda+(1)_{\Lambda}};
  (12,0)*{\dred\bbe{}};
  (12,-12)*{\scs 1};
  (-6,-2)*{\dblue\ccbub{\black -(\overline{\lambda +(1)_{\Lambda}})_n-2+m}{\black n}};
   \endxy
\label{eq:2ndbubbslide1n}
\end{align}

\begin{align}
\text{$ 
   \xy
  (14,8)*{\lambda};
  (0,0)*{\dred\bbe{}};
  (0,-12)*{\scs n};
  (12,-2)*{\dblue\ccbub{\black -\llambda_1-1+m}{\black 1}};
  (0,6)*{ }+(7,-1)*{\scs  };
 \endxy
$}
&= \quad
-   \xy
  (-4,8)*{\lambda+(n)_{\Lambda}};
  (12,0)*{\dred\bbe{}};
  (12,-12)*{\scs n};
  (-6,-2)*{\dblue\ccbub{\black -(\overline{\lambda +(n)_{\Lambda}})_1-2+m}{\black 1}};
  (12,6)*{\dred\bullet}+(5,-1)*{\scs };
 \endxy
 \quad + \quad
  \xy
  (-4,8)*{\lambda+(n)_{\Lambda}};
  (12,0)*{\red\bbe{}};
  (11,-12)*{\scs n};
  (-6,-2)*{\dblue\ccbub{\black -(\overline{\lambda+(n)_{\Lambda}})_1 -1+m}{\black 1}};
 \endxy \nn
\\ 
&- \quad \bbox{y}
\xy
  (-4,8)*{\lambda+(n)_{\Lambda}};
  (12,0)*{\dred\bbe{}};
  (12,-12)*{\scs n};
  (-6,-2)*{\dblue\ccbub{\black -(\overline{\lambda +(n)_{\Lambda}})_1-2+m}{\black 1}};
   \endxy
\label{eq:2ndbubbslide21n}
\end{align}

\begin{align}
\text{$ 
   \xy
  (6,8)*{\lambda};
  (12,0)*{\dred\bbe{}};
  (12,-12)*{\scs 1};
  (0,-2)*{\dblue\ccbub{\black -\llambda_n-1+m}{\black n}};
  (-12,6)*{ }+(7,-1)*{\scs  };
 \endxy
$}
&= 
-\sum\limits_{f+g=m}\xsum{p=0}{f} \binom{f}{p} (-\bbox{y})^{f-p}
   \xy
  (18,8)*{\lambda-(1)_{\Lambda}};
  (0,0)*{\dred\bbe{}};
  (0,-12)*{\scs 1};
  (16,-2)*{\dblue\ccbub{\black -(\overline{\lambda -(1)_{\Lambda}})_n-1+g}{\black n}};
  (0,6)*{\dred\bullet}+(-3,-1)*{\scs p};
 \endxy 
\label{eq:extrabubble41n}
\end{align}
\vskip0.2cm
\begin{align}
\text{$ 
   \xy
  (6,8)*{\lambda};
  (12,0)*{\dred\bbe{}};
  (12,-12)*{\scs n};
  (0,-2)*{\dblue\ccbub{\black -\llambda_1-1+m}{\black 1}};
  (-12,6)*{ }+(7,-1)*{\scs  };
 \endxy
$}
&= 
\sum\limits_{f+g=m}\xsum{p=0}{f} \binom{f}{p} \bbox{y}^{f-p}
   \xy
  (18,8)*{\lambda-(n)_{\Lambda}};
  (0,0)*{\dred\bbe{}};
  (0,-12)*{\scs n};
  (16,-2)*{\dblue\ccbub{\black -(\overline{\lambda -(n)_{\Lambda}})_1-1+g}{\black 1}};
  (0,6)*{\dred\bullet}+(-3,-1)*{\scs p};
 \endxy 
\label{eq:extrabubble421n}
\end{align}
\vskip0.2cm
\begin{align}
\text{$ 
   \xy
  (14,8)*{\lambda};
  (0,0)*{\dred\bbe{}};
  (0,-12)*{\scs 1};
  (12,-2)*{\dblue\cbub{\black \llambda_n-1+m}{\black n}};
  (0,6)*{ }+(7,-1)*{\scs  };
 \endxy
$}
&= -\sum\limits_{f+g=m}\ \xsum{p=0}{f} \binom{f}{p} (-\bbox{y})^{f-p}
   \xy
  (-4,8)*{\lambda+(1)_{\Lambda}};
  (12,0)*{\dred\bbe{}};
  (12,-12)*{\scs 1};
  (-6,-2)*{\dblue\cbub{\black (\overline{\lambda +(1)_{\Lambda}})_n-1+g}{\black n}};
  (12,6)*{\dred\bullet}+(2,-1)*{\scs p};
 \endxy
\label{eq:extrabubble431n}
\end{align}

\begin{align}
\text{$ 
   \xy
  (14,8)*{\lambda};
  (0,0)*{\dred\bbe{}};
  (0,-12)*{\scs n};
  (12,-2)*{\dblue\cbub{\black \llambda_1-1+m}{\black 1}};
  (0,6)*{ }+(7,-1)*{\scs  };
 \endxy
$}
&= \sum\limits_{f+g=m}\ \xsum{p=0}{f} \binom{f}{p} \bbox{y}^{f-p}
   \xy
  (-4,8)*{\lambda+(n)_{\Lambda}};
  (12,0)*{\dred\bbe{}};
  (12,-12)*{\scs n};
  (-6,-2)*{\dblue\cbub{\black (\overline{\lambda +(n)_{\Lambda}})_1-1+g}{\black 1}};
  (12,6)*{\dred\bullet}+(2,-1)*{\scs p};
 \endxy
\label{eq:extrabubble441n}
\end{align}

\begin{align}
\text{$ 
   \xy
  (6,8)*{\lambda};
  (12,0)*{\dred\bbe{}};
  (12,-12)*{\scs 1};
  (0,-2)*{\dblue\cbub{\black \llambda_n-1+m}{\black n}};
  (-12,6)*{ }+(7,-1)*{\scs  };
 \endxy
$}
&= 
\quad
   \xy
  (18,8)*{\lambda-(1)_{\Lambda}};
  (0,0)*{\dred\bbe{}};
  (0,-12)*{\scs 1};
  (16,-2)*{\dblue\cbub{\black (\overline{\lambda -(1)_{\Lambda}})_n-2+m}{\black n}};
  (0,6)*{\dred\bullet}+(-3,-1)*{\scs };
 \endxy
\quad-\quad
   \xy
  (18,8)*{\lambda-(1)_{\Lambda}};
  (0,0)*{\dred\bbe{}};
  (0,-12)*{\scs 1};
  (16,-2)*{\dblue\cbub{\black (\overline{\lambda -(1)_{\Lambda}})_n-1+m}{\black n}};
 \endxy \nn
\\ 
& -\quad \bbox{y}
   \xy
  (18,8)*{\lambda-(1)_{\Lambda}};
  (0,0)*{\dred\bbe{}};
  (0,-12)*{\scs 1};
  (16,-2)*{\dblue\cbub{\black (\overline{\lambda -(1)_{\Lambda}})_n-2+m}{\black n}};
 \endxy
\label{eq:extrabubble451n}
\end{align}

\begin{align}
\text{$ 
   \xy
  (6,8)*{\lambda};
  (12,0)*{\dred\bbe{}};
  (12,-12)*{\scs n};
  (0,-2)*{\dblue\cbub{\black \llambda_1-1+m}{\black 1}};
  (-12,6)*{ }+(7,-1)*{\scs  };
 \endxy
$}
&= 
\quad -
   \xy
  (18,8)*{\lambda-(n)_{\Lambda}};
  (0,0)*{\dred\bbe{}};
  (0,-12)*{\scs n};
  (16,-2)*{\dblue\cbub{\black (\overline{\lambda -(n)_{\Lambda}})_1-2+m}{\black 1}};
  (0,6)*{\dred\bullet}+(-3,-1)*{\scs };
 \endxy
\quad+\quad
   \xy
  (18,8)*{\lambda-(n)_{\Lambda}};
  (0,0)*{\dred\bbe{}};
  (0,-12)*{\scs n};
  (16,-2)*{\dblue\cbub{\black (\overline{\lambda -(n)_{\Lambda}})_1-1+m}{\black 1}};
 \endxy \nn
\\ 
&-\quad \bbox{y}
   \xy
  (18,8)*{\lambda-(n)_{\Lambda}};
  (0,0)*{\dred\bbe{}};
  (0,-12)*{\scs n};
  (16,-2)*{\dblue\cbub{\black (\overline{\lambda -(n)_{\Lambda}})_1-2+m}{\black 1}};
 \endxy
\label{eq:extrabubble461n}
\end{align}

%%%%%%%%%%%%%%%%%%
\n $\star$ More Reidemeister $3$ like relations:

\n Unless $i=k=j$ we have
\begin{equation} \label{eq_other_r3_1}
\text{$
 \vcenter{
 \xy 0;/r.18pc/:
(0,0)*{\dblue\xybox{
    (-4,-4)*{};(4,4)*{} **\crv{(-4,-1) & (4,1)}?(1)*\dir{>};
    (4,4)*{};(12,12)*{} **\crv{(4,7) & (12,9)}?(1)*\dir{>};
    (12,12)*{}; (12,20) **\dir{-};
}};
(-4,0)*{\dred\xybox{
    (4,-4)*{};(-4,4)*{} **\crv{(4,-1) & (-4,1)}?(0)*\dir{<};
    (-4,12)*{};(4,20)*{} **\crv{(-4,15) & (4,17)}?(0)*\dir{<};
    (-4,4)*{}; (-4,12) **\dir{-};
}};
(0,0)*{\dgreen\xybox{
    (12,4)*{};(4,12)*{} **\crv{(12,7) & (4,9)}?(1)*\dir{>};
    (4,12)*{};(-4,20)*{} **\crv{(4,15) & (-4,17)}?(1)*\dir{>};
    (12,-4)*{}; (12,4) **\dir{-};
}};
  (12,0)*{\lambda};
  ( -10,-11)*{\scs i};
  ( 2.5,-11)*{\scs j};
  (10.5,-11)*{\scs k};
\endxy}
 \;\; =\;\;
 \vcenter{
 \xy 0;/r.18pc/:
(0,0)*{\dgreen\xybox{
    (4,-4)*{};(-4,4)*{} **\crv{(4,-1) & (-4,1)}?(1)*\dir{>};
    (-4,4)*{};(-12,12)*{} **\crv{(-4,7) & (-12,9)}?(1)*\dir{>};
    (-12,12)*{}; (-12,20) **\dir{-};
}};
(4,0)*{\dred\xybox{
    (-4,-4)*{};(4,4)*{} **\crv{(-4,-1) & (4,1)}?(0)*\dir{<};
    (4,12)*{};(-4,20)*{} **\crv{(4,15) & (-4,17)}?(0)*\dir{<};
    (4,4)*{}; (4,12) **\dir{-};
}};
(0,0)*{\dblue\xybox{
    (-12,4)*{};(-4,12)*{} **\crv{(-12,7) & (-4,9)}?(1)*\dir{>};
    (-4,12)*{};(4,20)*{} **\crv{(-4,15) & (4,17)}?(1)*\dir{>};
    (-12,-4)*{}; (-12,4) **\dir{-};
}};
  (12,0)*{\lambda};
  (10,-11)*{\scs k};
  (-2.5,-11)*{\scs j};
  (-9.5,-11)*{\scs i};
\endxy}
$}
\end{equation}
and when $i=j=k$ we have  
\begin{equation} \label{eq_r3_extra}
\text{$
 \vcenter{
 \xy 0;/r.17pc/:
 (0,0)*{\dblue\xybox{
    (-4,-4)*{};(4,4)*{} **\crv{(-4,-1) & (4,1)}?(1)*\dir{>};
    (4,-4)*{};(-4,4)*{} **\crv{(4,-1) & (-4,1)}?(1)*\dir{<};
    ?(0)*\dir{<};
    (4,4)*{};(12,12)*{} **\crv{(4,7) & (12,9)}?(1)*\dir{>};
    (12,4)*{};(4,12)*{} **\crv{(12,7) & (4,9)}?(1)*\dir{>};
    (-4,12)*{};(4,20)*{} **\crv{(-4,15) & (4,17)};
    (4,12)*{};(-4,20)*{} **\crv{(4,15) & (-4,17)}?(1)*\dir{>};
    (-4,4)*{}; (-4,12) **\dir{-};
    (12,-4)*{}; (12,4) **\dir{-};
    (12,12)*{}; (12,20) **\dir{-};
}};
  (12,0)*{\lambda};
  (-10,-11)*{\scs i};
  ( 2.5,-11)*{\scs i};
  ( 9.5,-11)*{\scs i};
\endxy}
-\;
   \vcenter{
 \xy 0;/r.17pc/:
(0,0)*{\dblue\xybox{ 
   (4,-4)*{};(-4,4)*{} **\crv{(4,-1) & (-4,1)}?(1)*\dir{>};
    (-4,-4)*{};(4,4)*{} **\crv{(-4,-1) & (4,1)}?(0)*\dir{<};
    (-4,4)*{};(-12,12)*{} **\crv{(-4,7) & (-12,9)}?(1)*\dir{>};
    (-12,4)*{};(-4,12)*{} **\crv{(-12,7) & (-4,9)}?(1)*\dir{>};
    (4,12)*{};(-4,20)*{} **\crv{(4,15) & (-4,17)};
    (-4,12)*{};(4,20)*{} **\crv{(-4,15) & (4,17)}?(1)*\dir{>};
    (4,4)*{}; (4,12) **\dir{-} ?(.5)*\dir{<};
    (-12,-4)*{}; (-12,4) **\dir{-};
    (-12,12)*{}; (-12,20) **\dir{-};
}};
  (12,0)*{\lambda};
  (10  ,-11)*{\scs i};
  (-2.5,-11)*{\scs i};
  (-9.5,-11)*{\scs i};
\endxy}
  \; = \;
 \sum_{} \; \xy 0;/r.17pc/:
(0,0)*{\dblue\xybox{ 
   (-4,12)*{}="t1";
    (4,12)*{}="t2";
  "t2";"t1" **\crv{(5,5) & (-5,5)}; ?(.15)*\dir{} ?(.9)*\dir{>}
  ?(.2)*\dir{}+(0,-.2)*{\bullet}+(3,-2)*{\bscs f_1};
    (-4,-12)*{}="t1";
    (4,-12)*{}="t2";
  "t2";"t1" **\crv{(5,-5) & (-5,-5)}; ?(.05)*\dir{} ?(.9)*\dir{<}
  ?(.15)*\dir{}+(0,-.2)*{\bullet}+(3,2)*{\bscs f_3};
    (-8.5,0.5)*{\ccbub{\bscs -\llambda_i-3+f_4}{\black i}};
    (13,12)*{};(13,-12)*{} **\dir{-} ?(.5)*\dir{<};
    (13,8)*{\bullet}+(3,2)*{\bscs f_2};
}};
(18,-6)*{\lambda};
  \endxy
+\;
  \sum_{}
\; \xy 0;/r.17pc/: 
(0,0)*{\dblue\xybox{
  (-10,12)*{};(-10,-12)*{} **\dir{-} ?(.5)*\dir{<};
  (-10,8)*{\bullet}+(-3,2)*{\bscs g_2};
  (-4,12)*{}="t1";
  (4,12)*{}="t2";
  "t1";"t2" **\crv{(-4,5) & (4,5)}; ?(.15)*\dir{>} ?(.9)*\dir{>}
  ?(.4)*\dir{}+(0,-.2)*{\bullet}+(3,-2)*{\bscs \;\; g_1};
  (-4,-12)*{}="t1";
  (4,-12)*{}="t2";
  "t2";"t1" **\crv{(4,-5) & (-4,-5)}; ?(.12)*\dir{>} ?(.97)*\dir{>}
  ?(.8)*\dir{}+(0,-.2)*{\bullet}+(1,4)*{\bscs g_3};
  (16.6,-4.5)*{\cbub{\bscs \llambda_i-1+g_4}{\black i}};
}};
  (18,6)*{\lambda};
  \endxy
$}
\end{equation}
where the first sum is over all $f_1, f_2, f_3, f_4 \geq 0$ with
$f_1+f_2+f_3+f_4=\llambda_i$ and the second sum is over all $g_1, g_2,
g_3, g_4 \geq 0$ with $g_1+g_2+g_3+g_4=\llambda_i -2$. Note that the 
first summation is zero if $\llambda_i<0$ and the second is zero 
when $\llambda_i<2$.

Reidemeister $3$ like relations for all other orientations are determined from
\eqref{eq_r3_easy-gen}, \eqref{eq_r3_hard-gen}, and the above relations using
duality.

%%%%%%%%%%%%%%%%%%%%%%%%%%%%%%%%%%%%%%%%%%%%%%%%%%%%%%%%%%%%%%%%%%%%%

\subsubsection{The 2-category $\Scat(n,r)_{[y]}$}
As explained in Section~\ref{SchAlg}, the $q$-Schur algebra $\hat{\SD}(n,r)$ can be seen as a quotient of $\Ugla$ by the ideal generated by all idempotents corresponding to the weights that do not belong to $\Lambda(n,r)$.

It is then natural to define the 2-category $\Scat(n,r)_{[y]}$ as a quotient of 
$\Uglcataffy$ as follows. 

\begin{defn}
The 2-category $\Scat(n,r)_{[y]}$ is the quotient of $\Uglcataffy$ by the ideal 
generated by all 2-morphisms containing a region with a label not in $\Lambda(n,r)$. 
\end{defn}

We remark that we only put real bubbles, whose interior has a label outside 
$\Lambda(n,r)$, equal to zero. To see what happens to a fake bubble, one 
first has to write it in terms of real bubbles with the opposite orientation 
using the infinite Grassmannian relation~\eqref{eq_infinite_Grass}.

\vskip0.2cm 
As in~\cite{MSVschur}, we define $\hat{\mathcal{S}}(n,r)_{[y]}$ as a quotient 
of $\Uglcataffy$, rather than $\Ucataffy$. Therefore, 
just as in~\cite{MSVschur} (see the introduction of Sections $3$ and 
$4.3$ in that paper), we have to show that there exists a 
full and essentially surjective $2$-functor 
$$\Psi_{n,r}\colon \Ucataffy\to 
\hat{\mathcal{S}}(n,r)_{[y]},$$ 
which categorifies the surjective homomorphism 
$$\psi_{n,r}\colon \Uaff \to \hat{\mathbf S}(n,r)$$
defined in~\eqref{eq:psi}. 

On objects $\Psi_{n,r}$ maps $\mu$ to $\lambda:=\phi_{n,r}(\mu)$, which was 
defined in Section~\ref{genalg}. On $1$ and $2$--morphisms $\Psi_{n,r}$ is 
defined to be the identity except for the left cups and caps, on which 
it is given by
\begin{equation}
\label{eq:signs}
\Ucapli_{i,\mu}\mapsto (-1)^{\lambda_{i+1}+1}\,\,\Ucapli_{i,\lambda}
\quad\mbox{and}\quad \Ucupli_{i,\mu}\mapsto (-1)^{\lambda_{i+1}}\,\,
\Ucupli_{i,\lambda}.
\end{equation}  
Note that here we are simply extending the $2$--functor used 
in~\cite{MSVschur}. 

Just for completeness, we now state the following result without proof. 
\begin{prop}
\label{prop:affslnquotient}
The $2$-functor 
$$\Psi_{n,r}\colon \Ucataffy\to 
\hat{\mathcal{S}}(n,r)_{[y]}$$
is well-defined, full and essentially surjective. 
\end{prop} 

Just as for $\Ucataffy$ and $\Uglcataffy$, we can put $y=0$.
\begin{defn}
Let $\Scat(n,r)$ be the quotient of $\Scat(n,r)_{[y]}$ by the ideal generated 
by $y$. 
\end{defn}
\n Of course there also exists a full and essentially surjective $2$-functor 
$$\Psi_{n,r}\colon \Ucataff\to \Scat(n,r),$$
which is defined and denoted just as above.

\section{A functor from $\debim_{\hat{A}_{r-1}}^*$ to $\Scat(n,r)^*_{[y]}$}
\label{sec:embed}
In this section, we define a functor 
$$\Sigma_{n,r}\colon\debim_{\hat{A}_{r-1}}^*\to \Scat(n,r)^*_{[y]},$$
which categorifies the embedding 
$$\sigma_{n,r}\colon\hat{\He}_{\hat{A}_{r-1}}\to 1_{r}\hat{\SD}(n,r)1_{r}.$$  
Actually its target will be the one-object sub-2-category $\Scat(n,r)^*_{[y]}((1^r),(1^r))$ of~$\Scat(n,r)^*_{[y]}$. Since $\Scat(n,r)^*_{[y]}((1^r),(1^r))$ has only one object, 
it can be seen as a monoidal category. 

This functor is the affine analogue of $\Sigma_{n,d}$ in Section 6.5 in~\cite{MSVschur}. For diagrams with only unoriented $i$-colored strands for 
$i=1,\ldots, r-1$ the definitions in that paper and in this one coincide, 
for diagrams with unoriented $r$--colored strands or 
oriented strands the definition here is new.

In Section~\ref{GrAlg}, we will prove that
$$\Sigma_{n,r}: \debim_{\hat{A}_{r-1}}^* \to \Scat(n,r)^*_{[y]}((1^r),(1^r))$$ 
is faithful. We conjecture that $\Sigma_{n,r}$ is also full and, therefore, that the two categories $\debim_{\hat{A}_{r-1}}$ and $\Scat(n,r)_{[y]}((1^r),(1^r))$ are equivalent. The latter equivalence would the affine analogue of the 
one proved in Proposition 6.9 in~\cite{MSVschur} for finite type $A$.  

\subsection{The definition of the functor}
\label{ssec:scqs}

The functor $\Sigma_{n,r}$ is $\Q$-linear and monoidal, so it is sufficient to specify the image of the generating objects and 
morphisms. 

The functor $\Sigma_{n,r}$ is defined on objects by  
\begin{eqnarray*}
 \emptyset & \mapsto& {\mathbf 1}_r={\mathbf 1}_{(1^r)} \\
 i & \mapsto & \mathcal{E}_{-i}\mathcal{E}_{i}{\mathbf 1}_r \\
 r & \mapsto & \mathcal{E}_{-n} \dots \mathcal{E}_{-r} \mathcal{E}_{r} \dots \mathcal{E}_{n}{\mathbf 1}_r \\
 + & \mapsto & \mathcal{E}_{-n} \dots \mathcal{E}_{-r-1} \mathcal{E}_{-1} \dots \mathcal{E}_{-r}{\mathbf 1}_r \\
 - & \mapsto & \mathcal{E}_{r} \dots \mathcal{E}_{1} \mathcal{E}_{r+1} \dots \mathcal{E}_{n}{\mathbf 1}_r.
\end{eqnarray*}
The functor $\Sigma_{n,r}$ is defined on morphisms as follows:
\begin{itemize}
\item The empty diagram is sent to the empty diagram in the region labeled $(1^r)$

\item The vertical line colored $i$ is sent to the identity $2$-morphism on $\mathcal{E}_{-i}\mathcal{E}_{i}{\mathbf 1}_r$:
\begin{equation*}
\labellist
\tiny\hair 2pt
\pinlabel $i$   at -10  60
\endlabellist
\figins{-16}{0.5}{line}\ \
\longmapsto\ \ \
\text{$
 \xy 
 (0,0)*{\dblue\xybox{
 (-5,7);(-5,-7); **\dir{-} ?(.5)*\dir{>}+(2.3,0)*{\scriptstyle{}};}};
 (-6.3,-9)*{\scs i};
 (0,0)*{\dblue\xybox{
 (10,7);(10,-7); **\dir{-} ?(.5)*\dir{<}+(12.3,0)*{\scriptstyle{}};}};
 ( -1.2,-9)*{\scs i};
 (6,0)*{ (1^r)};
  \endxy
$}
\vspace*{2ex}
\end{equation*}

\item The vertical line colored $r$ is sent to the identity $2$-morphism on $\mathcal{E}_{-n} \dots \mathcal{E}_{-r} \mathcal{E}_{r} \dots \mathcal{E}_{n}{\mathbf 1}_r$:
\begin{equation*}
\labellist
\tiny\hair 2pt
\pinlabel $r$   at -10  60
\endlabellist
\figins{-16}{0.5}{line}\ \
\longmapsto\ \ \
 \xy 
 (-28,0)*{\dblue\xybox{(-28,8);(-28,-8); **\dir{-} ?(.5)*\dir{<};}};
 (-23,0)*{\cdots};
 (-18,0)*{\dred\xybox{(-18,8);(-18,-8); **\dir{-} ?(.5)*\dir{<};}};
 (-13,0)*{\dred\xybox{(-13,8);(-13,-8); **\dir{-} ?(.5)*\dir{>};}};
 (-8,0)*{\cdots};
 (-3,0)*{\dblue\xybox{(-3,8);(-3,-8); **\dir{-} ?(.5)*\dir{>};}};
 (-28,-10)*{\scs n}; (-18,-10)*{\scs r }; (-13,-10)*{ \scs r};(-3,-10)*{\scs n};
 (5,0)*{ (1^r)}; 
  \endxy
\vspace*{2ex}
\end{equation*}

\item The vertical line colored $+$ is sent to the identity $2$-morphism on $\mathcal{E}_{-n} \dots \mathcal{E}_{-r-1} \mathcal{E}_{-1} \dots \mathcal{E}_{-r}{\mathbf 1}_r$:
\begin{equation*}
\figins{-16}{0.5}{rho}\ \
\longmapsto\ \ \
 \xy 
 (-28,0)*{\dblue\xybox{(-28,8);(-28,-8); **\dir{-} ?(.5)*\dir{<};}};
 (-23,0)*{\cdots};
 (-18,0)*{\dpink\xybox{(-18,8);(-18,-8); **\dir{-} ?(.5)*\dir{<};}};
 (-13,0)*{\dgreen\xybox{(-13,8);(-13,-8); **\dir{-} ?(.5)*\dir{<};}};
 (-8,0)*{\cdots};
 (-3,0)*{\dred\xybox{(-3,8);(-3,-8); **\dir{-} ?(.5)*\dir{<};}};
 (-28,-10)*{\scs n}; (-18,-10)*{\scs r+1 }; (-13,-10)*{ \scs 1};(-3,-10)*{\scs r};
 (5,0)*{ (1^r)}; 
  \endxy
\vspace*{2ex}
\end{equation*}

\item The vertical line colored $-$ is sent to the identity $2$-morphism on $\mathcal{E}_{r} \dots \mathcal{E}_{1} \mathcal{E}_{r+1} \dots \mathcal{E}_{n}{\mathbf 1}_r$:
\begin{equation*}
\figins{-16}{0.5}{rho-1}\ \
\longmapsto\ \ \
 \xy 
 (-28,0)*{\dred\xybox{(-28,8);(-28,-8); **\dir{-} ?(.5)*\dir{>};}};
 (-23,0)*{\cdots};
 (-18,0)*{\dgreen\xybox{(-18,8);(-18,-8); **\dir{-} ?(.5)*\dir{>};}};
 (-13,0)*{\dpink\xybox{(-13,8);(-13,-8); **\dir{-} ?(.5)*\dir{>};}};
 (-8,0)*{\cdots};
 (-3,0)*{\dblue\xybox{(-3,8);(-3,-8); **\dir{-} ?(.5)*\dir{>};}};
 (-28,-10)*{\scs r}; (-18,-10)*{\scs 1 }; (-13,-10)*{ \scs r+1};(-3,-10)*{\scs n};
 (5,0)*{ (1^r)}; 
  \endxy
\vspace*{2ex}
\end{equation*}

\item The images of the startdot$_i$ and enddot$_i$ morphisms for $i\neq r$:
\begin{equation*}
\labellist
\tiny\hair 2pt
\pinlabel $i$   at -10  60
\endlabellist
\figins{-16}{0.5}{startdot}
\longmapsto\ 
    {\dblue \xy
    (0,2)*{\bbpef{\black i}};
    (10,2)*{\black (1^r) };
    \endxy}
\mspace{140mu}
\labellist
\tiny\hair 2pt
\pinlabel $i$   at -10  60
\endlabellist
\figins{-16}{0.5}{enddot}
\longmapsto\
    {\dblue \xy
    (0,-2.5)*{\bbcef{\black i}};
    (8,0.5)*{ \black (1^r) };
    \endxy}
\vspace*{2ex}
\end{equation*}

\item The images of the startdot$_r$ and enddot$_r$ morphisms:
\begin{equation*}
\labellist
\tiny\hair 2pt
\pinlabel $r$   at -10  60
\endlabellist
\figins{-16}{0.5}{startdot}
\longmapsto\ 
\xy 
 (0,3)*{\dred\xybox{(-3,0)*{};(3,0)*{} **\crv{(-3,-5) & (3,-5)}; ?(.5)*\dir{>};}};
 (0,-4)*{\dblue\xybox{(-15,0)*{};(15,0)*{} **\crv{(-15,-20) & (15,-20)}; ?(.5)*\dir{>};}};
 (-9,4)*{\cdots};
  (9,4)*{\cdots};
  (-15,8)*{\scs n}; (-3,8)*{\scs r }; (3,8)*{ \scs r};(15,8)*{\scs n};
 (20,0)*{ (1^r)}; 
  \endxy
\mspace{100mu}
\labellist
\tiny\hair 2pt
\pinlabel $r$   at -10  60
\endlabellist
\figins{-16}{0.5}{enddot}
\longmapsto\
\xy 
 (0,-3)*{\dred\xybox{(-3,0)*{};(3,0)*{} **\crv{(-3,5) & (3,5)}; ?(.5)*\dir{<};}};
 (0,4)*{\dblue\xybox{(-15,0)*{};(15,0)*{} **\crv{(-15,20) & (15,20)}; ?(.5)*\dir{<};}};
 (-9,-4)*{\cdots};
  (9,-4)*{\cdots};
  (-15,-8)*{\scs n}; (-3,-8)*{\scs r }; (3,-8)*{ \scs r};(15,-8)*{\scs n};
 (20,0)*{ (1^r)}; 
  \endxy
    \vspace*{2ex}
\end{equation*}

\item The images of the merge$_i$ and split$_i$ morphisms for $i\neq r$: 
\begin{equation*}
\labellist
\hair 2pt
\pinlabel $\scs i$   at  45  95
\endlabellist
\figins{-16}{0.6}{merge}
\longmapsto\
\xy 0;/r.16pc/; 
    (0,-1.5)*{\dblue\bbcfe{\black i}};
    (14,4)*{(1^r) };    
    (-12,3)*{};(12,3)*{};
    ( 7.5,2)*{\dblue\xybox{
    (-3,-5)*{}; (-10,8.5) **\crv{(-3,1) & (-10,3)}?(1)*\dir{>};}};
    ( -7.5,2)*{\dblue\xybox{
    ( 3,-5)*{}; ( 10,8.5) **\crv{( 3,1) & ( 10,3)}?(0)*\dir{<};}};
    (-11,-7)*{\scs i};
    ( 11,-7)*{\scs i};
    \endxy
\mspace{80mu}
\labellist
\hair 2pt
\pinlabel $\scs i$   at  45  45
\endlabellist
\figins{-16}{0.6}{split}
\longmapsto \
\xy 0;/r.16pc/; 
    (0,5)*{\dblue\bbpfe{\black i}};
    (14,0)*{(1^r) };    
    (-12,3)*{};(12,3)*{};
    (-7.5,2)*{\dblue\xybox{
    (-3,-5)*{}; (-10,8.5) **\crv{(-3,1) & (-10,3)}?(0)*\dir{<};}};
    ( 7.5,2)*{\dblue\xybox{
    ( 3,-5)*{}; ( 10,8.5) **\crv{( 3,1) & ( 10,3)}?(1)*\dir{>};}};
    (-4,-7)*{\scs i};
    ( 4,-7)*{\scs i};
    \endxy
\vspace*{2ex}
\end{equation*}

\item The images of the merge$_r$ and split$_r$ morphisms: 
\begin{equation*}
\labellist
\hair 2pt
\pinlabel $\scs r$   at  45  95
\endlabellist
\figins{-16}{0.6}{merge}
\longmapsto\
\xy 
 (-12,6)*{\dred\xybox{(-3,18)*{};(-21,-6)*{} **\crv{(-3,10) & (-21,10)}; ?(.5)*\dir{>};}};
 (12,6)*{\dred\xybox{(3,18)*{};(21,-6)*{} **\crv{(3,10) & (21,10)}; ?(.5)*\dir{<};}};  
 (-24,6)*{\dblue\xybox{(-15,18)*{};(-33,-6)*{} **\crv{(-15,10) & (-33,10)}; ?(.5)*\dir{>};}};
 (24,6)*{\dblue\xybox{(15,18)*{};(33,-6)*{} **\crv{(15,10) & (33,10)}; ?(.5)*\dir{<};}};
 (0,-3)*{\dblue\xybox{(-3,0)*{};(3,0)*{} **\crv{(-3,5) & (3,5)}; ?(.5)*\dir{>};}};
 (0,4)*{\dred\xybox{(-15,0)*{};(15,0)*{} **\crv{(-15,20) & (15,20)}; ?(.5)*\dir{>};}};
 (-9,-4)*{\cdots};
  (9,-4)*{\cdots};
  (-9,16)*{\cdots};
  (9,16)*{\cdots};
 (-27,-4)*{\cdots};
  (27,-4)*{\cdots};
  (-15,-8)*{\scs r}; (-3,-8)*{\scs n }; (3,-8)*{ \scs n};(15,-8)*{\scs r}; (-21,-8)*{\scs r}; (-33,-8)*{\scs n }; (33,-8)*{ \scs n};(21,-8)*{\scs r};
 (38,4)*{ (1^r)}; 
  \endxy
\end{equation*}

\begin{equation*}
\labellist
\hair 2pt
\pinlabel $\scs r$   at  45  45
\endlabellist
\figins{-16}{0.6}{split}
\longmapsto \
\xy 
 (-12,-6)*{\dred\xybox{(-3,-18)*{};(-21,6)*{} **\crv{(-3,-10) & (-21,-10)}; ?(.5)*\dir{<};}};
 (12,-6)*{\dred\xybox{(3,-18)*{};(21,6)*{} **\crv{(3,-10) & (21,-10)}; ?(.5)*\dir{>};}};  
 (-24,-6)*{\dblue\xybox{(-15,-18)*{};(-33,6)*{} **\crv{(-15,-10) & (-33,-10)}; ?(.5)*\dir{<};}};
 (24,-6)*{\dblue\xybox{(15,-18)*{};(33,6)*{} **\crv{(15,-10) & (33,-10)}; ?(.5)*\dir{>};}};
 (0,3)*{\dblue\xybox{(-3,0)*{};(3,0)*{} **\crv{(-3,-5) & (3,-5)}; ?(.5)*\dir{<};}};
 (0,-4)*{\dred\xybox{(-15,0)*{};(15,0)*{} **\crv{(-15,-20) & (15,-20)}; ?(.5)*\dir{<};}};
 (-9,4)*{\cdots};
  (9,4)*{\cdots};
 (-9,-16)*{\cdots};
  (9,-16)*{\cdots};
 (-27,4)*{\cdots};
  (27,4)*{\cdots};
  (-15,8)*{\scs r}; (-3,8)*{\scs n }; (3,8)*{ \scs n};(15,8)*{\scs r}; (-21,8)*{\scs r}; (-33,8)*{\scs n }; (33,8)*{ \scs n};(21,8)*{\scs r};
 (38,-4)*{ (1^r)}; 
  \endxy
\end{equation*}

\item The image of the $4$--valent vertex morphism $4$vert$_{i,j}$ with distant colors $i$ and $j$ different from $r$:
\begin{equation*}
\labellist
\hair 2pt
\pinlabel $\scs i$   at  -5 -12
\pinlabel $\scs j$   at 128 -10
\endlabellist
\figins{-16}{0.6}{4vert}
\longmapsto\ \ \ 
\text{$
\xy 0;/r.16pc/; 
    (14,0)*{(1^r) };    
    ( 0,0)*{\dblue\xybox{
    ( 0,-11)*{}; (-15,8.5) **\crv{( 0,-11) & (-15,8.5)}?(0)*\dir{<};
    ( 5,-11)*{}; (-10,8.5) **\crv{( 5,-11) & (-10,8.5)}?(1)*\dir{>};}};
    ( 0,0)*{\dred\xybox{  
    ( 0,-11)*{}; ( 15,8.5) **\crv{( 0,-11) & ( 15,8.5)}?(0)*\dir{<};
    ( 5,-11)*{}; ( 20,8.5) **\crv{( 5,-11) & ( 20,8.5)}?(1)*\dir{>};}};
    (-10,-12)*{\scs i}; (-5,-12)*{\scs i};
    ( 10,-12)*{\scs j}; ( 5,-12)*{\scs j};
    \endxy
$}
\vspace*{2ex}
\end{equation*}

\item The images of the $4$--valent vertex morphisms $4$vert$_{r,j}$ and $4$vert$_{j,r}$ with distant colors $r$ and $j$:
\begin{equation*}
\labellist
\hair 2pt
\pinlabel $\scs r$   at  -5 -12
\pinlabel $\scs j$   at 128 -10
\endlabellist
\figins{-16}{0.6}{4vert}
\longmapsto\ \ \
 \xy 
 (-28,0)*{\dblue\xybox{(-18,8);(-28,-8); **\dir{-} ?(0)*\dir{<};}};
 (-27,-6)*{\cdots};
 (-19,6)*{\cdots};
 (-18,0)*{\dred\xybox{(-8,8);(-18,-8); **\dir{-} ?(0)*\dir{<};}};
 (-13,0)*{\dred\xybox{(-3,8);(-13,-8); **\dir{-} ?(1)*\dir{>};}};
 (-12,-6)*{\cdots};
 (-4,6)*{\cdots};
 (-3,0)*{\dblue\xybox{(7,8);(-3,-8); **\dir{-} ?(1)*\dir{>};}};
 (-18,0)*{\dgreen\xybox{(-28,8);(2,-8); **\dir{-} ?(0)*\dir{<};}};
 (-13,0)*{\dgreen\xybox{(-23,8);(7,-8); **\dir{-} ?(1)*\dir{>};}};
 (-33,-10)*{\scs n}; (-23,-10)*{\scs r }; (-18,-10)*{ \scs r};(-8,-10)*{\scs n}; (-3,-10)*{\scs j}; (2,-10)*{\scs j};
 (7,0)*{ (1^r)}; 
  \endxy
\vspace*{2ex}
\end{equation*}

\begin{equation*}
\labellist
\hair 2pt
\pinlabel $\scs j$   at  -5 -12
\pinlabel $\scs r$   at 128 -10
\endlabellist
\figins{-16}{0.6}{4vert}
\longmapsto\ \ \
 \xy 
 (-28,0)*{\dblue\xybox{(-18,-8);(-28,8); **\dir{-} ?(1)*\dir{>};}};
 (-19,-6)*{\cdots};
 (-27,6)*{\cdots};
 (-18,0)*{\dred\xybox{(-8,-8);(-18,8); **\dir{-} ?(1)*\dir{>};}};
 (-13,0)*{\dred\xybox{(-3,-8);(-13,8); **\dir{-} ?(0)*\dir{<};}};
 (-4,-6)*{\cdots};
 (-12,6)*{\cdots};
 (-3,0)*{\dblue\xybox{(7,-8);(-3,8); **\dir{-} ?(0)*\dir{<};}};
 (-18,0)*{\dgreen\xybox{(-28,-8);(2,8); **\dir{-} ?(1)*\dir{>};}};
 (-13,0)*{\dgreen\xybox{(-23,-8);(7,8); **\dir{-} ?(0)*\dir{<};}};
 (-33,-10)*{\scs j}; (-23,-10)*{\scs n }; (-28,-10)*{ \scs j};(-8,-10)*{\scs r}; (-13,-10)*{\scs r}; (2,-10)*{\scs n};
 (7,0)*{ (1^r)}; 
  \endxy
\vspace*{2ex}
\end{equation*}

\item The images of the $6$--valent vertex morphisms $6$vert$_{i+1,i}$ and $6$vert$_{i,i+1}$ with colors $i$ and $i+1$ different from $r$:
\begin{equation}\label{eq:sixval}
\labellist
\tiny\hair 2pt
\pinlabel $i+1$ at -5 -10
\pinlabel $i$   at 65 -10
\pinlabel $i+1$ at 135 -10
\endlabellist
\figins{-18}{0.6}{6vertu}\
\longmapsto\ \ \
\text{$
\xy 0;/r.16pc/; 
    (16,0)*{(1^r) };    
    ( 0,0)*{\dblue\xybox{
    (-7.5,10)*{}; (    5,-10) **\crv{(-4.5, 7) & ( 7.5,0) & ( 5,-9)}?(0)*\dir{<};
    (12.5,10)*{}; (    0,-10) **\crv{( 9.5, 7) & (-2.5,0) & ( 0,-9)}?(1)*\dir{>};
    (17.5,10)*{}; (-12.5, 10) **\crv{(   8, 0) & ( 2.5,-6) & (-3,0)}?(0)*\dir{<};
}};
    ( 0,0)*{\dred\xybox{  
    (-10,-20)*{};( 10,-20) **\crv{(-9,-19) & (0,-12) & (8,-19)}?(.2)*\dir{>} ?(.8)*\dir{>};
    ( 2.5, 0)*{};( 15,-20) **\crv{( 2.5,0) & ( 2,-10) & ( 15,-20)}?(0)*\dir{<};
    (-2.5, 0)*{};(-15,-20) **\crv{(-2.5,0) & (-2,-10) & (-15,-20)}?(1)*\dir{>};
}};  
    ( -17,-12)*{\scs i+1}; (-10,-12)*{\scs i+1};
    (-2.5,-12)*{\scs i }; (2.5,-12)*{\scs i };
    ( 16,-12)*{\scs i+1}; 
    ( 16, 12)*{\scs i};
    \endxy
$}
\vspace*{2ex}
\end{equation}

\begin{equation*}
\labellist
\tiny\hair 2pt
\pinlabel $i$ at -5 -10
\pinlabel $i+1$   at 65 -10
\pinlabel $i$ at 135 -10
\endlabellist
\figins{-18}{0.6}{6vertd}\
\longmapsto\ \ \
\text{$
\xy 0;/r.16pc/; 
    (16,0)*{(1^r) };    
    ( 0,0)*{\dblue\xybox{
    (-7.5,-10)*{}; (  5, 10) **\crv{(-4.5,-7) & ( 7.5,0) & ( 5,9)}?(1)*\dir{>};
    (12.5,-10)*{}; (  0, 10) **\crv{( 9.5,-7) & (-2.5,0) & ( 0,9)}?(0)*\dir{<};
    (17.5,-10)*{}; (-12.5,-10) **\crv{(  8,  0) & (2.5,6) & (-3, 0)}?(1)*\dir{>};
}};
    ( 0,0)*{\dred\xybox{  
    (-10,0)*{};(10,0) **\crv{(-9,-1) & (0,-8) & (8,-1)}?(.2)*\dir{<} ?(.8)*\dir{<};
    ( 2.5,-20)*{}; (  15, 0) **\crv{( 2.5,-20) & ( 2,-10) & ( 15,0)}?(1)*\dir{>};
    (-2.5,-20)*{}; ( -15, 0) **\crv{(-2.5,-20) & (-2,-10) & (-15,0)}?(0)*\dir{<};
}};  
    ( -16,-12)*{\scs   i}; (-11,-12)*{\scs   i};
    (-3.5,-12)*{\scs i+1}; (3.5,-12)*{\scs i+1};
    (  11,-12)*{\scs i  }; ( 16,-12)*{\scs   i}; 
    (  10, 12)*{\scs i+1};
    \endxy
$}
\vspace*{2ex}
\end{equation*}

\item The images of the $6$--valent vertex morphisms $6$vert$_{r,1}$, $6$vert$_{1,r}$, $6$vert$_{r-1,r}$ and $6$vert$_{r,r-1}$:
\begin{eqnarray*}
 \labellist
\tiny\hair 2pt
\pinlabel $r$ at -5 -10
\pinlabel $1$   at 65 -10
\pinlabel $r$ at 135 -10
\endlabellist
\figins{-55}{0.6}{6vertd}\
& \xy (0,0)*{}; (0,-12)*{\longmapsto}; \endxy &
 \xy
(-15,-23)*{\cdots};
(-50,-12.5)*{\dred\xybox{(-75,-25);(-55,0) **\crv{(-75,-5) & (-55,-10)} ?(1)*\dir{>};}};
(-40,-12.5)*{\dblue\xybox{(-65,-25);(-45,0) **\crv{(-65,-5) & (-45,-10)} ?(1)*\dir{>};}};
(-45,-12.5)*{\dblue\xybox{(-20,-25)*{};(-40,0)*{} **\crv{(-20,-5) & (-40,-10)}; ?(0)*\dir{<};}};
(-30,-23)*{\cdots};
(-35,-12.5)*{\dred\xybox{(-10,-25);(-30,0) **\crv{(-10,-5) & (-30,-10)}; ?(0)*\dir{<};}};
%(-12.5,12.5)*{\doran\xybox{(-5,25);(-20,0) **\crv{(-3,5) & (-20,10)}; ?(1)*\dir{>};}};
(-55,-23)*{\cdots};
%(-42.5,12.5)*{\dred\xybox{(-60,25);(-25,0) **\crv{(-62,5) & (-25,10)}; ?(1)*\dir{>};}};
(-70,-23)*{\cdots};
(-35,-2)*{\cdots};
(-50,-2)*{\cdots};
(-42.5,-12.5)*{\dred\xybox{(-35,-25);(-80,-25) **\crv{(-35,-20) & (-80,-20)}; ?(1)*\dir{>};}};
(-42.5,-2.5)*{\dgreen\xybox{(-15,0);(-50,0) **\crv{(-15,-5) & (-50,-5)}; ?(1)*\dir{>};}};
(-42.5,-12.5)*{\dblue\xybox{(-25,-25);(-90,-25) **\crv{(-25,-15) & (-90,-15)}; ?(1)*\dir{>};}};
%(-27.5,12.5)*{\dred\xybox{(0,25);(-55,25) **\crv{(0,10) & (-55,10)}; ?(1)*\dir{>};}};
(-30,-12.5)*{\dgreen\xybox{(-65,-25);(-45,0) **\crv{(-65,-5) & (-45,-10)} ?(1)*\dir{>};}};
(-55,-12.5)*{\dgreen\xybox{(-20,-25)*{};(-40,0)*{} **\crv{(-20,-5) & (-40,-10)}; ?(0)*\dir{<};}};
%(0,27)*{\scs r}; (-5,27)*{\scs r-1}; 
(-10,-27)*{\scs n}; (-20,-27)*{\scs r}; (-25,-27)*{\scs r}; (-35,-27)*{\scs n}; (-40,-27)*{\scs 1}; (-50,-27)*{\scs n}; (-45,-27)*{\scs 1}; (-60,-27)*{\scs r}; (-65,-27)*{\scs r}; (-75,-27)*{\scs n}; (-25,2)*{\scs 1};
(0,-12)*{ (1^r)};
\endxy
\end{eqnarray*}

\begin{eqnarray*}
  \labellist
\tiny\hair 2pt
\pinlabel $1$ at -5 -10
\pinlabel $r$   at 65 -10
\pinlabel $1$ at 135 -10
\endlabellist
\figins{15}{0.6}{6vertu}\
& \xy (0,0)*{}; (0,12)*{\longmapsto}; \endxy &
 \xy
(-15,23)*{\cdots};
(-50,12.5)*{\dred\xybox{(-75,25);(-55,0) **\crv{(-75,5) & (-55,10)} ?(0)*\dir{<};}};
(-40,12.5)*{\dblue\xybox{(-65,25);(-45,0) **\crv{(-65,5) & (-45,10)} ?(0)*\dir{<};}};
(-45,12.5)*{\dblue\xybox{(-20,25)*{};(-40,0)*{} **\crv{(-20,5) & (-40,10)}; ?(1)*\dir{>};}};
(-30,23)*{\cdots};
(-35,12.5)*{\dred\xybox{(-10,25);(-30,0) **\crv{(-10,5) & (-30,10)}; ?(1)*\dir{>};}};
%(-12.5,12.5)*{\doran\xybox{(-5,25);(-20,0) **\crv{(-3,5) & (-20,10)}; ?(1)*\dir{>};}};
(-55,23)*{\cdots};
%(-42.5,12.5)*{\dred\xybox{(-60,25);(-25,0) **\crv{(-62,5) & (-25,10)}; ?(1)*\dir{>};}};
(-70,23)*{\cdots};
(-35,2)*{\cdots};
(-50,2)*{\cdots};
(-42.5,12.5)*{\dred\xybox{(-35,25);(-80,25) **\crv{(-35,20) & (-80,20)}; ?(0)*\dir{<};}};
(-42.5,2.5)*{\dgreen\xybox{(-15,0);(-50,0) **\crv{(-15,5) & (-50,5)}; ?(0)*\dir{<};}};
(-42.5,12.5)*{\dblue\xybox{(-25,25);(-90,25) **\crv{(-25,15) & (-90,15)}; ?(0)*\dir{<};}};
%(-27.5,12.5)*{\dred\xybox{(0,25);(-55,25) **\crv{(0,10) & (-55,10)}; ?(1)*\dir{>};}};
(-30,12.5)*{\dgreen\xybox{(-65,25);(-45,0) **\crv{(-65,5) & (-45,10)} ?(0)*\dir{<};}};
(-55,12.5)*{\dgreen\xybox{(-20,25)*{};(-40,0)*{} **\crv{(-20,5) & (-40,10)}; ?(1)*\dir{>};}};
%(0,27)*{\scs r}; (-5,27)*{\scs r-1}; 
(-10,27)*{\scs n}; (-20,27)*{\scs r}; (-25,27)*{\scs r}; (-35,27)*{\scs n}; (-40,27)*{\scs 1}; (-50,27)*{\scs n}; (-45,27)*{\scs 1}; (-60,27)*{\scs r}; (-65,27)*{\scs r}; (-75,27)*{\scs n}; (-25,-2)*{\scs 1};
(0,12)*{ (1^r)};
\endxy
\end{eqnarray*}

\begin{eqnarray*}
 \labellist
\tiny\hair 2pt
\pinlabel $r-1$ at -5 -10
\pinlabel $r$   at 65 -10
\pinlabel $r-1$ at 135 -10
\endlabellist
\figins{15}{0.6}{6vertd}\
& \xy (0,0)*{}; (0,12)*{\longmapsto}; \endxy &
 \xy
(-15,23)*{\cdots};
(-65,12.5)*{\dblue\xybox{(-75,25);(-55,0) **\crv{(-75,15) & (-55,10)} ?(1)*\dir{>};}};
(-55,12.5)*{\dred\xybox{(-65,25);(-45,0) **\crv{(-65,15) & (-45,10)} ?(1)*\dir{>};}};
(-30,12.5)*{\dred\xybox{(-20,25)*{};(-40,0)*{} **\crv{(-20,15) & (-40,10)}; ?(0)*\dir{<};}};
(-30,23)*{\cdots};
(-20,12.5)*{\dblue\xybox{(-10,25);(-30,0) **\crv{(-10,15) & (-30,10)}; ?(0)*\dir{<};}};
%(-12.5,12.5)*{\doran\xybox{(-5,25);(-20,0) **\crv{(-3,5) & (-20,10)}; ?(1)*\dir{>};}};
(-55,23)*{\cdots};
%(-42.5,12.5)*{\dred\xybox{(-60,25);(-25,0) **\crv{(-62,5) & (-25,10)}; ?(1)*\dir{>};}};
(-70,23)*{\cdots};
(-35,2)*{\cdots};
(-50,2)*{\cdots};
(-42.5,12.5)*{\dblue\xybox{(-35,25);(-50,25) **\crv{(-35,20) & (-50,20)}; ?(1)*\dir{>};}};
(-42.5,7.5)*{\doran\xybox{(-15,0);(-60,0) **\crv{(-15,15) & (-60,15)}; ?(1)*\dir{>};}};
(-42.5,12.5)*{\dred\xybox{(-25,25);(-60,25) **\crv{(-25,15) & (-60,15)}; ?(1)*\dir{>};}};
%(-27.5,12.5)*{\dred\xybox{(0,25);(-55,25) **\crv{(0,10) & (-55,10)}; ?(1)*\dir{>};}};
(-35,12.5)*{\doran\xybox{(-65,25);(-45,0) **\crv{(-65,5) & (-45,10)} ?(1)*\dir{>};}};
(-50,12.5)*{\doran\xybox{(-20,25)*{};(-40,0)*{} **\crv{(-20,5) & (-40,10)}; ?(0)*\dir{<};}};
%(0,27)*{\scs r}; (-5,27)*{\scs r-1}; 
(-10,27)*{\scs n}; (-20,27)*{\scs r}; (-25,27)*{\scs r}; (-35,27)*{\scs n}; (-40,27)*{\scs r-1}; (-50,27)*{\scs n}; (-45,27)*{\scs r-1}; (-60,27)*{\scs r}; (-65,27)*{\scs r}; (-75,27)*{\scs n}; (-20,-2)*{\scs r-1};
(0,12)*{ (1^r)};
\endxy
\end{eqnarray*}

\begin{eqnarray*}
  \labellist
\tiny\hair 2pt
\pinlabel $r$ at -5 -10
\pinlabel $r-1$   at 65 -10
\pinlabel $r$ at 135 -10
\endlabellist
\figins{-55}{0.6}{6vertu}\
& \xy (0,0)*{}; (0,-12)*{\longmapsto}; \endxy &
 \xy
(-15,-23)*{\cdots};
(-65,-12.5)*{\dblue\xybox{(-75,-25);(-55,0) **\crv{(-75,-15) & (-55,-10)} ?(0)*\dir{<};}};
(-55,-12.5)*{\dred\xybox{(-65,-25);(-45,0) **\crv{(-65,-15) & (-45,-10)} ?(0)*\dir{<};}};
(-30,-12.5)*{\dred\xybox{(-20,-25)*{};(-40,0)*{} **\crv{(-20,-15) & (-40,-10)}; ?(1)*\dir{>};}};
(-30,-23)*{\cdots};
(-20,-12.5)*{\dblue\xybox{(-10,-25);(-30,0) **\crv{(-10,-15) & (-30,-10)}; ?(1)*\dir{>};}};
%(-12.5,12.5)*{\doran\xybox{(-5,25);(-20,0) **\crv{(-3,5) & (-20,10)}; ?(1)*\dir{>};}};
(-55,-23)*{\cdots};
%(-42.5,12.5)*{\dred\xybox{(-60,25);(-25,0) **\crv{(-62,5) & (-25,10)}; ?(1)*\dir{>};}};
(-70,-23)*{\cdots};
(-35,-2)*{\cdots};
(-50,-2)*{\cdots};
(-42.5,-12.5)*{\dblue\xybox{(-35,-25);(-50,-25) **\crv{(-35,-20) & (-50,-20)}; ?(0)*\dir{<};}};
(-42.5,-7.5)*{\doran\xybox{(-15,0);(-60,0) **\crv{(-15,-15) & (-60,-15)}; ?(0)*\dir{<};}};
(-42.5,-12.5)*{\dred\xybox{(-25,-25);(-60,-25) **\crv{(-25,-15) & (-60,-15)}; ?(0)*\dir{<};}};
%(-27.5,12.5)*{\dred\xybox{(0,25);(-55,25) **\crv{(0,10) & (-55,10)}; ?(1)*\dir{>};}};
(-35,-12.5)*{\doran\xybox{(-65,-25);(-45,0) **\crv{(-65,-5) & (-45,-10)} ?(0)*\dir{<};}};
(-50,-12.5)*{\doran\xybox{(-20,-25)*{};(-40,0)*{} **\crv{(-20,-5) & (-40,-10)}; ?(1)*\dir{>};}};
%(0,27)*{\scs r}; (-5,27)*{\scs r-1}; 
(-10,-27)*{\scs n}; (-20,-27)*{\scs r}; (-25,-27)*{\scs r}; (-35,-27)*{\scs n}; (-40,-27)*{\scs r-1}; (-50,-27)*{\scs n}; (-45,-27)*{\scs r-1}; (-60,-27)*{\scs r}; (-65,-27)*{\scs r}; (-75,-27)*{\scs n}; (-20,2)*{\scs r-1};
(0,-12)*{ (1^r)};
\endxy
\end{eqnarray*}

\item The images of the oriented caps and cups $+$cap, $-$cap, $+$cup and $-$cup:
\begin{equation*}
\figins{-16}{0.5}{cap-r}
\longmapsto\
\xy 
 (0,-4)*{\dred\xybox{(-2,0)*{};(2,0)*{} **\crv{(-2,3) & (2,3)}; ?(.5)*\dir{<};}};
 (0,1.5)*{\dgreen\xybox{(-10,0)*{};(10,0)*{} **\crv{(-10,14) & (10,14)}; ?(.5)*\dir{<};}};
 (-6,-4)*{\cdots};
  (6,-4)*{\cdots};
(0,4)*{\dpink\xybox{(-14,0)*{};(14,0)*{} **\crv{(-14,19) & (14,19)}; ?(.5)*\dir{<};}};
 (0,9.5)*{\dblue\xybox{(-24,0)*{};(24,0)*{} **\crv{(-24,30) & (24,30)}; ?(.5)*\dir{<};}};
 (-19,-4)*{\cdots};
  (19,-4)*{\cdots};
  (-10,-8)*{\scs 1}; (-2,-8)*{\scs r}; (2,-8)*{ \scs r};(10,-8)*{\scs 1};(-24,-8)*{\scs n}; (-14,-8)*{\scs r+1}; (14,-8)*{ \scs r+1};(24,-8)*{\scs n};
 (30,0)*{ (1^r)}; 
  \endxy
    \vspace*{2ex}
\end{equation*}
\begin{equation*}
\figins{-16}{0.5}{cap-l}
\longmapsto\
\xy 
 (0,-4)*{\dblue\xybox{(-2,0)*{};(2,0)*{} **\crv{(-2,3) & (2,3)}; ?(.5)*\dir{>};}};
 (0,1.5)*{\dpink\xybox{(-10,0)*{};(10,0)*{} **\crv{(-10,14) & (10,14)}; ?(.5)*\dir{>};}};
 (-6,-4)*{\cdots};
  (6,-4)*{\cdots};
(0,4)*{\dgreen\xybox{(-14,0)*{};(14,0)*{} **\crv{(-14,19) & (14,19)}; ?(.5)*\dir{>};}};
 (0,9.5)*{\dred\xybox{(-24,0)*{};(24,0)*{} **\crv{(-24,30) & (24,30)}; ?(.5)*\dir{>};}};
 (-19,-4)*{\cdots};
  (19,-4)*{\cdots};
  (-10,-8)*{\scs r+1}; (-2,-8)*{\scs n}; (2,-8)*{ \scs n};(10,-8)*{\scs r+1};(-24,-8)*{\scs r}; (-14,-8)*{\scs 1}; (14,-8)*{ \scs 1};(24,-8)*{\scs r};
 (30,0)*{ (1^r)}; 
  \endxy
    \vspace*{2ex}
\end{equation*}
\begin{equation*}
\figins{-16}{0.5}{cup-r}
\longmapsto\
\xy 
 (0,4)*{\dblue\xybox{(-2,0)*{};(2,0)*{} **\crv{(-2,-3) & (2,-3)}; ?(.5)*\dir{<};}};
 (0,-1.5)*{\dpink\xybox{(-10,0)*{};(10,0)*{} **\crv{(-10,-14) & (10,-14)}; ?(.5)*\dir{<};}};
 (-6,4)*{\cdots};
  (6,4)*{\cdots};
(0,-4)*{\dgreen\xybox{(-14,0)*{};(14,0)*{} **\crv{(-14,-19) & (14,-19)}; ?(.5)*\dir{<};}};
 (0,-9.5)*{\dred\xybox{(-24,0)*{};(24,0)*{} **\crv{(-24,-30) & (24,-30)}; ?(.5)*\dir{<};}};
 (-19,4)*{\cdots};
  (19,4)*{\cdots};
  (-10,8)*{\scs r+1}; (-2,8)*{\scs n}; (2,8)*{ \scs n};(10,8)*{\scs r+1};(-24,8)*{\scs r}; (-14,8)*{\scs 1}; (14,8)*{ \scs 1};(24,8)*{\scs r};
 (30,0)*{ (1^r)}; 
  \endxy
    \vspace*{2ex} 
\end{equation*}
\begin{equation*}
\figins{-16}{0.5}{cup-l}
\longmapsto\
\xy 
 (0,4)*{\dred\xybox{(-2,0)*{};(2,0)*{} **\crv{(-2,-3) & (2,-3)}; ?(.5)*\dir{>};}};
 (0,-1.5)*{\dgreen\xybox{(-10,0)*{};(10,0)*{} **\crv{(-10,-14) & (10,-14)}; ?(.5)*\dir{>};}};
 (-6,4)*{\cdots};
  (6,4)*{\cdots};
(0,-4)*{\dpink\xybox{(-14,0)*{};(14,0)*{} **\crv{(-14,-19) & (14,-19)}; ?(.5)*\dir{>};}};
 (0,-9.5)*{\dblue\xybox{(-24,0)*{};(24,0)*{} **\crv{(-24,-30) & (24,-30)}; ?(.5)*\dir{>};}};
 (-19,4)*{\cdots};
  (19,4)*{\cdots};
  (-10,8)*{\scs 1}; (-2,8)*{\scs r}; (2,8)*{ \scs r};(10,8)*{\scs 1};(-24,8)*{\scs n}; (-14,8)*{\scs r+1}; (14,8)*{ \scs r+1};(24,8)*{\scs n};
 (30,0)*{ (1^r)}; 
  \endxy
    \vspace*{2ex}
\end{equation*}

\item The images of the $4$--valent vertex morphisms $4$vert$_{+,i}$, $4$vert$_{i+1,+}$, $4$vert$_{i,-}$ and $4$vert$_{-,i+1}$ with colors $i$ different from $r-1$ and $r$ and $i+1$ different from $r$ and $1$:
\begin{equation*}
\labellist
\hair 2pt
\pinlabel $\scs i+1$   at  -5 140
\pinlabel $\scs i$   at 128 -10
\endlabellist
\figins{-16}{0.6}{4mvert-ur}\ \
\longmapsto\ \ \
 \xy 
 (-28,0)*{\dblue\xybox{(-18,8);(-28,-8); **\dir{-} ?(0)*\dir{<};}};
 (-27,-6)*{\cdots};
 (-19,6)*{\cdots};
 (-18,0)*{\dpink\xybox{(-8,8);(-18,-8); **\dir{-} ?(0)*\dir{<};}};
 (-13,0)*{\dgreen\xybox{(-3,8);(-13,-8); **\dir{-} ?(0)*\dir{<};}};
 (-12,-6)*{\cdots};
 (-4,6)*{\cdots};
 (-3,0)*{\doran\xybox{(7,8);(-3,-8); **\dir{-} ?(0)*\dir{<};}};
 (12,0)*{\dbrun\xybox{(22,8);(12,-8); **\dir{-} ?(0)*\dir{<};}};
 (13,-6)*{\cdots};
 (21,6)*{\cdots};
 (22,0)*{\dred\xybox{(32,8);(22,-8); **\dir{-} ?(0)*\dir{<};}};
 (12,-4)*{\dturq\xybox{(-3,-8)*{};(27,-8)*{} **\crv{(-2,0) & (27,0)}; ?(.5)*\dir{<};}};
 (14.5,0)*{\dturq\xybox{(7,8)*{};(22,-8)*{} **\crv{(5,0) & (24,6)}; ?(1)*\dir{>};}};
 (-8,3.5)*{\dyellow\xybox{(-23,8)*{};(17,8)*{} **\crv{(-23,0) & (17,0)}; ?(.5)*\dir{<};}};
 (-15.5,0)*{\dyellow\xybox{(-28,8)*{};(7,-8)*{} **\crv{(-30,-4) & (9,2)}; ?(1)*\dir{>};}};
 (-33,-10)*{\scs n}; (-23,-10)*{\scs r+1 }; (-18,-10)*{ \scs 1};(-8,-10)*{\scs i-1}; (-3,-10)*{\scs i}; (2,-10)*{\scs i+1}; (7,-10)*{\scs i+2}; (17,-10)*{\scs r};(22,-10)*{\scs i};(27,-10)*{\scs i};
 (32,0)*{ (1^r)}; 
  \endxy
\vspace*{2ex}
\end{equation*}

\begin{equation*}
\labellist
\hair 2pt
\pinlabel $\scs i+1$   at  -5 -10
\pinlabel $\scs i$   at 128 140
\endlabellist
\figins{-16}{0.6}{4mvert-ul}\ \
\longmapsto\ \ \
 \xy 
 (-28,0)*{\dblue\xybox{(-18,-8);(-28,8); **\dir{-} ?(1)*\dir{>};}};
 (-27,6)*{\cdots};
 (-19,-6)*{\cdots};
 (-18,0)*{\dpink\xybox{(-8,-8);(-18,8); **\dir{-} ?(1)*\dir{>};}};
 (-13,0)*{\dgreen\xybox{(-3,-8);(-13,8); **\dir{-} ?(1)*\dir{>};}};
 (-12,6)*{\cdots};
 (-4,-6)*{\cdots};
 (-3,0)*{\doran\xybox{(7,-8);(-3,8); **\dir{-} ?(1)*\dir{>};}};
 (12,0)*{\dbrun\xybox{(22,-8);(12,8); **\dir{-} ?(1)*\dir{>};}};
 (13,6)*{\cdots};
 (21,-6)*{\cdots};
 (22,0)*{\dred\xybox{(32,-8);(22,8); **\dir{-} ?(1)*\dir{>};}};
 (12,4)*{\dturq\xybox{(-3,8)*{};(27,8)*{} **\crv{(-2,0) & (27,0)}; ?(.5)*\dir{>};}};
 (14.5,0)*{\dturq\xybox{(7,-8)*{};(22,8)*{} **\crv{(5,0) & (24,-6)}; ?(0)*\dir{<};}};
 (-8,-3.5)*{\dyellow\xybox{(-23,-8)*{};(17,-8)*{} **\crv{(-23,0) & (17,0)}; ?(.5)*\dir{>};}};
 (-15.5,0)*{\dyellow\xybox{(-28,-8)*{};(7,8)*{} **\crv{(-30,4) & (9,-2)}; ?(0)*\dir{<};}};
 (-23,-10)*{\scs n}; (-13,-10)*{\scs r+1 }; (-8,-10)*{ \scs 1};(2,-10)*{\scs i-1}; (7,-10)*{\scs i}; (12,-10)*{\scs i+1}; (17,-10)*{\scs i+2}; (27,-10)*{\scs r};(-33,-10)*{\scs i+1};(-27,-10)*{\scs i+1};
 (32,0)*{ (1^r)}; 
  \endxy
\vspace*{2ex}
\end{equation*}

\begin{equation*}
\labellist
\hair 2pt
\pinlabel $\scs i$   at  -5 -10
\pinlabel $\scs i+1$   at 128 140
\endlabellist
\figins{-16}{0.6}{4mvert-dr}\ \
\longmapsto\ \ \
\xy 
 (27,0)*{\dblue\xybox{(17,8);(27,-8); **\dir{-} ?(1)*\dir{>};}};
 (-15,-6)*{\cdots};
 (-22,6)*{\cdots};
 (17,0)*{\dpink\xybox{(7,8);(17,-8); **\dir{-} ?(1)*\dir{>};}};
 (12,0)*{\dgreen\xybox{(2,8);(12,-8); **\dir{-} ?(1)*\dir{>};}};
 (10,-6)*{\cdots};
 (3,6)*{\cdots};
 (2,0)*{\doran\xybox{(-8,8);(2,-8); **\dir{-} ?(1)*\dir{>};}};
 (-13,0)*{\dbrun\xybox{(-23,8);(-13,-8); **\dir{-} ?(1)*\dir{>};}};
 (25,-6)*{\cdots};
 (18,6)*{\cdots};
 (-23,0)*{\dred\xybox{(-33,8);(-23,-8); **\dir{-} ?(1)*\dir{>};}};
 (-13,-4)*{\dturq\xybox{(-3,-8)*{};(27,-8)*{} **\crv{(-2,0) & (27,0)}; ?(.5)*\dir{<};}};
 (-15.5,0)*{\dturq\xybox{(-13,8)*{};(-28,-8)*{} **\crv{(-11,0) & (-30,6)}; ?(0)*\dir{<};}};
 (7,3.5)*{\dyellow\xybox{(-23,8)*{};(17,8)*{} **\crv{(-23,0) & (17,0)}; ?(.5)*\dir{<};}};
 (14.5,0)*{\dyellow\xybox{(27,8)*{};(-8,-8)*{} **\crv{(29,-4) & (-10,2)}; ?(0)*\dir{<};}};
 (32,-10)*{\scs n}; (22,-10)*{\scs r+1 }; (17,-10)*{ \scs 1};(7,-10)*{\scs i-1}; (2,-10)*{\scs i}; (-3,-10)*{\scs i+1}; (-8,-10)*{\scs i+2}; (-18,-10)*{\scs r};(-23,-10)*{\scs i};(-28,-10)*{\scs i};
 (37,0)*{ (1^r)}; 
  \endxy
\vspace*{2ex}
\end{equation*}

\begin{equation*}
\labellist
\hair 2pt
\pinlabel $\scs i$   at  -5 140
\pinlabel $\scs i+1$   at 128 -10
\endlabellist
\figins{-16}{0.6}{4mvert-dl}\ \
\longmapsto\ \ \
\xy 
 (27,0)*{\dblue\xybox{(17,-8);(27,8); **\dir{-} ?(0)*\dir{<};}};
 (-15,6)*{\cdots};
 (-22,-6)*{\cdots};
 (17,0)*{\dpink\xybox{(7,-8);(17,8); **\dir{-} ?(0)*\dir{<};}};
 (12,0)*{\dgreen\xybox{(2,-8);(12,8); **\dir{-} ?(0)*\dir{<};}};
 (10,6)*{\cdots};
 (3,-6)*{\cdots};
 (2,0)*{\doran\xybox{(-8,-8);(2,8); **\dir{-} ?(0)*\dir{<};}};
 (-13,0)*{\dbrun\xybox{(-23,-8);(-13,8); **\dir{-} ?(0)*\dir{<};}};
 (25,6)*{\cdots};
 (18,-6)*{\cdots};
 (-23,0)*{\dred\xybox{(-33,-8);(-23,8); **\dir{-} ?(0)*\dir{<};}};
 (-13,4)*{\dturq\xybox{(-3,8)*{};(27,8)*{} **\crv{(-2,0) & (27,0)}; ?(.5)*\dir{>};}};
 (-15.5,0)*{\dturq\xybox{(-13,-8)*{};(-28,8)*{} **\crv{(-11,0) & (-30,-6)}; ?(1)*\dir{>};}};
 (7,-3.5)*{\dyellow\xybox{(-23,-8)*{};(17,-8)*{} **\crv{(-23,0) & (17,0)}; ?(.5)*\dir{>};}};
 (14.5,0)*{\dyellow\xybox{(27,-8)*{};(-8,8)*{} **\crv{(29,4) & (-10,-2)}; ?(1)*\dir{>};}};
 (22,-10)*{\scs n}; (12,-10)*{\scs r+1 }; (7,-10)*{ \scs 1};(-3,-10)*{\scs i-1}; (-8,-10)*{\scs i}; (-13,-10)*{\scs i+1}; (-18,-10)*{\scs i+2}; (-28,-10)*{\scs r};(32,-10)*{\scs i+1};(27,-10)*{\scs i+1};
 (37,0)*{ (1^r)}; 
  \endxy
\vspace*{2ex}
\end{equation*}

\item The images of the $4$--valent vertex morphisms $4$vert$_{+,r}$, $4$vert$_{1,+}$, $4$vert$_{r,-}$ and $4$vert$_{-,1}$:
\begin{eqnarray*}
\labellist
\hair 2pt
\pinlabel $\scs 1$   at  -5 140
\pinlabel $\scs r$   at 128 -10
\endlabellist
\figins{15}{0.6}{4mvert-ur}
& \xy (0,0)*{}; (0,12)*{\longmapsto}; \endxy &
 \xy
(5,2)*{\cdots};
(15,12.5)*{\dgreen\xybox{(15,0);(15,25); **\dir{-} ?(1)*\dir{>};}};
(32.5,12.5)*{\dturq\xybox{(20,0)*{};(45,25)*{} **\crv{(18,20) & (45,15)}; ?(0)*\dir{<};}};
(25,2)*{\cdots};
(42.5,12.5)*{\doran\xybox{(30,0);(55,25) **\crv{(28,20) & (55,15)}; ?(0)*\dir{<};}};
(32.5,12.5)*{\dblue\xybox{(40,0);(25,25) **\crv{(42,20) & (25,15)}; ?(0)*\dir{<};}};
(45,2)*{\cdots};
(42.5,12.5)*{\dpink\xybox{(50,0);(35,25) **\crv{(52,20) & (35,15)}; ?(0)*\dir{<};}};
(57.5,12.5)*{\dred\xybox{(55,0);(60,25); **\dir{-} ?(1)*\dir{>};}};
(70,2)*{\cdots};
(30,23)*{\cdots};
(50,23)*{\cdots};
(47.5,2.5)*{\dred\xybox{(35,0);(60,0) **\crv{(35,5) & (60,5)}; ?(0)*\dir{<};}};
(30,12.5)*{\dgreen\xybox{(20,25);(40,25) **\crv{(20,20) & (40,20)}; ?(0)*\dir{<};}};
(37.5,5)*{\dpink\xybox{(10,0);(65,0) **\crv{(10,10) & (65,10)}; ?(0)*\dir{<};}};
(37.5,7.5)*{\dblue\xybox{(0,0);(75,0) **\crv{(0,15) & (75,15)}; ?(0)*\dir{<};}};
(0,-2)*{\scs n}; (10,-2)*{\scs r+1}; (15,-2)*{\scs 1}; (20,-2)*{\scs 2}; (30,-2)*{\scs r-1}; (35,-2)*{\scs r}; (40,-2)*{\scs n}; (50,-2)*{\scs r+1}; (55,-2)*{\scs r}; (60,-2)*{\scs r}; (65,-2)*{\scs r+1}; (75,-2)*{\scs n}; (20,27)*{\scs 1}; 
(82,12)*{ (1^r)};
\endxy
\end{eqnarray*}

\begin{eqnarray*}
\labellist
\hair 2pt
\pinlabel $\scs 1$   at  -5 -10
\pinlabel $\scs r$   at 128 140
\endlabellist
\figins{15}{0.6}{4mvert-ul}
& \xy (0,0)*{}; (0,12)*{\longmapsto}; \endxy &
 \xy
(5,23)*{\cdots};
(15,12.5)*{\dgreen\xybox{(15,25);(15,0); **\dir{-} ?(0)*\dir{<};}};
(32.5,12.5)*{\dturq\xybox{(20,25)*{};(45,0)*{} **\crv{(18,5) & (45,10)}; ?(1)*\dir{>};}};
(25,23)*{\cdots};
(42.5,12.5)*{\doran\xybox{(30,25);(55,0) **\crv{(28,5) & (55,10)}; ?(1)*\dir{>};}};
(32.5,12.5)*{\dblue\xybox{(40,25);(25,0) **\crv{(42,5) & (25,10)}; ?(1)*\dir{>};}};
(45,23)*{\cdots};
(42.5,12.5)*{\dpink\xybox{(50,25);(35,0) **\crv{(52,5) & (35,10)}; ?(1)*\dir{>};}};
(57.5,12.5)*{\dred\xybox{(55,25);(60,0); **\dir{-} ?(0)*\dir{<};}};
(70,23)*{\cdots};
(30,2)*{\cdots};
(50,2)*{\cdots};
(47.5,12.5)*{\dred\xybox{(35,25);(60,25) **\crv{(35,20) & (60,20)}; ?(1)*\dir{>};}};
(30,2.5)*{\dgreen\xybox{(20,0);(40,0) **\crv{(20,5) & (40,5)}; ?(1)*\dir{>};}};
(37.5,12.5)*{\dpink\xybox{(10,25);(65,25) **\crv{(10,15) & (65,15)}; ?(1)*\dir{>};}};
(37.5,12.5)*{\dblue\xybox{(0,25);(75,25) **\crv{(0,10) & (75,10)}; ?(1)*\dir{>};}};
(0,27)*{\scs n}; (10,27)*{\scs r+1}; (15,27)*{\scs 1}; (20,27)*{\scs 2}; (30,27)*{\scs r-1}; (35,27)*{\scs r}; (40,27)*{\scs n}; (50,27)*{\scs r+1}; (55,27)*{\scs r}; (60,27)*{\scs r}; (65,27)*{\scs r+1}; (75,27)*{\scs n}; (20,-2)*{\scs 1}; 
(82,12)*{ (1^r)};
\endxy
\end{eqnarray*}

\begin{eqnarray*}
\labellist
\hair 2pt
\pinlabel $\scs 1$   at 128 140
\pinlabel $\scs r$   at -5 -10
\endlabellist
\figins{15}{0.6}{4mvert-dr}
& \xy (0,0)*{}; (0,12)*{\longmapsto}; \endxy &
 \xy
(-5,2)*{\cdots};
(-15,12.5)*{\dgreen\xybox{(-15,0);(-15,25); **\dir{-} ?(0)*\dir{<};}};
(-32.5,12.5)*{\dturq\xybox{(-20,0)*{};(-45,25)*{} **\crv{(-18,20) & (-45,15)}; ?(1)*\dir{>};}};
(-25,2)*{\cdots};
(-42.5,12.5)*{\doran\xybox{(-30,0);(-55,25) **\crv{(-28,20) & (-55,15)}; ?(1)*\dir{>};}};
(-32.5,12.5)*{\dblue\xybox{(-40,0);(-25,25) **\crv{(-42,20) & (-25,15)}; ?(1)*\dir{>};}};
(-45,2)*{\cdots};
(-42.5,12.5)*{\dpink\xybox{(-50,0);(-35,25) **\crv{(-52,20) & (-35,15)}; ?(1)*\dir{>};}};
(-57.5,12.5)*{\dred\xybox{(-55,0);(-60,25); **\dir{-} ?(0)*\dir{<};}};
(-70,2)*{\cdots};
(-30,23)*{\cdots};
(-50,23)*{\cdots};
(-47.5,2.5)*{\dred\xybox{(-35,0);(-60,0) **\crv{(-35,5) & (-60,5)}; ?(1)*\dir{>};}};
(-30,12.5)*{\dgreen\xybox{(-20,25);(-40,25) **\crv{(-20,20) & (-40,20)}; ?(1)*\dir{>};}};
(-37.5,5)*{\dpink\xybox{(-10,0);(-65,0) **\crv{(-10,10) & (-65,10)}; ?(1)*\dir{>};}};
(-37.5,7.5)*{\dblue\xybox{(0,0);(-75,0) **\crv{(0,15) & (-75,15)}; ?(1)*\dir{>};}};
(0,-2)*{\scs n}; (-10,-2)*{\scs r+1}; (-15,-2)*{\scs 1}; (-20,-2)*{\scs 2}; (-30,-2)*{\scs r-1}; (-35,-2)*{\scs r}; (-40,-2)*{\scs n}; (-50,-2)*{\scs r+1}; (-55,-2)*{\scs r}; (-60,-2)*{\scs r}; (-65,-2)*{\scs r+1}; (-75,-2)*{\scs n}; (-20,27)*{\scs 1}; 
(7,12)*{ (1^r)};
\endxy
\end{eqnarray*}

\begin{eqnarray*}
\labellist
\hair 2pt
\pinlabel $\scs 1$   at 128 -10
\pinlabel $\scs r$   at -5 140
\endlabellist
\figins{15}{0.6}{4mvert-dl}
& \xy (0,0)*{}; (0,12)*{\longmapsto}; \endxy &
 \xy
(-5,23)*{\cdots};
(-15,12.5)*{\dgreen\xybox{(-15,25);(-15,0); **\dir{-} ?(1)*\dir{>};}};
(-32.5,12.5)*{\dturq\xybox{(-20,25)*{};(-45,0)*{} **\crv{(-18,5) & (-45,10)}; ?(0)*\dir{<};}};
(-25,23)*{\cdots};
(-42.5,12.5)*{\doran\xybox{(-30,25);(-55,0) **\crv{(-28,5) & (-55,10)}; ?(0)*\dir{<};}};
(-32.5,12.5)*{\dblue\xybox{(-40,25);(-25,0) **\crv{(-42,5) & (-25,10)}; ?(0)*\dir{<};}};
(-45,23)*{\cdots};
(-42.5,12.5)*{\dpink\xybox{(-50,25);(-35,0) **\crv{(-52,5) & (-35,10)}; ?(0)*\dir{<};}};
(-57.5,12.5)*{\dred\xybox{(-55,25);(-60,0); **\dir{-} ?(1)*\dir{>};}};
(-70,23)*{\cdots};
(-30,2)*{\cdots};
(-50,2)*{\cdots};
(-47.5,12.5)*{\dred\xybox{(-35,25);(-60,25) **\crv{(-35,20) & (-60,20)}; ?(0)*\dir{<};}};
(-30,2.5)*{\dgreen\xybox{(-20,0);(-40,0) **\crv{(-20,5) & (-40,5)}; ?(0)*\dir{<};}};
(-37.5,12.5)*{\dpink\xybox{(-10,25);(-65,25) **\crv{(-10,15) & (-65,15)}; ?(0)*\dir{<};}};
(-37.5,12.5)*{\dblue\xybox{(0,25);(-75,25) **\crv{(0,10) & (-75,10)}; ?(0)*\dir{<};}};
(0,27)*{\scs n}; (-10,27)*{\scs r+1}; (-15,27)*{\scs 1}; (-20,27)*{\scs 2}; (-30,27)*{\scs r-1}; (-35,27)*{\scs r}; (-40,27)*{\scs n}; (-50,27)*{\scs r+1}; (-55,27)*{\scs r}; (-60,27)*{\scs r}; (-65,27)*{\scs r+1}; (-75,27)*{\scs n}; (-20,-2)*{\scs 1}; 
(7,12)*{ (1^r)};
\endxy
\end{eqnarray*}

\item The images of the $4$--valent vertex morphisms $4$vert$_{+,r-1}$, $4$vert$_{r,+}$, $4$vert$_{r-1,-}$ and $4$vert$_{-,r}$:
\begin{eqnarray*}
\labellist
\hair 2pt
\pinlabel $\scs r-1$   at 128 -10
\pinlabel $\scs r$   at -5 140
\endlabellist
\figins{15}{0.6}{4mvert-ur}
& \xy (0,0)*{}; (0,12)*{\longmapsto}; \endxy &
 \xy
(-15,23)*{\cdots};
(-67.5,12.5)*{\dblue\xybox{(-75,25);(-60,0) **\crv{(-75,15) & (-60,10)} ?(1)*\dir{>};}};
(-57.5,12.5)*{\dpink\xybox{(-65,25);(-50,0) **\crv{(-65,15) & (-50,10)} ?(1)*\dir{>};}};
(-32.5,12.5)*{\dgreen\xybox{(-20,25)*{};(-45,0)*{} **\crv{(-18,5) & (-45,10)}; ?(1)*\dir{>};}};
(-30,23)*{\cdots};
(-22.5,12.5)*{\dyellow\xybox{(-10,25);(-35,0) **\crv{(-8,5) & (-35,10)}; ?(1)*\dir{>};}};
(-12.5,12.5)*{\doran\xybox{(-5,25);(-20,0) **\crv{(-3,5) & (-20,10)}; ?(1)*\dir{>};}};
(-45,23)*{\cdots};
(-42.5,12.5)*{\dred\xybox{(-60,25);(-25,0) **\crv{(-62,5) & (-25,10)}; ?(1)*\dir{>};}};
(-70,23)*{\cdots};
(-40,2)*{\cdots};
(-55,2)*{\cdots};
(-37.5,12.5)*{\dblue\xybox{(-35,25);(-40,25) **\crv{(-35,20) & (-40,20)}; ?(1)*\dir{>};}};
(-22.5,2.5)*{\doran\xybox{(-15,0);(-30,0) **\crv{(-15,5) & (-30,5)}; ?(1)*\dir{>};}};
(-37.5,12.5)*{\dpink\xybox{(-25,25);(-50,25) **\crv{(-25,15) & (-50,15)}; ?(1)*\dir{>};}};
(-27.5,12.5)*{\dred\xybox{(0,25);(-55,25) **\crv{(0,10) & (-55,10)}; ?(1)*\dir{>};}};
(0,27)*{\scs r}; (-5,27)*{\scs r-1}; (-10,27)*{\scs r-2}; (-20,27)*{\scs 1}; (-25,27)*{\scs r+1}; (-35,27)*{\scs n}; (-40,27)*{\scs n}; (-50,27)*{\scs r+1}; (-55,27)*{\scs r}; (-60,27)*{\scs r}; (-65,27)*{\scs r+1}; (-75,27)*{\scs n}; (-15,-2)*{\scs r-1}; 
(7,12)*{ (1^r)};
\endxy
\end{eqnarray*}

\begin{eqnarray*}
\labellist
\hair 2pt
\pinlabel $\scs r-1$   at 128 140
\pinlabel $\scs r$   at -5 -10
\endlabellist
\figins{15}{0.6}{4mvert-ul}
& \xy (0,0)*{}; (0,12)*{\longmapsto}; \endxy &
 \xy
(-15,2)*{\cdots};
(-67.5,12.5)*{\dblue\xybox{(-75,0);(-60,25) **\crv{(-75,10) & (-60,15)} ?(0)*\dir{<};}};
(-57.5,12.5)*{\dpink\xybox{(-65,0);(-50,25) **\crv{(-65,10) & (-50,15)} ?(0)*\dir{<};}};
(-32.5,12.5)*{\dgreen\xybox{(-20,0)*{};(-45,25)*{} **\crv{(-18,20) & (-45,15)}; ?(0)*\dir{<};}};
(-30,2)*{\cdots};
(-22.5,12.5)*{\dyellow\xybox{(-10,0);(-35,25) **\crv{(-8,20) & (-35,15)}; ?(0)*\dir{<};}};
(-12.5,12.5)*{\doran\xybox{(-5,0);(-20,25) **\crv{(-3,20) & (-20,15)}; ?(0)*\dir{<};}};
(-45,2)*{\cdots};
(-42.5,12.5)*{\dred\xybox{(-60,0);(-25,25) **\crv{(-62,20) & (-25,15)}; ?(0)*\dir{<};}};
(-70,2)*{\cdots};
(-40,23)*{\cdots};
(-55,23)*{\cdots};
(-37.5,2.5)*{\dblue\xybox{(-35,0);(-40,0) **\crv{(-35,5) & (-40,5)}; ?(0)*\dir{<};}};
(-22.5,12.5)*{\doran\xybox{(-15,25);(-30,25) **\crv{(-15,20) & (-30,20)}; ?(0)*\dir{<};}};
(-37.5,5.5)*{\dpink\xybox{(-25,0);(-50,0) **\crv{(-25,10) & (-50,10)}; ?(0)*\dir{<};}};
(-27.5,7.5)*{\dred\xybox{(0,0);(-55,0) **\crv{(0,15) & (-55,15)}; ?(0)*\dir{<};}};
(0,-2)*{\scs r}; (-5,-2)*{\scs r-1}; (-10,-2)*{\scs r-2}; (-20,-2)*{\scs 1}; (-25,-2)*{\scs r+1}; (-35,-2)*{\scs n}; (-40,-2)*{\scs n}; (-50,-2)*{\scs r+1}; (-55,-2)*{\scs r}; (-60,-2)*{\scs r}; (-65,-2)*{\scs r+1}; (-75,-2)*{\scs n}; (-15,27)*{\scs r-1}; 
(7,12)*{ (1^r)};
\endxy
\end{eqnarray*}

\begin{eqnarray*}
\labellist
\hair 2pt
\pinlabel $\scs r-1$   at -5 -10
\pinlabel $\scs r$   at 128 140
\endlabellist
\figins{15}{0.6}{4mvert-dr}
& \xy (0,0)*{}; (0,12)*{\longmapsto}; \endxy &
 \xy
(15,23)*{\cdots};
(67.5,12.5)*{\dblue\xybox{(75,25);(60,0) **\crv{(75,15) & (60,10)} ?(0)*\dir{<};}};
(57.5,12.5)*{\dpink\xybox{(65,25);(50,0) **\crv{(65,15) & (50,10)} ?(0)*\dir{<};}};
(32.5,12.5)*{\dgreen\xybox{(20,25)*{};(45,0)*{} **\crv{(18,5) & (45,10)}; ?(0)*\dir{<};}};
(30,23)*{\cdots};
(22.5,12.5)*{\dyellow\xybox{(10,25);(35,0) **\crv{(8,5) & (35,10)}; ?(0)*\dir{<};}};
(12.5,12.5)*{\doran\xybox{(5,25);(20,0) **\crv{(3,5) & (20,10)}; ?(0)*\dir{<};}};
(45,23)*{\cdots};
(42.5,12.5)*{\dred\xybox{(60,25);(25,0) **\crv{(62,5) & (25,10)}; ?(0)*\dir{<};}};
(70,23)*{\cdots};
(40,2)*{\cdots};
(55,2)*{\cdots};
(37.5,12.5)*{\dblue\xybox{(35,25);(40,25) **\crv{(35,20) & (40,20)}; ?(0)*\dir{<};}};
(22.5,2.5)*{\doran\xybox{(15,0);(30,0) **\crv{(15,5) & (30,5)}; ?(0)*\dir{<};}};
(37.5,12.5)*{\dpink\xybox{(25,25);(50,25) **\crv{(25,15) & (50,15)}; ?(0)*\dir{<};}};
(27.5,12.5)*{\dred\xybox{(0,25);(55,25) **\crv{(0,10) & (55,10)}; ?(0)*\dir{<};}};
(0,27)*{\scs r}; (5,27)*{\scs r-1}; (10,27)*{\scs r-2}; (20,27)*{\scs 1}; (25,27)*{\scs r+1}; (35,27)*{\scs n}; (40,27)*{\scs n}; (50,27)*{\scs r+1}; (55,27)*{\scs r}; (60,27)*{\scs r}; (65,27)*{\scs r+1}; (75,27)*{\scs n}; (15,-2)*{\scs r-1}; 
(82,12)*{ (1^r)};
\endxy
\end{eqnarray*}

\begin{eqnarray*}
\labellist
\hair 2pt
\pinlabel $\scs r-1$   at -5 140
\pinlabel $\scs r$   at 128 -10
\endlabellist
\figins{15}{0.6}{4mvert-dl}
& \xy (0,0)*{}; (0,12)*{\longmapsto}; \endxy &
 \xy
(15,2)*{\cdots};
(67.5,12.5)*{\dblue\xybox{(75,0);(60,25) **\crv{(75,10) & (60,15)} ?(1)*\dir{>};}};
(57.5,12.5)*{\dpink\xybox{(65,0);(50,25) **\crv{(65,10) & (50,15)} ?(1)*\dir{>};}};
(32.5,12.5)*{\dgreen\xybox{(20,0)*{};(45,25)*{} **\crv{(18,20) & (45,15)}; ?(1)*\dir{>};}};
(30,2)*{\cdots};
(22.5,12.5)*{\dyellow\xybox{(10,0);(35,25) **\crv{(8,20) & (35,15)}; ?(1)*\dir{>};}};
(12.5,12.5)*{\doran\xybox{(5,0);(20,25) **\crv{(3,20) & (20,15)}; ?(1)*\dir{>};}};
(45,2)*{\cdots};
(42.5,12.5)*{\dred\xybox{(60,0);(25,25) **\crv{(62,20) & (25,15)}; ?(1)*\dir{>};}};
(70,2)*{\cdots};
(40,23)*{\cdots};
(55,23)*{\cdots};
(37.5,2.5)*{\dblue\xybox{(35,0);(40,0) **\crv{(35,5) & (40,5)}; ?(1)*\dir{>};}};
(22.5,12.5)*{\doran\xybox{(15,25);(30,25) **\crv{(15,20) & (30,20)}; ?(1)*\dir{>};}};
(37.5,5.5)*{\dpink\xybox{(25,0);(50,0) **\crv{(25,10) & (50,10)}; ?(1)*\dir{>};}};
(27.5,7.5)*{\dred\xybox{(0,0);(55,0) **\crv{(0,15) & (55,15)}; ?(1)*\dir{>};}};
(0,-2)*{\scs r}; (5,-2)*{\scs r-1}; (10,-2)*{\scs r-2}; (20,-2)*{\scs 1}; (25,-2)*{\scs r+1}; (35,-2)*{\scs n}; (40,-2)*{\scs n}; (50,-2)*{\scs r+1}; (55,-2)*{\scs r}; (60,-2)*{\scs r}; (65,-2)*{\scs r+1}; (75,-2)*{\scs n}; (15,27)*{\scs r-1}; 
(82,12)*{ (1^r)};
\endxy
\end{eqnarray*}

\item The images of the box morphisms box$_{i}$ for $i = 1, \dots , r$:

\begin{equation*}
\xy (0,0)*{\bbox{i}} \endxy \longmapsto 
- \sum\limits_{j=i}^{r-1}\
\xy
(0,0)*{\dpurple\xybox{%
    (3,0);(-3,0) **\crv{(3,4.2) & (-3,4.2)};
    ?(.05)*\dir{>} ?(1)*\dir{>}; 
    (3,0);(-3,0) **\crv{(3,-4.2) & (-3,-4.2)} ?(.3)*\dir{}+(2,0)*{\bscs j};
}};
(6,3)*{\scs (1^r)},
\endxy + \xy 0;/r.18pc/:
 (0,-1)*{\dred\ccbub{\black -1}{\black r}};
  (8,4)*{\scs(1^r)};
 \endxy
\end{equation*}

\item The image of the box morphism box$_{y}$:

\begin{equation*}
\xy (0,0)*{\bbox{y}} \endxy \longmapsto 
 \xy
(0,0)*{\bbox{y}};
(6,2)*{\scs (1^r)}
\endxy
\end{equation*}

\end{itemize}

It is easy to check that $\Sigma_{n,r}$ is degree preserving and monoidal.  
\begin{lem}
$\Sigma_{n,r}$ is well-defined.
\end{lem}

\begin{proof} 
We check that $\Sigma_{n,r}$ preserves all relations in $\debim_{\hat{A}_{r-1}}^*$. 

First of all, note that all ``finite type A'' relations, i.e. the relations between diagrams without $r$-colored or oriented strands, 
are preserved by precisely the same arguments as in~\cite{MSVschur}.  

Let us also remark that, except in the checks of the box relations, we will use neither Relation~\eqref{eq_r2_ij-gen} nor the bubble slides Relations~\eqref{eq:2ndbubbslide1n}--\eqref{eq:extrabubble461n} for $\{i,j\} = \{1,n\}$.

Let us now go through the list of the remaining relations and explain why they are preserved: 

\noindent $\bullet$ The Isotopy Relations~\eqref{eq:adj}--\eqref{eq:v4mrotd} are straightforward. 
\\
\noindent $\bullet$ Relations \eqref{eq:dumbrot}--\eqref{eq:dumbdumbsquare} with at least one $r$--colored strand follow from 
the same relations without any $r$--colored strand together with Relations \eqref{eq:orbub}--\eqref{eq:slide6mv2}. So it suffices to 
prove the latter relations.  

Indeed, for any of these relations with an $r$--colored strand, one can add an oriented bubble (or if necessary two nested ones) to the diagram on the left-hand side of the equation.  
By Relation \eqref{eq:orbub} the oriented bubbles are equal to one and we will show that that relation is preserved by $\Sigma_{n,r}$ below. 
Using the appropriate relations involving colored and oriented strands 
(which are checked below), slide that bubble across the diagram until it encloses that part of the diagram which differs from the diagram on the 
right-hand side of the equation. It is always possible to choose the orientation of the bubble so that the diagram in its interior does not have any $r$--colored 
strands. Then apply the relevant relation for colors different from $r$ to the part of the diagram contained inside the bubble. Finally, slide the bubble aside again. 
This works since all these operations are proved to be preserved by $\Sigma_{n,r}$ below.
\\
\noindent $\bullet$ Relation~\eqref{eq:orbub} follows directly from the fact that in $\Scat(n,r)^*_{[y]}((1^r),(1^r))$ all dotted bubbles of degree zero are equal to 
$\pm 1$. Indeed we can apply successively Relations \eqref{eq:bubb_deg0} to the nested bubbles in the image of 
$\figins{-5}{0.2}{bubble-p.eps}$ and $\figins{-5}{0.2}{bubble-n.eps}$. 
\\
\noindent $\bullet$ For Relations~\eqref{eq:capcupud} and~\eqref{eq:capcupdu} use repeatedly Relations~\eqref{eq_ident_decomp} and~\eqref{eq_ident_decomp0}. 
We only give the details for Relation~\eqref{eq:capcupud}, the other relation being completely analogous.  

The diagram $\Sigma_{n,r} \bigl( \figins{-5}{0.2}{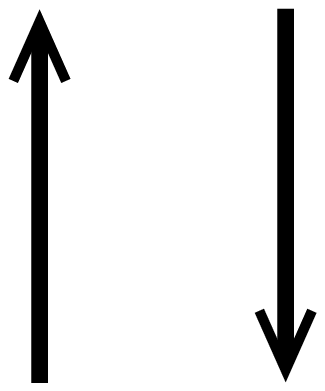} \bigr) $ is as follows
\begin{equation*}
 \xy 
 (-28,0)*{\dblue\xybox{(-28,8);(-28,-8); **\dir{-} ?(.5)*\dir{<};}};
 (-23,0)*{\cdots};
 (-18,0)*{\dpink\xybox{(-18,8);(-18,-8); **\dir{-} ?(.5)*\dir{<};}};
 (-13,0)*{\dgreen\xybox{(-13,8);(-13,-8); **\dir{-} ?(.5)*\dir{<};}};
 (-8,0)*{\cdots};
 (-3,0)*{\dred\xybox{(-3,8);(-3,-8); **\dir{-} ?(.5)*\dir{<};}};
 (-28,-10)*{\scs n}; (-18,-10)*{\scs r+1 }; (-13,-10)*{ \scs 1};(-3,-10)*{\scs r};
   \endxy
\vspace*{2ex}
\quad
 \xy 
 (-28,0)*{\dred\xybox{(-28,8);(-28,-8); **\dir{-} ?(.5)*\dir{>};}};
 (-23,0)*{\cdots};
 (-18,0)*{\dgreen\xybox{(-18,8);(-18,-8); **\dir{-} ?(.5)*\dir{>};}};
 (-13,0)*{\dpink\xybox{(-13,8);(-13,-8); **\dir{-} ?(.5)*\dir{>};}};
 (-8,0)*{\cdots};
 (-3,0)*{\dblue\xybox{(-3,8);(-3,-8); **\dir{-} ?(.5)*\dir{>};}};
 (-28,-10)*{\scs r}; (-18,-10)*{\scs 1 }; (-13,-10)*{ \scs r+1};(-3,-10)*{\scs n};
 (5,0)*{ (1^r)}; 
  \endxy.
\vspace*{2ex}
\end{equation*}
Apply successively Relation \eqref{eq_ident_decomp} to each pair of the form 
\begin{equation*}
\text{$
 \xy 
 (0,0)*{\dturq\xybox{
 (-5,7);(-5,-7); **\dir{-} ?(.5)*\dir{<}+(2.3,0)*{\scriptstyle{}};}};
 (-6.3,-9)*{\scs i};
 (0,0)*{\dturq\xybox{
 (10,7);(10,-7); **\dir{-} ?(.5)*\dir{>}+(12.3,0)*{\scriptstyle{}};}};
 ( -1.2,-9)*{\scs i};
 (6,0)*{ \lambda };
  \endxy
$}.
\vspace*{2ex}
\end{equation*}
For $i=r+1, \dots,n$, we have $\lambda = (0,1, \dots,1, 0, \dots,0,1,0, \dots , 0)$, where the entries which are equal to one are the $2$nd until the $r$th and the 
$i+1$st (mod n). For $i=1, \dots,r$, we have $\lambda = (1, \dots,1, 0,1, \dots,1,0, \dots , 0)$ where the entries which are equal to zero are the $r+2$nd until the 
$n$th and the $i$th.

For these $\lambda$, Relation \eqref{eq_ident_decomp} becomes 

\begin{eqnarray*}
 \vcenter{\xy 0;/r.18pc/:
  (0,0)*{\dturq\xybox{
  (-8,0)*{};(8,0)*{};
  (-4,10)*{}="t1";
  (4,10)*{}="t2";
  (-4,-10)*{}="b1";
  (4,-10)*{}="b2";
  "t1";"b1" **\dir{-} ?(.5)*\dir{>};
  "t2";"b2" **\dir{-} ?(.5)*\dir{<};}};
  (-6,-8)*{\scs i};(6,-8)*{\scs i};
  (10,2)*{\lambda};
    \endxy}
&\quad = \quad&
   \vcenter{\xy 0;/r.18pc/:
    (0,0)*{\dturq\xybox{
    (-4,-4)*{};(4,4)*{} **\crv{(-4,-1) & (4,1)}?(1)*\dir{<};?(0)*\dir{<};
    (4,-4)*{};(-4,4)*{} **\crv{(4,-1) & (-4,1)}?(1)*\dir{>};
    (-4,4)*{};(4,12)*{} **\crv{(-4,7) & (4,9)}?(1)*\dir{>};
    (4,4)*{};(-4,12)*{} **\crv{(4,7) & (-4,9)};}};
    (8,8)*{\lambda};(-6.8,-7)*{\scs i};(6,-7)*{\scs i};
 \endxy}
  \quad - \quad
    \vcenter{\xy 0;/r.18pc/:
  (0,0)*{\dturq\xybox{ 
  (-8,0)*{}; (8,0)*{};
  (-4,-15)*{}="b1";
  (4,-15)*{}="b2";
  "b2";"b1" **\crv{(5,-8) & (-5,-8)}; ?(.1)*\dir{>} ?(.95)*\dir{>}
  ?(.8)*\dir{};
  (-4,15)*{}="t1";
  (4,15)*{}="t2";
  "t2";"t1" **\crv{(5,8) & (-5,8)}; ?(.15)*\dir{<} ?(.97)*\dir{<}
  ?(.4)*\dir{};
  (0,0)*{\cbub{\black\scs \quad\; -2}{i}};}};
  (-10,10)*{\lambda};
  \endxy} .
\end{eqnarray*}
The first term on the right-hand side of this relation is equal to zero because the label appearing in interior region contains a negative entry. 
The bubble appearing in the second term on the right-hand side is equal to $-1$. Thus, for all $i$, we are left with
\begin{equation}\label{id_decomp_red}
\vcenter{\xy 0;/r.18pc/:
  (0,0)*{\dturq\xybox{
  (-8,0)*{};(8,0)*{};
  (-4,10)*{}="t1";
  (4,10)*{}="t2";
  (-4,-10)*{}="b1";
  (4,-10)*{}="b2";
  "t1";"b1" **\dir{-} ?(.5)*\dir{>};
  "t2";"b2" **\dir{-} ?(.5)*\dir{<};}};
  (-6,-8)*{\scs i};(6,-8)*{\scs i};
  (10,2)*{\lambda};
    \endxy}
\quad = \quad
    {\dturq \xy
    (0,4)*{\bbpef{\black i}};
    (0,-4.5)*{\bbcef{\black i}};
    (8,0.5)*{ \black \lambda };
    \endxy}.
\vspace*{2ex}
\end{equation}
Taking all strands together, we get the following nested cups and caps:
\begin{equation*}
\xy 
 (0,4)*{\dred\xybox{(-2,0)*{};(2,0)*{} **\crv{(-2,-3) & (2,-3)}; ?(.5)*\dir{>};}};
 (0,-1.5)*{\dgreen\xybox{(-10,0)*{};(10,0)*{} **\crv{(-10,-14) & (10,-14)}; ?(.5)*\dir{>};}};
 (-6,4)*{\cdots};
  (6,4)*{\cdots};
(0,-4)*{\dpink\xybox{(-14,0)*{};(14,0)*{} **\crv{(-14,-19) & (14,-19)}; ?(.5)*\dir{>};}};
 (0,-9.5)*{\dblue\xybox{(-24,0)*{};(24,0)*{} **\crv{(-24,-30) & (24,-30)}; ?(.5)*\dir{>};}};
 (-19,4)*{\cdots};
  (19,4)*{\cdots};
  (-10,8)*{\scs 1}; (-2,8)*{\scs r}; (2,8)*{ \scs r};(10,8)*{\scs 1};(-24,8)*{\scs n}; (-14,8)*{\scs r+1}; (14,8)*{ \scs r+1};(24,8)*{\scs n};
 (0,-20)*{ \lambda}; 
 (0,-54)*{\dred\xybox{(-2,0)*{};(2,0)*{} **\crv{(-2,3) & (2,3)}; ?(.5)*\dir{<};}};
 (0,-48.5)*{\dgreen\xybox{(-10,0)*{};(10,0)*{} **\crv{(-10,14) & (10,14)}; ?(.5)*\dir{<};}};
 (-6,-54)*{\cdots};
  (6,-54)*{\cdots};
(0,-46)*{\dpink\xybox{(-14,0)*{};(14,0)*{} **\crv{(-14,19) & (14,19)}; ?(.5)*\dir{<};}};
 (0,-40.5)*{\dblue\xybox{(-24,0)*{};(24,0)*{} **\crv{(-24,30) & (24,30)}; ?(.5)*\dir{<};}};
 (-19,-54)*{\cdots};
  (19,-54)*{\cdots};
  (-10,-58)*{\scs 1}; (-2,-58)*{\scs r}; (2,-58)*{ \scs r};(10,-58)*{\scs 1};(-24,-58)*{\scs n}; (-14,-58)*{\scs r+1}; (14,-58)*{ \scs r+1};(24,-58)*{\scs n};
   \endxy
    \vspace*{2ex}
\end{equation*}
which is equal to $\Sigma_{n,r} \bigl( \figins{-3}{0.25}{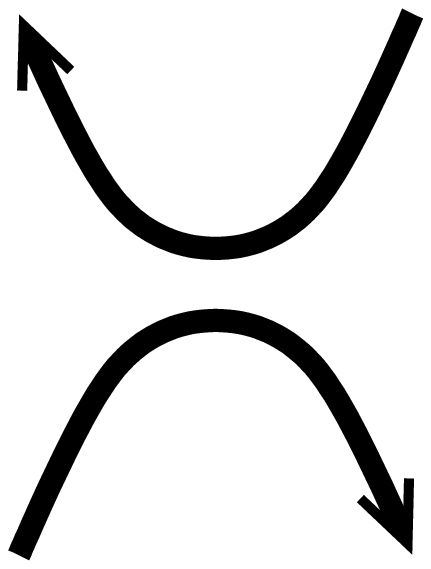} \bigr)$.
\\
\noindent $\bullet$ Relation \eqref{eq:slide4mv}

\begin{equation*}
\labellist
\tiny\hair 2pt
\pinlabel $i$ at  -4 -10
\pinlabel $j$ at 134 -10
\pinlabel $i+1$ at 134 140
\pinlabel $j+1$ at -4 140
\pinlabel $i$ at  191 -10
\pinlabel $j$ at 329 -10
\pinlabel $i+1$ at 329 140
\pinlabel $j+1$ at 191 140
\endlabellist
\figins{-18}{0.6}{4mvert-slide-u}\
=\
\figins{-18}{0.6}{4mvert-slide-d}.
\end{equation*}
For $j < i $ and $i,j \neq r$, this relation follows from the fact that, using repeatedly Relations \eqref{eq_downup_ij-gen}, \eqref{eq_r2_ij-gen} 
for distant $i,j$, \eqref{eq_r3_easy-gen} and \eqref{eq_other_r3_1}, one can reduce both 
$\Sigma_{n,r} \Bigl(  \figins{-7}{0.3}{4mvert-slide-u}  \, \Bigr)$ and $\Sigma_{n,r} \Bigl(  \figins{-7}{0.3}{4mvert-slide-d}  \, \Bigr)$ 
to the following diagram

\begin{equation*}
 \xy 
 (-28,0)*{\dblue\xybox{(-8,24);(-38,-24); **\dir{-} ?(0)*\dir{<};}};
 (-37,-22)*{\cdots};
 (-9,22)*{\cdots};
 (-18,0)*{\dpink\xybox{(2,24);(-28,-24); **\dir{-} ?(0)*\dir{<};}};
 (-13,0)*{\dgreen\xybox{(7,24);(-23,-24); **\dir{-} ?(0)*\dir{<};}};
 (-22,-22)*{\cdots};
 (6,22)*{\cdots};
 (-3,0)*{\dgrey\xybox{(17,24);(-13,-24); **\dir{-} ?(0)*\dir{<};}};
 (37,0)*{\dbrun\xybox{(32,24);(2,-24); **\dir{-} ?(0)*\dir{<};}};
 (28,-22)*{\cdots};
 (56,22)*{\cdots};
 (47,0)*{\dred\xybox{(42,24);(12,-24); **\dir{-} ?(0)*\dir{<};}};
  %(-5.5,11)*{\dyellow\xybox{(-18,24)*{};(7,15)*{} **\crv{(-18,20) & (2,15)}; ?(0)*\dir{<};}};
 (-10,11)*{\dyellow\xybox{(-23,24)*{};(3.1,9)*{} **\crv{(-23,17) & (-4,9)}; ?(0)*\dir{<};}};
 (-14.4,11)*{\dyellow\xybox{(-28,24)*{};(-0.8,3)*{} **\crv{(-28,14) & (-10,3)}; ?(1)*\dir{>};}};
 (-18.7,10)*{\dblack\xybox{(-33,24)*{};(-4.2,-3)*{} **\crv{(-33,11) & (-16,-3)}; ?(0)*\dir{<};}};
 (-23.1,7)*{\dblack\xybox{(-38,24)*{};(-8.2,-9)*{} **\crv{(-38,8) & (-22,-9)}; ?(1)*\dir{>};}};
 %(-27.6,4)*{\dyellow\xybox{(-43,24)*{};(-12,-15)*{} **\crv{(-43,5) & (-28,-15)}; ?(1)*\dir{>};}};
 %(31.6,-12.5)*{\doran\xybox{(18,-24)*{};(8,-15)*{} **\crv{(18,-20) & (13,-15)}; ?(0)*\dir{<};}};
 (36.1,-12.5)*{\dturq\xybox{(23,-24)*{};(11.9,-9)*{} **\crv{(23,-17) & (19,-9)}; ?(0)*\dir{<};}};
 (40.5,-12.5)*{\dturq\xybox{(28,-24)*{};(15.8,-3)*{} **\crv{(28,-14) & (25,-3)}; ?(1)*\dir{>};}};
 (44.8,-10.5)*{\dpurple\xybox{(33,-24)*{};(19.2,3)*{} **\crv{(33,-11) & (31,3)}; ?(0)*\dir{<};}};
 (49.2,-7.5)*{\dpurple\xybox{(38,-24)*{};(23.2,9)*{} **\crv{(38,-8) & (37,9)}; ?(1)*\dir{>};}};
 %(53.7,-4.5)*{\doran\xybox{(43,-24)*{};(27,15)*{} **\crv{(43,-5) & (43,15)}; ?(1)*\dir{>};}};
%merdier au milieu
(12,0)*{\dduck\xybox{(7,24);(-23,-24); **\dir{-} ?(0)*\dir{<};}};
 (3,-22)*{\cdots};
 (31,22)*{\cdots};
 (22,0)*{\doran\xybox{(17,24);(-13,-24); **\dir{-} ?(0)*\dir{<};}};
 (24.8,11)*{\dyellow\xybox{(46.5,24)*{};(3.1,9)*{} **\crv{(46.5,20) & (20,9)}; ?(0.7)*\dir{>};}};
 %(24.5,10.95)*{\dblack\xybox{(42,24)*{};(7,15)*{} **\crv{(22,0) & (22,15)}; ?(0.5)*\dir{>};}};
 (-6,-13.05)*{\dblack\xybox{(-12,-24)*{};(-12,-9)*{} **\crv{(-10,-20) & (-6,-9)}; ?(0.5)*\dir{<};}};
 (15,-7.5)*{\dpurple\xybox{(-32,-24)*{};(23.2,9)*{} **\crv{(-32,-20) & (-10,9)}; ?(0.5)*\dir{<};}};
 (33,7.5)*{\dturq\xybox{(22,24)*{};(11.9,-9)*{} **\crv{(22,20) & (2,-9)}; ?(0.5)*\dir{>};}};
 %(24.2,0)*{\doran\xybox{(8,-15)*{};(27,15)*{} **\crv{(-20,-15) & (-12,15)}; ?(0.5)*\dir{<};}};
 %(17,-1)*{\dturq\xybox{(-18,-24)*{};(12,24)*{} **\crv{(-15,15) & (-10,20)}; ?(0.5)*\dir{<};}};
 (8,10)*{\dblack\xybox{(-22,-3)*{};(4,24)*{} **\crv{(-10,-3) & (4,20)}; ?(0.7)*\dir{<};}};
 (8,-11)*{\dyellow\xybox{(-18,3)*{};(0,-24)*{} **\crv{(-10,3) & (0,-20)}; ?(0.7)*\dir{>};}};
 (23.5,-12.5)*{\dturq\xybox{(-7,-24)*{};(16.25,-3)*{} **\crv{(-7,-20) & (5,-3)}; ?(0.7)*\dir{<};}};
 (27,11.4)*{\dpurple\xybox{(-3,24)*{};(19.2,3)*{} **\crv{(19.2,20) & (16,3)}; ?(0.3)*\dir{>};}};
 (-43,-26)*{\scs n}; (-33,-26)*{\scs r+1 }; (-28,-26)*{ \scs 1};(-18,-26)*{\scs j-1}; (-13,-26)*{\scs j}; (-8,-26)*{\scs j+1}; (-3,-26)*{\scs j+2}; (7,-26)*{\scs i-1}; (12,-26)*{\scs i}; (17,-26)*{\scs i+1}; (22,-26)*{\scs i+2}; (32,-26)*{\scs r};
%(37,-26)*{\scs i-1};
(42,-26)*{\scs i}; (47,-26)*{\scs i};(52,-26)*{\scs j}; (57,-26)*{\scs j};
%(62,-26)*{\scs i-1};
% (-43,26)*{\scs i+1};
(-38,26)*{\scs j+1}; (-33,26)*{\scs j+1};(-28,26)*{\scs i+1}; (-23,26)*{\scs i+1};
%(-18,26)*{\scs i+1};
 (72,0)*{ (1^r)}; 
  \endxy
\vspace*{2ex}
\end{equation*}

Similarly, for $j > i $ and $i,j \neq r$, Relation \eqref{eq:slide4mv} follows from the fact that, using repeatedly 
Relations \eqref{eq_downup_ij-gen}, \eqref{eq_r2_ij-gen} for distant $i,j$, \eqref{eq_r3_easy-gen} and \eqref{eq_other_r3_1}, 
one can reduce both $\Sigma_{n,r} \Bigl(  \figins{-7}{0.3}{4mvert-slide-u}  \, \Bigr)$ and $\Sigma_{n,r} \Bigl(  \figins{-7}{0.3}{4mvert-slide-d}  \, \Bigr)$ 
to the following diagram

\begin{equation*}
 \xy 
 (-28,0)*{\dblue\xybox{(-8,24);(-38,-24); **\dir{-} ?(0)*\dir{<};}};
 (-37,-22)*{\cdots};
 (-9,22)*{\cdots};
 (-18,0)*{\dpink\xybox{(2,24);(-28,-24); **\dir{-} ?(0)*\dir{<};}};
 (-13,0)*{\dgreen\xybox{(7,24);(-23,-24); **\dir{-} ?(0)*\dir{<};}};
 (-22,-22)*{\cdots};
 (6,22)*{\cdots};
 (-3,0)*{\doran\xybox{(17,24);(-13,-24); **\dir{-} ?(0)*\dir{<};}};
 (37,0)*{\dduck\xybox{(32,24);(2,-24); **\dir{-} ?(0)*\dir{<};}};
 (28,-22)*{\cdots};
 (56,22)*{\cdots};
 (47,0)*{\dred\xybox{(42,24);(12,-24); **\dir{-} ?(0)*\dir{<};}};
  %(-5.5,11)*{\dyellow\xybox{(-18,24)*{};(7,15)*{} **\crv{(-18,20) & (2,15)}; ?(0)*\dir{<};}};
 (-10,11)*{\dyellow\xybox{(-23,24)*{};(3.1,9)*{} **\crv{(-23,17) & (-4,9)}; ?(0)*\dir{<};}};
 (-14.4,11)*{\dyellow\xybox{(-28,24)*{};(-0.8,3)*{} **\crv{(-28,14) & (-10,3)}; ?(1)*\dir{>};}};
 (-18.7,10)*{\dblack\xybox{(-33,24)*{};(-4.2,-3)*{} **\crv{(-33,11) & (-16,-3)}; ?(0)*\dir{<};}};
 (-23.1,7)*{\dblack\xybox{(-38,24)*{};(-8.2,-9)*{} **\crv{(-38,8) & (-22,-9)}; ?(1)*\dir{>};}};
 %(-27.6,4)*{\dyellow\xybox{(-43,24)*{};(-12,-15)*{} **\crv{(-43,5) & (-28,-15)}; ?(1)*\dir{>};}};
 %(31.6,-12.5)*{\doran\xybox{(18,-24)*{};(8,-15)*{} **\crv{(18,-20) & (13,-15)}; ?(0)*\dir{<};}};
 (36.1,-12.5)*{\dturq\xybox{(23,-24)*{};(11.9,-9)*{} **\crv{(23,-17) & (19,-9)}; ?(0)*\dir{<};}};
 (40.5,-12.5)*{\dturq\xybox{(28,-24)*{};(15.8,-3)*{} **\crv{(28,-14) & (25,-3)}; ?(1)*\dir{>};}};
 (44.8,-10.5)*{\dpurple\xybox{(33,-24)*{};(19.2,3)*{} **\crv{(33,-11) & (31,3)}; ?(0)*\dir{<};}};
 (49.2,-7.5)*{\dpurple\xybox{(38,-24)*{};(23.2,9)*{} **\crv{(38,-8) & (37,9)}; ?(1)*\dir{>};}};
 %(53.7,-4.5)*{\doran\xybox{(43,-24)*{};(27,15)*{} **\crv{(43,-5) & (43,15)}; ?(1)*\dir{>};}};
%merdier au milieu
(12,0)*{\dbrun\xybox{(7,24);(-23,-24); **\dir{-} ?(0)*\dir{<};}};
 (3,-22)*{\cdots};
 (31,22)*{\cdots};
 (22,0)*{\dgrey\xybox{(17,24);(-13,-24); **\dir{-} ?(0)*\dir{<};}};
 (4.5,-13)*{\dblack\xybox{(28.5,-24)*{};(3.1,-9)*{} **\crv{(28.5,-15) & (20,-9)}; ?(0.7)*\dir{<};}};
 (12,10.95)*{\dyellow\xybox{(6,24)*{};(-12,9)*{} **\crv{(6,20) & (-3,9)}; ?(0.5)*\dir{>};}};
 (18.5,7.45)*{\dturq\xybox{(8,24)*{};(23.2,-9)*{} **\crv{(8,20) & (-10,-9)}; ?(0.5)*\dir{>};}};
 (27,-7.5)*{\dpurple\xybox{(-18,-24)*{};(11.9,9)*{} **\crv{(-18,-20) & (-2,9)}; ?(0.5)*\dir{<};}};
  (-1.5,-11)*{\dyellow\xybox{(-22,3)*{};(-29,-24)*{} **\crv{(-10,3) & (-29,-20)}; ?(0.7)*\dir{>};}};
 (20,10)*{\dblack\xybox{(-18,-3)*{};(33.5,24)*{} **\crv{(0,-3) & (33.5,20)}; ?(0.7)*\dir{<};}};
 (36.8,11.5)*{\dpurple\xybox{(20,24)*{};(16.25,3)*{} **\crv{(20,20) & (5,3)}; ?(0.7)*\dir{>};}};
 (11,-12.6)*{\dturq\xybox{(-28,-24)*{};(19.2,-3)*{} **\crv{(-28,-18) &(19.2,-24) & (14,-3)}; ?(0.3)*\dir{<};}};
 (-43,-26)*{\scs n}; (-33,-26)*{\scs r+1 }; (-28,-26)*{ \scs 1};(-18,-26)*{\scs i-1}; (-13,-26)*{\scs i}; (-8,-26)*{\scs i+1}; (-3,-26)*{\scs i+2}; (7,-26)*{\scs j-1}; (12,-26)*{\scs j}; (17,-26)*{\scs j+1}; (22,-26)*{\scs j+2}; (32,-26)*{\scs r};
%(37,-26)*{\scs i-1};
(42,-26)*{\scs i}; (47,-26)*{\scs i};(52,-26)*{\scs j}; (57,-26)*{\scs j};
%(62,-26)*{\scs i-1};
% (-43,26)*{\scs i+1};
(-38,26)*{\scs j+1}; (-33,26)*{\scs j+1};(-28,26)*{\scs i+1}; (-23,26)*{\scs i+1};
%(-18,26)*{\scs i+1};
 (72,0)*{ (1^r)}; 
  \endxy
\vspace*{2ex}
\end{equation*}

One can prove the remaining cases, in which one of the integers $i$ or $j$ is equal to $r-1$ or $r$, in exactly the same way. Just use repeatedly Relations 
\eqref{eq_downup_ij-gen}, \eqref{eq_r2_ij-gen} for $i \cdot j=0$, \eqref{eq_r3_easy-gen} and \eqref{eq_other_r3_1}.
\\
\noindent $\bullet$ Relations~\eqref{eq:reid2ml} and~\eqref{eq:reid2mr}. 

First when the colors $(i, i+1)$ differ from $(r, 1)$ and $(r-1, r)$. We only give the details for 
Relation~\eqref{eq:reid2ml}, because the proof of the other relation is very similar. 

The image under $\Sigma_{n,r}$ of $\figins{-4}{0.25}{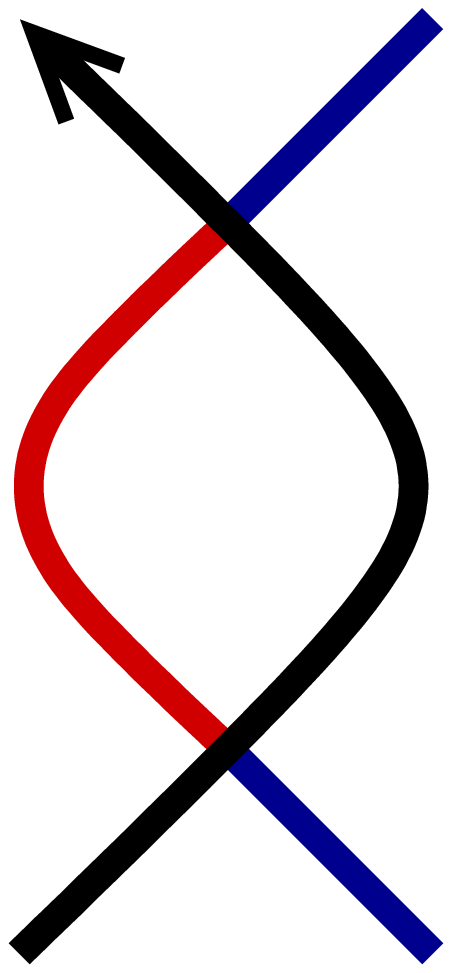} $ is equal to 

\begin{equation*}
 \xy 
 (-28,0)*{\dblue\xybox{(-28,-24)*{};(-28,8)*{} **\crv{(-18,-16) & (-18,0)}; ?(0)*\dir{<};}};
 (-18,0)*{\dpink\xybox{(-18,-24)*{};(-18,8)*{} **\crv{(-8,-16) & (-8,0)}; ?(0)*\dir{<};}};
 (-13,0)*{\dgreen\xybox{(-13,-24)*{};(-13,8)*{} **\crv{(-3,-16) & (-3,0)}; ?(0)*\dir{<};}};
 (-3,0)*{\doran\xybox{(-3,-24)*{};(-3,8)*{} **\crv{(7,-16) & (7,0)}; ?(0)*\dir{<};}};
 (12,0)*{\dbrun\xybox{(12,-24)*{};(12,8)*{} **\crv{(22,-16) & (22,0)}; ?(0)*\dir{<};}};
 (22,0)*{\dred\xybox{(22,-24)*{};(22,8)*{} **\crv{(32,-16) & (32,0)}; ?(0)*\dir{<};}};
 (-26,14)*{\cdots};
 %(-19,-6)*{\cdots};
 %(-18,8)*{\dpink\xybox{(-8,-8);(-18,8); **\dir{-} ?(1)*\dir{>};}};
 %(-13,8)*{\dgreen\xybox{(-3,-8);(-13,8); **\dir{-} ?(1)*\dir{>};}};
 (-11,14)*{\cdots};
 %(-4,-6)*{\cdots};
 %(-3,8)*{\doran\xybox{(7,-8);(-3,8); **\dir{-} ?(1)*\dir{>};}};
 %(12,8)*{\dbrun\xybox{(22,-8);(12,8); **\dir{-} ?(1)*\dir{>};}};
 (14,14)*{\cdots};
 %(21,-6)*{\cdots};
 %(22,8)*{\dred\xybox{(32,-8);(22,8); **\dir{-} ?(1)*\dir{>};}};
 (12,12)*{\dturq\xybox{(-3,8)*{};(27,8)*{} **\crv{(-2,0) & (27,0)}; ?(.5)*\dir{>};}};
 (14.5,8)*{\dturq\xybox{(7,-8)*{};(22,8)*{} **\crv{(5,0) & (24,-6)}; ?(0)*\dir{<};}};
 (-8,4.3)*{\dyellow\xybox{(-23,-8)*{};(17,-8)*{} **\crv{(-23,0) & (17,0)}; ?(.5)*\dir{>};}};
 (-15.5,8)*{\dyellow\xybox{(-28,-8)*{};(7,8)*{} **\crv{(-30,4) & (9,-2)}; ?(0)*\dir{<};}};
 %(-28,-8)*{\dblue\xybox{(-18,8);(-28,-8); **\dir{-} ?(0)*\dir{<};}};
 (-26,-14)*{\cdots};
 %(-19,6)*{\cdots};
 %(-18,-8)*{\dpink\xybox{(-8,8);(-18,-8); **\dir{-} ?(0)*\dir{<};}};
 %(-13,-8)*{\dgreen\xybox{(-3,8);(-13,-8); **\dir{-} ?(0)*\dir{<};}};
 (-11,-14)*{\cdots};
 %(-4,6)*{\cdots};
 %(-3,-8)*{\doran\xybox{(7,8);(-3,-8); **\dir{-} ?(0)*\dir{<};}};
 %(12,-8)*{\dbrun\xybox{(22,8);(12,-8); **\dir{-} ?(0)*\dir{<};}};
 (14,-14)*{\cdots};
 %(21,6)*{\cdots};
 %(22,-8)*{\dred\xybox{(32,8);(22,-8); **\dir{-} ?(0)*\dir{<};}};
 (12,-12)*{\dturq\xybox{(-3,-8)*{};(27,-8)*{} **\crv{(-2,0) & (27,0)}; ?(.5)*\dir{<};}};
 (14.5,-8)*{\dturq\xybox{(7,8)*{};(22,-8)*{} **\crv{(5,0) & (24,6)}; ?(1)*\dir{>};}};
 (-8,-4.3)*{\dyellow\xybox{(-23,8)*{};(17,8)*{} **\crv{(-23,0) & (17,0)}; ?(.5)*\dir{<};}};
 (-15.5,-8)*{\dyellow\xybox{(-28,8)*{};(7,-8)*{} **\crv{(-30,-4) & (9,2)}; ?(1)*\dir{>};}};
 (-33,-18)*{\scs n}; (-23,-18)*{\scs r+1 }; (-18,-18)*{ \scs 1};(-8,-18)*{\scs i-1}; (-3,-18)*{\scs i}; (2,-18)*{\scs i+1}; (7,-18)*{\scs i+2}; (17,-18)*{\scs r};(22,-18)*{\scs i};(27,-18)*{\scs i}; (4,7)*{\scs i+1};
 (32,0)*{ (1^r)}; 
  \endxy
\vspace*{2ex}
\end{equation*}

Thanks to Relation \eqref{eq_downup_ij-gen}, the left part of the central bubble can be slid to the right until we end up in the position
\begin{equation*} 
 \xy
 (-12,0)*{\dyellow\ncbub{}{\black \scs \mspace{-100mu} i+1}};
 (4,0)*{\lambda};
 (-13,6)*{\scs i};
 (-16,0)*{\dturq\xybox{(8,-5)*{};(8,5)*{} **\crv{(-2,-2) & (-2,2)}; ?(0)*\dir{<};}};
 \endxy
 \end{equation*}
with $\lambda=(1,\dots,1,0,1,\dots,1,0,\dots,0)$ where the first $0$ is in $i+2$nd position and the last $1$ in $r+1$st position. That degree-zero bubble 
is equal to one, by Relation \eqref{eq:bubb_deg0}.

After removing the bubble, we can slide the $i$--colored strand over the $r-i-1$ rightmost strands 
and slide the $i+1$--colored strand over the $n-r+i-1$ leftmost strands, using Relation \eqref{eq_r2_ij-gen}. We get
\begin{equation*}
 \xy 
 (-28,0)*{\dblue\xybox{(-28,8);(-28,-8); **\dir{-} ?(.5)*\dir{<};}};
 (-23,0)*{\cdots};
 (-18,0)*{\dpink\xybox{(-18,8);(-18,-8); **\dir{-} ?(.5)*\dir{<};}};
 (-13,0)*{\dgreen\xybox{(-13,8);(-13,-8); **\dir{-} ?(.5)*\dir{<};}};
 (-8,0)*{\cdots};
 (-3,0)*{\doran\xybox{(-3,8);(-3,-8); **\dir{-} ?(.5)*\dir{<};}};
 (18,-4)*{\dturq\xybox{(-8,-8)*{};(27,-8)*{} **\crv{(-12,6) & (10,-6)}; ?(.35)*\dir{<};}};
 (18,3)*{\dturq\xybox{(-8,8)*{};(27,8)*{} **\crv{(-12,-6) & (10,6)}; ?(.35)*\dir{>};}};
 (12,0)*{\dyellow\xybox{(12,8);(12,-8); **\dir{-} ?(.5)*\dir{<};}};
 (17,0)*{\dbrun\xybox{(17,8);(17,-8); **\dir{-} ?(.5)*\dir{<};}};
 (22,0)*{\cdots};
 (27,0)*{\dred\xybox{(27,8);(27,-8); **\dir{-} ?(.5)*\dir{<};}};
 (32,0)*{\dturq\xybox{(32,8);(32,-8); **\dir{-} ?(.5)*\dir{<};}};
 (-28,-10)*{\scs n}; (-18,-10)*{\scs r+1 }; (-13,-10)*{ \scs 1};(-3,-10)*{\scs i-1}; (2,-10)*{\scs i}; (12,-10)*{\scs i+1}; (17,-10)*{\scs i+2}; (27,-10)*{\scs r}; (32,-10)*{\scs i};
 (40,0)*{ (1^r)}; (37,-10)*{\scs i}; (0,0)*{ \lambda};
  \endxy
\vspace*{2ex}
\end{equation*}
with $\lambda= (1,\dots, 1,0,1 ,\dots,1,0, \dots, 0)$ where the first $0$ is in $i$th position and the last $1$ in $r+1$st position. 
Apply Relation \eqref{id_decomp_red}, which has only one non-zero term as before. We get
\begin{equation*}
 \xy 
 (-28,0)*{\dblue\xybox{(-28,8);(-28,-8); **\dir{-} ?(.5)*\dir{<};}};
 (-23,0)*{\cdots};
 (-18,0)*{\dpink\xybox{(-18,8);(-18,-8); **\dir{-} ?(.5)*\dir{<};}};
 (-13,0)*{\dgreen\xybox{(-13,8);(-13,-8); **\dir{-} ?(.5)*\dir{<};}};
 (-8,0)*{\cdots};
 (-3,0)*{\doran\xybox{(-3,8);(-3,-8); **\dir{-} ?(.5)*\dir{<};}};
 (2,0)*{\dturq\xybox{(2,8);(2,-8); **\dir{-} ?(.5)*\dir{<};}};
 (17,0)*{\dturq\xybox{(37,-16)*{};(37,0)*{} **\crv{(-3,-9) & (-3,-7)}; ?(0.5)*\dir{>};}};
 %(18,-4)*{\dturq\xybox{(-8,-8)*{};(27,-8)*{} **\crv{(-12,6) & (10,-6)}; ?(.35)*\dir{<};}};
 %(18,3)*{\dturq\xybox{(-8,8)*{};(27,8)*{} **\crv{(-12,-6) & (10,6)}; ?(.35)*\dir{>};}};
 (12,0)*{\dyellow\xybox{(12,8);(12,-8); **\dir{-} ?(.5)*\dir{<};}};
 (17,0)*{\dbrun\xybox{(17,8);(17,-8); **\dir{-} ?(.5)*\dir{<};}};
 (22,0)*{\cdots};
 (27,0)*{\dred\xybox{(27,8);(27,-8); **\dir{-} ?(.5)*\dir{<};}};
 (32,0)*{\dturq\xybox{(32,8);(32,-8); **\dir{-} ?(.5)*\dir{<};}};
 (-28,-10)*{\scs n}; (-18,-10)*{\scs r+1 }; (-13,-10)*{ \scs 1};(-3,-10)*{\scs i-1}; (2,-10)*{\scs i}; (12,-10)*{\scs i+1}; (17,-10)*{\scs i+2}; (27,-10)*{\scs r}; (32,-10)*{\scs i};
 (42,0)*{ (1^r)}; (37,-10)*{\scs i}; (0,0)*{ \lambda};
  \endxy
\vspace*{2ex}
\end{equation*}
Then the strand with the rightmost endpoints can be slid all the way to the right, first using Relation \eqref{eq_downup_ij-gen} and 
then Relation \eqref{eq_ident_decomp}, which again is simply equal to  
\begin{eqnarray}\label{id_decomp_red2}
 \vcenter{\xy 0;/r.18pc/:
  (0,0)*{\dturq\xybox{
  (-8,0)*{};(8,0)*{};
  (-4,10)*{}="t1";
  (4,10)*{}="t2";
  (-4,-10)*{}="b1";
  (4,-10)*{}="b2";
  "t1";"b1" **\dir{-} ?(.5)*\dir{>};
  "t2";"b2" **\dir{-} ?(.5)*\dir{<};}};
  (-6,-8)*{\scs i};(6,-8)*{\scs i};
  (10,2)*{\lambda};
    \endxy}
&\quad = \quad&
   \vcenter{\xy 0;/r.18pc/:
    (0,0)*{\dturq\xybox{
    (-4,-4)*{};(4,4)*{} **\crv{(-4,-1) & (4,1)}?(1)*\dir{<};?(0)*\dir{<};
    (4,-4)*{};(-4,4)*{} **\crv{(4,-1) & (-4,1)}?(1)*\dir{>};
    (-4,4)*{};(4,12)*{} **\crv{(-4,7) & (4,9)}?(1)*\dir{>};
    (4,4)*{};(-4,12)*{} **\crv{(4,7) & (-4,9)};}};
    (10,2)*{\lambda};(-6.8,-7)*{\scs i};(6,-7)*{\scs i};
 \endxy} .
\end{eqnarray}
because $\lambda= (1^r)$. Finally, we end up with a diagram which is indeed equal to $\Sigma_{n,r} \Bigl( \figins{-7}{0.3}{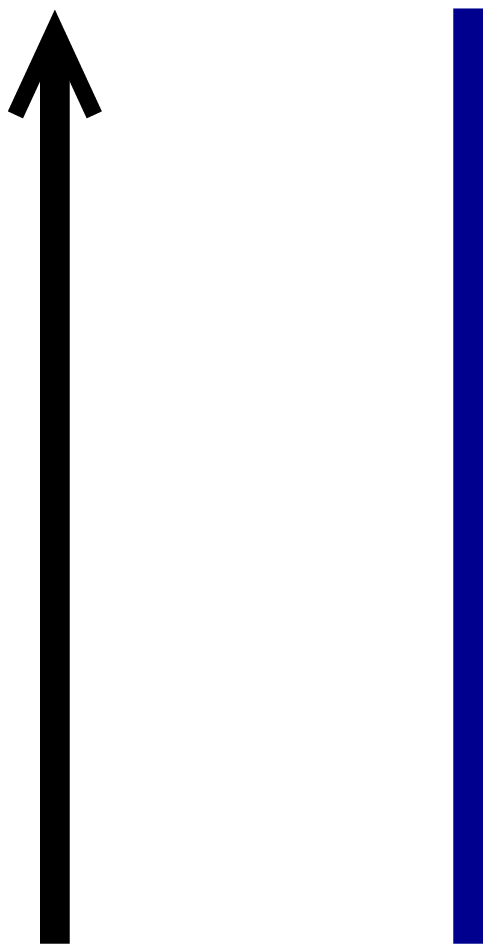} \, \Bigr)$.

We can prove Relations \eqref{eq:reid2ml} and \eqref{eq:reid2mr} with colors $(r, 1)$ and $(r-1, r)$ in almost the same way. 
We can slide bubbles and strands using Relations \eqref{eq_downup_ij-gen} and \eqref{eq_r2_ij-gen}, evaluate bubbles using Relation~\eqref{eq:bubb_deg0} 
and use Relations \eqref{eq:EF} and \eqref{eq_ident_decomp}, which again are particularly simple for the relevant labels:  
\begin{equation*}
\vcenter{\xy 0;/r.18pc/:
  (0,0)*{\dturq\xybox{
  (-8,0)*{};(8,0)*{};
  (-4,10)*{}="t1";
  (4,10)*{}="t2";
  (-4,-10)*{}="b1";
  (4,-10)*{}="b2";
  "t1";"b1" **\dir{-} ?(.5)*\dir{<};
  "t2";"b2" **\dir{-} ?(.5)*\dir{>};}};
  (-6,-8)*{\scs i};(6,-8)*{\scs i};
  (10,2)*{\lambda};
    \endxy}
\, = \,
   \vcenter{\xy 0;/r.18pc/:
    (0,0)*{\dturq\xybox{
    (-4,-4)*{};(4,4)*{} **\crv{(-4,-1) & (4,1)}?(1)*\dir{>};?(0)*\dir{>};
    (4,-4)*{};(-4,4)*{} **\crv{(4,-1) & (-4,1)}?(1)*\dir{<};
    (-4,4)*{};(4,12)*{} **\crv{(-4,7) & (4,9)}?(1)*\dir{<};
    (4,4)*{};(-4,12)*{} **\crv{(4,7) & (-4,9)};}};
    (10,2)*{\lambda};(-6.8,-7)*{\scs i};(6,-7)*{\scs i};
 \endxy}
 \quad \mbox{or} \quad
 \vcenter{\xy 0;/r.18pc/:
  (0,0)*{\dturq\xybox{
  (-8,0)*{};(8,0)*{};
  (-4,10)*{}="t1";
  (4,10)*{}="t2";
  (-4,-10)*{}="b1";
  (4,-10)*{}="b2";
  "t1";"b1" **\dir{-} ?(.5)*\dir{>};
  "t2";"b2" **\dir{-} ?(.5)*\dir{<};}};
  (-6,-8)*{\scs i};(6,-8)*{\scs i};
  (10,2)*{\lambda};
    \endxy}
\, = \,
   \vcenter{\xy 0;/r.18pc/:
    (0,0)*{\dturq\xybox{
    (-4,-4)*{};(4,4)*{} **\crv{(-4,-1) & (4,1)}?(1)*\dir{<};?(0)*\dir{<};
    (4,-4)*{};(-4,4)*{} **\crv{(4,-1) & (-4,1)}?(1)*\dir{>};
    (-4,4)*{};(4,12)*{} **\crv{(-4,7) & (4,9)}?(1)*\dir{>};
    (4,4)*{};(-4,12)*{} **\crv{(4,7) & (-4,9)};}};
    (10,2)*{\lambda};(-6.8,-7)*{\scs i};(6,-7)*{\scs i};
 \endxy}
\quad \mbox{or} \quad
\vcenter{\xy 0;/r.18pc/:
  (0,0)*{\dturq\xybox{
  (-8,0)*{};(8,0)*{};
  (-4,10)*{}="t1";
  (4,10)*{}="t2";
  (-4,-10)*{}="b1";
  (4,-10)*{}="b2";
  "t1";"b1" **\dir{-} ?(.5)*\dir{>};
  "t2";"b2" **\dir{-} ?(.5)*\dir{<};}};
  (-6,-8)*{\scs i};(6,-8)*{\scs i};
  (10,2)*{\lambda};
    \endxy}
\, = 
    {\dturq \xy
    (0,4)*{\bbpef{\black i}};
    (0,-4.5)*{\bbcef{\black i}};
    (8,0.5)*{ \black \lambda };
    \endxy}
\end{equation*}
depending on the value of $\lambda$. The difference here is that we have to iterate some of the steps that we used above, e.g. because we get nested bubbles or 
various pairs of strands to which we can apply Relations~\eqref{eq:EF} and~\eqref{eq_ident_decomp}. 
\\
\noindent $\bullet$ Relations~\eqref{eq:slidedotdist-md} and \eqref{eq:slidedotdist-mu}. 

Again, first when $(i, i+1)$ differ from $(r, 1)$ and $(r-1, r)$. 
The proofs for these two relations are very similar. We only give the proof for Relation \eqref{eq:slidedotdist-mu}. 

The image under $\Sigma_{n,r}$ of  $\figins{-4}{0.25}{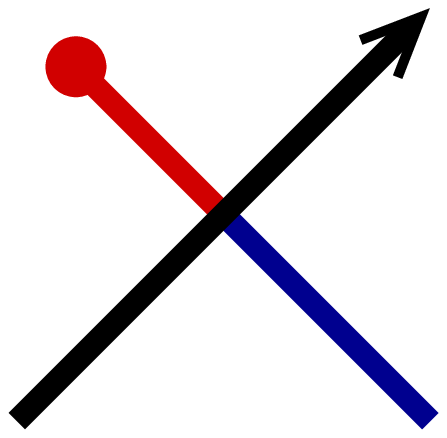} $ is equal to 
\begin{equation*}
 \xy 
 (-30.5,5.5)*{\dyellow\xybox{(-33,8)*{};(-28,8)*{} **\crv{(-33,12) & (-28,12)}; ?(.5)*\dir{<};}};
 (-28,0)*{\dblue\xybox{(-18,8);(-28,-8); **\dir{-} ?(0)*\dir{<};}};
 (-27,-6)*{\cdots};
 (-19,6)*{\cdots};
 (-18,0)*{\dpink\xybox{(-8,8);(-18,-8); **\dir{-} ?(0)*\dir{<};}};
 (-13,0)*{\dgreen\xybox{(-3,8);(-13,-8); **\dir{-} ?(0)*\dir{<};}};
 (-12,-6)*{\cdots};
 (-4,6)*{\cdots};
 (-3,0)*{\doran\xybox{(7,8);(-3,-8); **\dir{-} ?(0)*\dir{<};}};
 (12,0)*{\dbrun\xybox{(22,8);(12,-8); **\dir{-} ?(0)*\dir{<};}};
 (13,-6)*{\cdots};
 (21,6)*{\cdots};
 (22,0)*{\dred\xybox{(32,8);(22,-8); **\dir{-} ?(0)*\dir{<};}};
 (12,-4)*{\dturq\xybox{(-3,-8)*{};(27,-8)*{} **\crv{(-2,0) & (27,0)}; ?(.5)*\dir{<};}};
 (14.5,0)*{\dturq\xybox{(7,8)*{};(22,-8)*{} **\crv{(5,0) & (24,6)}; ?(1)*\dir{>};}};
 (-8,3.5)*{\dyellow\xybox{(-23,8)*{};(17,8)*{} **\crv{(-23,0) & (17,0)}; ?(.5)*\dir{<};}};
 (-15.5,0)*{\dyellow\xybox{(-28,8)*{};(7,-8)*{} **\crv{(-30,-4) & (9,2)}; ?(1)*\dir{>};}};
 (-33,-10)*{\scs n}; (-23,-10)*{\scs r+1 }; (-18,-10)*{ \scs 1};(-8,-10)*{\scs i-1}; (-3,-10)*{\scs i}; (2,-10)*{\scs i+1}; (7,-10)*{\scs i+2}; (17,-10)*{\scs r};(22,-10)*{\scs i};(27,-10)*{\scs i};
 (32,0)*{ (1^r)}; 
  \endxy
\vspace*{2ex}
\end{equation*}
We can slide the $i+1$--colored strand over the $n-r+i-1$ leftmost strands using Relation \eqref{eq_r2_ij-gen}, 
so that we end up with the following picture in the central part of the diagram:
\begin{equation*}
 \xy 
 (11,-4)*{\dturq\xybox{(-3,-8)*{};(15,-2)*{} **\crv{(-2,0) & (12,-2)}; ?(.5)*\dir{<};}};
 (11,4)*{\dturq\xybox{(-3,8)*{};(15,2)*{} **\crv{(-2,0) & (12,2)}; ?(.5)*\dir{>};}};
 (7,0)*{\dyellow\xybox{(7,8);(7,-8); **\dir{-} ?(0)*\dir{<};}};
 (2,-10)*{\scs i}; (7,-10)*{\scs i+1}; (2,10)*{\scs i};
 (7,0)*{ \lambda}; 
  \endxy
\vspace*{2ex}
\end{equation*}
with $\lambda= (1,\dots, 1,0,1 ,\dots,1,0, \dots, 0)$ where the first $0$ is in $i$th position and the last $1$ in $r+1$st position. 
We now use Relation \eqref{eq_ident_decomp}, which again reduces to \eqref{id_decomp_red}. Then we apply Relation \eqref{eq_downup_ij-gen} 
to the two rightmost strands. After doing all this, the central part of the diagram becomes equal to  
\begin{equation*}
 \xy 
 (12,0)*{\dturq\xybox{(7,8);(7,-8); **\dir{-} ?(1)*\dir{>};}};
 (7,0)*{\dyellow\xybox{(7,8);(7,-8); **\dir{-} ?(0)*\dir{<};}};
 (2,0)*{\dturq\xybox{(7,8);(7,-8); **\dir{-} ?(0)*\dir{<};}};
 (2,-10)*{\scs i}; (7,-10)*{\scs i+1}; (12,-10)*{\scs i};
   \endxy
\vspace*{2ex}
\end{equation*}
Then, in the full diagram, the rightmost $i$-colored strand above can be slid over the $r-i-1$ rightmost parallel strands 
using Relation \eqref{eq_downup_ij-gen}, so that we end up with
\begin{equation}
\label{eq:kink}
 \xy 
 (-28,0)*{\dblue\xybox{(-28,8);(-28,-8); **\dir{-} ?(.5)*\dir{<};}};
 (-23,0)*{\cdots};
 (-18,0)*{\dpink\xybox{(-18,8);(-18,-8); **\dir{-} ?(.5)*\dir{<};}};
 (-13,0)*{\dgreen\xybox{(-13,8);(-13,-8); **\dir{-} ?(.5)*\dir{<};}};
 (-8,0)*{\cdots};
 (-3,0)*{\dred\xybox{(-3,8);(-3,-8); **\dir{-} ?(.5)*\dir{<};}};
(5,-4.5)*{\dturq\xybox{(17,-4)*{};(22,-8)*{} **\crv{(17,-6) & (22,-6)}; ?(1)*\dir{};}};
(5,-4.5)*{\dturq\xybox{(22,-4)*{};(17,-8)*{} **\crv{(22,-6) & (17,-6)}; ?(0)*\dir{};}};
(5,-6.5)*{\dturq\xybox{(-33,8)*{};(-28,8)*{} **\crv{(-33,12) & (-28,12)}; ?(.5)*\dir{>};}};
 (-28,-10)*{\scs n}; (-18,-10)*{\scs r+1 }; (-13,-10)*{ \scs 1};(-3,-10)*{\scs r}; (2,-10)*{\scs i}; (7,-10)*{\scs i};
 (5,3)*{ (1^r)}; 
  \endxy
\vspace*{2ex}
\end{equation}
We apply Relation \eqref{eq:redtobubbles} to the curl. For $\lambda=(1^r)$, this relation becomes 
\begin{eqnarray*}
  \text{$\xy 0;/r.18pc/:
  (12,8)*{(1^r)};
  %(0,0)*{\twoIu{i}};
  (0,-3)*{\dturq\xybox{
  (-3,-8)*{};(3,8)*{} **\crv{(-3,-1) & (3,1)}?(1)*\dir{>};?(0)*\dir{>};
    (3,-8)*{};(-3,8)*{} **\crv{(3,-1) & (-3,1)}?(1)*\dir{>};
  (-3,-12)*{\bbsid};
  (-3,8)*{\bbsid};
  (3,8)*{}="t1";
  (9,8)*{}="t2";
  (3,-8)*{}="t1'";
  (9,-8)*{}="t2'";
   "t1";"t2" **\crv{(3,14) & (9, 14)};
   "t1'";"t2'" **\crv{(3,-14) & (9, -14)};
   "t2'";"t2" **\dir{-} ?(.5)*\dir{<};}};
   (9,0)*{}; (-7.5,-12)*{\scs i};
 \endxy$} \; = \; -
   \xy
  (19,4)*{(1^r)};
  (0,0)*{\dturq\bbe{}};(-2,-8)*{\scs i};
  (12,-2)*{\dturq\cbub{\black-1}{\black i}};
  (0,6)*{\dturq\bullet}+(2,1)*{\scs 0};
 \endxy
\end{eqnarray*}
Since the bubble is equal to $-1$, the diagram in~\eqref{eq:kink} becomes equal to 
\begin{equation*}
 \xy 
 (-28,0)*{\dblue\xybox{(-28,8);(-28,-8); **\dir{-} ?(.5)*\dir{<};}};
 (-23,0)*{\cdots};
 (-18,0)*{\dpink\xybox{(-18,8);(-18,-8); **\dir{-} ?(.5)*\dir{<};}};
 (-13,0)*{\dgreen\xybox{(-13,8);(-13,-8); **\dir{-} ?(.5)*\dir{<};}};
 (-8,0)*{\cdots};
 (-3,0)*{\dred\xybox{(-3,8);(-3,-8); **\dir{-} ?(.5)*\dir{<};}};
(4.5,-10.5)*{\dturq\xybox{(-33,8)*{};(-28,8)*{} **\crv{(-33,12) & (-28,12)}; ?(.5)*\dir{<};}};
 (-28,-10)*{\scs n}; (-18,-10)*{\scs r+1 }; (-13,-10)*{ \scs 1};(-3,-10)*{\scs r}; (2,-10)*{\scs i}; (7,-10)*{\scs i};
 (5,3)*{ (1^r)}; 
  \endxy
\vspace*{2ex}
\end{equation*}
which is indeed the image under $\Sigma_{n,r}$ of  $\figins{-4}{0.25}{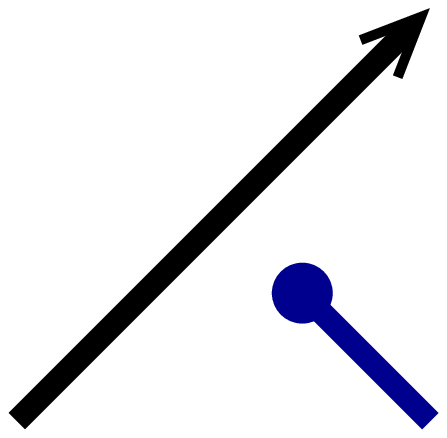} $.
\\
To prove Relations \eqref{eq:slidedotdist-md} and  \eqref{eq:slidedotdist-mu} with colors $(r, 1)$ and $(r-1, r)$, we use exactly the same ingredients as above: 
sliding strands using Relations \eqref{eq_downup_ij-gen} and \eqref{eq_r2_ij-gen}, removing curls using 
\begin{equation*}
  \text{$ \xy 0;/r.18pc/:
  (-12,8)*{\lambda};
   (0,-2)*{\dturq\xybox{
   (-3,-8)*{};(3,8)*{} **\crv{(-3,-1) & (3,1)}?(1)*\dir{>};?(0)*\dir{>};
    (3,-8)*{};(-3,8)*{} **\crv{(3,-1) & (-3,1)}?(1)*\dir{>};
  (3,-12)*{\bbsid};
  (3,8)*{\bbsid}; %(6,-8)*{\scs i};
  (-9,8)*{}="t1";
  (-3,8)*{}="t2";
  (-9,-8)*{}="t1'";
  (-3,-8)*{}="t2'";
   "t1";"t2" **\crv{(-9,14) & (-3, 14)};
   "t1'";"t2'" **\crv{(-9,-14) & (-3, -14)};
  "t1'";"t1" **\dir{-} ?(.5)*\dir{<};}};(7.5,-11)*{\scs i};
 \endxy$} \; = \;
    \xy
  (-3,6)*{\lambda};
  (0,0)*{\dturq\bbe{}};(2,-8)*{\scs i};
  %(-12,-2)*{\dturq\ccbub{\black -1}{\black i}};
  %(0,6)*{\dturq\bullet}+(2,-1)*{\scs 0};
 \endxy
\quad \mbox{or} \quad
\text{$\xy 0;/r.18pc/:
  (12,8)*{\lambda};
  %(0,0)*{\twoIu{i}};
  (0,-3)*{\dturq\xybox{
  (-3,-8)*{};(3,8)*{} **\crv{(-3,-1) & (3,1)}?(1)*\dir{>};?(0)*\dir{>};
    (3,-8)*{};(-3,8)*{} **\crv{(3,-1) & (-3,1)}?(1)*\dir{>};
  (-3,-12)*{\bbsid};
  (-3,8)*{\bbsid};
  (3,8)*{}="t1";
  (9,8)*{}="t2";
  (3,-8)*{}="t1'";
  (9,-8)*{}="t2'";
   "t1";"t2" **\crv{(3,14) & (9, 14)};
   "t1'";"t2'" **\crv{(3,-14) & (9, -14)};
   "t2'";"t2" **\dir{-} ?(.5)*\dir{<};}};
   (9,0)*{}; (-7.5,-12)*{\scs i};
 \endxy$} \; = \; 
   \xy
  (3,6)*{\lambda};
  (0,0)*{\dturq\bbe{}};(-2,-8)*{\scs i};
  %(12,-2)*{\dturq\cbub{\black-1}{\black i}};
  %(0,6)*{\dturq\bullet}+(2,1)*{\scs 0};
 \endxy
\end{equation*}
and applying Relation \eqref{eq_ident_decomp}, which is equal to  
\begin{equation*}
{\dturq \xy
    (0,4)*{\bbpef{\black i}};
    (0,-4.5)*{\bbcef{\black i}};
    (8,0.5)*{ \black \lambda };
    \endxy}
\, = 
\vcenter{\xy 0;/r.18pc/:
  (0,0)*{\dturq\xybox{
  (-8,0)*{};(8,0)*{};
  (-4,10)*{}="t1";
  (4,10)*{}="t2";
  (-4,-10)*{}="b1";
  (4,-10)*{}="b2";
  "t1";"b1" **\dir{-} ?(.5)*\dir{>};
  "t2";"b2" **\dir{-} ?(.5)*\dir{<};}};
  (-6,-8)*{\scs i};(6,-8)*{\scs i};
  (10,2)*{\lambda};
    \endxy}
\end{equation*}
for the relevant value of $\lambda$. 
\\
\noindent $\bullet$ Relation \eqref{eq:mslide3v}. Actually, we will prove 
\begin{equation}\label{eq:mslide3v2}
\ \figins{-17}{0.55}{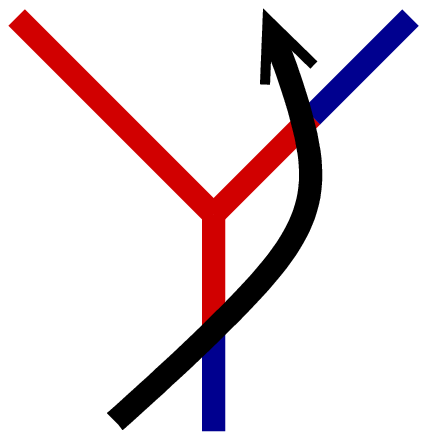}\
=\
\figins{-17}{0.55}{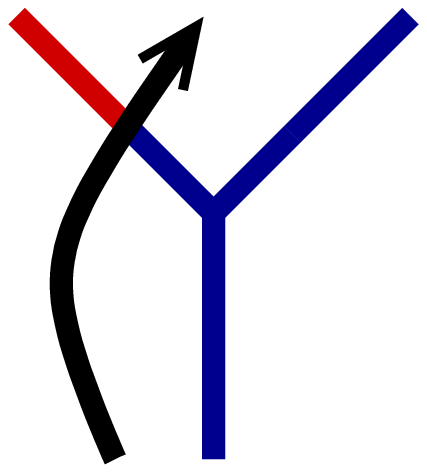}
\end{equation}
which is equivalent. 

Let us start with colors $(i, i+1)$ different from $(r, 1)$ and $(r-1, r)$.

The image under $\Sigma_{n,r}$ of $ \figins{-7}{0.3}{msplitslide-d2.eps} $ is
\begin{equation*}
 \xy 
 (-28,0)*{\dblue\xybox{(-28,-24)*{};(-28,8)*{} **\crv{(-18,-16) & (-18,0)}; ?(0)*\dir{<};}};
 (-18,0)*{\dpink\xybox{(-18,-24)*{};(-18,8)*{} **\crv{(-8,-16) & (-8,0)}; ?(0)*\dir{<};}};
 (-13,0)*{\dgreen\xybox{(-13,-24)*{};(-13,8)*{} **\crv{(-3,-16) & (-3,0)}; ?(0)*\dir{<};}};
 (-3,0)*{\doran\xybox{(-3,-24)*{};(-3,8)*{} **\crv{(7,-16) & (7,0)}; ?(0)*\dir{<};}};
 (12,0)*{\dbrun\xybox{(12,-24)*{};(12,8)*{} **\crv{(22,-16) & (22,0)}; ?(0)*\dir{<};}};
 (22,0)*{\dred\xybox{(22,-24)*{};(22,8)*{} **\crv{(32,-16) & (32,0)}; ?(0)*\dir{<};}};
 (-26,14)*{\cdots};
 %(-19,-6)*{\cdots};
 %(-18,8)*{\dpink\xybox{(-8,-8);(-18,8); **\dir{-} ?(1)*\dir{>};}};
 %(-13,8)*{\dgreen\xybox{(-3,-8);(-13,8); **\dir{-} ?(1)*\dir{>};}};
 (-11,14)*{\cdots};
 %(-4,-6)*{\cdots};
 %(-3,8)*{\doran\xybox{(7,-8);(-3,8); **\dir{-} ?(1)*\dir{>};}};
 %(12,8)*{\dbrun\xybox{(22,-8);(12,8); **\dir{-} ?(1)*\dir{>};}};
 (14,14)*{\cdots};
 %(21,-6)*{\cdots};
 %(22,8)*{\dred\xybox{(32,-8);(22,8); **\dir{-} ?(1)*\dir{>};}};
 (12,11.5)*{\dturq\xybox{(-3,8)*{};(27,8)*{} **\crv{(-2,0) & (27,0)}; ?(.5)*\dir{>};}};
 (14.5,8)*{\dturq\xybox{(7,-8)*{};(22,8)*{} **\crv{(5,0) & (24,-6)}; ?(0)*\dir{<};}};
 (-8,4.3)*{\dyellow\xybox{(-23,-8)*{};(17,-8)*{} **\crv{(-23,0) & (17,0)}; ?(.5)*\dir{>};}};
 (-17.5,11.5)*{\dyellow\xybox{(-33,8)*{};(7,8)*{} **\crv{(-33,0) & (7,0)}; ?(0)*\dir{<};}};
 %(-28,-8)*{\dblue\xybox{(-18,8);(-28,-8); **\dir{-} ?(0)*\dir{<};}};
 (-26,-14)*{\cdots};
 %(-19,6)*{\cdots};
 %(-18,-8)*{\dpink\xybox{(-8,8);(-18,-8); **\dir{-} ?(0)*\dir{<};}};
 %(-13,-8)*{\dgreen\xybox{(-3,8);(-13,-8); **\dir{-} ?(0)*\dir{<};}};
 (-11,-14)*{\cdots};
 %(-4,6)*{\cdots};
 %(-3,-8)*{\doran\xybox{(7,8);(-3,-8); **\dir{-} ?(0)*\dir{<};}};
 %(12,-8)*{\dbrun\xybox{(22,8);(12,-8); **\dir{-} ?(0)*\dir{<};}};
 (14,-14)*{\cdots};
 %(21,6)*{\cdots};
 %(22,-8)*{\dred\xybox{(32,8);(22,-8); **\dir{-} ?(0)*\dir{<};}};
 (12,-12)*{\dturq\xybox{(-3,-8)*{};(27,-8)*{} **\crv{(-2,0) & (27,0)}; ?(.5)*\dir{<};}};
 (14.5,-8)*{\dturq\xybox{(7,8)*{};(22,-8)*{} **\crv{(5,0) & (24,6)}; ?(1)*\dir{>};}};
 (-8,-4.3)*{\dyellow\xybox{(-23,8)*{};(17,8)*{} **\crv{(-23,0) & (17,0)}; ?(.5)*\dir{<};}};
 (-20.5,0)*{\dyellow\xybox{(-38,24)*{};(7,-8)*{} **\crv{(-42,-6) & (9,2)}; ?(1)*\dir{>};}};
 (-33,-18)*{\scs n}; (-23,-18)*{\scs r+1 }; (-18,-18)*{ \scs 1};(-8,-18)*{\scs i-1}; (-3,-18)*{\scs i}; (2,-18)*{\scs i+1}; (7,-18)*{\scs i+2}; (17,-18)*{\scs r};(22,-18)*{\scs i};(27,-18)*{\scs i}; (4,7)*{\scs i+1};
 (32,0)*{ (1^r)}; 
  \endxy
\vspace*{2ex}
\end{equation*}

The image under $\Sigma_{n,r}$ of $ \figins{-7}{0.3}{msplitslide-u2.eps} $ is
\begin{equation*}
 \xy 
 (-28,-5)*{\dblue\xybox{(-18,8);(-28,-18); **\dir{-} ?(0)*\dir{<};}};
 (-27,-16)*{\cdots};
 (-19,6)*{\cdots};
 (-18,-5)*{\dpink\xybox{(-8,8);(-18,-18); **\dir{-} ?(0)*\dir{<};}};
 (-13,-5)*{\dgreen\xybox{(-3,8);(-13,-18); **\dir{-} ?(0)*\dir{<};}};
 (-12,-16)*{\cdots};
 (-4,6)*{\cdots};
 (-3,-5)*{\doran\xybox{(7,8);(-3,-18); **\dir{-} ?(0)*\dir{<};}};
 (12,-5)*{\dbrun\xybox{(22,8);(12,-18); **\dir{-} ?(0)*\dir{<};}};
 (13,-16)*{\cdots};
 (21,6)*{\cdots};
 (22,-5)*{\dred\xybox{(32,8);(22,-18); **\dir{-} ?(0)*\dir{<};}};
 (-3,-9)*{\dturq\xybox{(-3,-13);(-3,-18.5); **\dir{-} ?(0)*\dir{<};}};
 (32,-2.2)*{\dturq\xybox{(32,8);(32,-12.5); **\dir{-} ?(0)*\dir{<};}};
 (12,-4)*{\dturq\xybox{(-3,-18)*{};(27,-18)*{} **\crv{(-2,-12) & (27,-12)}; ?(.5)*\dir{<};}};
 (29.5,-5)*{\dturq\xybox{(-3,-18)*{};(2,-18)*{} **\crv{(-2,-20) & (1,-20)}; ?(.5)*\dir{};}};
 (14.5,-5)*{\dturq\xybox{(7,8)*{};(22,-18)*{} **\crv{(5,0) & (24,6)}; ?(1)*\dir{>};}};
 (32,-5.5)*{\dturq\xybox{(32,8)*{};(22,-18)*{} **\crv{(34,-28) & (20,-10)}; ?(0)*\dir{<};}};
 (-8,3.5)*{\dyellow\xybox{(-23,8)*{};(17,8)*{} **\crv{(-23,0) & (17,0)}; ?(.5)*\dir{<};}};
 (-15.5,-5)*{\dyellow\xybox{(-28,8)*{};(7,-18)*{} **\crv{(-30,-4) & (9,2)}; ?(1)*\dir{>};}};
 (-33,-20)*{\scs n}; (-23,-20)*{\scs r+1 }; (-18,-20)*{ \scs 1};(-8,-20)*{\scs i-1}; (-3,-20)*{\scs i}; (2,-20)*{\scs i+1}; (7,-20)*{\scs i+2}; (17,-20)*{\scs r};(22,-20)*{\scs i};(27,-20)*{\scs i};
 (44,0)*{ (1^r)}; 
  \endxy
\vspace*{2ex}
\end{equation*}

We can apply to the first of these two diagrams the same arguments as in the beginning of the proof of Relation \eqref{eq:reid2ml}, in order to simplify 
it to
\begin{equation*}
 \xy 
 (-28,0)*{\dblue\xybox{(-28,-24)*{};(-28,8)*{} **\crv{(-18,-16) & (-18,0)}; ?(0)*\dir{<};}};
 (-18,0)*{\dpink\xybox{(-18,-24)*{};(-18,8)*{} **\crv{(-8,-16) & (-8,0)}; ?(0)*\dir{<};}};
 (-13,0)*{\dgreen\xybox{(-13,-24)*{};(-13,8)*{} **\crv{(-3,-16) & (-3,0)}; ?(0)*\dir{<};}};
 (-3,0)*{\doran\xybox{(-3,-24)*{};(-3,8)*{} **\crv{(7,-16) & (7,0)}; ?(0)*\dir{<};}};
 (12,0)*{\dbrun\xybox{(12,-24)*{};(12,8)*{} **\crv{(22,-16) & (22,0)}; ?(0)*\dir{<};}};
 (22,0)*{\dred\xybox{(22,-24)*{};(22,8)*{} **\crv{(32,-16) & (32,0)}; ?(0)*\dir{<};}};
 (27,0)*{\dturq\xybox{(22,-24)*{};(22,8)*{} **\crv{(32,-16) & (32,0)}; ?(0)*\dir{<};}};
 (-26,14)*{\cdots};
 %(-19,-6)*{\cdots};
 %(-18,8)*{\dpink\xybox{(-8,-8);(-18,8); **\dir{-} ?(1)*\dir{>};}};
 %(-13,8)*{\dgreen\xybox{(-3,-8);(-13,8); **\dir{-} ?(1)*\dir{>};}};
 (-11,14)*{\cdots};
 %(-4,-6)*{\cdots};
 %(-3,8)*{\doran\xybox{(7,-8);(-3,8); **\dir{-} ?(1)*\dir{>};}};
 %(12,8)*{\dbrun\xybox{(22,-8);(12,8); **\dir{-} ?(1)*\dir{>};}};
 (14,14)*{\cdots};
 %(21,-6)*{\cdots};
 %(22,8)*{\dred\xybox{(32,-8);(22,8); **\dir{-} ?(1)*\dir{>};}};
 (12,11.5)*{\dturq\xybox{(-3,8)*{};(27,8)*{} **\crv{(-2,0) & (27,0)}; ?(.5)*\dir{>};}};
 %(14.5,8)*{\dturq\xybox{(7,-8)*{};(22,8)*{} **\crv{(5,0) & (24,-6)}; ?(0)*\dir{<};}};
 %(-8,4.3)*{\dyellow\xybox{(-23,-8)*{};(17,-8)*{} **\crv{(-23,0) & (17,0)}; ?(.5)*\dir{>};}};
 (-17.5,11.5)*{\dyellow\xybox{(-33,8)*{};(7,8)*{} **\crv{(-33,0) & (7,0)}; ?(0)*\dir{<};}};
 %(-28,-8)*{\dblue\xybox{(-18,8);(-28,-8); **\dir{-} ?(0)*\dir{<};}};
 (-26,-14)*{\cdots};
 %(-19,6)*{\cdots};
 %(-18,-8)*{\dpink\xybox{(-8,8);(-18,-8); **\dir{-} ?(0)*\dir{<};}};
 %(-13,-8)*{\dgreen\xybox{(-3,8);(-13,-8); **\dir{-} ?(0)*\dir{<};}};
 (-11,-14)*{\cdots};
 %(-4,6)*{\cdots};
 %(-3,-8)*{\doran\xybox{(7,8);(-3,-8); **\dir{-} ?(0)*\dir{<};}};
 %(12,-8)*{\dbrun\xybox{(22,8);(12,-8); **\dir{-} ?(0)*\dir{<};}};
 (14,-14)*{\cdots};
 %(21,6)*{\cdots};
 %(22,-8)*{\dred\xybox{(32,8);(22,-8); **\dir{-} ?(0)*\dir{<};}};
 (12,-12)*{\dturq\xybox{(-3,-8)*{};(27,-8)*{} **\crv{(-2,0) & (27,0)}; ?(.5)*\dir{<};}};
 %(14.5,-8)*{\dturq\xybox{(7,8)*{};(22,-8)*{} **\crv{(5,0) & (24,6)}; ?(1)*\dir{>};}};
 %(-8,-4.3)*{\dyellow\xybox{(-23,8)*{};(17,8)*{} **\crv{(-23,0) & (17,0)}; ?(.5)*\dir{<};}};
 (-20.5,0)*{\dyellow\xybox{(-38,24)*{};(7,-8)*{} **\crv{(-42,-6) & (9,2)}; ?(1)*\dir{>};}};
 (-33,-18)*{\scs n}; (-23,-18)*{\scs r+1 }; (-18,-18)*{ \scs 1};(-8,-18)*{\scs i-1}; (-3,-18)*{\scs i}; (2,-18)*{\scs i+1}; (7,-18)*{\scs i+2}; (17,-18)*{\scs r};(22,-18)*{\scs i};(27,-18)*{\scs i}; 
 (35,0)*{ (1^r)}; 
  \endxy
\vspace*{2ex}
\end{equation*}
In order to prove that this diagram is equal to the second one, it only remains to check that
\begin{equation}\label{eqtriple1}
 \xy 
 (11,-4)*{\dturq\xybox{(15,-8)*{};(-3,-2)*{} **\crv{(12,0) & (-2,-2)}; ?(.5)*\dir{>};}};
 (11,4)*{\dturq\xybox{(15,8)*{};(-3,2)*{} **\crv{(12,0) & (-2,2)}; ?(.5)*\dir{<};}};
 (15,0)*{\dturq\xybox{(7,8);(7,-8); **\dir{-} ?(0)*\dir{<};}};
 (20,-10)*{\scs i}; (15,-10)*{\scs i}; (20,10.5)*{\scs i};
 (24,4)*{ \lambda}; 
  \endxy
\vspace*{2ex}
\quad = \quad 
\xy 
 (11,-3)*{\dturq\xybox{(5,-8)*{};(-3,2)*{} **\crv{(2,0) & (-2,2)}; ?(.1)*\dir{<};}};
 (11,3)*{\dturq\xybox{(5,8)*{};(-3,-2)*{} **\crv{(2,0) & (-2,-2)}; ?(.3)*\dir{>};}};
 (20,0)*{\dturq\xybox{(7,8);(7,-8); **\dir{-} ?(1)*\dir{>};}};
 (20,-10)*{\scs i}; (15,-10)*{\scs i}; (15,10.5)*{\scs i};
 (24,4)*{ \lambda}; 
  \endxy
\end{equation}
where $\lambda = (1^r)$. Because Relation \eqref{eq:EF} reduces to 
\begin{eqnarray}\label{id_decomp_red3}
 \vcenter{\xy 0;/r.18pc/:
  (0,0)*{\dturq\xybox{
  (-8,0)*{};(8,0)*{};
  (-4,10)*{}="t1";
  (4,10)*{}="t2";
  (-4,-10)*{}="b1";
  (4,-10)*{}="b2";
  "t1";"b1" **\dir{-} ?(.5)*\dir{<};
  "t2";"b2" **\dir{-} ?(.5)*\dir{>};}};
  (-6,-8)*{\scs i};(6,-8)*{\scs i};
  (10,2)*{\lambda};
    \endxy}
&\quad = \quad&
   \vcenter{\xy 0;/r.18pc/:
    (0,0)*{\dturq\xybox{
    (-4,-4)*{};(4,4)*{} **\crv{(-4,-1) & (4,1)}?(1)*\dir{>};?(0)*\dir{>};
    (4,-4)*{};(-4,4)*{} **\crv{(4,-1) & (-4,1)}?(1)*\dir{<};
    (-4,4)*{};(4,12)*{} **\crv{(-4,7) & (4,9)}?(1)*\dir{<};
    (4,4)*{};(-4,12)*{} **\crv{(4,7) & (-4,9)};}};
    (10,2)*{\lambda};(-6.8,-7)*{\scs i};(6,-7)*{\scs i};
 \endxy} 
\end{eqnarray}
when $\lambda= (1^r)$, the left hand side of \eqref{eqtriple1} is equal to
\begin{equation*}
 \xy 
 (11,-4)*{\dturq\xybox{(15,-8)*{};(-3,-2)*{} **\crv{(12,0) & (-2,-2)}; ?(.5)*\dir{>};}};
 (11,4)*{\dturq\xybox{(15,8)*{};(-3,2)*{} **\crv{(12,0) & (-2,2)}; ?(.5)*\dir{<};}};
 (15,0)*{\dturq\xybox{(7,8);(7,-8); **\dir{-} ?(0)*\dir{<};}};
 (20,-10)*{\scs i}; (15,-10)*{\scs i}; (20,10.5)*{\scs i};
 (24,4)*{ (1^r)}; 
  \endxy
\vspace*{2ex}
\quad = \quad 
\xy 
 (11,-2)*{\dturq\xybox{(15,-8)*{};(-3,-2)*{} **\crv{(3,5) & (0,5)}; ?(.5)*\dir{>};}};
 (11,2)*{\dturq\xybox{(15,8)*{};(-3,2)*{} **\crv{(3,-5) & (0,-5)}; ?(.5)*\dir{<};}};
 (15,0)*{\dturq\xybox{(7,8);(7,-8); **\dir{-} ?(0)*\dir{<};}};
 (20,-10)*{\scs i}; (15,-10)*{\scs i}; (20,10.5)*{\scs i};
 (24,4)*{ (1^r)}; 
  \endxy
\end{equation*}
So proving \eqref{eqtriple1} is equivalent to proving that:
\begin{equation}\label{eqtriple2}
 \xy 
 (11,-3)*{\dturq\xybox{(15,-8)*{};(-3,2)*{} **\crv{(12,-2) & (0,0)}; ?(.5)*\dir{>};}};
 (11,3)*{\dturq\xybox{(15,8)*{};(-3,-2)*{} **\crv{(12,2) & (0,0)}; ?(.5)*\dir{<};}};
 (15,0)*{\dturq\xybox{(7,8);(7,-8); **\dir{-} ?(0)*\dir{<};}};
 (20,-10)*{\scs i}; (15,-10)*{\scs i}; (20,10.5)*{\scs i};
 (24,4)*{ (1^r)}; 
  \endxy
\vspace*{2ex}
\quad = \quad 
\xy 
 (11,-4)*{\dturq\xybox{(5,-8)*{};(-3,-2)*{} **\crv{(4,-5) & (0,-3)}; ?(0.5)*\dir{<};}};
 (11,3)*{\dturq\xybox{(5,8)*{};(-3,2)*{} **\crv{(4,5) & (0,3)}; ?(.5)*\dir{>};}};
 (20,0)*{\dturq\xybox{(7,8);(7,-8); **\dir{-} ?(1)*\dir{>};}};
 (20,-10)*{\scs i}; (15,-10)*{\scs i}; (15,10.5)*{\scs i};
 (24,4)*{ (1^r)}; 
  \endxy
\end{equation}
Let us apply twice Relation \eqref{id_decomp_red2} to the right hand side. This gives us
\begin{equation*}
 \xy 
 (11,-5)*{\dturq\xybox{(15,-12)*{};(-3,-2)*{} **\crv{(20,0) & (-2,-2)}; ?(.3)*\dir{<};}};
 (11,4)*{\dturq\xybox{(15,12)*{};(-3,2)*{} **\crv{(20,0) & (-2,2)}; ?(.3)*\dir{>};}};
 (15,0)*{\dturq\xybox{(15,8)*{};(15,-8)*{} **\crv{(0,3) & (0,-3)}; ?(.7)*\dir{<};}};
 (20,-10)*{\scs i}; (22,-6)*{\scs i}; (20,10.5)*{\scs i};
 (24,4)*{ (1^r)}; 
  \endxy.
\vspace*{2ex}
\end{equation*}
Then we apply Relation \eqref{eq_ident_decomp}, with $\lambda = (1,\dots,1,0,2,1,\dots,1,0,\dots,0)$ where $2$ is in $i+1$st position and the 
last $1$ in $r$th position, which gives us  
\begin{equation}\label{eqtriple3}
-
\xy 
 (11,0)*{\dturq\xybox{(-3,2)*{};(-3,-2)*{} **\crv{(10,2) & (10,-2)}; ?(.7)*\dir{<};}};
 (21,0)*{\dturq\xybox{(15,12)*{};(15,-12)*{} **\crv{(20,3) & (20,-3)}; ?(.3)*\dir{>}+(0.5,-5)*{\bullet};}};
 (15,0)*{\dturq\xybox{(15,8)*{};(15,-8)*{} **\crv{(0,3) & (0,-3)}; ?(.7)*\dir{<};}};
 (21,-11)*{\scs i}; (24,-7)*{\scs i}; (3,-3)*{\scs i};
 (26,4)*{ (1^r)}; 
  \endxy
\quad - \quad 
\xy 
 (11,0)*{\dturq\xybox{(-3,2)*{};(-3,-2)*{} **\crv{(10,2) & (10,-2)}; ?(.7)*\dir{<}+(2,1)*{\bullet};}};
 (21,0)*{\dturq\xybox{(15,12)*{};(15,-12)*{} **\crv{(20,3) & (20,-3)}; ?(.3)*\dir{>};}};
 (15,0)*{\dturq\xybox{(15,8)*{};(15,-8)*{} **\crv{(0,3) & (0,-3)}; ?(.7)*\dir{<};}};
 (21,-11)*{\scs i}; (24,-7)*{\scs i}; (3,-3)*{\scs i};
 (26,4)*{ (1^r)}; 
  \endxy
\quad - \quad 
\xy 
 (11,0)*{\dturq\xybox{(-3,2)*{};(-3,-2)*{} **\crv{(10,2) & (10,-2)}; ?(.7)*\dir{<};}};
 (21,0)*{\dturq\xybox{(15,12)*{};(15,-12)*{} **\crv{(20,3) & (20,-3)}; ?(.3)*\dir{>};}};
 (15,0)*{\dturq\xybox{(15,8)*{};(15,-8)*{} **\crv{(0,3) & (0,-3)}; ?(.7)*\dir{<};}};
(21,-11)*{\scs i}; (24,-7)*{\scs i}; (3,-3)*{\scs i};
 (26,4)*{ (1^r)}; 
 0;/r.15pc/:(28,-1)*{\dturq\xybox{(0,0)*{\ccbub{\black 2}{\black i}};}};
   \endxy.
\end{equation}
Here the first and last term are equal to zero, due to Relation \eqref{eq_nil_rels}. To the second term we can apply 
Relation \eqref{eq_nil_dotslide}, so that we get 
\begin{equation*}
\xy 
 (11,-4)*{\dturq\xybox{(5,-8)*{};(-3,-2)*{} **\crv{(4,-5) & (0,-3)}; ?(0.5)*\dir{<};}};
 (11,3)*{\dturq\xybox{(5,8)*{};(-3,2)*{} **\crv{(4,5) & (0,3)}; ?(.5)*\dir{>};}};
 (20,0)*{\dturq\xybox{(7,8);(7,-8); **\dir{-} ?(1)*\dir{>};}};
 (20,-10)*{\scs i}; (15,-10)*{\scs i}; (15,10.5)*{\scs i};
 (24,4)*{ (1^r)}; 
  \endxy
\vspace*{2ex}
\quad = \quad -
\xy 
 (11,0)*{\dturq\xybox{(-3,2)*{};(-3,-2)*{} **\crv{(10,2) & (10,-2)}; ?(.7)*\dir{<}+(-5,-1)*{\bullet};}};
 (21,0)*{\dturq\xybox{(15,12)*{};(15,-12)*{} **\crv{(20,3) & (20,-3)}; ?(.3)*\dir{>};}};
 (15,0)*{\dturq\xybox{(15,8)*{};(15,-8)*{} **\crv{(0,3) & (0,-3)}; ?(.7)*\dir{<};}};
 (21,-11)*{\scs i}; (24,-7)*{\scs i}; (3,-3)*{\scs i};
 (26,4)*{ (1^r)}; 
  \endxy
\quad + \quad 
\xy 
 (11,-3)*{\dturq\xybox{(15,-8)*{};(-3,2)*{} **\crv{(12,-2) & (0,0)}; ?(.5)*\dir{>};}};
 (11,3)*{\dturq\xybox{(15,8)*{};(-3,-2)*{} **\crv{(12,2) & (0,0)}; ?(.5)*\dir{<};}};
 (15,0)*{\dturq\xybox{(7,8);(7,-8); **\dir{-} ?(0)*\dir{<};}};
 (20,-10)*{\scs i}; (15,-10)*{\scs i}; (20,10.5)*{\scs i};
 (24,4)*{ (1^r)}; 
  \endxy
\end{equation*}
Here again the first term on the right-hand side is killed by Relation \eqref{eq_nil_rels}, so Relation \eqref{eqtriple2} is satisfied 
and $\Sigma_{n,r} \Bigl( \figins{-7}{0.3}{msplitslide-d2.eps} \, \Bigr) $ and $\Sigma_{n,r} \Bigl( \figins{-7}{0.3}{msplitslide-u2.eps} \, \Bigr) $ are equal.

Proving Relation \eqref{eq:mslide3v2} for colors $(r-1,r)$ is not much more complicated and is left to the reader. 
\\
\noindent $\bullet$ Relations \eqref{eq:slide6mv} and \eqref{eq:slide6mv2} involving oriented strands and three adjacent colored strands.

In order to prove these relations we bend the left end of the oriented strand downward and the right end upward in the diagrams. 

Let us start with the case when there are no $r$-colored strands. We prove Relation \eqref{eq:slide6mv} for the case in which the bottom strands 
are colored $(i-1,i,i-1)$. Relation \eqref{eq:slide6mv2} for the case in which the bottom strands are colored $(i,i-1,i)$ can be proved similarly. 

By applying Relations \eqref{eq_downup_ij-gen}, \eqref{eq_r2_ij-gen}, \eqref{eq_r3_easy-gen} and \eqref{eq_other_r3_1} 
to the diagrams $\Sigma_{n,r} \Bigl( \figins{-7}{0.3}{6mvert-slide-d.eps}  \,  \Bigr)$ and $\Sigma_{n,r} \Bigl( \figins{-7}{0.3}{6mvert-slide-u.eps} \, \Bigr)$, 
we can slide the entangled parts of the strands colored $i-1$, $i$ and $i+1$ to the middle of the diagrams. In this way, we can turn  
$\Sigma_{n,r} \Bigl( \figins{-7}{0.3}{6mvert-slide-d.eps}  \,  \Bigr)$ and $\Sigma_{n,r} \Bigl( \figins{-7}{0.3}{6mvert-slide-u.eps} \, \Bigr)$ 
into the following two diagrams: 

the diagram $\Sigma_{n,r} \Bigl( \figins{-7}{0.3}{6mvert-slide-d.eps}  \,  \Bigr)$ becomes 
\begin{equation*}
 \xy 
 (-28,0)*{\dblue\xybox{(-8,24);(-38,-24); **\dir{-} ?(0)*\dir{<};}};
 (-37,-22)*{\cdots};
 (-9,22)*{\cdots};
 (-18,0)*{\dpink\xybox{(2,24);(-28,-24); **\dir{-} ?(0)*\dir{<};}};
 (-13,0)*{\dgreen\xybox{(7,24);(-23,-24); **\dir{-} ?(0)*\dir{<};}};
 (-22,-22)*{\cdots};
 (6,22)*{\cdots};
 (-3,0)*{\dgrey\xybox{(17,24);(-13,-24); **\dir{-} ?(0)*\dir{<};}};
 (37,0)*{\dbrun\xybox{(32,24);(2,-24); **\dir{-} ?(0)*\dir{<};}};
 (28,-22)*{\cdots};
 (56,22)*{\cdots};
 (47,0)*{\dred\xybox{(42,24);(12,-24); **\dir{-} ?(0)*\dir{<};}};
 %(12,-4)*{\dturq\xybox{(-3,-8)*{};(27,-8)*{} **\crv{(-2,0) & (27,0)}; ?(.5)*\dir{<};}};
 %(14.5,0)*{\dturq\xybox{(7,8)*{};(22,-8)*{} **\crv{(5,0) & (24,6)}; ?(1)*\dir{>};}};
 %(-8,3.5)*{\dyellow\xybox{(-23,8)*{};(17,8)*{} **\crv{(-23,0) & (17,0)}; ?(.5)*\dir{<};}};
 (-5.5,11)*{\dyellow\xybox{(-18,24)*{};(7,15)*{} **\crv{(-18,20) & (2,15)}; ?(0)*\dir{<};}};
 (-10,11)*{\dyellow\xybox{(-23,24)*{};(3.1,9)*{} **\crv{(-23,17) & (-4,9)}; ?(1)*\dir{>};}};
 (-14.4,11)*{\dturq\xybox{(-28,24)*{};(-0.8,3)*{} **\crv{(-28,14) & (-10,3)}; ?(0)*\dir{<};}};
 (-18.7,10)*{\dturq\xybox{(-33,24)*{};(-4.2,-3)*{} **\crv{(-33,11) & (-16,-3)}; ?(1)*\dir{>};}};
 (-23.1,7)*{\dyellow\xybox{(-38,24)*{};(-8.2,-9)*{} **\crv{(-38,8) & (-22,-9)}; ?(0)*\dir{<};}};
 (-27.6,4)*{\dyellow\xybox{(-43,24)*{};(-12,-15)*{} **\crv{(-43,5) & (-28,-15)}; ?(1)*\dir{>};}};
 (31.6,-12.5)*{\doran\xybox{(18,-24)*{};(8,-15)*{} **\crv{(18,-20) & (13,-15)}; ?(0)*\dir{<};}};
 (36.1,-12.5)*{\doran\xybox{(23,-24)*{};(11.9,-9)*{} **\crv{(23,-17) & (19,-9)}; ?(1)*\dir{>};}};
 (40.5,-12.5)*{\dturq\xybox{(28,-24)*{};(15.8,-3)*{} **\crv{(28,-14) & (25,-3)}; ?(0)*\dir{<};}};
 (44.8,-10.5)*{\dturq\xybox{(33,-24)*{};(19.2,3)*{} **\crv{(33,-11) & (31,3)}; ?(1)*\dir{>};}};
 (49.2,-7.5)*{\doran\xybox{(38,-24)*{};(23.2,9)*{} **\crv{(38,-8) & (37,9)}; ?(0)*\dir{<};}};
 (53.7,-4.5)*{\doran\xybox{(43,-24)*{};(27,15)*{} **\crv{(43,-5) & (43,15)}; ?(1)*\dir{>};}};
%merdier au milieu
 (-2,-0.5)*{\dyellow\xybox{(-8.2,-9)*{};(3.1,9)*{} **\crv{(5,-9) & (5,9)}; ?(0.5)*\dir{<};}};
 (24.5,10.95)*{\dyellow\xybox{(42,24)*{};(7,15)*{} **\crv{(22,0) & (22,15)}; ?(0.5)*\dir{>};}};
 (5,-13.05)*{\dyellow\xybox{(12,-24)*{};(-12,-15)*{} **\crv{(32,0) & (3,-15)}; ?(0.5)*\dir{<};}};
 (31,11.6)*{\doran\xybox{(2,24)*{};(23.2,9)*{} **\crv{(21,20) & (20,9)}; ?(0.5)*\dir{>};}};
 (11.1,-12.4)*{\doran\xybox{(-28,-24)*{};(11.9,-9)*{} **\crv{(9,-20) & (6,-9)}; ?(0.5)*\dir{<};}};
 (24.2,0)*{\doran\xybox{(8,-15)*{};(27,15)*{} **\crv{(-20,-15) & (-12,15)}; ?(0.5)*\dir{<};}};
 (17,-1)*{\dturq\xybox{(-18,-24)*{};(12,24)*{} **\crv{(-15,15) & (-10,20)}; ?(0.5)*\dir{<};}};
 (15.1,7.3)*{\dturq\xybox{(-22.2,-3)*{};(16.25,-2.6)*{} **\crv{(-6,-3) & (10,40) & (10,-3)}; ?(0.7)*\dir{>};}};
 (18.8,-6.5)*{\dturq\xybox{(-19.7,2.45)*{};(19.2,3)*{} **\crv{(-5,3) & (-17,-35) & (7,3)}; ?(0.7)*\dir{<};}};
 %(-14.4,11)*{\dturq\xybox{(52,24)*{};(-0.8,3)*{} **\crv{(52,14) & (10,3)}; ?(0)*\dir{<};}};
 (-43,-26)*{\scs n}; (-33,-26)*{\scs r+1 }; (-28,-26)*{ \scs 1};(-18,-26)*{\scs i-2}; (-8,-26)*{\scs i-1}; (2,-26)*{\scs i}; (12,-26)*{\scs i+1}; (22,-26)*{\scs i+2}; (32,-26)*{\scs r};(37,-26)*{\scs i-1};(42,-26)*{\scs i-1}; (47,-26)*{\scs i};(52,-26)*{\scs i}; (57,-26)*{\scs i-1};(62,-26)*{\scs i-1}; (-43,26)*{\scs i+1};(-38,26)*{\scs i+1}; (-33,26)*{\scs i};(-28,26)*{\scs i}; (-23,26)*{\scs i+1};(-18,26)*{\scs i+1};
 (72,0)*{ (1^r)}; 
  \endxy
\vspace*{2ex}
\end{equation*}
and $\Sigma_{n,r} \Bigl( \figins{-7}{0.3}{6mvert-slide-u.eps} \, \Bigr)$ becomes
\begin{equation*}
 \xy 
 (-28,0)*{\dblue\xybox{(-8,24);(-38,-24); **\dir{-} ?(0)*\dir{<};}};
 (-37,-22)*{\cdots};
 (-9,22)*{\cdots};
 (-18,0)*{\dpink\xybox{(2,24);(-28,-24); **\dir{-} ?(0)*\dir{<};}};
 (-13,0)*{\dgreen\xybox{(7,24);(-23,-24); **\dir{-} ?(0)*\dir{<};}};
 (-22,-22)*{\cdots};
 (6,22)*{\cdots};
 (-3,0)*{\dgrey\xybox{(17,24);(-13,-24); **\dir{-} ?(0)*\dir{<};}};
 (37,0)*{\dbrun\xybox{(32,24);(2,-24); **\dir{-} ?(0)*\dir{<};}};
 (28,-22)*{\cdots};
 (56,22)*{\cdots};
 (47,0)*{\dred\xybox{(42,24);(12,-24); **\dir{-} ?(0)*\dir{<};}};
 %(12,-4)*{\dturq\xybox{(-3,-8)*{};(27,-8)*{} **\crv{(-2,0) & (27,0)}; ?(.5)*\dir{<};}};
 %(14.5,0)*{\dturq\xybox{(7,8)*{};(22,-8)*{} **\crv{(5,0) & (24,6)}; ?(1)*\dir{>};}};
 %(-8,3.5)*{\dyellow\xybox{(-23,8)*{};(17,8)*{} **\crv{(-23,0) & (17,0)}; ?(.5)*\dir{<};}};
 (-5.5,11)*{\dyellow\xybox{(-18,24)*{};(7,15)*{} **\crv{(-18,20) & (2,15)}; ?(0)*\dir{<};}};
 (-10,11)*{\dyellow\xybox{(-23,24)*{};(3.1,9)*{} **\crv{(-23,17) & (-4,9)}; ?(1)*\dir{>};}};
 (-14.4,11)*{\dturq\xybox{(-28,24)*{};(-0.8,3)*{} **\crv{(-28,14) & (-10,3)}; ?(0)*\dir{<};}};
 (-18.7,10)*{\dturq\xybox{(-33,24)*{};(-4.2,-3)*{} **\crv{(-33,11) & (-16,-3)}; ?(1)*\dir{>};}};
 (-23.1,7)*{\dyellow\xybox{(-38,24)*{};(-8.2,-9)*{} **\crv{(-38,8) & (-22,-9)}; ?(0)*\dir{<};}};
 (-27.6,4)*{\dyellow\xybox{(-43,24)*{};(-12,-15)*{} **\crv{(-43,5) & (-28,-15)}; ?(1)*\dir{>};}};
 (31.6,-12.5)*{\doran\xybox{(18,-24)*{};(8,-15)*{} **\crv{(18,-20) & (13,-15)}; ?(0)*\dir{<};}};
 (36.1,-12.5)*{\doran\xybox{(23,-24)*{};(11.9,-9)*{} **\crv{(23,-17) & (19,-9)}; ?(1)*\dir{>};}};
 (40.5,-12.5)*{\dturq\xybox{(28,-24)*{};(15.8,-3)*{} **\crv{(28,-14) & (25,-3)}; ?(0)*\dir{<};}};
 (44.8,-10.5)*{\dturq\xybox{(33,-24)*{};(19.2,3)*{} **\crv{(33,-11) & (31,3)}; ?(1)*\dir{>};}};
 (49.2,-7.5)*{\doran\xybox{(38,-24)*{};(23.2,9)*{} **\crv{(38,-8) & (37,9)}; ?(0)*\dir{<};}};
 (53.7,-4.5)*{\doran\xybox{(43,-24)*{};(27,15)*{} **\crv{(43,-5) & (43,15)}; ?(1)*\dir{>};}};
%merdier au milieu
 (5.5,-0.5)*{\dyellow\xybox{(-8.2,-9)*{};(3.1,9)*{} **\crv{(25,-9) & (20,9)}; ?(0.5)*\dir{<};}};
 (24.5,10.95)*{\dyellow\xybox{(42,24)*{};(7,15)*{} **\crv{(32,20) & (20,15)}; ?(0.5)*\dir{>};}};
 (-0.2,-13)*{\dyellow\xybox{(12,-24)*{};(-12,-15)*{} **\crv{(2,-20) & (-7,-15)}; ?(0.5)*\dir{<};}};
 (24,9.35)*{\doran\xybox{(2,24)*{};(23.2,9)*{} **\crv{(-20,10) & (10,-20) & (18,9)}; ?(0.7)*\dir{>};}};
 (11.1,-12.4)*{\doran\xybox{(-28,-24)*{};(11.9,-9)*{} **\crv{(-30,-15) & (0,5) & (6,-9)}; ?(0.2)*\dir{<};}};
 (31.7,0)*{\doran\xybox{(8,-15)*{};(27,15)*{} **\crv{(-5,-15) & (3,15)}; ?(0.5)*\dir{<};}};
 (17,-0.9)*{\dturq\xybox{(-18,-24)*{};(12,24)*{} **\crv{(0,-3) & (3,3)}; ?(0.5)*\dir{<};}};
 (15.1,-9.25)*{\dturq\xybox{(-22.2,-3)*{};(16.25,-2.55)*{} **\crv{(-6,-3) & (0,-30) & (10,-3)}; ?(0.35)*\dir{>};}};
 (18.8,6.6)*{\dturq\xybox{(-19.7,2.45)*{};(19.2,3)*{} **\crv{(-5,3) & (15,25) & (7,3)}; ?(0.3)*\dir{<};}};
 %(-14.4,11)*{\dturq\xybox{(52,24)*{};(-0.8,3)*{} **\crv{(52,14) & (10,3)}; ?(0)*\dir{<};}};
 (-43,-26)*{\scs n}; (-33,-26)*{\scs r+1 }; (-28,-26)*{ \scs 1};(-18,-26)*{\scs i-2}; (-8,-26)*{\scs i-1}; (2,-26)*{\scs i}; (12,-26)*{\scs i+1}; (22,-26)*{\scs i+2}; (32,-26)*{\scs r};(37,-26)*{\scs i-1};(42,-26)*{\scs i-1}; (47,-26)*{\scs i};(52,-26)*{\scs i}; (57,-26)*{\scs i-1};(62,-26)*{\scs i-1}; (-43,26)*{\scs i+1};(-38,26)*{\scs i+1}; (-33,26)*{\scs i};(-28,26)*{\scs i}; (-23,26)*{\scs i+1};(-18,26)*{\scs i+1};
 (72,0)*{ (1^r)}; 
  \endxy.
\vspace*{2ex}
\end{equation*}
We have to prove that these two diagrams are equivalent. 

First consider the diagram above which is equivalent to $\Sigma_{n,r} \Bigl( \figins{-7}{0.3}{6mvert-slide-d.eps}  \,  \Bigr)$. 
Apply Relation~\eqref{eq_r3_extra} to the triangle in the center formed by three $i$-colored strands. Since 
$\lambda = (1,\dots,1,0,1,2,0,1,\dots,1,0,\dots,0)$, where $2$ is in the $i+1$st position and the last $1$ in the $r+1$st position, that relation is equal to 
\begin{equation} \label{R3red}
\text{$
 \vcenter{
 \xy 0;/r.17pc/:
 (0,0)*{\dturq\xybox{
    (-4,-4)*{};(4,4)*{} **\crv{(-4,-1) & (4,1)}?(1)*\dir{>};
    (4,-4)*{};(-4,4)*{} **\crv{(4,-1) & (-4,1)}?(1)*\dir{<};
    ?(0)*\dir{<};
    (4,4)*{};(12,12)*{} **\crv{(4,7) & (12,9)}?(1)*\dir{>};
    (12,4)*{};(4,12)*{} **\crv{(12,7) & (4,9)}?(1)*\dir{>};
    (-4,12)*{};(4,20)*{} **\crv{(-4,15) & (4,17)};
    (4,12)*{};(-4,20)*{} **\crv{(4,15) & (-4,17)}?(1)*\dir{>};
    (-4,4)*{}; (-4,12) **\dir{-};
    (12,-4)*{}; (12,4) **\dir{-};
    (12,12)*{}; (12,20) **\dir{-};
}};
  (12,0)*{\lambda};
  (-10,-11)*{\scs i};
  ( 2.5,-11)*{\scs i};
  ( 9.5,-11)*{\scs i};
\endxy}
=\;
   \vcenter{
 \xy 0;/r.17pc/:
(0,0)*{\dturq\xybox{ 
   (4,-4)*{};(-4,4)*{} **\crv{(4,-1) & (-4,1)}?(1)*\dir{>};
    (-4,-4)*{};(4,4)*{} **\crv{(-4,-1) & (4,1)}?(0)*\dir{<};
    (-4,4)*{};(-12,12)*{} **\crv{(-4,7) & (-12,9)}?(1)*\dir{>};
    (-12,4)*{};(-4,12)*{} **\crv{(-12,7) & (-4,9)}?(1)*\dir{>};
    (4,12)*{};(-4,20)*{} **\crv{(4,15) & (-4,17)};
    (-4,12)*{};(4,20)*{} **\crv{(-4,15) & (4,17)}?(1)*\dir{>};
    (4,4)*{}; (4,12) **\dir{-} ?(.5)*\dir{<};
    (-12,-4)*{}; (-12,4) **\dir{-};
    (-12,12)*{}; (-12,20) **\dir{-};
}};
  (12,0)*{\lambda};
  (10  ,-11)*{\scs i};
  (-2.5,-11)*{\scs i};
  (-9.5,-11)*{\scs i};
\endxy}
$}
\end{equation}
Then apply Relation~\eqref{eq_r3_easy-gen} to the two triangles formed by strands colored $(i-1, i, i+1)$. Both triangles are slightly to the right of the center 
and one is higher and the other is lower than the center. Sliding the $i$-colored strands to the left using this relation creates four bigons, 
two between strands colored $i$ and $i+1$ and the other two between strands colored $i-1$ and $i+1$. 
The first two bigons can be solved using Relation~\eqref{eq_downup_ij-gen}, the other two using Relation \eqref{eq_r2_ij-gen}. Finally, 
we apply Relation~\eqref{eq_r3_easy-gen} to the top and bottom central triangles of the diagram. This proves that 
$\Sigma_{n,r} \Bigl( \figins{-7}{0.3}{6mvert-slide-d.eps}  \,  \Bigr)$ is equivalent to

\begin{equation*}
 \xy 
 (-28,0)*{\dblue\xybox{(-8,24);(-38,-24); **\dir{-} ?(0)*\dir{<};}};
 (-37,-22)*{\cdots};
 (-9,22)*{\cdots};
 (-18,0)*{\dpink\xybox{(2,24);(-28,-24); **\dir{-} ?(0)*\dir{<};}};
 (-13,0)*{\dgreen\xybox{(7,24);(-23,-24); **\dir{-} ?(0)*\dir{<};}};
 (-22,-22)*{\cdots};
 (6,22)*{\cdots};
 (-3,0)*{\dgrey\xybox{(17,24);(-13,-24); **\dir{-} ?(0)*\dir{<};}};
 (37,0)*{\dbrun\xybox{(32,24);(2,-24); **\dir{-} ?(0)*\dir{<};}};
 (28,-22)*{\cdots};
 (56,22)*{\cdots};
 (47,0)*{\dred\xybox{(42,24);(12,-24); **\dir{-} ?(0)*\dir{<};}};
 %(12,-4)*{\dturq\xybox{(-3,-8)*{};(27,-8)*{} **\crv{(-2,0) & (27,0)}; ?(.5)*\dir{<};}};
 %(14.5,0)*{\dturq\xybox{(7,8)*{};(22,-8)*{} **\crv{(5,0) & (24,6)}; ?(1)*\dir{>};}};
 %(-8,3.5)*{\dyellow\xybox{(-23,8)*{};(17,8)*{} **\crv{(-23,0) & (17,0)}; ?(.5)*\dir{<};}};
 (-5.5,11)*{\dyellow\xybox{(-18,24)*{};(7,15)*{} **\crv{(-18,20) & (2,15)}; ?(0)*\dir{<};}};
 (-10,11)*{\dyellow\xybox{(-23,24)*{};(3.1,9)*{} **\crv{(-23,17) & (-4,9)}; ?(1)*\dir{>};}};
 (-14.4,11)*{\dturq\xybox{(-28,24)*{};(-0.8,3)*{} **\crv{(-28,14) & (-10,3)}; ?(0)*\dir{<};}};
 (-18.7,10)*{\dturq\xybox{(-33,24)*{};(-4.2,-3)*{} **\crv{(-33,11) & (-16,-3)}; ?(1)*\dir{>};}};
 (-23.1,7)*{\dyellow\xybox{(-38,24)*{};(-8.2,-9)*{} **\crv{(-38,8) & (-22,-9)}; ?(0)*\dir{<};}};
 (-27.6,4)*{\dyellow\xybox{(-43,24)*{};(-12,-15)*{} **\crv{(-43,5) & (-28,-15)}; ?(1)*\dir{>};}};
 (31.6,-12.5)*{\doran\xybox{(18,-24)*{};(8,-15)*{} **\crv{(18,-20) & (13,-15)}; ?(0)*\dir{<};}};
 (36.1,-12.5)*{\doran\xybox{(23,-24)*{};(11.9,-9)*{} **\crv{(23,-17) & (19,-9)}; ?(1)*\dir{>};}};
 (40.5,-12.5)*{\dturq\xybox{(28,-24)*{};(15.8,-3)*{} **\crv{(28,-14) & (25,-3)}; ?(0)*\dir{<};}};
 (44.8,-10.5)*{\dturq\xybox{(33,-24)*{};(19.2,3)*{} **\crv{(33,-11) & (31,3)}; ?(1)*\dir{>};}};
 (49.2,-7.5)*{\doran\xybox{(38,-24)*{};(23.2,9)*{} **\crv{(38,-8) & (37,9)}; ?(0)*\dir{<};}};
 (53.7,-4.5)*{\doran\xybox{(43,-24)*{};(27,15)*{} **\crv{(43,-5) & (43,15)}; ?(1)*\dir{>};}};
%merdier au milieu
 (-2,-0.5)*{\dyellow\xybox{(-8.2,-9)*{};(3.1,9)*{} **\crv{(5,-9) & (5,9)}; ?(0.5)*\dir{<};}};
 (24.5,10.95)*{\dyellow\xybox{(42,24)*{};(7,15)*{} **\crv{(32,20) & (20,15)}; ?(0.5)*\dir{>};}};
 (-0.2,-13)*{\dyellow\xybox{(12,-24)*{};(-12,-15)*{} **\crv{(2,-20) & (-7,-15)}; ?(0.5)*\dir{<};}};
 (31,11.6)*{\doran\xybox{(2,24)*{};(23.2,9)*{} **\crv{(2,20) & (11,9)}; ?(0.5)*\dir{>};}};
 (11.1,-12.4)*{\doran\xybox{(-28,-24)*{};(11.9,-9)*{} **\crv{(-28,-20) & (-10,-9)}; ?(0.5)*\dir{<};}};
 (34.2,0)*{\doran\xybox{(8,-15)*{};(27,15)*{} **\crv{(0,-15) & (8,15)}; ?(0.5)*\dir{<};}};
 (17,-0.9)*{\dturq\xybox{(-18,-24)*{};(12,24)*{} **\crv{(0,-3) & (3,3)}; ?(0.5)*\dir{<};}};
 (15.1,2.3)*{\dturq\xybox{(-22.2,-3)*{};(16.25,-2.6)*{} **\crv{(-6,-3) & (10,20) & (10,-3)}; ?(0.7)*\dir{>};}};
 (18.8,-1.5)*{\dturq\xybox{(-19.7,2.45)*{};(19.2,3)*{} **\crv{(-5,3) & (-17,-15) & (7,3)}; ?(0.7)*\dir{<};}};
 %(-14.4,11)*{\dturq\xybox{(52,24)*{};(-0.8,3)*{} **\crv{(52,14) & (10,3)}; ?(0)*\dir{<};}};
 (-43,-26)*{\scs n}; (-33,-26)*{\scs r+1 }; (-28,-26)*{ \scs 1};(-18,-26)*{\scs i-2}; (-8,-26)*{\scs i-1}; (2,-26)*{\scs i}; (12,-26)*{\scs i+1}; (22,-26)*{\scs i+2}; (32,-26)*{\scs r};(37,-26)*{\scs i-1};(42,-26)*{\scs i-1}; (47,-26)*{\scs i};(52,-26)*{\scs i}; (57,-26)*{\scs i-1};(62,-26)*{\scs i-1}; (-43,26)*{\scs i+1};(-38,26)*{\scs i+1}; (-33,26)*{\scs i};(-28,26)*{\scs i}; (-23,26)*{\scs i+1};(-18,26)*{\scs i+1};
 (72,0)*{ (1^r)}; 
  \endxy
\vspace*{2ex}
\end{equation*}
Now apply Relation \eqref{eq_r3_hard-gen}
\begin{equation*}
 \vcenter{
 \xy 0;/r.18pc/:
(0,0)*{\dturq\xybox{
    (-4,-4)*{};(4,4)*{} **\crv{(-4,-1) & (4,1)}?(1)*\dir{<};
    (4,4)*{};(12,12)*{} **\crv{(4,7) & (12,9)}?(1)*\dir{<};
    (12,4)*{};(4,12)*{} **\crv{(12,7) & (4,9)}?(1)*\dir{<};
    (4,12)*{};(-4,20)*{} **\crv{(4,15) & (-4,17)}?(1)*\dir{<};
    (12,-4)*{}; (12,4) **\dir{-};
    (12,12)*{}; (12,20) **\dir{-};
}};
(-4,0)*{\doran\xybox{
    (4,-4)*{};(-4,4)*{} **\crv{(4,-1) & (-4,1)}?(1)*\dir{<};
    (-4,12)*{};(4,20)*{} **\crv{(-4,15) & (4,17)}?(1)*\dir{<};
    (-4,4)*{}; (-4,12) **\dir{-};
}};
  (12,0)*{\lambda};
  (-10,-11)*{\scs i};
  (3.5,-11)*{\scs i-1};
  (9.5,-11)*{\scs i};
\endxy}
\quad = \quad
 \vcenter{
 \xy 0;/r.18pc/:
(0,0)*{\dturq\xybox{
    (4,-4)*{};(-4,4)*{} **\crv{(4,-1) & (-4,1)}?(1)*\dir{<};
    (-4,4)*{};(-12,12)*{} **\crv{(-4,7) & (-12,9)}?(1)*\dir{<};
    (-12,4)*{};(-4,12)*{} **\crv{(-12,7) & (-4,9)}?(1)*\dir{<};
    (-4,12)*{};(4,20)*{} **\crv{(-4,15) & (4,17)}?(1)*\dir{<};
    (-12,-4)*{}; (-12,4) **\dir{-};
    (-12,12)*{}; (-12,20) **\dir{-};
}};
(4,0)*{\doran\xybox{
    (-4,-4)*{};(4,4)*{} **\crv{(-4,-1) & (4,1)}?(1)*\dir{<};
    (4,12)*{};(-4,20)*{} **\crv{(4,15) & (-4,17)}?(1)*\dir{<};
    (4,4)*{}; (4,12) **\dir{-};
}};
  (12,0)*{\lambda};
  (10,-11)*{\scs i};
  (-4,-11)*{\scs i-1};
  (-9.5,-11)*{\scs i};
\endxy}
 \;\; -\;\;\;\; 
\xy 0;/r.18pc/:
(10,0)*{\dturq\xybox{
  (4,12);(4,-12) **\dir{-}?(.5)*\dir{>};
  (22,12);(22,-12) **\dir{-}?(.5)*\dir{>}?(.25)*\dir{}+(0,0)*{}+(10,0)*{\scs};
}};
(3.5,0)*{\doran\xybox{
  (-4,12);(-4,-12) **\dir{-}?(.5)*\dir{>}?(.25)*\dir{}+(0,0)*{}+(-3,0)*{\scs };
}};
  (20,0)*{\lambda}; 
  (-5.6,-11)*{\scs i};
  (0.5,-11)*{\scs i-1};
  (15.2,-11)*{\scs i};
 \endxy
\end{equation*}
to the triangle in the central right part of the diagram. This gives us two terms. The second term is killed because it contains a bigon between 
two $i$-colored strands with the same orientation, which is zero by Relation \eqref{eq_nil_rels}. In the remaining term we can slide the vertical $i$-colored 
strand to the left using Relation \eqref{eq_r3_extra}, as we did in~\eqref{R3red}. 
This leaves us with a bigon between two $i$-colored strands with opposite orientations. Use Relation~\eqref{eq_ident_decomp} in order to remove this bigon. 
Note that this relation is equal to   
\begin{eqnarray*}
 \vcenter{\xy 0;/r.18pc/:
  (0,0)*{\dturq\xybox{
  (-8,0)*{};(8,0)*{};
  (-4,10)*{}="t1";
  (4,10)*{}="t2";
  (-4,-10)*{}="b1";
  (4,-10)*{}="b2";
  "t1";"b1" **\dir{-} ?(.5)*\dir{>};
  "t2";"b2" **\dir{-} ?(.5)*\dir{<};}};
  (-6,-8)*{\scs i};(6,-8)*{\scs i};
  (10,2)*{\lambda};
    \endxy}
&\quad = \quad&
   \vcenter{\xy 0;/r.18pc/:
    (0,0)*{\dturq\xybox{
    (-4,-4)*{};(4,4)*{} **\crv{(-4,-1) & (4,1)}?(1)*\dir{<};?(0)*\dir{<};
    (4,-4)*{};(-4,4)*{} **\crv{(4,-1) & (-4,1)}?(1)*\dir{>};
    (-4,4)*{};(4,12)*{} **\crv{(-4,7) & (4,9)}?(1)*\dir{>};
    (4,4)*{};(-4,12)*{} **\crv{(4,7) & (-4,9)};}};
    (10,2)*{\lambda};(-6.8,-7)*{\scs i};(6,-7)*{\scs i};
 \endxy}, 
\end{eqnarray*}
since $\lambda = (1,\dots,1,0,2,1,0,1,\dots,1,0,\dots,0)$ with $2$ in the $i$th position and the last $1$ in the $r+1$st position. 

This tells us that $\Sigma_{n,r} \Bigl( \figins{-7}{0.3}{6mvert-slide-d.eps}  \,  \Bigr)$ is equivalent to
\begin{equation}\label{diagfinal}
 \xy 
 (-28,0)*{\dblue\xybox{(-8,24);(-38,-24); **\dir{-} ?(0)*\dir{<};}};
 (-37,-22)*{\cdots};
 (-9,22)*{\cdots};
 (-18,0)*{\dpink\xybox{(2,24);(-28,-24); **\dir{-} ?(0)*\dir{<};}};
 (-13,0)*{\dgreen\xybox{(7,24);(-23,-24); **\dir{-} ?(0)*\dir{<};}};
 (-22,-22)*{\cdots};
 (6,22)*{\cdots};
 (-3,0)*{\dgrey\xybox{(17,24);(-13,-24); **\dir{-} ?(0)*\dir{<};}};
 (37,0)*{\dbrun\xybox{(32,24);(2,-24); **\dir{-} ?(0)*\dir{<};}};
 (28,-22)*{\cdots};
 (56,22)*{\cdots};
 (47,0)*{\dred\xybox{(42,24);(12,-24); **\dir{-} ?(0)*\dir{<};}};
 %(12,-4)*{\dturq\xybox{(-3,-8)*{};(27,-8)*{} **\crv{(-2,0) & (27,0)}; ?(.5)*\dir{<};}};
 %(14.5,0)*{\dturq\xybox{(7,8)*{};(22,-8)*{} **\crv{(5,0) & (24,6)}; ?(1)*\dir{>};}};
 %(-8,3.5)*{\dyellow\xybox{(-23,8)*{};(17,8)*{} **\crv{(-23,0) & (17,0)}; ?(.5)*\dir{<};}};
 (-5.5,11)*{\dyellow\xybox{(-18,24)*{};(7,15)*{} **\crv{(-18,20) & (2,15)}; ?(0)*\dir{<};}};
 (-10,11)*{\dyellow\xybox{(-23,24)*{};(3.1,9)*{} **\crv{(-23,17) & (-4,9)}; ?(1)*\dir{>};}};
 (-14.4,11)*{\dturq\xybox{(-28,24)*{};(-0.8,3)*{} **\crv{(-28,14) & (-10,3)}; ?(0)*\dir{<};}};
 (-18.7,10)*{\dturq\xybox{(-33,24)*{};(-4.2,-3)*{} **\crv{(-33,11) & (-16,-3)}; ?(1)*\dir{>};}};
 (-23.1,7)*{\dyellow\xybox{(-38,24)*{};(-8.2,-9)*{} **\crv{(-38,8) & (-22,-9)}; ?(0)*\dir{<};}};
 (-27.6,4)*{\dyellow\xybox{(-43,24)*{};(-12,-15)*{} **\crv{(-43,5) & (-28,-15)}; ?(1)*\dir{>};}};
 (31.6,-12.5)*{\doran\xybox{(18,-24)*{};(8,-15)*{} **\crv{(18,-20) & (13,-15)}; ?(0)*\dir{<};}};
 (36.1,-12.5)*{\doran\xybox{(23,-24)*{};(11.9,-9)*{} **\crv{(23,-17) & (19,-9)}; ?(1)*\dir{>};}};
 (40.5,-12.5)*{\dturq\xybox{(28,-24)*{};(15.8,-3)*{} **\crv{(28,-14) & (25,-3)}; ?(0)*\dir{<};}};
 (44.8,-10.5)*{\dturq\xybox{(33,-24)*{};(19.2,3)*{} **\crv{(33,-11) & (31,3)}; ?(1)*\dir{>};}};
 (49.2,-7.5)*{\doran\xybox{(38,-24)*{};(23.2,9)*{} **\crv{(38,-8) & (37,9)}; ?(0)*\dir{<};}};
 (53.7,-4.5)*{\doran\xybox{(43,-24)*{};(27,15)*{} **\crv{(43,-5) & (43,15)}; ?(1)*\dir{>};}};
%merdier au milieu
 (-2,-0.5)*{\dyellow\xybox{(-8.2,-9)*{};(3.1,9)*{} **\crv{(5,-9) & (5,9)}; ?(0.5)*\dir{<};}};
 (24.5,10.95)*{\dyellow\xybox{(42,24)*{};(7,15)*{} **\crv{(32,20) & (20,15)}; ?(0.5)*\dir{>};}};
 (-0.2,-13)*{\dyellow\xybox{(12,-24)*{};(-12,-15)*{} **\crv{(2,-20) & (-7,-15)}; ?(0.5)*\dir{<};}};
 (31,11.6)*{\doran\xybox{(2,24)*{};(23.2,9)*{} **\crv{(2,20) & (11,9)}; ?(0.5)*\dir{>};}};
 (11.1,-12.4)*{\doran\xybox{(-28,-24)*{};(11.9,-9)*{} **\crv{(-28,-20) & (-10,-9)}; ?(0.5)*\dir{<};}};
 (34.2,0)*{\doran\xybox{(8,-15)*{};(27,15)*{} **\crv{(0,-15) & (8,15)}; ?(0.5)*\dir{<};}};
 (17,-0.9)*{\dturq\xybox{(-18,-24)*{};(12,24)*{} **\crv{(0,-3) & (3,3)}; ?(0.5)*\dir{<};}};
 (15,-2.3)*{\dturq\xybox{(-22.2,-3.5)*{};(16.25,-3)*{}; **\dir{-} ?(0)*\dir{<};}};
 (18.6,0.7)*{\dturq\xybox{(-19.7,3)*{};(19.2,3.5)*{}; **\dir{-} ?(1)*\dir{>};}};
 %(-14.4,11)*{\dturq\xybox{(52,24)*{};(-0.8,3)*{} **\crv{(52,14) & (10,3)}; ?(0)*\dir{<};}};
 (-43,-26)*{\scs n}; (-33,-26)*{\scs r+1 }; (-28,-26)*{ \scs 1};(-18,-26)*{\scs i-2}; (-8,-26)*{\scs i-1}; (2,-26)*{\scs i}; (12,-26)*{\scs i+1}; (22,-26)*{\scs i+2}; (32,-26)*{\scs r};(37,-26)*{\scs i-1};(42,-26)*{\scs i-1}; (47,-26)*{\scs i};(52,-26)*{\scs i}; (57,-26)*{\scs i-1};(62,-26)*{\scs i-1}; (-43,26)*{\scs i+1};(-38,26)*{\scs i+1}; (-33,26)*{\scs i};(-28,26)*{\scs i}; (-23,26)*{\scs i+1};(-18,26)*{\scs i+1};
 (72,0)*{ (1^r)}; 
  \endxy
\vspace*{2ex}
\end{equation}

Let us now prove that $\Sigma_{n,r} \Bigl( \figins{-7}{0.3}{6mvert-slide-u.eps} \, \Bigr)$ is equivalent to this same diagram. First apply 
Relation~\eqref{eq_r3_extra} to the triangle formed by the three $i-1$-colored strands just right of the center. Since 
$\lambda = (1,\dots,1,2,0,0,1,\dots,1,0,\dots,0)$ with $2$ in the $i$th position and the last $1$ in the $r+1$st position, 
this relation becomes 
\begin{equation*} 
\text{$
 \vcenter{
 \xy 0;/r.17pc/:
 (0,0)*{\doran\xybox{
    (-4,-4)*{};(4,4)*{} **\crv{(-4,-1) & (4,1)}?(1)*\dir{>};
    (4,-4)*{};(-4,4)*{} **\crv{(4,-1) & (-4,1)}?(1)*\dir{<};
    ?(0)*\dir{<};
    (4,4)*{};(12,12)*{} **\crv{(4,7) & (12,9)}?(1)*\dir{>};
    (12,4)*{};(4,12)*{} **\crv{(12,7) & (4,9)}?(1)*\dir{>};
    (-4,12)*{};(4,20)*{} **\crv{(-4,15) & (4,17)};
    (4,12)*{};(-4,20)*{} **\crv{(4,15) & (-4,17)}?(1)*\dir{>};
    (-4,4)*{}; (-4,12) **\dir{-};
    (12,-4)*{}; (12,4) **\dir{-};
    (12,12)*{}; (12,20) **\dir{-};
}};
  (12,0)*{\lambda};
  (-11.5,-11)*{\scs i-1};
  ( 4,-11)*{\scs i-1};
  ( 12,-11)*{\scs i-1};
\endxy}
=\;
   \vcenter{
 \xy 0;/r.17pc/:
(0,0)*{\doran\xybox{ 
   (4,-4)*{};(-4,4)*{} **\crv{(4,-1) & (-4,1)}?(1)*\dir{>};
    (-4,-4)*{};(4,4)*{} **\crv{(-4,-1) & (4,1)}?(0)*\dir{<};
    (-4,4)*{};(-12,12)*{} **\crv{(-4,7) & (-12,9)}?(1)*\dir{>};
    (-12,4)*{};(-4,12)*{} **\crv{(-12,7) & (-4,9)}?(1)*\dir{>};
    (4,12)*{};(-4,20)*{} **\crv{(4,15) & (-4,17)};
    (-4,12)*{};(4,20)*{} **\crv{(-4,15) & (4,17)}?(1)*\dir{>};
    (4,4)*{}; (4,12) **\dir{-} ?(.5)*\dir{<};
    (-12,-4)*{}; (-12,4) **\dir{-};
    (-12,12)*{}; (-12,20) **\dir{-};
}};
  (12.5,0)*{\lambda};
  (11.5,-11)*{\scs i-1};
  (-3.5,-11)*{\scs i-1};
  (-11.5,-11)*{\scs i-1};
\endxy}
$}
\end{equation*}
Then apply Relation~\eqref{eq_r3_easy-gen} to the two central triangles formed by strands colored $(i-1, i, i+1)$. 
Sliding the $i+1$-colored strands to the left using this relation creates two bigons between strands colored $i-1$ and $i+1$ respectively. 
These bigons can be removed using Relation~\eqref{eq_r2_ij-gen}. 

In this way, we have proved that $\Sigma_{n,r} \Bigl( \figins{-7}{0.3}{6mvert-slide-u.eps} \, \Bigr)$ is equivalent to
\begin{equation*}
 \xy 
 (-28,0)*{\dblue\xybox{(-8,24);(-38,-24); **\dir{-} ?(0)*\dir{<};}};
 (-37,-22)*{\cdots};
 (-9,22)*{\cdots};
 (-18,0)*{\dpink\xybox{(2,24);(-28,-24); **\dir{-} ?(0)*\dir{<};}};
 (-13,0)*{\dgreen\xybox{(7,24);(-23,-24); **\dir{-} ?(0)*\dir{<};}};
 (-22,-22)*{\cdots};
 (6,22)*{\cdots};
 (-3,0)*{\dgrey\xybox{(17,24);(-13,-24); **\dir{-} ?(0)*\dir{<};}};
 (37,0)*{\dbrun\xybox{(32,24);(2,-24); **\dir{-} ?(0)*\dir{<};}};
 (28,-22)*{\cdots};
 (56,22)*{\cdots};
 (47,0)*{\dred\xybox{(42,24);(12,-24); **\dir{-} ?(0)*\dir{<};}};
 %(12,-4)*{\dturq\xybox{(-3,-8)*{};(27,-8)*{} **\crv{(-2,0) & (27,0)}; ?(.5)*\dir{<};}};
 %(14.5,0)*{\dturq\xybox{(7,8)*{};(22,-8)*{} **\crv{(5,0) & (24,6)}; ?(1)*\dir{>};}};
 %(-8,3.5)*{\dyellow\xybox{(-23,8)*{};(17,8)*{} **\crv{(-23,0) & (17,0)}; ?(.5)*\dir{<};}};
 (-5.5,11)*{\dyellow\xybox{(-18,24)*{};(7,15)*{} **\crv{(-18,20) & (2,15)}; ?(0)*\dir{<};}};
 (-10,11)*{\dyellow\xybox{(-23,24)*{};(3.1,9)*{} **\crv{(-23,17) & (-4,9)}; ?(1)*\dir{>};}};
 (-14.4,11)*{\dturq\xybox{(-28,24)*{};(-0.8,3)*{} **\crv{(-28,14) & (-10,3)}; ?(0)*\dir{<};}};
 (-18.7,10)*{\dturq\xybox{(-33,24)*{};(-4.2,-3)*{} **\crv{(-33,11) & (-16,-3)}; ?(1)*\dir{>};}};
 (-23.1,7)*{\dyellow\xybox{(-38,24)*{};(-8.2,-9)*{} **\crv{(-38,8) & (-22,-9)}; ?(0)*\dir{<};}};
 (-27.6,4)*{\dyellow\xybox{(-43,24)*{};(-12,-15)*{} **\crv{(-43,5) & (-28,-15)}; ?(1)*\dir{>};}};
 (31.6,-12.5)*{\doran\xybox{(18,-24)*{};(8,-15)*{} **\crv{(18,-20) & (13,-15)}; ?(0)*\dir{<};}};
 (36.1,-12.5)*{\doran\xybox{(23,-24)*{};(11.9,-9)*{} **\crv{(23,-17) & (19,-9)}; ?(1)*\dir{>};}};
 (40.5,-12.5)*{\dturq\xybox{(28,-24)*{};(15.8,-3)*{} **\crv{(28,-14) & (25,-3)}; ?(0)*\dir{<};}};
 (44.8,-10.5)*{\dturq\xybox{(33,-24)*{};(19.2,3)*{} **\crv{(33,-11) & (31,3)}; ?(1)*\dir{>};}};
 (49.2,-7.5)*{\doran\xybox{(38,-24)*{};(23.2,9)*{} **\crv{(38,-8) & (37,9)}; ?(0)*\dir{<};}};
 (53.7,-4.5)*{\doran\xybox{(43,-24)*{};(27,15)*{} **\crv{(43,-5) & (43,15)}; ?(1)*\dir{>};}};
%merdier au milieu
 (0.5,-0.5)*{\dyellow\xybox{(-8.2,-9)*{};(3.1,9)*{} **\crv{(15,-9) & (10,9)}; ?(0.5)*\dir{<};}};
 (24.5,10.95)*{\dyellow\xybox{(42,24)*{};(7,15)*{} **\crv{(32,20) & (20,15)}; ?(0.5)*\dir{>};}};
 (-0.2,-13)*{\dyellow\xybox{(12,-24)*{};(-12,-15)*{} **\crv{(2,-20) & (-7,-15)}; ?(0.5)*\dir{<};}};
 (24,9.35)*{\doran\xybox{(2,24)*{};(23.2,9)*{} **\crv{(-20,10) & (10,-20) & (18,9)}; ?(0.7)*\dir{>};}};
 (11.1,-12.4)*{\doran\xybox{(-28,-24)*{};(11.9,-9)*{} **\crv{(-30,-15) & (0,5) & (6,-9)}; ?(0.2)*\dir{<};}};
 (35,0)*{\doran\xybox{(8,-15)*{};(27,15)*{} **\crv{(3,-15) & (11,15)}; ?(0.5)*\dir{<};}};
 (17,-0.9)*{\dturq\xybox{(-18,-24)*{};(12,24)*{} **\crv{(0,-3) & (3,3)}; ?(0.5)*\dir{<};}};
 (15.1,-9.25)*{\dturq\xybox{(-22.2,-3)*{};(16.25,-2.55)*{} **\crv{(-6,-3) & (0,-30) & (10,-3)}; ?(0.35)*\dir{>};}};
 (18.8,6.6)*{\dturq\xybox{(-19.7,2.45)*{};(19.2,3)*{} **\crv{(-5,3) & (15,25) & (7,3)}; ?(0.3)*\dir{<};}};
 %(-14.4,11)*{\dturq\xybox{(52,24)*{};(-0.8,3)*{} **\crv{(52,14) & (10,3)}; ?(0)*\dir{<};}};
 (-43,-26)*{\scs n}; (-33,-26)*{\scs r+1 }; (-28,-26)*{ \scs 1};(-18,-26)*{\scs i-2}; (-8,-26)*{\scs i-1}; (2,-26)*{\scs i}; (12,-26)*{\scs i+1}; (22,-26)*{\scs i+2}; (32,-26)*{\scs r};(37,-26)*{\scs i-1};(42,-26)*{\scs i-1}; (47,-26)*{\scs i};(52,-26)*{\scs i}; (57,-26)*{\scs i-1};(62,-26)*{\scs i-1}; (-43,26)*{\scs i+1};(-38,26)*{\scs i+1}; (-33,26)*{\scs i};(-28,26)*{\scs i}; (-23,26)*{\scs i+1};(-18,26)*{\scs i+1};
 (72,0)*{ (1^r)}; 
  \endxy
\vspace*{2ex}
\end{equation*}

Next, apply Relation~\eqref{eq_r3_hard-gen} to the central left part of the diagram. Locally we end up with a sum of two terms:
\begin{equation*}
\vcenter{
 \xy 0;/r.18pc/:
(0,0)*{\doran\xybox{
    (4,-4)*{};(-4,4)*{} **\crv{(4,-1) & (-4,1)}?(1)*\dir{<};
    (-4,4)*{};(-12,12)*{} **\crv{(-4,7) & (-12,9)}?(1)*\dir{<};
    (-12,4)*{};(-4,12)*{} **\crv{(-12,7) & (-4,9)}?(1)*\dir{<};
    (-4,12)*{};(4,20)*{} **\crv{(-4,15) & (4,17)}?(1)*\dir{<};
    (-12,-4)*{}; (-12,4) **\dir{-};
    (-12,12)*{}; (-12,20) **\dir{-};
}};
(4,0)*{\dturq\xybox{
    (-4,-4)*{};(4,4)*{} **\crv{(-4,-1) & (4,1)}?(1)*\dir{<};
    (4,12)*{};(-4,20)*{} **\crv{(4,15) & (-4,17)}?(1)*\dir{<};
    (4,4)*{}; (4,12) **\dir{-};
}};
  (12,0)*{\lambda};
  (11.5,-11)*{\scs i-1};
  (-2,-11)*{\scs i};
  (-11.5,-11)*{\scs i-1};
\endxy}
\quad = \quad
 \vcenter{
 \xy 0;/r.18pc/:
(0,0)*{\doran\xybox{
    (-4,-4)*{};(4,4)*{} **\crv{(-4,-1) & (4,1)}?(1)*\dir{<};
    (4,4)*{};(12,12)*{} **\crv{(4,7) & (12,9)}?(1)*\dir{<};
    (12,4)*{};(4,12)*{} **\crv{(12,7) & (4,9)}?(1)*\dir{<};
    (4,12)*{};(-4,20)*{} **\crv{(4,15) & (-4,17)}?(1)*\dir{<};
    (12,-4)*{}; (12,4) **\dir{-};
    (12,12)*{}; (12,20) **\dir{-};
}};
(-4,0)*{\dturq\xybox{
    (4,-4)*{};(-4,4)*{} **\crv{(4,-1) & (-4,1)}?(1)*\dir{<};
    (-4,12)*{};(4,20)*{} **\crv{(-4,15) & (4,17)}?(1)*\dir{<};
    (-4,4)*{}; (-4,12) **\dir{-};
}};
  (12,0)*{\lambda};
  (-11.5,-11)*{\scs i-1};
  (2.5,-11)*{\scs i};
  (11.5,-11)*{\scs i-1};
\endxy}
 \;\; -\;\;\;\; 
\xy 0;/r.18pc/:
(10,0)*{\doran\xybox{
  (4,12);(4,-12) **\dir{-}?(.5)*\dir{>};
  (22,12);(22,-12) **\dir{-}?(.5)*\dir{>}?(.25)*\dir{}+(0,0)*{}+(10,0)*{\scs};
}};
(3.5,0)*{\dturq\xybox{
  (-4,12);(-4,-12) **\dir{-}?(.5)*\dir{>}?(.25)*\dir{}+(0,0)*{}+(-3,0)*{\scs };
}};
  (20,0)*{\lambda}; 
  (-7.5,-11)*{\scs i-1};
  (2,-11)*{\scs i};
  (17.5,-11)*{\scs i-1};
 \endxy
\end{equation*}
The second term is killed because it contains an $i-1$-colored curl which is zero by Relation \eqref{eq_nil_rels}. Note that $\lambda = (1,\dots,1,2,0,0,1,\dots,1,0,\dots,0)$ where $2$ is in the $i$th position and the last $1$ is in the $r+1$st position, 
so $\bar{\lambda}_{i-1} = -1$. The first term contains a bigon between two $i-1$-colored strands with opposite orientations. Remove this bigon using 
Relation~\eqref{eq:EF}, which in this case equals 
\begin{eqnarray*}
 \vcenter{\xy 0;/r.18pc/:
  (0,0)*{\doran\xybox{
  (-8,0)*{};(8,0)*{};
  (-4,10)*{}="t1";
  (4,10)*{}="t2";
  (-4,-10)*{}="b1";
  (4,-10)*{}="b2";
  "t1";"b1" **\dir{-} ?(.5)*\dir{<};
  "t2";"b2" **\dir{-} ?(.5)*\dir{>};}};
  (-8,-8)*{\scs i-1};(8,-8)*{\scs i-1};
  (10,2)*{\lambda};
    \endxy}
&\quad = \quad&
   \vcenter{\xy 0;/r.18pc/:
    (0,0)*{\doran\xybox{
    (-4,-4)*{};(4,4)*{} **\crv{(-4,-1) & (4,1)}?(1)*\dir{>};?(0)*\dir{>};
    (4,-4)*{};(-4,4)*{} **\crv{(4,-1) & (-4,1)}?(1)*\dir{<};
    (-4,4)*{};(4,12)*{} **\crv{(-4,7) & (4,9)}?(1)*\dir{<};
    (4,4)*{};(-4,12)*{} **\crv{(4,7) & (-4,9)};}};
    (10,2)*{\lambda};(-7.8,-7)*{\scs i-1};(8,-7)*{\scs i-1};
 \endxy}.  
\end{eqnarray*}

Now apply Relation~\eqref{eq_r3_easy-gen} to the top central and bottom central parts of the diagram which we have obtained so far. Apply 
Relation~\eqref{eq_other_r3_1} to the top right and bottom right parts of the diagram. In this way we get two more bigons between strands 
colored $i-1$ and $i$ with opposite orientations, which we remove using Relation~\eqref{eq_downup_ij-gen}. 
This finishes our proof that $\Sigma_{n,r} \Bigl( \figins{-7}{0.3}{6mvert-slide-u.eps} \, \Bigr)$ is equivalent to the same diagram in~\eqref{diagfinal}.
\\

Let us now consider Relation~\eqref{eq:slide6mv} when one of the strands has color $r$ (Relation \eqref{eq:slide6mv2} can be dealt with in a similar way). Three 
cases have to be considered: when the bottom strands are colored $(r-2,r-1,r-2)$, $(r,1,r)$ or $(r-1,r,r-1)$. 

Using the same sort of reasoning 
as above, one can show that $\Sigma_{n,r} \Bigl( \figins{-7}{0.3}{6mvert-slide-d.eps}  \,  \Bigr)$ and 
$\Sigma_{n,r} \Bigl( \figins{-7}{0.3}{6mvert-slide-u.eps} \, \Bigr)$ both reduce to the same diagram. We omit the precise computations, but just 
give the diagram to which both images can be reduced. 

In the case of the colors $(r-2,r-1,r-2)$, one can reduce both images to the following diagram
\begin{equation*}
 \xy
(-75,12.5)*{\dblue\xybox{(-75,25);(-55,0) **\crv{(-75,15) & (-55,10)} ?(1)*\dir{>};}};
(-65,12.5)*{\dpink\xybox{(-65,25);(-45,0) **\crv{(-65,15) & (-45,10)} ?(1)*\dir{>};}};
(-30,23)*{\cdots};
(-55,23)*{\cdots};
(-80,23)*{\cdots};
(-45,2)*{\cdots};
(-60,2)*{\cdots};
(-42.5,12.5)*{\dblue\xybox{(-35,25);(-50,25) **\crv{(-35,20) & (-50,20)}; ?(1)*\dir{>};}};
(-42.5,12.5)*{\dpink\xybox{(-25,25);(-60,25) **\crv{(-25,15) & (-60,15)}; ?(1)*\dir{>};}};
(-42.5,12.5)*{\dred\xybox{(-20,25);(-65,25) **\crv{(-20,12.5) & (-65,12.5)}; ?(1)*\dir{>};}};
(-25,2.5)*{\dyellow\xybox{(0,0);(-20,0) **\crv{(0,5) & (-20,5)}; ?(1)*\dir{>};}};
(-27.5,12.5)*{\doran\xybox{(-65,25);(-30,0) **\crv{(-65,5) & (-30,10)} ?(1)*\dir{>};}};
(-22.5,12.5)*{\doran\xybox{(-65,25);(-30,0) **\crv{(-65,5) & (-30,10)} ?(0)*\dir{<};}};
(-20,27)*{\scs r}; (-25,27)*{\scs r+1}; (-35,27)*{\scs n}; (-40,27)*{\scs r-1};  (-45,27)*{\scs r-1}; (-50,27)*{\scs n};
(-60,27)*{\scs r+1};(-65,27)*{\scs r}; (-70,27)*{\scs r}; (-75,27)*{\scs r+1};  (-85,27)*{\scs n}; 
(10,23)*{\cdots};
(-15,12.5)*{\dgreen\xybox{(-20,25)*{};(-90,0)*{} **\crv{(-20,0) & (-90,10)}; ?(1)*\dir{>};}};
(-5,23)*{\cdots};
(-5,12.5)*{\dduck\xybox{(-10,25);(-80,0) **\crv{(-10,0) & (-80,10)}; ?(1)*\dir{>};}};
(17.5,12.5)*{\dyellow\xybox{(-5,25);(-40,0) **\crv{(-5,5) & (-40,10)}; ?(1)*\dir{>};}};
(5,12.5)*{\doran\xybox{(0,25);(-70,0) **\crv{(0,0) & (-70,10)}; ?(1)*\dir{>};}};
(25,23)*{\cdots};
(-47.5,12.5)*{\dred\xybox{(-60,25);(-15,0) **\crv{(-60,10) & (-15,10)}; ?(1)*\dir{>};}};
(2.5,12.5)*{\dblue\xybox{(-35,25);(-40,25) **\crv{(-35,20) & (-40,20)}; ?(1)*\dir{>};}};
(-7.5,2.5)*{\dyellow\xybox{(-15,0);(-40,0) **\crv{(-15,5) & (-40,5)}; ?(1)*\dir{>};}};
(2.5,12.5)*{\dpink\xybox{(-25,25);(-50,25) **\crv{(-25,15) & (-50,15)}; ?(1)*\dir{>};}};
(15,12.5)*{\dred\xybox{(5,25);(-55,25) **\crv{(5,10) & (-55,10)}; ?(1)*\dir{>};}};
(-15,27)*{\scs r};(-10,27)*{\scs r+1}; (0,27)*{\scs n}; (5,27)*{\scs n}; (15,27)*{\scs r+1};
(20,27)*{\scs 1}; (30,27)*{\scs r-3}; (35,27)*{\scs r-2}; (40,27)*{\scs r-1}; (45,27)*{\scs r};
(5,-2)*{\scs r-2}; (-15,-2)*{\scs r-2}; 
(30,5)*{ (1^r)};
\endxy.
\end{equation*}

In the case of colors $(r,1,r)$, we get 

\begin{equation*}
 \xy
(-125,2)*{\cdots};
(-92.5,12.5)*{\dturq\xybox{(15,0)*{};(50,25)*{} **\crv{(15,20) & (50,15)}; ?(0)*\dir{<};}};
(-72.5,12.5)*{\dgreen\xybox{(20,0)*{};(105,25)*{} **\crv{(20,20) & (105,10)}; ?(0)*\dir{<};}};
(-62.5,12.5)*{\dpurple\xybox{(20,0)*{};(105,25)*{} **\crv{(20,20) & (105,10)}; ?(0)*\dir{<};}};
(-100,2)*{\cdots};
(-52.5,12.5)*{\doran\xybox{(30,0);(115,25) **\crv{(30,20) & (115,10)}; ?(0)*\dir{<};}};
(-12.5,12.5)*{\dred\xybox{(35,0);(50,25) **\crv{(35,20) & (50,15)}; ?(0)*\dir{<};}};
(-80,2)*{\cdots};
(-55,2)*{\cdots};
(-77.5,2.5)*{\dred\xybox{(35,0);(60,0) **\crv{(35,5) & (60,5)}; ?(0)*\dir{<};}};
(-62.5,12.5)*{\dturq\xybox{(20,25);(35,25) **\crv{(20,20) & (35,20)}; ?(0)*\dir{<};}};
(-90,3.75)*{\dpink\xybox{(5,0);(65,0) **\crv{(5,7.5) & (65,7.5)}; ?(0)*\dir{<};}};
(-90,5)*{\dblue\xybox{(-5,0);(75,0) **\crv{(-5,10) & (75,10)}; ?(0)*\dir{<};}};
(-5,2)*{\cdots};
(-40,12.5)*{\dblue\xybox{(-20,0)*{};(-30,25)*{} **\crv{(-20,20) & (-30,15)}; ?(0)*\dir{<};}};
(-30,2)*{\cdots};
(-30,12.5)*{\dpink\xybox{(-10,0);(-20,25) **\crv{(-10,20) & (-20,15)}; ?(0)*\dir{<};}};
(-40,23)*{\cdots};
(-15,23)*{\cdots};
(-42.5,2.5)*{\dred\xybox{(-25,0);(-80,0) **\crv{(-25,5) & (-80,5)}; ?(1)*\dir{>};}};
(-37.5,12.5)*{\dturq\xybox{(-20,25);(-45,25) **\crv{(-20,20) & (-45,20)}; ?(1)*\dir{>};}};
(-42.5,3.75)*{\dpink\xybox{(-25,0);(-90,0) **\crv{(-25,7.5) & (-90,7.5)}; ?(1)*\dir{>};}};
(-42.5,5)*{\dblue\xybox{(-15,0);(-100,0) **\crv{(-15,10) & (-100,10)}; ?(1)*\dir{>};}};
(-55,12.5)*{\dgreen\xybox{(-20,0)*{};(-40,25)*{} **\crv{(-20,20) & (-40,15)}; ?(0)*\dir{<};}};
(-50,12.5)*{\dgreen\xybox{(-20,0)*{};(-40,25)*{} **\crv{(-20,20) & (-40,15)}; ?(1)*\dir{>};}};
(0,-2)*{\scs n}; (-10,-2)*{\scs r+1}; (-15,-2)*{\scs r}; (-20,-2)*{\scs r}; (-25,-2)*{\scs r+1}; (-35,-2)*{\scs n}; (-40,-2)*{\scs 1}; (-50,-2)*{\scs n}; (-45,-2)*{\scs 1}; (-60,-2)*{\scs r+1}; (-65,-2)*{\scs r}; (-70,-2)*{\scs r}; (-75,-2)*{\scs r+1}; (-85,-2)*{\scs n}; (-90,-2)*{\scs r}; (-95,-2)*{\scs r-1}; (-105,-2)*{\scs 3}; (-110,-2)*{\scs 2}; (-115,-2)*{\scs 1}; (-120,-2)*{\scs r+1}; (-130,-2)*{\scs n}; (-25,27)*{\scs 2};  (-55,27)*{\scs 2};
(-3,12)*{ (1^r)};
\endxy.
\end{equation*}

Finally, in the case of colors $(r-1,r,r-1)$, we get 

\begin{equation*}
  \xy
(0,23)*{\cdots};
(-65,12.5)*{\dblue\xybox{(-75,25);(-55,0) **\crv{(-75,15) & (-55,10)} ?(1)*\dir{>};}};
(-55,12.5)*{\dpink\xybox{(-65,25);(-45,0) **\crv{(-65,15) & (-45,10)} ?(1)*\dir{>};}};
(-50,12.5)*{\dred\xybox{(-65,25);(-45,0) **\crv{(-65,15) & (-45,10)} ?(1)*\dir{>};}};
(-45,12.5)*{\dred\xybox{(-65,25);(-45,0) **\crv{(-65,15) & (-45,10)} ?(0)*\dir{<};}};
(-40,12.5)*{\dpink\xybox{(-65,25);(-45,0) **\crv{(-65,15) & (-45,10)} ?(0)*\dir{<};}};
(-30,12.5)*{\dblue\xybox{(-65,25);(-45,0) **\crv{(-65,15) & (-45,10)} ?(0)*\dir{<};}};
(-67.5,12.5)*{\dblue\xybox{(-20,25)*{};(-105,0)*{} **\crv{(-20,15) & (-105,10)}; ?(1)*\dir{>};}};
(-57.5,12.5)*{\dpink\xybox{(-20,25)*{};(-105,0)*{} **\crv{(-20,15) & (-105,10)}; ?(1)*\dir{>};}};
(-20,23)*{\cdots};
(-47.5,12.5)*{\dturq\xybox{(-10,25);(-95,0) **\crv{(-10,15) & (-95,10)}; ?(1)*\dir{>};}};
(-37.5,12.5)*{\dyellow\xybox{(-10,25);(-95,0) **\crv{(-10,15) & (-95,10)}; ?(1)*\dir{>};}};
(-27.5,12.5)*{\dred\xybox{(-10,25);(-95,0) **\crv{(-10,15) & (-95,10)}; ?(1)*\dir{>};}};
(-2.5,12.5)*{\doran\xybox{(-10,25);(-35,0) **\crv{(-10,15) & (-35,10)}; ?(1)*\dir{>};}};
(-90,12.5)*{\dgreen\xybox{(-10,25);(-20,0) **\crv{(-10,15) & (-20,10)}; ?(1)*\dir{>};}};
(-45,23)*{\cdots};
(-70,23)*{\cdots};
(-105,2)*{\cdots};
(-85,2)*{\cdots};
(-25,2)*{\cdots};
(-50,2)*{\cdots};
(-20,12.5)*{\dgreen\xybox{(-30,25);(-50,25) **\crv{(-30,20) & (-50,20)}; ?(1)*\dir{>};}};
(-37.5,2.5)*{\doran\xybox{(-10,0);(-65,0) **\crv{(-10,5) & (-65,5)}; ?(1)*\dir{>};}};
(-67.5,2.5)*{\doran\xybox{(-30,0);(-45,0) **\crv{(-30,5) & (-45,5)}; ?(1)*\dir{>};}};
(-57.5,12.5)*{\dgreen\xybox{(-20,25);(-65,25) **\crv{(-20,20) & (-65,20)}; ?(1)*\dir{>};}};
(15,27)*{\scs r}; (10,27)*{\scs r-1}; (5,27)*{\scs r-2}; (-5,27)*{\scs 2}; (-10,27)*{\scs 1}; (-15,27)*{\scs r+1}; (-25,27)*{\scs n}; (-35,27)*{\scs 1}; (-30,27)*{\scs 1}; (-40,27)*{\scs n}; (-50,27)*{\scs r+1}; (-55,27)*{\scs r}; (-60,27)*{\scs r}; (-65,27)*{\scs r+1}; (-75,27)*{\scs n}; (-80,27)*{\scs 1}; (-85,27)*{\scs 1};
(-10,-2)*{\scs r-1}; (-60,-2)*{\scs r-1};
(3,9)*{ (1^r)};
\endxy.
\end{equation*}

\noindent $\bullet$ Box relations. Just as for Relations~\eqref{eq:dumbrot}--\eqref{eq:dumbdumbsquare}, 
some of the box relations with $i=r$ or $j=r$ follow from the same box relations for $i,j \neq r$ together with some other 
box relations. Taking into account the observations in Remark~\ref{uselessrels} too, we see that it suffices to prove 
Relations \eqref{eq:box12} and \eqref{eq:box32} here.

Let us start with Relation \eqref{eq:box32}. It is sufficient to prove 
this relation for $i = r-1$, because we can write 
$$\bbox{i+1}=\bbox{i+1}-\bbox{i+2}+\bbox{i+2}-\bbox{i+3}+\ldots+\bbox{r-1}-\bbox{r}+\bbox{r}$$
and
$$\bbox{i}=\bbox{i}-\bbox{i+1}+\bbox{i+1}-\bbox{i+2}+\ldots+\bbox{r-2}-
\bbox{r-1}+\bbox{r-1}$$
and use Relations~\eqref{eq:slidedotdist-md},~\eqref{eq:slidedotdist-mu} 
and~\eqref{eq:box1}. 

Let us prove Relation \eqref{eq:box32} for $i = r-1$, i.e.  
\begin{equation*}
 \;
\xy 0;/r.18pc/:
 (0,-1)*{\dred\ccbub{\black -1}{\black r}};
  %(8,4)*{\scs(1^r)};
 \endxy
 \;
 \xy 
 (-28,0)*{\dblue\xybox{(-28,8);(-28,-8); **\dir{-} ?(.5)*\dir{<};}};
 (-23,0)*{\cdots};
 (-18,0)*{\dpink\xybox{(-18,8);(-18,-8); **\dir{-} ?(.5)*\dir{<};}};
 (-13,0)*{\dgreen\xybox{(-13,8);(-13,-8); **\dir{-} ?(.5)*\dir{<};}};
 (-8,0)*{\cdots};
 (-3,0)*{\dred\xybox{(-3,8);(-3,-8); **\dir{-} ?(.5)*\dir{<};}};
 (-28,-10)*{\scs n}; (-18,-10)*{\scs r+1 }; (-13,-10)*{ \scs 1};(-3,-10)*{\scs r};
 %(5,0)*{ (1^r)}; 
  \endxy
\; =  \;
 \xy 
 (-28,0)*{\dblue\xybox{(-28,8);(-28,-8); **\dir{-} ?(.5)*\dir{<};}};
 (-23,0)*{\cdots};
 (-18,0)*{\dpink\xybox{(-18,8);(-18,-8); **\dir{-} ?(.5)*\dir{<};}};
 (-13,0)*{\dgreen\xybox{(-13,8);(-13,-8); **\dir{-} ?(.5)*\dir{<};}};
 (-8,0)*{\cdots};
 (-3,0)*{\dred\xybox{(-3,8);(-3,-8); **\dir{-} ?(.5)*\dir{<};}};
 (-28,-10)*{\scs n}; (-18,-10)*{\scs r+1 }; (-13,-10)*{ \scs 1};(-3,-10)*{\scs r};
 %(5,0)*{ (1^r)}; 
  \endxy
\quad - \;
\xy
(0,0)*{\dduck\xybox{%
    (3,0);(-3,0) **\crv{(3,4.2) & (-3,4.2)};
    ?(.05)*\dir{>} ?(1)*\dir{>}; 
    (3,0);(-3,0) **\crv{(3,-4.2) & (-3,-4.2)} ?(.3)*\dir{}+(4,0)*{\bscs r-1};
}};
%(6,3)*{\scs (1^r)},
\endxy
\; + \; 
\xy 0;/r.18pc/:
 (0,-1)*{\dred\ccbub{\black -1}{\black r}};
  %(8,4)*{\scs(1^r)};
 \endxy.
\end{equation*}

By Relations~\eqref{eq:slidedotdist-md} and \eqref{eq:slidedotdist-mu}, 
we have   
\begin{equation*}
\xy 
 (0,-1.75)*{\dred\xybox{(-3,0)*{};(3,0)*{} **\crv{(-3,-5) & (3,-5)}; ?(.5)*\dir{>};}};
 (0,-9.5)*{\dblue\xybox{(-15,0)*{};(15,0)*{} **\crv{(-15,-20) & (15,-20)}; ?(.5)*\dir{>};}};
 (-9,0)*{\cdots};
 (9,0)*{\cdots};
 (0,1.75)*{\dred\xybox{(-3,0)*{};(3,0)*{} **\crv{(-3,5) & (3,5)}; ?(.5)*\dir{<};}};
 (0,9.5)*{\dblue\xybox{(-15,0)*{};(15,0)*{} **\crv{(-15,20) & (15,20)}; ?(.5)*\dir{<};}};
 (17,-2)*{\scs n}; (4,-2)*{\scs r }; 
 %(17,10)*{ \scs (1^r)}; 
  \endxy
 \;
 \xy 
 (-28,0)*{\dblue\xybox{(-28,8);(-28,-8); **\dir{-} ?(.5)*\dir{<};}};
 (-23,0)*{\cdots};
 (-18,0)*{\dpink\xybox{(-18,8);(-18,-8); **\dir{-} ?(.5)*\dir{<};}};
 (-13,0)*{\dgreen\xybox{(-13,8);(-13,-8); **\dir{-} ?(.5)*\dir{<};}};
 (-8,0)*{\cdots};
 (-3,0)*{\dred\xybox{(-3,8);(-3,-8); **\dir{-} ?(.5)*\dir{<};}};
 (-28,-10)*{\scs n}; (-18,-10)*{\scs r+1 }; (-13,-10)*{ \scs 1};(-3,-10)*{\scs r};
 %(5,0)*{ (1^r)}; 
  \endxy
\; =  \;
 \xy 
 (-28,0)*{\dblue\xybox{(-28,8);(-28,-8); **\dir{-} ?(.5)*\dir{<};}};
 (-23,0)*{\cdots};
 (-18,0)*{\dpink\xybox{(-18,8);(-18,-8); **\dir{-} ?(.5)*\dir{<};}};
 (-13,0)*{\dgreen\xybox{(-13,8);(-13,-8); **\dir{-} ?(.5)*\dir{<};}};
 (-8,0)*{\cdots};
 (-3,0)*{\dred\xybox{(-3,8);(-3,-8); **\dir{-} ?(.5)*\dir{<};}};
 (-28,-10)*{\scs n}; (-18,-10)*{\scs r+1 }; (-13,-10)*{ \scs 1};(-3,-10)*{\scs r};
 %(5,0)*{ (1^r)}; 
  \endxy
\quad
\xy
(0,0)*{\dduck\xybox{%
    (3,0);(-3,0) **\crv{(3,4.2) & (-3,4.2)};
    ?(.05)*\dir{>} ?(1)*\dir{>}; 
    (3,0);(-3,0) **\crv{(3,-4.2) & (-3,-4.2)} ?(.3)*\dir{}+(4,0)*{\bscs r-1};
}};
%(6,3)*{\scs (1^r)},
\endxy.
\end{equation*}
Therefore, it suffices to prove 
\begin{equation*}
\xy 
 (0,-1.75)*{\dred\xybox{(-3,0)*{};(3,0)*{} **\crv{(-3,-5) & (3,-5)}; ?(.5)*\dir{>};}};
 (0,-9.5)*{\dblue\xybox{(-15,0)*{};(15,0)*{} **\crv{(-15,-20) & (15,-20)}; ?(.5)*\dir{>};}};
 (-9,0)*{\cdots};
 (9,0)*{\cdots};
 (0,1.75)*{\dred\xybox{(-3,0)*{};(3,0)*{} **\crv{(-3,5) & (3,5)}; ?(.5)*\dir{<};}};
 (0,9.5)*{\dblue\xybox{(-15,0)*{};(15,0)*{} **\crv{(-15,20) & (15,20)}; ?(.5)*\dir{<};}};
 (17,-2)*{\scs n}; (4,-2)*{\scs r }; 
 %(17,10)*{ \scs (1^r)}; 
  \endxy
 \; + \;
\xy 0;/r.18pc/:
 (0,-1)*{\dred\ccbub{\black -1}{\black r}};
  %(8,4)*{\scs(1^r)};
 \endxy
 \;
 \xy 
 (-28,0)*{\dblue\xybox{(-28,8);(-28,-8); **\dir{-} ?(.5)*\dir{<};}};
 (-23,0)*{\cdots};
 (-18,0)*{\dpink\xybox{(-18,8);(-18,-8); **\dir{-} ?(.5)*\dir{<};}};
 (-13,0)*{\dgreen\xybox{(-13,8);(-13,-8); **\dir{-} ?(.5)*\dir{<};}};
 (-8,0)*{\cdots};
 (-3,0)*{\dred\xybox{(-3,8);(-3,-8); **\dir{-} ?(.5)*\dir{<};}};
 (-28,-10)*{\scs n}; (-18,-10)*{\scs r+1 }; (-13,-10)*{ \scs 1};(-3,-10)*{\scs r};
 %(5,0)*{ (1^r)}; 
  \endxy
\; = \;
 \xy 
 (-28,0)*{\dblue\xybox{(-28,8);(-28,-8); **\dir{-} ?(.5)*\dir{<};}};
 (-23,0)*{\cdots};
 (-18,0)*{\dpink\xybox{(-18,8);(-18,-8); **\dir{-} ?(.5)*\dir{<};}};
 (-13,0)*{\dgreen\xybox{(-13,8);(-13,-8); **\dir{-} ?(.5)*\dir{<};}};
 (-8,0)*{\cdots};
 (-3,0)*{\dred\xybox{(-3,8);(-3,-8); **\dir{-} ?(.5)*\dir{<};}};
 (-28,-10)*{\scs n}; (-18,-10)*{\scs r+1 }; (-13,-10)*{ \scs 1};(-3,-10)*{\scs r};
 %(5,0)*{ (1^r)}; 
  \endxy
\quad
\xy 0;/r.18pc/:
 (0,-1)*{\dred\ccbub{\black -1}{\black r}};
  %(8,4)*{\scs(1^r)};
 \endxy.
\end{equation*}
On the one hand, we observe that 
\begin{equation*}
\xy 0;/r.18pc/:
 (0,-1)*{\dred\ccbub{\black -1}{\black r}};
  %(8,4)*{\scs(1^r)};
 \endxy
 \;
 \xy 
 (-28,0)*{\dblue\xybox{(-28,8);(-28,-8); **\dir{-} ?(.5)*\dir{<};}};
 (-23,0)*{\cdots};
 (-18,0)*{\dpink\xybox{(-18,8);(-18,-8); **\dir{-} ?(.5)*\dir{<};}};
 (-13,0)*{\dgreen\xybox{(-13,8);(-13,-8); **\dir{-} ?(.5)*\dir{<};}};
 (-8,0)*{\cdots};
 (-3,0)*{\dred\xybox{(-3,8);(-3,-8); **\dir{-} ?(.5)*\dir{<};}};
 (-28,-10)*{\scs n}; (-18,-10)*{\scs r+1 }; (-13,-10)*{ \scs 1};(-3,-10)*{\scs r};
 %(5,0)*{ (1^r)}; 
  \endxy
\; = \;
 \xy 
 (-28,0)*{\dblue\xybox{(-28,8);(-28,-8); **\dir{-} ?(.5)*\dir{<};}};
 (-23,0)*{\cdots};
 (-18,0)*{\dpink\xybox{(-18,8);(-18,-8); **\dir{-} ?(.5)*\dir{<};}};
 (-13,0)*{\dgreen\xybox{(-13,8);(-13,-8); **\dir{-} ?(.5)*\dir{<};}};
 (-8,0)*{\cdots};
 (-3,0)*{\dred\xybox{(-3,8);(-3,-8); **\dir{-} ?(.5)*\dir{<};}};
 (-28,-10)*{\scs n}; (-18,-10)*{\scs r+1 }; (-13,-10)*{ \scs 1};(-3,-10)*{\scs r};
 (-18,-4)*{\dpink\bullet}
 %(5,0)*{ (1^r)}; 
  \endxy
\; - \;
\xy 
 (-28,0)*{\dblue\xybox{(-28,8);(-28,-8); **\dir{-} ?(.5)*\dir{<};}};
 (-23,0)*{\cdots};
 (-18,0)*{\dpink\xybox{(-18,8);(-18,-8); **\dir{-} ?(.5)*\dir{<};}};
 (-8,0)*{\dred\xybox{%
    (3,0);(-3,0) **\crv{(3,4.2) & (-3,4.2)};
    ?(.05)*\dir{>} ?(1)*\dir{>}; 
    (3,0);(-3,0) **\crv{(3,-4.2) & (-3,-4.2)} ?(.3)*\dir{}+(2,0)*{\bscs r}; }};
 (2,0)*{\dgreen\xybox{(2,8);(2,-8); **\dir{-} ?(.5)*\dir{<};}};
 (7,0)*{\cdots};
 (12,0)*{\dred\xybox{(12,8);(12,-8); **\dir{-} ?(.5)*\dir{<};}};
 (-28,-10)*{\scs n}; (-18,-10)*{\scs r+1 }; (2,-10)*{ \scs 1};(12,-10)*{\scs r};
 %(5,0)*{ (1^r)}; 
  \endxy
\end{equation*}
since the bubble can be slid through the first $n-r-1$ left strands 
and then bubble slide Relation~\eqref{eq:2ndbubbslide} can be applied 
to the bubble and the strand colored $r+1$. 
On the other hand, we have 
\begin{equation*}
\xy 
 (0,-1.75)*{\dred\xybox{(-3,0)*{};(3,0)*{} **\crv{(-3,-5) & (3,-5)}; ?(.5)*\dir{>};}};
 (0,-9.5)*{\dblue\xybox{(-15,0)*{};(15,0)*{} **\crv{(-15,-20) & (15,-20)}; ?(.5)*\dir{>};}};
 (-9,0)*{\cdots};
 (9,0)*{\cdots};
 (0,1.75)*{\dred\xybox{(-3,0)*{};(3,0)*{} **\crv{(-3,5) & (3,5)}; ?(.5)*\dir{<};}};
 (0,9.5)*{\dblue\xybox{(-15,0)*{};(15,0)*{} **\crv{(-15,20) & (15,20)}; ?(.5)*\dir{<};}};
 (17,-2)*{\scs n}; (4,-2)*{\scs r }; 
 %(17,10)*{ \scs (1^r)}; 
  \endxy
 \;
 \xy 
 (-28,0)*{\dblue\xybox{(-28,8);(-28,-8); **\dir{-} ?(.5)*\dir{<};}};
 (-23,0)*{\cdots};
 (-18,0)*{\dpink\xybox{(-18,8);(-18,-8); **\dir{-} ?(.5)*\dir{<};}};
 (-13,0)*{\dgreen\xybox{(-13,8);(-13,-8); **\dir{-} ?(.5)*\dir{<};}};
 (-8,0)*{\cdots};
 (-3,0)*{\dred\xybox{(-3,8);(-3,-8); **\dir{-} ?(.5)*\dir{<};}};
 (-28,-10)*{\scs n}; (-18,-10)*{\scs r+1 }; (-13,-10)*{ \scs 1};(-3,-10)*{\scs r};
 %(5,0)*{ (1^r)}; 
  \endxy
\; =  \;
\xy 
 (-28,0)*{\dblue\xybox{(-28,8);(-28,-8); **\dir{-} ?(.5)*\dir{<};}};
 (-23,0)*{\cdots};
 (-18,0)*{\dpink\xybox{(-18,8);(-18,-8); **\dir{-} ?(.5)*\dir{<};}};
 (-8,0)*{\dred\xybox{%
    (3,0);(-3,0) **\crv{(3,4.2) & (-3,4.2)};
    ?(.05)*\dir{>} ?(1)*\dir{>}; 
    (3,0);(-3,0) **\crv{(3,-4.2) & (-3,-4.2)} ?(.3)*\dir{}+(2,0)*{\bscs r}; }};
 (2,0)*{\dgreen\xybox{(2,8);(2,-8); **\dir{-} ?(.5)*\dir{<};}};
 (7,0)*{\cdots};
 (12,0)*{\dred\xybox{(12,8);(12,-8); **\dir{-} ?(.5)*\dir{<};}};
 (-28,-10)*{\scs n}; (-18,-10)*{\scs r+1 }; (2,-10)*{ \scs 1};(12,-10)*{\scs r};
 %(5,0)*{ (1^r)}; 
  \endxy
\end{equation*}
which is obtained by using repeatedly 
Relation~\eqref{eq:EF}, which for the relevant labels $\lambda$ reduces to 
\begin{equation*}
\vcenter{\xy 0;/r.18pc/:
  (0,0)*{\dturq\xybox{
  (-8,0)*{};(8,0)*{};
  (-4,10)*{}="t1";
  (4,10)*{}="t2";
  (-4,-10)*{}="b1";
  (4,-10)*{}="b2";
  "t1";"b1" **\dir{-} ?(.5)*\dir{<};
  "t2";"b2" **\dir{-} ?(.5)*\dir{>};}};
  (-6,-8)*{\scs i};(6,-8)*{\scs i};
  (10,2)*{\lambda};
    \endxy}
\, = 
    {\dturq \xy
    (0,4)*{\bbpfe{\black i}};
    (0,-4.5)*{\bbcfe{\black i}};
    (8,0.5)*{ \black \lambda };
    \endxy}.
\end{equation*}
So it only remains to prove that
\begin{equation*}
 \xy 
 (-28,0)*{\dblue\xybox{(-28,8);(-28,-8); **\dir{-} ?(.5)*\dir{<};}};
 (-23,0)*{\cdots};
 (-18,0)*{\dpink\xybox{(-18,8);(-18,-8); **\dir{-} ?(.5)*\dir{<};}};
 (-13,0)*{\dgreen\xybox{(-13,8);(-13,-8); **\dir{-} ?(.5)*\dir{<};}};
 (-8,0)*{\cdots};
 (-3,0)*{\dred\xybox{(-3,8);(-3,-8); **\dir{-} ?(.5)*\dir{<};}};
 (-28,-10)*{\scs n}; (-18,-10)*{\scs r+1 }; (-13,-10)*{ \scs 1};(-3,-10)*{\scs r};
 (-18,-4)*{\dpink\bullet}
 %(5,0)*{ (1^r)}; 
  \endxy
\; = \;
\xy 
 (-28,0)*{\dblue\xybox{(-28,8);(-28,-8); **\dir{-} ?(.5)*\dir{<};}};
 (-23,0)*{\cdots};
 (-18,0)*{\dpink\xybox{(-18,8);(-18,-8); **\dir{-} ?(.5)*\dir{<};}};
 (-13,0)*{\dgreen\xybox{(-13,8);(-13,-8); **\dir{-} ?(.5)*\dir{<};}};
 (-8,0)*{\cdots};
 (-3,0)*{\dred\xybox{(-3,8);(-3,-8); **\dir{-} ?(.5)*\dir{<};}};
 (-28,-10)*{\scs n}; (-18,-10)*{\scs r+1 }; (-13,-10)*{ \scs 1};(-3,-10)*{\scs r};
 %(5,0)*{ (1^r)}; 
  \endxy
\quad
\xy 0;/r.18pc/:
 (0,-1)*{\dred\ccbub{\black -1}{\black r}};
  %(8,4)*{\scs(1^r)};
 \endxy.
\end{equation*}
Notice that 
\begin{equation*}
 \xy 
 (-28,0)*{\dblue\xybox{(-28,8);(-28,-8); **\dir{-} ?(.5)*\dir{<};}};
 (-23,0)*{\cdots};
 (-18,0)*{\dpink\xybox{(-18,8);(-18,-8); **\dir{-} ?(.5)*\dir{<};}};
 (-13,0)*{\dgreen\xybox{(-13,8);(-13,-8); **\dir{-} ?(.5)*\dir{<};}};
 (-8,0)*{\cdots};
 (-3,0)*{\dred\xybox{(-3,8);(-3,-8); **\dir{-} ?(.5)*\dir{<};}};
 (-28,-10)*{\scs n}; (-18,-10)*{\scs r+1 }; (-13,-10)*{ \scs 1};(-3,-10)*{\scs r};
 %(5,0)*{ (1^r)}; 
  \endxy
\quad
\xy 0;/r.18pc/:
 (0,-1)*{\dred\ccbub{\black -1}{\black r}};
  %(8,4)*{\scs(1^r)};
 \endxy
\; = \;
\xy 
 (-28,0)*{\dblue\xybox{(-28,8);(-28,-8); **\dir{-} ?(.5)*\dir{<};}};
 (-23,0)*{\cdots};
 (-18,0)*{\dpink\xybox{(-18,8);(-18,-8); **\dir{-} ?(.5)*\dir{<};}};
 (-13,0)*{\dgreen\xybox{(-13,8);(-13,-8); **\dir{-} ?(.5)*\dir{<};}};
 (-8,0)*{\cdots};
 (-3,0)*{\dred\xybox{(-3,8);(-3,-8); **\dir{-} ?(.5)*\dir{<};}};
 (-28,-10)*{\scs n}; (-18,-10)*{\scs r+1 }; (-13,-10)*{ \scs 1};(-3,-10)*{\scs r};
 %(5,0)*{ (1^r)}; 
  \endxy
\quad
\xy 0;/r.18pc/:
 (0,-1)*{\dred\cbub{\black 1}{\black r}};
  %(8,4)*{\scs(1^r)};
 \endxy,
\end{equation*}
by the infinite Grassmannian relation. We can 
apply Relation~\eqref{eq:EF} to the bubble and the strand 
colored $r$. The first term that appears is killed, 
because the weight inside the bigon has a negative entry. The bubble 
appearing in the second term satisfies
\begin{equation*}
 \xy
 (-12,0)*{\dred\cbub{\black -2}{\black r}};
 (-8,8)*{\lambda};
 \endxy
  = -1,
\end{equation*}
since $\lambda = (1, \dots,1 , 0 , 1 , 0 , \dots, 0)$ 
with the last one in $r+1$st position. 
Thus we get
\begin{equation*}
 \xy 
 (-28,0)*{\dblue\xybox{(-28,8);(-28,-8); **\dir{-} ?(.5)*\dir{<};}};
 (-23,0)*{\cdots};
 (-18,0)*{\dpink\xybox{(-18,8);(-18,-8); **\dir{-} ?(.5)*\dir{<};}};
 (-13,0)*{\dgreen\xybox{(-13,8);(-13,-8); **\dir{-} ?(.5)*\dir{<};}};
 (-8,0)*{\cdots};
 (-3,0)*{\dred\xybox{(-3,8);(-3,-8); **\dir{-} ?(.5)*\dir{<};}};
 (-28,-10)*{\scs n}; (-18,-10)*{\scs r+1 }; (-13,-10)*{ \scs 1};(-3,-10)*{\scs r};
 %(5,0)*{ (1^r)}; 
  \endxy
\quad
\xy 0;/r.18pc/:
 (0,-1)*{\dred\ccbub{\black -1}{\black r}};
  %(8,4)*{\scs(1^r)};
 \endxy
\; = \;
 \xy 
 (-28,0)*{\dblue\xybox{(-28,8);(-28,-8); **\dir{-} ?(.5)*\dir{<};}};
 (-23,0)*{\cdots};
 (-18,0)*{\dpink\xybox{(-18,8);(-18,-8); **\dir{-} ?(.5)*\dir{<};}};
 (-13,0)*{\dgreen\xybox{(-13,8);(-13,-8); **\dir{-} ?(.5)*\dir{<};}};
 (-8,0)*{\cdots};
 (-3,0)*{\dred\xybox{(-3,8);(-3,-8); **\dir{-} ?(.5)*\dir{<};}};
 (-28,-10)*{\scs n}; (-18,-10)*{\scs r+1 }; (-13,-10)*{ \scs 1};(-3,-10)*{\scs r};
 (-3,-4)*{\dred\bullet}
 %(5,0)*{ (1^r)}; 
  \endxy
\end{equation*}
Finally, we have to verify that
\begin{equation}
\label{eqexpy}
 \xy 
 (-28,0)*{\dblue\xybox{(-28,8);(-28,-8); **\dir{-} ?(.5)*\dir{<};}};
 (-23,0)*{\cdots};
 (-18,0)*{\dpink\xybox{(-18,8);(-18,-8); **\dir{-} ?(.5)*\dir{<};}};
 (-13,0)*{\dgreen\xybox{(-13,8);(-13,-8); **\dir{-} ?(.5)*\dir{<};}};
 (-8,0)*{\cdots};
 (-3,0)*{\dred\xybox{(-3,8);(-3,-8); **\dir{-} ?(.5)*\dir{<};}};
 (-28,-10)*{\scs n}; (-18,-10)*{\scs r+1 }; (-13,-10)*{ \scs 1};(-3,-10)*{\scs r};
 (-18,-4)*{\dpink\bullet}
 %(5,0)*{ (1^r)}; 
  \endxy
\; = \;
 \xy 
 (-28,0)*{\dblue\xybox{(-28,8);(-28,-8); **\dir{-} ?(.5)*\dir{<};}};
 (-23,0)*{\cdots};
 (-18,0)*{\dpink\xybox{(-18,8);(-18,-8); **\dir{-} ?(.5)*\dir{<};}};
 (-13,0)*{\dgreen\xybox{(-13,8);(-13,-8); **\dir{-} ?(.5)*\dir{<};}};
 (-8,0)*{\cdots};
 (-3,0)*{\dred\xybox{(-3,8);(-3,-8); **\dir{-} ?(.5)*\dir{<};}};
 (-28,-10)*{\scs n}; (-18,-10)*{\scs r+1 }; (-13,-10)*{ \scs 1};(-3,-10)*{\scs r};
 (-3,-4)*{\dred\bullet}
 %(5,0)*{ (1^r)}; 
  \endxy
\end{equation}
which is true: 
just slide the left dotted strand over all the strands colored 
$1, \dots , r-1$, using the first case of Relation~\eqref{eq_r2_ij-gen}, 
and then apply the second case of Relation~\eqref{eq_r2_ij-gen}. 
Observe that the term with the bigon is killed, 
because the weight inside the bigon has a negative entry. 
Note that this argument is not valid if $r = n-1$. Indeed in this case, $r+1$ and $1$ are adjacent colors, hence the left dotted strand cannot be simply slid over the strand colored $1$. In order to prove this remaining case together with Relation~\eqref{eq:box12}, let us remark that the right hand side of Equation~\eqref{eqexpy} can also be expressed as follows:
\begin{multline}
\label{eqexpy2}
\xy 
 (-28,0)*{\dblue\xybox{(-28,8);(-28,-8); **\dir{-} ?(.5)*\dir{<};}};
 (-23,0)*{\cdots};
 (-18,0)*{\dpink\xybox{(-18,8);(-18,-8); **\dir{-} ?(.5)*\dir{<};}};
 (-13,0)*{\dgreen\xybox{(-13,8);(-13,-8); **\dir{-} ?(.5)*\dir{<};}};
 (-8,0)*{\cdots};
 (-3,0)*{\dred\xybox{(-3,8);(-3,-8); **\dir{-} ?(.5)*\dir{<};}};
 (-28,-10)*{\scs n}; (-18,-10)*{\scs r+1 }; (-13,-10)*{ \scs 1};(-3,-10)*{\scs r};
 (-3,-4)*{\dred\bullet}
 %(5,0)*{ (1^r)}; 
  \endxy
\; = \;
 \xy 
 (-28,0)*{\dblue\xybox{(-28,8);(-28,-8); **\dir{-} ?(.5)*\dir{<};}};
 (-23,0)*{\cdots};
 (-18,0)*{\dpink\xybox{(-18,8);(-18,-8); **\dir{-} ?(.5)*\dir{<};}};
 (-13,0)*{\dgreen\xybox{(-13,8);(-13,-8); **\dir{-} ?(.5)*\dir{<};}};
 (-8,0)*{\cdots};
 (-3,0)*{\dred\xybox{(-3,8);(-3,-8); **\dir{-} ?(.5)*\dir{<};}};
 (-28,-10)*{\scs n}; (-18,-10)*{\scs r+1 }; (-13,-10)*{ \scs 1};(-3,-10)*{\scs r};
 (-18,-4)*{\dpink\bullet}
 %(5,0)*{ (1^r)}; 
  \endxy
\ +    \\
\left( 
\sum\limits_{j=1}^{r-1}\
\xy
(0,0)*{\dpurple\xybox{%
    (3,0);(-3,0) **\crv{(3,4.2) & (-3,4.2)};
    ?(.05)*\dir{>} ?(1)*\dir{>}; 
    (3,0);(-3,0) **\crv{(3,-4.2) & (-3,-4.2)} ?(.3)*\dir{}+(2,0)*{\bscs j};
}};
(6,3)*{\scs (1^r)},
\endxy- \xy 0;/r.18pc/:
 (0,-1)*{\dred\ccbub{\black -1}{\black r}};
  (8,4)*{\scs(1^r)};
 \endxy +
\xy 0;/r.18pc/:
 (0,-1)*{\dblue\ccbub{\black 1}{\black n}};
  (8,4)*{\scs(1^r)};
 \endxy + \xy (0,0)*{\bbox{y}} \endxy 
\right)
\xy 
 (-28,0)*{\dblue\xybox{(-28,8);(-28,-8); **\dir{-} ?(.5)*\dir{<};}};
 (-23,0)*{\cdots};
 (-18,0)*{\dpink\xybox{(-18,8);(-18,-8); **\dir{-} ?(.5)*\dir{<};}};
 (-13,0)*{\dgreen\xybox{(-13,8);(-13,-8); **\dir{-} ?(.5)*\dir{<};}};
 (-8,0)*{\cdots};
 (-3,0)*{\dred\xybox{(-3,8);(-3,-8); **\dir{-} ?(.5)*\dir{<};}};
 (-28,-10)*{\scs n}; (-18,-10)*{\scs r+1 }; (-13,-10)*{ \scs 1};(-3,-10)*{\scs r}
 %(5,0)*{ (1^r)}; 
  \endxy 
\end{multline}
This expression is obtained using repeatedly kink resolutions and bubble slides. Indeed the kink on the left hand side of Relation \eqref{eq:redtobubbles} for $i=r$ is equal to zero here since the label inside the kink possesses a negative entry. Hence one can express the dotted $r$ strand as a non-dotted strand times a bubble colored $r$ on the left. This bubble can then be slid through the $r-1$ strand using Relation \eqref{eq:extrabubblelast}, creating two terms: one is a dotted $r-1$ strand while the other is a non-dotted $r-1$ strand times a bubble colored $r$ on the left. This bubble can be slid all the way to the left using \eqref{eq:extrabubblelast} making appear an extra term which is the dotted $r+1$ strand on the right hand side of \eqref{eqexpy2}. One applies this same trick successively to the dotted strands for colors $r-1$ to $1$. The only exception is that, in the end, to slide the bubble colored $1$ through the strand colored $n$, we have to use the deformed Relation \eqref{eq:extrabubble461n}, which brings out the $y$ term in \eqref{eqexpy2}. Finally the dotted $n$ strand that thus appears can also be expressed as a non-dotted strand times a $n$-colored bubble on the left using Relation \eqref{eq:redtobubbles}.

Therefore, by~\eqref{eqexpy2}, it suffices to prove
\begin{equation}
\label{eqexpy3}
  \sum\limits_{j=1}^{r-1}\
\xy
(0,0)*{\dpurple\xybox{%
    (3,0);(-3,0) **\crv{(3,4.2) & (-3,4.2)};
    ?(.05)*\dir{>} ?(1)*\dir{>}; 
    (3,0);(-3,0) **\crv{(3,-4.2) & (-3,-4.2)} ?(.3)*\dir{}+(2,0)*{\bscs j};
}};
(6,3)*{\scs (1^r)},
\endxy- \xy 0;/r.18pc/:
 (0,-1)*{\dred\ccbub{\black -1}{\black r}};
  (8,4)*{\scs(1^r)};
 \endxy +
\xy 0;/r.18pc/:
 (0,-1)*{\dblue\ccbub{\black 1}{\black n}};
  (8,4)*{\scs(1^r)};
 \endxy +\xy (0,0)*{\bbox{y}} \endxy  = 0
\end{equation}
in order to show~\eqref{eqexpy}. In Section~\ref{sec:2rep} we 
define a $2$-representation of $\Scat(n,r)^*_{[y]}$. Its restriction 
to $\mbox{END}({\mathbf 1}_r)$ gives an algebra homomorphism 
$$\mathcal{F}'\colon \mbox{END}({\mathbf 1}_r)\to \Q[y,x_1,\ldots,x_{r}]$$
which on the elements of degree two is determined by 
\begin{eqnarray*}
\bbox{y}\;{\scs (1^r)}&\mapsto&y\\
\xy 
(-8,0)*{\dpurple\xybox{%
    (3,0);(-3,0) **\crv{(3,4.2) & (-3,4.2)};
    ?(.05)*\dir{>} ?(1)*\dir{>}; 
    (3,0);(-3,0) **\crv{(3,-4.2) & (-3,-4.2)} ?(.3)*\dir{}+(2,0)*{\bscs i}; 
}};
(-2,2)*{\scs (1^r)};
\endxy
&\mapsto&
\begin{cases}
x_{i+1}-x_{i}& 1\leq i\leq r-1\\
0& r+1\leq i\leq n-1
\end{cases}
\\
\xy 0;/r.18pc/:
 (0,-1)*{\red\ccbub{\black -1}{\black r}};
  (8,4)*{\scs(1^r)};
 \endxy
&\mapsto& x_r\\
\xy 0;/r.18pc/:
 (0,-1)*{\dblue\ccbub{\black 1}{\black n}};
  (8,4)*{\scs(1^r)};
 \endxy
&\mapsto& x_1-y.
\end{eqnarray*}
Note that $\mathcal{F}'$ maps the l.h.s. of~\eqref{eqexpy3} to zero. 
We are going to show that this implies~\eqref{eqexpy3} by showing 
that $\mathcal{F}'$ is an isomorphism. From the definition 
it is clear that $\mathcal{F}'$ is surjective. Injectivity follows if we 
can prove that $\mathrm{END}({\mathbf 1}_r)$ is generated by 
\begin{equation}
\label{eq:bubblegenerators}
\bbox{y}\;{\scs (1^r)}, 
\quad
\xy 
(-8,0)*{\dpurple\xybox{%
    (3,0);(-3,0) **\crv{(3,4.2) & (-3,4.2)};
    ?(.05)*\dir{>} ?(1)*\dir{>}; 
    (3,0);(-3,0) **\crv{(3,-4.2) & (-3,-4.2)} ?(.3)*\dir{}+(2,0)*{\bscs i}; 
}};
(-2,2)*{\scs (1^r)};
\endxy,
\quad 
\xy 0;/r.18pc/:
 (0,-1)*{\red\ccbub{\black -1}{\black r}};
  (8,4)*{\scs(1^r)};
 \endxy,
\quad
\xy 0;/r.18pc/:
 (0,-1)*{\dblue\ccbub{\black 1}{\black n}};
  (8,4)*{\scs(1^r)};
\endxy
\end{equation}
for $i=1,\ldots,r-1$, because that implies 
that $\mathrm{END}({\mathbf 1}_r)$ is isomorphic to a 
quotient of $\Q[y,x_1,\ldots,x_r]$ by the surjectivity of $\mathcal{F}'$, 
which means that the two algebras have to be isomorphic.  

In order to prove that $\mathrm{END}({\mathbf 1}_r)$ is indeed generated by 
the $2$-morphisms in~\eqref{eq:bubblegenerators}, first note that 
$\mathrm{END}({\mathbf 1}_r)$ is generated by all counter-clockwise 
bubbles of arbitrary degree and $\bbox{y}$. It therefore suffices 
to prove that any counter-clockwise bubble is in the span 
of the 2-morphisms in~\eqref{eq:bubblegenerators}.

We first prove this fact for counter-clockwise $i$-bubbles with 
$1\leq i\leq r-1$. The result follows from the recursive formula 
\begin{equation}
\label{eq:expcci}
\xy 0;/r.18pc/:
 (0,-1)*{\dpurple\ccbub{\black t+1}{\black i}};
  (8,4)*{\scs(1^r)};
 \endxy
\;=\;-\;
\xy 0;/r.18pc/:
 (0,-1)*{\dpurple\ccbub{\black t}{\black i}};
  (8,4)*{\scs(1^r)};
 \endxy
\left(\xy 
(-8,0)*{\dturq\xybox{%
    (3,0);(-3,0) **\crv{(3,4.2) & (-3,4.2)};
    ?(.05)*\dir{>} ?(1)*\dir{>}; 
    (3,0);(-3,0) **\crv{(3,-4.2) & (-3,-4.2)} ?(.3)*\dir{}+(4,0)*{\bscs i+1}; 
}};
(-2,2)*{\scs (1^r)};
\endxy+\cdots+\xy 
(-8,0)*{\doran\xybox{%
    (3,0);(-3,0) **\crv{(3,4.2) & (-3,4.2)};
    ?(.05)*\dir{>} ?(1)*\dir{>}; 
    (3,0);(-3,0) **\crv{(3,-4.2) & (-3,-4.2)} ?(.3)*\dir{}+(4,0)*{\bscs r-1}; 
}};
(-2,2)*{\scs (1^r)};
\endxy-
\xy 0;/r.18pc/:
 (0,-1)*{\red\ccbub{\black -1}{\black r}};
  (8,4)*{\scs(1^r)};
 \endxy
\right)
\end{equation}
for $t\geq 0$. The equation 
in~\eqref{eq:expcci} can be obtained by unnesting the l.h.s. of 
\begin{equation}
\label{eq:expcci2}
 \xy 
 (0,-1.75)*{\dred\xybox{(-3,0)*{};(3,0)*{} **\crv{(-3,-5) & (3,-5)}; ?(.5)*\dir{>};}};
 (0,-9.5)*{\dpurple\xybox{(-15,0)*{};(15,0)*{} **\crv{(-15,-20) & (15,-20)}; ?(.5)*\dir{>};}};
 (-9,0)*{\cdots};
 (9,0)*{\cdots};
 (0,1.75)*{\dred\xybox{(-3,0)*{};(3,0)*{} **\crv{(-3,5) & (3,5)}; ?(.5)*\dir{<};}};
 (0,9.5)*{\dpurple\xybox{(-15,0)*{};(15,0)*{} **\crv{(-15,20) & (15,20)}; ?(.5)*\dir{<};}};
 (17,-2)*{\scs i}; (4,-2)*{\scs r }; 
 (17,10)*{ \scs (1^r)}; 
 (12,-9)*{\dpurple \bullet}; 
 (13,-11)*{\scs t};
\endxy
\;=\;0
\end{equation}
using bubble slides from the inside to the outside. 
Equation~\eqref{eq:expcci2} holds, because the region in the 
center has label $\lambda$ with $\lambda_{r+1}=-1$. 

For $i=r$ the argument is simpler, because 
$$
\xy 0;/r.18pc/:
 (0,-1)*{\red\ccbub{\black -2+t}{\black r}};
  (8,4)*{\scs(1^r)};
 \endxy
\;=\;0
$$ 
for any $t\geq 2$. This holds because the inner region has label $\lambda$ 
with $\lambda_{r+1}=-1$. 

Similarly, for $i=n$ we have 
$$
\xy 0;/r.18pc/:
 (0,-1)*{\dblue\cbub{\black -2+t}{\black n}};
  (8,4)*{\scs(1^r)};
 \endxy
\;=\;0
$$ 
for any $t\geq 2$. By the infinite Grassmannian relation, this 
implies that 
$$
\xy 0;/r.18pc/:
 (0,-1)*{\dblue\ccbub{\black t}{\black n}};
  (8,4)*{\scs(1^r)};
 \endxy
\;=\;
\left(
\xy 0;/r.18pc/:
 (0,-1)*{\dblue\ccbub{\black 1}{\black n}};
  (8,4)*{\scs(1^r)};
 \endxy
\right)^t
$$ 
for any $t\geq 2$. 

For $r+1\leq i\leq n-1$ there is nothing to prove. 
In that case, the 
counter-clockwise $i$-bubbles of positive degree are all zero, 
because their interior is labeled by $\lambda$ with $\lambda_{i+1}=-1$. 
This finishes the proof that  
$$\mbox{END}({\mathbf 1}_r)\cong \Q[y,x_1,\ldots,x_r],$$
which implies~\eqref{eqexpy3}.

Now let us consider Relation \eqref{eq:box12}. The image under $\Sigma_{n,r}$ of 
$\figins{-6}{0.25}{startenddot.eps}_{\,r}$ is 
\begin{equation*}
 \xy 
 (0,-1.75)*{\dred\xybox{(-3,0)*{};(3,0)*{} **\crv{(-3,-5) & (3,-5)}; ?(.5)*\dir{>};}};
 (0,-9.5)*{\dblue\xybox{(-15,0)*{};(15,0)*{} **\crv{(-15,-20) & (15,-20)}; ?(.5)*\dir{>};}};
 (-9,0)*{\cdots};
 (9,0)*{\cdots};
 (0,1.75)*{\dred\xybox{(-3,0)*{};(3,0)*{} **\crv{(-3,5) & (3,5)}; ?(.5)*\dir{<};}};
 (0,9.5)*{\dblue\xybox{(-15,0)*{};(15,0)*{} **\crv{(-15,20) & (15,20)}; ?(.5)*\dir{<};}};
 (17,-2)*{\scs n}; (4,-2)*{\scs r }; 
 (17,10)*{ \scs (1^r)}; 
  \endxy
= - \xy 0;/r.18pc/:
 (0,-1)*{\dred\ccbub{\black -1}{\black r}};
  (8,4)*{\scs(1^r)};
 \endxy +
\xy 0;/r.18pc/:
 (0,-1)*{\dblue\ccbub{\black 1}{\black n}};
  (8,4)*{\scs(1^r)};
 \endxy 
\end{equation*}
This expression can be obtained using repeatedly bubble slide Relation \eqref{eq:extrabubble3}. Observe that, at each step but the first one, only one term survives, the second being systematically zero since it includes a real bubble whose inside is a label with a negative entry. 

If one replaces, in this expression, the bubble colored $n$ using Equation \eqref{eqexpy3}, one recognizes precisely the image under $\Sigma_{n,r}$ of $  \bbox{1} -\bbox{r} - \bbox{y}$.

We have checked that $\Sigma_{n,r}$ preserves all the relations of $\debim_{\hat{A}_{r-1}}^*$, 
so this ends the proof.

\end{proof}

\section{A 2-representation of $\Scat(n,r)^*_{[y]}$}\label{2rep}
\label{sec:2rep}

In this section we define a $2$--category $\esbim_{\hat{A}_{r-1}}$ and a $2$-functor
$$\mathcal{F}': \Scat{(n,r)}^{*}_{[y]}\to {\esbim^{*}_{\hat{A}_{r-1}}}.$$

The $2$-category $\esbim_{\hat{A}_{r-1}}$ is an extension of the 
category of singular Soergel bimodules in affine type $A$ 
considered by Williamson in~\cite{Wi} (see also \cite{MSVcolhom, MSVschur}). 
The $2$--functor $\mathcal{F}'$ is a generalization of Khovanov and Lauda's 
2-representation $\Gamma^G_r$ defined in~\cite{KL3, KLer}.

\subsection{Extended singular bimodules}\label{extbim}

Let $R=\Q[y][x_1,\ldots,x_r]$. As in 
Section~\ref{bim}, there is a grading on $R$ defined by $\deg y=\deg x_i=2$, 
for $i=1,\ldots, r$, and a degree preserving action of 
$\hat{\mathcal{W}}_{\hat{A}_{r-1}}$ on $R$. 
For any partition $(i_1,\ldots,i_k)$ of $r$, let 
$$S_{i_1}\times\cdots\times S_{i_k}\subseteq W_{A_{r-1}}\subset 
\hat{\mathcal{W}}_{\hat{A}_{r-1}}$$
be the parabolic subgroup which is contained in the finite Weyl group.  
let $R_{i_1\cdots i_k}\subseteq R$ denote the subring of $S_{i_1}\times\cdots\times S_{i_k}$--invariant polynomials.   

Using these rings of partially symmetric polynomials, 
we can construct bimodules by induction and restriction. Induction 
is defined as follows: 
suppose $i_j$ splits into $i_j^0$ plus $i_j^1$, 
then we define the {\em induction functor} by 
$$\mbox{Ind}_{i_j}^{i_j^0, i_j^1}(R_{i_1\cdots i_k}):=
R_{i_1 \dots i_j^0 i_j^1 \cdots i_k} \otimes_{i_1\cdots i_k} R_{i_1\cdots i_k}.$$
The subscript of the tensor product means that it is taken over 
$R_{i_1\cdots i_k}$. The {\em restriction functor} is defined by  
$$\mbox{Res}_{i_j, i_{j+1}}^{i_j + i_{j+1}}(R_{i_1\cdots i_k})
:=_{R_{i_1 \dots i_j+i_{j+1}\cdots i_k}|}R_{i_1\cdots i_k}.$$

\begin{defn}
Let $\esbim_{\hat{A}_{r-1}}$ be the $2$--category whose objects are the rings 
$R_{i_1\cdots i_k}$, for all partitions $(i_1,\ldots,i_k)$ of $r$. 

For any two partitions~$(i_1,\ldots,i_k)$ and~$(j_1,\ldots,j_l)$ of~$r$, 
the $1$--morphisms between $R_{i_1\cdots i_k}$ and $R_{j_1\cdots j_l}$ are by 
definition the direct sums and shifts of tensor products of 
$R_{i_1\cdots i_k}$--$R_{j_1\cdots j_l}$--bimodules obtained by induction and 
restriction and by tensoring with the 
twisted bimodules~$R_{i_ki_1\cdots i_{k-1},\rho}$ and~$R_{i_2\cdots i_ki_1,\rho^{-1}}$, 
which we will define below. 

The $2$--morphisms are the degree-preserving bimodule maps.
\end{defn}

The twisted bimodules are defined like the bimodules 
$B_{\rho^{\pm 1}}$ of section~\ref{bim}: 
$R_{i_ki_1\cdots i_{k-1},\rho}$ (resp.~$R_{i_2\cdots i_ki_1,\rho^{-1}}$) is equal to~$R_{i_ki_1\cdots i_{k-1}}$ (resp.~$R_{i_2\cdots i_ki_1}$) as a left~$R_{i_ki_1\cdots i_{k-1}}$--module 
(resp. as a left~$R_{i_2\cdots i_ki_1}$--module) whereas the action on the right is twisted. The right action of any~$a \in R_{i_1\cdots i_k }$ on~$R_{i_ki_1\cdots i_{k-1},\rho}$ 
and~$R_{i_2\cdots i_ki_1,\rho^{-1}}$ is given by multiplication by~$\rho^{i_k}(a)$ and~$\rho^{-i_1}(a)$ respectively. 
Recall that the action of $\rho^{\pm 1}$ was defined in section~\ref{action}.
The functors defined by tensoring with twisted bimodules are denote by 
$$R_{\rho^{i_k}}(R_{i_1\cdots i_k}):=R_{i_ki_1\cdots i_{k-1},\rho}
\otimes_{i_1\cdots i_k}R_{i_1\cdots i_k}
$$
and
$$
R_{\rho^{-i_1}}(R_{i_1\cdots i_k}):=R_{i_2\cdots i_ki_1,\rho^{-1}}\otimes_{i_1\cdots i_k}
R_{i_1\cdots i_k}
$$
respectively. 

\begin{lem}
The map $\rho^{i_k}$ gives an isomorphism between $R_{i_1\cdots i_k}$ and $R_{i_k i_1\cdots i_{k-1}}$, while $\rho^{-i_1}$ gives an isomorphism between 
$R_{i_1\cdots i_k}$ and $ R_{ i_2\cdots  i_{k}i_1}$. 

This implies that the twisted bimodules~$R_{i_k i_1\cdots i_{k-1},\rho}$ and $R_{ i_2 \cdots i_{k}i_1,\rho^{-1}}$ are well-defined.
\end{lem}

\begin{proof}
We only prove the lemma for $\rho^{i_k}$. The proof for $\rho^{-i_1}$ is similar and is left to the reader. 

It is clear that $\rho^{i_k}$ is a bijection. What remains to be shown is that the image of $R_{i_1\cdots i_k}$ is indeed $R_{i_k i_1\cdots i_{k-1}}$. 

The ring $R_{i_1\cdots i_k}$ is generated by the following elementary symmetric polynomials:
\begin{multline*}
 \epsilon_{p_j}(x_{i_1 + \dots + i_{j-1} + 1}, \dots, x_{i_1 + \dots + i_{j-1} + i_{j}}) = \\
\underset{\substack{q_1 < \dots < q_{p_j}\\ q_l \in \{ 1,\dots, i_j \} } }{\sum}x_{i_1 + \dots + i_{j-1} + q_1} \dots x_{i_1 + \dots + i_{j-1} + q_{p_j}}
\end{multline*}
for all $j = 1, \dots, k$ and $p_j =1, \dots, i_j$.

Similarly, the ring $R_{i_k i_1\cdots i_{k-1}}$ is generated by the elementary symmetric polynomials 
\begin{multline*}
 \epsilon_{p_j}(x_{i_k + i_1 + \dots + i_{j-1} + 1}, \dots, x_{i_k + i_1 + \dots + i_{j-1} + i_{j}}) = \\
 \underset{\substack{q_1 < \dots < q_{p_j}\\ q_l \in \{ 1,\dots, i_j \} } }{\sum}x_{i_1 + \dots + i_{j-1} + q_1 + i_k} \dots x_{i_1 + \dots + i_{j-1} + q_{p_j} + i_k}
\end{multline*}
and 
$$ \epsilon_{p_{k}}(x_1, \dots, x_{i_k}) = \underset{\substack{q_1 < \dots < q_{p_k}\\q_l \in \{ 1,\dots, i_k \}}}{\sum}x_{q_1} \dots x_{q_{p_k}}$$
for all $j = 1, \dots, k-1$, $p_j =1, \dots, i_{j}$ and $p_k = 1, \dots, i_k$.

Note that 
$$
\rho^{i_k}(x_q)=\begin{cases}
x_{q+i_k}&\text{for}\;q = 1, \dots, i_1 + \dots + i_{k-1}\\
x_{q - i_1 - \dots - i_{k-1} } -y&\text{for}\;q = i_1 + \dots + i_{k-1} +1, \dots, i_1 + \dots + i_{k}
\end{cases}.
$$

Let us look at the image under $\rho^{i_k}$ of the elementary symmetric polynomials generating $R_{i_1\cdots i_k}$. For $j = 1, \dots, k-1$, we have  
\begin{eqnarray*}
\rho^{i_k}(\epsilon_{p_j}(x_{i_1 + \dots + i_{j-1} + 1}, \dots, x_{i_1 + \dots + i_{j-1} + i_{j}}))&=\\ 
 \epsilon_{p_j}(x_{i_1 + \dots + i_{j-1} + 1 + i_k}, \dots, x_{i_1 + \dots + i_{j-1} + i_{j} + i_k}),&
\end{eqnarray*}
which indeed belongs to $R_{i_k i_1\cdots i_{k-1}}$. 

For $j=k$ and $p_k =1, \dots, i_k$, we get 
\begin{eqnarray*}
\rho^{i_k}(\epsilon_{p_k}(x_{i_1 + \dots + i_{k-1} + 1}, \dots, x_{i_1 + \dots +  i_{k}}))&=\\ 
\underset{\substack{q_1 < \dots < q_{p_k}\\q_l \in \{ 1,\dots, i_k \}}}{\sum}(x_{q_1}-y) \dots (x_{q_{p_k}}-y).&
\end{eqnarray*}
This product is equal to
\begin{multline*} \underset{\substack{q_1 < \dots < q_{p_k}\\q_l \in \{ 1,\dots, i_k \}}}{\sum}x_{q_1} \dots x_{q_{p_k}} -y \underset{\substack{q_1 < \dots < q_{p_k-1}\\q_l \in \{ 1,\dots, i_k \}}}{\sum}x_{q_1} \dots x_{q_{p_k -1}}  \\
+y^2 \underset{\substack{q_1 < \dots < q_{p_k-2}\\q_l \in \{ 1,\dots, i_k \}}}{\sum}x_{q_1} \dots x_{q_{p_k-2}} + \dots + (-1)^{p_k -1}y^{p_k-1} (x_1 + \dots + x_{i_k})+(-1)^{p_k}y^{p_k} = \\
\epsilon_{p_k}(x_1, \dots, x_{i_k})-y\epsilon_{p_k-1}(x_1, \dots, x_{i_k}) +y^2\epsilon_{p_k-2}(x_1, \dots, x_{i_k}) \\
+ \dots+ (-1)^{p_k -1}y^{p_k-1} \epsilon_{p_k-2}(x_1, \dots, x_{i_k})+(-1)^{p_k}y^{p_k},
\end{multline*}
which again belongs to $R_{i_k i_1\cdots i_{k-1}}$. 

This shows that $\rho^{i_k}$ sends the ring $R_{i_1\cdots i_k}$ isomorphically to the ring $R_{i_k i_1\cdots i_{k-1}}$.
\end{proof}

The proof of the following lemma is straightforward and is left to the reader. 
\begin{lem}\label{comsingbim}
We have the following isomorphisms of bimodules 
relating twisting, induction and restriction:

$$R_{\rho^{-i_k}}R_{\rho^{i_k}}(R_{i_1 \cdots i_{k}})
\cong R_{\rho^{i_1}}R_{\rho^{-i_1}}(R_{i_1 \cdots i_{k}})\cong R_{i_1 \cdots i_{k}} ,$$

$$\mbox{Ind}_{i_j}^{i_j^0 , i_j^1} R_{\rho^{i_k}}(R_{i_1 \cdots i_{k}}) \cong 
R_{\rho^{i_k}}\mbox{Ind}_{i_j}^{i_j^0 , i_j^1}(R_{i_1 \cdots i_{k}})\quad
\text{for}\; j \neq k,$$

$$
\mbox{Ind}_{i_k}^{i_k^0 , i_k^1}R_{\rho^{i_k}}(R_{i_1 \cdots i_{k}})
 \cong R_{\rho^{i_k^0}}R_{\rho^{i_k^1}}\mbox{Ind}_{i_k}^{i_k^0 , i_k^1}(R_{i_1 \cdots i_{k}}),
$$

$$\mbox{Res}_{i_j, i_{j+1}}^{i_j + i_{j+1}}R_{\rho^{i_k}} (R_{i_1 \cdots i_{k}})
\cong R_{\rho^{i_k}}\mbox{Res}_{i_j, i_{j+1}}^{i_j + i_{j+1}}(R_{i_1 \cdots i_{k}})
\quad\text{for}\; j \neq k-1,$$

$$
\mbox{Res}_{i_{k-1}, i_{k}}^{i_{k-1} + i_{k}}R_{\rho^{i_{k-1}}}R_{\rho^{i_k}} (R_{i_1 \cdots i_{k}})
\cong R_{\rho^{i_{k-1}+i_k}}\mbox{Res}_{i_{k-1}, i_{k}}^{i_{k-1} + i_{k}}(R_{i_1 \cdots i_{k}}).
$$

There exist analogous isomorphisms for the negative twists. 
\end{lem}

\begin{lem}
 The category $\ebim_{\hat{A}_{r-1}}$ is a full subcategory of $\esbim_{\hat{A}_{r-1}}$.
\end{lem}
\begin{proof}
For $i=1,\ldots, r-1$, the full embedding of $\ebim_{\hat{A}_{r-1}}$ into $\esbim_{\hat{A}_{r-1}}$ sends 
$B_i$ to $B_i$ and $B_{\rho^{\pm 1}}$ to $B_{\rho^{\pm 1}} = R_{(1^r),\rho^{\pm 1}}$. For $i=r$, the bimodule $B_r$ is sent to 
$$B_{\rho} \ot_R B_{r-1} \ot_R B_{\rho^{-1}}\in\esbim_{\hat{A}_{r-1}}.$$ 
The fact that the isomorphism in $\ebim_{\hat{A}_{r-1}}$ between $B_r$ and 
$B_{\rho} \ot_R B_{r-1} \ot_R B_{\rho^{-1}} $ is unique 
ensures that $\ebim_{\hat{A}_{r-1}}$ is a full subcategory of $\esbim_{\hat{A}_{r-1}}$.
\end{proof}

\subsection{The $2$-representation}~\label{functSchSBim}

We will mostly refer the reader to~\cite{KL3, KLer, MSVschur} for 
the definition of 
$$\mathcal{F}' : \Scat(n,r)^{*}_{[y]} \to \esbim_{\hat{A}_{r-1}}^{*},$$  
since $\mathcal{F}'$ is a straigthforward generalization of the 
equivariant 
Khovanov-Lauda $2$--representations discussed in those papers. 

\begin{rem}
\label{rem:relationLuGV}
Khovanov and Lauda used the equivariant cohomology rings of 
the varieties of partial flags in $\C^r$ for the definition of their 
equivariant $2$--representations. These cohomology rings are 
isomorphic to the finite type $A$ singular Soergel bimodules 
which were used in~\cite{MSVschur}. 

We do not know if the $2$--representation in this paper, which 
we define using the extended affine singular Soergel bimodules, can be 
defined in terms of equivariant cohomology rings of the varieties of 
cyclic partial flags (or periodic lattices) in 
$\C[\epsilon,\epsilon^{-1}]^r$ defined in~\cite{LuAff} and~\cite{GV}.      
\end{rem}

\subsubsection{Definition of $\mathcal{F}'$}
Note that for $y=0$ the restriction of $\F'\circ \Psi_{n,r}$ to 
$\mathcal{U}(\mathfrak{sl}_n)$ is simply 
equal to $\Gamma^G_r$, where 
$$\Psi_{n,r}\colon \mathcal{U}(\hat{\mathfrak{sl}_n})^*_{[y]}\to 
\hat{\mathcal{S}}(n,r)^*_{[y]}$$ was defined 
just before Proposition~\ref{prop:affslnquotient}.  

We will define the $2$--functor $\F'$ on all objects and $1$--morphisms of $\Scat(n,r)^{*}_{[y]}$, give explicitly the images of 
the $2$--morphisms for which the color $n$ appears and explain how they are related to Khovanov and Lauda's $2$--representation $\Gamma^G_r$. Here we are using their notation $k_i=\lambda_1+\dots+\lambda_i$, for $i=1,\dots,n$, with the convention that $k_0 = 0$.

\noindent $\bullet$. On objects $\lambda\in\Lambda(n,r)$, the 2-functor $\F'$ is given by:
$$\lambda=(\lambda_1,\cdots,\lambda_n) \mapsto R_{\lambda_1\cdots \lambda_n}.$$

\noindent $\bullet$. On $1$-morphisms we define $\F'$ as follows:
$$\onel \{t\}\mapsto R_{\lambda_1\cdots \lambda_n}\{t\}.$$

For $i=1,\ldots,n-1$ and $t \in \Z$ we define:

$$\cal{E}_{i}\onel\{t\}\mapsto \mbox{Res}_{\lambda_i,1}^{\lambda_i +1}\mbox{Ind}_{\lambda_{i+1}}^{1,\lambda_{i+1}-1}\left(R_{\lambda_1\cdots \lambda_n} \{t+1+k_{i-1}+k_i-k_{i+1}\}\right)$$
and 
$$\cal{E}_{-i}\onel\{t\}\mapsto 
\mbox{Res}_{1,\lambda_{i+1}}^{\lambda_{i+1} +1}\mbox{Ind}_{\lambda_i}^{\lambda_i-1,1}\left(R_{\lambda_1\cdots \lambda_n} \{t+1-k_i\}\right).
$$

For $i=n$ and $t \in \Z$ we define:

\begin{multline*}
\cal{E}_{n}\onel\{t\}\mapsto\\ 
\mbox{Res}_{\lambda_n, 1}^{\lambda_n +1}R_{\rho^{-1}}\mbox{Ind}_{\lambda_1}^{1,\lambda_1 -1}\left(R_{\lambda_1\cdots \lambda_n} \{t+ n - (r+k_1) -(k_1 + \cdots + k_{n-2}) \}\right)
\end{multline*}
and 
$$
\cal{E}_{-n}\onel\{t\}\mapsto 
\mbox{Res}_{1,\lambda_1}^{\lambda_1 +1}R_{\rho} \mbox{Ind}_{\lambda_n}^{\lambda_n-1,1}\left(R_{\lambda_1\cdots \lambda_n} \{t+ k_1 + \cdots + k_{n-1} \}\right).
$$

\begin{rem}
Let us explain where the shifts in the image of the new $1$--morphisms come from. We will denote by $A_n(\lambda)$ the shift appearing in 
$\F'(\cal{E}_{ n}\onel)$ and by $B_n(\lambda)$ the one appearing in $\F'(\cal{E}_{- n}\onel)$. To understand their origin, 
let us go back to the decategorified level (see Section~\ref{SchAlg}), where the embedding of $\hat{\SD}(n,r)$ into $\hat{\SD}(n+1,r)$ sends $E_n1_{\lambda}$ to $E_nE_{n+1}1_{(\lambda,0)}$ and $E_{-n}1_{\lambda}$ to $E_{-(n+1)}E_{-n}1_{(\lambda,0)}$. We want this embedding to have a categorical analogue. Although we will not 
work out the details of the corresponding functor in this paper, 
a necessary condition for the existence of such a functor is that 
the shifts satisfy the following recurrence relations:
\begin{eqnarray*}
A_n(\lambda) & = & k_{n-1} - 1 + A_{n+1}(\lambda) \\
B_n(\lambda) & = & B_{n+1}(\lambda') +1 - k_n
\end{eqnarray*}
where $\lambda' = (\lambda_1, \cdots, \lambda_{n-1}, \lambda_n -1,1)$. This determines the value of $A_n(\lambda)$ and $B_n(\lambda)$ up to a constant which 
does not depend on $n$. To fix these constants, we use the fact that we want the triangle of lemma~\ref{lem:comm} to be commutative and the $2$--functor $\F'$ to be degree preserving. 
Note that $\F(-)$ has a shift equal to zero, so $\F'(\mathcal{E}_{r} \dots \mathcal{E}_{1} \mathcal{E}_{r+1} \dots \mathcal{E}_{n}{\mathbf 1}_r) =\F' \Sigma_{n,r}(-)$ 
should have a shift equal to zero too. In this way we obtain a constraint on $A_n((1^r))$ and we deduce that the aforementioned overall constant is 
equal to $-(r+k_1)$. Similarly $\F(+)$ has a shift equal to zero, so $\F'(\mathcal{E}_{-n} \dots \mathcal{E}_{-r-1} \mathcal{E}_{-1} \dots \mathcal{E}_{-r}{\mathbf 1}_r)=\F' \Sigma_{n,r}(+)$ has to have a zero-shift too. This gives us a constraint on $ B_n((1^r)+\bar{\alpha}_n)$ and allows us to deduce that the aforementioned overall 
constant is equal to zero. The reader can verify that this choice of overall constants also fits with the $-1$--shift of 
$\F'(\mathcal{E}_{-n} \dots \mathcal{E}_{-r} \mathcal{E}_{r} \dots \mathcal{E}_{n}{\mathbf 1}_r)= \F'\Sigma_{n,r}(r)$.
\end{rem}

\noindent $\bullet$. To define $\F'$ on $2$-morphisms, we have to give the bimodule maps which correspond 
to the generating $2$-morphisms of $\Scat(n,r)^{*}_{[y]}$. 
The ones not involving color $n$ have the same image as under Khovanov and Lauda's 2-representation~\cite{KL3, KLer, MSVschur}. 

When the color $n$ occurs in a generating $2$--morphism, the $2$--functor $\F'$ is as follows. Here $\epsilon_{\alpha}$ and $\eta_{\alpha}$ 
denote the elementary and the complete symmetric polynomials respectively. 
$$
\begin{array}{l}
    \mathcal{F}'\left(\;\xy
    (0,0)*{\dblue\bbpef{\black n}};
    (6,-2)*{ \lambda};
    \endxy \; \right)
    \colon 1 \mapsto \\ \\
   \xsum{f=0}{\lambda_n}(-1)^{\lambda_n-f} \; x_{1}^{f} \otimes 1 \otimes 1 \otimes  \epsilon_{\lambda_n-f}(x_{r-\lambda_n+1}+y, \dots, x_{r}+y) =\\
\xsum{f=0}{\lambda_n}(-1)^{\lambda_n-f} \; (x_{1}-y)^{f}\otimes 1 \otimes 1 \otimes \epsilon_{\lambda_n-f}(x_{r-\lambda_n+1}, \dots, x_{r})
 %    \label{def_eq_FE_G}
\end{array}
$$

$$
\begin{array}{l}
   \mathcal{F}'\left(\;\xy
    (0,0)*{\dblue\bbpfe{\black n}};
    (6,-2)*{ \lambda};
    \endxy\;\right)
     \colon 
        1 \mapsto \\ \\
    (-1)^{\lambda_{1}}\xsum{f=0}{\lambda_{1}}(-1)^{\lambda_{1}-f} \; (x_{r}+y)^{f}\otimes 1 \otimes 1 \otimes \epsilon_{\lambda_{1}-f}(x_{1}, \dots, x_{\lambda_1}) = \\
(-1)^{\lambda_{1}}\xsum{f=0}{\lambda_{1}}(-1)^{\lambda_{1}-f} \; x_{r}^{f}\otimes 1 \otimes 1 \otimes \epsilon_{\lambda_{1}-f}(x_{1}-y, \dots, x_{\lambda_1}-y)
 %    \label{def_eq_EF_G}
\end{array}
$$

$$
\begin{array}{l}
  \mathcal{F}'\left(\;\xy
    (0,0)*{\dblue\bbcef{\black n}};
    (6,2)*{ \lambda};
    \endxy\;\right)
\colon  x_{1}^{\alpha_1} \otimes 1 \otimes 1\otimes x_{1}^{\alpha_2}
\mapsto \\ \\
(-1)^{\lambda_{1}+1} \eta_{\alpha_1+\alpha_2+1-\lambda_{1}}(x_{1}, \dots, x_{\lambda_1})= \\
(-1)^{\lambda_{1}+1} \xsum{p=0}{\alpha_{1}} \xsum{q=0}{\alpha_{2}}\binom{\alpha_{1}}{p} \binom{\alpha_{2}}{q} y^{\alpha_1+\alpha_2-p-q} 
\eta_{p+q+1-\lambda_{1}}(x_{1}-y, \dots, x_{\lambda_1}-y)
 %  \label{def_FE_cap}
\end{array}
$$

$$
\begin{array}{l}
  \mathcal{F}'\left(\;\xy
    (0,0)*{\dblue\bbcfe{\black n}};
    (6,2)*{ \lambda};
    \endxy \;\right)
\colon 
  x_{r}^{\alpha_1}\otimes 1 \otimes 1\otimes x_{r}^{\alpha_2}
\mapsto \\ \\
\xsum{p=0}{\alpha_{1}} \xsum{q=0}{\alpha_{2}}\binom{\alpha_{1}}{p} \binom{\alpha_{2}}{q} (-y)^{\alpha_1+\alpha_2-p-q} 
\eta_{\alpha_1+\alpha_2+1-\lambda_{n}}(x_{r-\lambda_n+1}+y, \dots, x_{r}+y)= \\
\eta_{\alpha_1+\alpha_2+1-\lambda_{n}}(x_{r-\lambda_n+1}, \dots, x_{r})
%\label{def_EF_cap}
 %%
\end{array}
$$

If $|n - j|>1$, 
\begin{eqnarray*}
\mathcal{F}'\left(
 \xy
  (0,0)*{\xybox{
    (0,0)*{\dblue\xybox{(-4,-4)*{};(4,4)*{} **\crv{(-4,-1) & (4,1)}?(1)*\dir{>}}} ;
    (0,0)*{\dpink\xybox{(4,-4)*{};(-4,4)*{} **\crv{(4,-1) & (-4,1)}?(1)*\dir{>}}};
    (-5,-3)*{\scs n};
     (5.1,-3)*{\scs j};
     (8,1)*{ \lambda};
     (-7,0)*{};(9,0)*{};
     }};
  \endxy
  \right)
 &\maps &
    \begin{array}{ccl}
       x_{r}^{\alpha_1}\otimes 1 \otimes x_{k_{j}+1}^{\alpha_2} & \mapsto &  x_{k_{j}}^{\alpha_2}\otimes 1 \otimes  (x_{1}-y)^{\alpha_1}  \\ 
     \end{array} 
     \nn
\end{eqnarray*}

\begin{eqnarray*}
\mathcal{F}'\left(
 \xy
  (0,0)*{\xybox{
    (0,0)*{\dpink\xybox{(-4,-4)*{};(4,4)*{} **\crv{(-4,-1) & (4,1)}?(1)*\dir{>}}} ;
    (0,0)*{\dblue\xybox{(4,-4)*{};(-4,4)*{} **\crv{(4,-1) & (-4,1)}?(1)*\dir{>}}};
    (-5,-3)*{\scs j};
     (5.1,-3)*{\scs n};
     (8,1)*{ \lambda};
     (-7,0)*{};(9,0)*{};
     }};
  \endxy
  \right)
 &\maps &
    \begin{array}{ccl}
       x_{k_{j}}^{\alpha_1}\otimes 1 \otimes x_{1}^{\alpha_2} & \mapsto &  (x_{r}+y)^{\alpha_2}\otimes 1 \otimes x_{k_{j}+1}^{\alpha_1}  \\ 
     \end{array} 
     \nn
\end{eqnarray*}

\begin{eqnarray*}
\mathcal{F}'\left(
\xy
  (0,0)*{\xybox{
    (0,0)*{\dblue\xybox{(-4,-4)*{};(4,4)*{} **\crv{(-4,-1) & (4,1)}?(0)*\dir{<}}} ;
    (0,0)*{\dpink\xybox{(4,-4)*{};(-4,4)*{} **\crv{(4,-1) & (-4,1)}?(0)*\dir{<}}};
    (-6,-3)*{\scs n};
     (6,-3)*{\scs j};
     (8,1)*{ \lambda};
     (-7,0)*{};(9,0)*{};
     }};
  \endxy
  \right)
 &\maps &
    \begin{array}{ccl}
       x_{1}^{\alpha_1}\otimes 1 \otimes x_{k_{j}}^{\alpha_2} & \mapsto &  x_{k_{j}+1}^{\alpha_2}\otimes 1 \otimes  (x_{r}+y)^{\alpha_1}  \\ 
     \end{array} 
     \nn
\end{eqnarray*}

\begin{eqnarray*}
\mathcal{F}'\left(
\xy
  (0,0)*{\xybox{
    (0,0)*{\dpink\xybox{(-4,-4)*{};(4,4)*{} **\crv{(-4,-1) & (4,1)}?(0)*\dir{<}}} ;
    (0,0)*{\dblue\xybox{(4,-4)*{};(-4,4)*{} **\crv{(4,-1) & (-4,1)}?(0)*\dir{<}}};
    (-6,-3)*{\scs j};
     (6,-3)*{\scs n};
     (8,1)*{ \lambda};
     (-7,0)*{};(9,0)*{};
     }};
  \endxy
  \right)
 &\maps &
    \begin{array}{ccl}
       x_{k_{j}1}^{\alpha_1}\otimes 1 \otimes x_{r}^{\alpha_2} & \mapsto &  (x_{1}-y)^{\alpha_2}\otimes 1 \otimes  x_{k_{j}}^{\alpha_1}  \\ 
     \end{array} 
     \nn
\end{eqnarray*}
%In the case $i =j=n$,
$$ 
\begin{array}{l}
\mathcal{F}'\left(
 \xy
  (0,0)*{\xybox{
    (0,0)*{\dblue\xybox{(-4,-4)*{};(4,4)*{} **\crv{(-4,-1) & (4,1)}?(1)*\dir{>}}} ;
    (0,0)*{\dblue\xybox{(4,-4)*{};(-4,4)*{} **\crv{(4,-1) & (-4,1)}?(1)*\dir{>}}};
    (-5,-3)*{\scs n};
     (5.1,-3)*{\scs n};
     (8,1)*{ \lambda};
     (-7,0)*{};(9,0)*{};
     }};
  \endxy
  \right)
 \colon x_{r}^{\alpha_1}\otimes 1 \otimes 1 \otimes x_{1}^{\alpha_2} \mapsto \\ \\
(x_r +y)^{\alpha_2}\xsum{p=0}{\alpha_1}\xsum{f=0}{p-1} \binom{\alpha_{1}}{p} (-y)^{\alpha_1-p}(x_{r}+y)^{p-1-f}\otimes 1 \otimes 1 \otimes x_{1}^{f} \\ 
-\; x_{r}^{\alpha_1}\xsum{g=0}{\alpha_2-1}(x_{r}+y)^{\alpha_2-1-g}\otimes 
1 \otimes 1 \otimes x_{1}^{g} = \\
(x_r +y)^{\alpha_2}\xsum{f=0}{\alpha_1-1} x_{r}^{\alpha_1-1-f}\otimes 1 \otimes 1 \otimes (x_{1}-y)^{f}\\ 
-\;x_{r}^{\alpha_1}\xsum{q=0}{\alpha_2}\xsum{g=0}{q-1}\binom{\alpha_{2}}{q} y^{\alpha_2-q}x_{r}^{q-1-g}\otimes 1 \otimes 1 \otimes (x_{1}-y)^{g}   
\end{array}
$$

$$
\begin{array}{l}
\mathcal{F}'\left(
 \xy
  (0,0)*{\xybox{
    (0,0)*{\dblue\xybox{(-4,-4)*{};(4,4)*{} **\crv{(-4,-1) & (4,1)}?(0)*\dir{<}}} ;
    (0,0)*{\dblue\xybox{(4,-4)*{};(-4,4)*{} **\crv{(4,-1) & (-4,1)}?(0)*\dir{<}}};
    (-6,-3)*{\scs n};
     (6,-3)*{\scs n};
     (8,1)*{ \lambda};
     (-7,0)*{};(9,0)*{};
     }};
  \endxy
  \right)
 \colon x_{1}^{\alpha_1}\otimes 1 \otimes 1 \otimes x_{r}^{\alpha_2}  \mapsto \\ \\
x_1^{\alpha_1}\xsum{p=0}{\alpha_2}\xsum{f=0}{p-1} 
\binom{\alpha_{2}}{p} (-y)^{\alpha_2-p}x_{1}^{p-1-f}\otimes 1 \otimes 1 \otimes (x_{r}+y)^{f} \\
-\;(x_{1}-y)^{\alpha_2}\xsum{g=0}{\alpha_1-1}x_{1}^{\alpha_1-1-g}\otimes 1 \otimes 1 \otimes (x_{r}+y)^{g} = \\
x_1^{\alpha_1}\xsum{f=0}{\alpha_2-1} (x_{1}-y)^{\alpha_2-1-f}\otimes 1 \otimes 1 \otimes x_{r}^{f} \\ 
-\; (x_{1}-y)^{\alpha_2}\xsum{q=0}{\alpha_1}\xsum{g=0}{q-1}\binom{\alpha_{1}}{q} y^{\alpha_1-q}(x_{1}-y)^{q-1-g}\otimes 1 \otimes 1 \otimes x_{r}^{g}   
\end{array}
$$
%In the case $i -1=j$,
$$
\begin{array}{l}
\mathcal{F}'\left(
 \xy
  (0,0)*{\xybox{
    (0,0)*{\dgreen\xybox{(-4,-4)*{};(4,4)*{} **\crv{(-4,-1) & (4,1)}?(1)*\dir{>}}} ;
    (0,0)*{\dblue\xybox{(4,-4)*{};(-4,4)*{} **\crv{(4,-1) & (-4,1)}?(1)*\dir{>}}};
    (-5,-3)*{\scs 1};
     (5.1,-3)*{\scs n};
     (8,1)*{ \lambda};
     (-7,0)*{};(9,0)*{};
     }};
  \endxy
  \right)
 \colon 
       x_{\lambda_1}^{\alpha_1}\otimes1 \otimes x_{1}^{\alpha_2}  \mapsto \\ \\
\left( (x_{r}+y)^{\alpha_2} \otimes 1\otimes x_{\lambda_1+1}^{\alpha_1+1}- 
(x_{r}+y)^{\alpha_2+1}\otimes1 \otimes     x_{\lambda_1+1}^{\alpha_1} \right) \{ -1\} 
=\\ 
\left( (x_{r}+y)^{\alpha_2} \otimes 1\otimes x_{\lambda_1+1}^{\alpha_1}(x_{\lambda_1+1}-y)- 
x_{r}(x_{r}+y)^{\alpha_2}\otimes1 \otimes     x_{\lambda_1+1}^{\alpha_1} \right) \{ -1\} 
\end{array}              
$$
$$
\begin{array}{l}
\mathcal{F}'\left(
 \xy
  (0,0)*{\xybox{
    (0,0)*{\dblue\xybox{(-4,-4)*{};(4,4)*{} **\crv{(-4,-1) & (4,1)}?(1)*\dir{>} }};
    (0,0)*{\dred\xybox{(4,-4)*{};(-4,4)*{} **\crv{(4,-1) & (-4,1)}?(1)*\dir{>}}};
    (-5,-3)*{\scs n};
     (7.1,-3)*{\scs n-1};
     (8,1)*{ \lambda};
     (-7,0)*{};(9,0)*{};
     }};
  \endxy
  \right)
 \colon 
        x_{r}^{\alpha_1}\otimes1 \otimes x_{r-\lambda_n+1}^{\alpha_2}  \mapsto \\ \\  
\left( (x_{r-\lambda_n}+y)^{\alpha_2} \otimes 1\otimes (x_{1}-y)^{\alpha_1}x_1-
(x_{r-\lambda_n}+y)x_{r-\lambda_n}^{\alpha_2}\otimes1 \otimes     (x_{1}-y)^{\alpha_1} 
\right) \{ -1\}=\\ 
\left( (x_{r-\lambda_n}+y)^{\alpha_2} \otimes 1\otimes (x_{1}-y)^{\alpha_1+1}- 
x_{r-\lambda_n}^{\alpha_2+1}\otimes1 \otimes     (x_{1}-y)^{\alpha_1} \right) \{ -1\} 
\end{array}      
$$
         
\begin{eqnarray*}
\mathcal{F}'\left(
 \xy
  (0,0)*{\xybox{
    (0,0)*{\dgreen\xybox{(-4,-4)*{};(4,4)*{} **\crv{(-4,-1) & (4,1)}?(0)*\dir{<} }};
    (0,0)*{\dblue\xybox{(4,-4)*{};(-4,4)*{} **\crv{(4,-1) & (-4,1)}?(0)*\dir{<}}};
    (-6,-3)*{\scs 1};
     (6,-3)*{\scs n};
     (8,1)*{ \lambda};
     (-7,0)*{};(9,0)*{};
     }};
  \endxy
  \right)
 &\maps &
        \begin{array}{ccl}
       x_{\lambda_1+1}^{\alpha_1}\otimes1 \otimes x_{r}^{\alpha_2} & \mapsto  & \left( (x_{1}-y)^{\alpha_2} \otimes 1\otimes x_{\lambda_1}^{\alpha_1} \right) \{ -1\} \\ 
     \end{array} 
\end{eqnarray*}              

\begin{eqnarray*}
\mathcal{F}'\left(
 \xy
  (0,0)*{\xybox{
    (0,0)*{\dblue\xybox{(-4,-4)*{};(4,4)*{} **\crv{(-4,-1) & (4,1)}?(0)*\dir{<}}} ;
    (0,0)*{\dred\xybox{(4,-4)*{};(-4,4)*{} **\crv{(4,-1) & (-4,1)}?(0)*\dir{<}}};
    (-6,-3)*{\scs n};
     (8,-3)*{\scs n-1};
     (8,1)*{ \lambda};
     (-7,0)*{};(9,0)*{};
     }};
  \endxy
  \right)
 &\maps &
        \begin{array}{ccl}
       x_{1}^{\alpha_1}\otimes1 \otimes x_{r-\lambda_n}^{\alpha_2} & \mapsto  & \left( x_{r-\lambda_n+1}^{\alpha_2} \otimes 1\otimes (x_{r}+y)^{\alpha_1} \right) \{ -1\} \\ 
     \end{array} 
\end{eqnarray*}   
%In the case $i +1=j$, 
\begin{eqnarray*}
\mathcal{F}'\left(
 \xy
  (0,0)*{\xybox{
    (0,0)*{\dblue\xybox{(-4,-4)*{};(4,4)*{} **\crv{(-4,-1) & (4,1)}?(1)*\dir{>}}} ;
    (0,0)*{\dgreen\xybox{(4,-4)*{};(-4,4)*{} **\crv{(4,-1) & (-4,1)}?(1)*\dir{>}}};
    (-5,-3)*{\scs n};
     (5.1,-3)*{\scs 1};
     (8,3)*{ \lambda};
     (-7,0)*{};(9,0)*{};
     }};
  \endxy
  \right)
 &\maps &
        \begin{array}{ccl}
       x_{r}^{\alpha_1}\otimes1 \otimes x_{\lambda_1+1}^{\alpha_2} & \mapsto  & \left( x_{\lambda_1}^{\alpha_2} \otimes 1\otimes (x_{1}-y)^{\alpha_1} \right) \{ 1\} \\ 
     \end{array} 
\end{eqnarray*}              

\begin{eqnarray*}
\mathcal{F}'\left(
 \xy
  (0,0)*{\xybox{
    (0,0)*{\dred\xybox{(-4,-4)*{};(4,4)*{} **\crv{(-4,-1) & (4,1)}?(1)*\dir{>} }};
    (0,0)*{\dblue\xybox{(4,-4)*{};(-4,4)*{} **\crv{(4,-1) & (-4,1)}?(1)*\dir{>}}};
    (-7,-3)*{\scs n-1};
     (5.1,-3)*{\scs n};
     (8,1)*{ \lambda};
     (-7,0)*{};(9,0)*{};
     }};
  \endxy
  \right)
 &\maps &
        \begin{array}{ccl}
       x_{r-\lambda_n}^{\alpha_1}\otimes1 \otimes x_{1}^{\alpha_2} & \mapsto  & \left( (x_{r}+y)^{\alpha_2} \otimes 1\otimes x_{r-\lambda_n+1}^{\alpha_1} \right) \{ 1\} \\ 
     \end{array} 
\end{eqnarray*}   

$$
\begin{array}{l}
\mathcal{F}'\left(
 \xy
  (0,0)*{\xybox{
    (0,0)*{\dblue\xybox{(-4,-4)*{};(4,4)*{} **\crv{(-4,-1) & (4,1)}?(0)*\dir{<} }};
    (0,0)*{\dgreen\xybox{(4,-4)*{};(-4,4)*{} **\crv{(4,-1) & (-4,1)}?(0)*\dir{<}}};
    (-6,-3)*{\scs n};
     (6,-3)*{\scs 1};
     (8,1)*{ \lambda};
     (-7,0)*{};(9,0)*{};
     }};
  \endxy
  \right)
 \colon 
       x_{1}^{\alpha_1}\otimes1 \otimes x_{\lambda_1}^{\alpha_2}  \mapsto \\ \\
\left( x_{\lambda_1+1}^{\alpha_2+1} \otimes 1\otimes (x_{r}+y)^{\alpha_1}-
x_{\lambda_1+1}^{\alpha_2}\otimes1 \otimes     (x_{r}+y)^{\alpha_1+1} \right) \{ 1\}=\\ 
\left( x_{\lambda_1+1}^{\alpha_2}(x_{\lambda_1+1}-y)\otimes 1\otimes (x_{r}+y)^{\alpha_1} -
x_{\lambda_1+1}^{\alpha_2}\otimes1 \otimes    x_{r}(x_{r}+y)^{\alpha_1}  \right) \{ 1\}  
\end{array}              
$$
$$
\begin{array}{l}
\mathcal{F}'\left(
 \xy
  (0,0)*{\xybox{
    (0,0)*{\dred\xybox{(-4,-4)*{};(4,4)*{} **\crv{(-4,-1) & (4,1)}?(0)*\dir{<}}} ;
    (0,0)*{\dblue\xybox{(4,-4)*{};(-4,4)*{} **\crv{(4,-1) & (-4,1)}?(0)*\dir{<}}};
    (-8,-3)*{\scs n-1};
     (6,-3)*{\scs n};
     (8,1)*{ \lambda};
     (-7,0)*{};(9,0)*{};
     }};
  \endxy
  \right)
 \colon 
    x_{r-\lambda_n+1}^{\alpha_1}\otimes1 \otimes x_{r}^{\alpha_2}  \mapsto \\ \\  
\left( (x_{1}-y)^{\alpha_2}x_1 \otimes 1\otimes(x_{r-\lambda_n}+y)^{\alpha_1} - 
(x_{1}-y)^{\alpha_2}\otimes1 \otimes  (x_{r-\lambda_n}+y)x_{r-\lambda_n}^{\alpha_1}
\right) \{ 1\} =\\ 
\left(  (x_{1}-y)^{\alpha_2+1}\otimes 1\otimes (x_{r-\lambda_n}+y)^{\alpha_1}-
(x_{1}-y)^{\alpha_2}\otimes1 \otimes  x_{r-\lambda_n}^{\alpha_1+1}   \right) \{ 1\} 
\end{array}      
$$
\begin{eqnarray*} %\label{eq_gamma_updot}
  \mathcal{F}'\left(
 \xy
  (0,0)*{\dblue\lineu{\black n}};
  (-0.25,4)*{\dblue\txt\large{$\bullet$}};
 (6,0)*{ \lambda};
 %(-8,0)*{ \lambda+i_X};
 (-10,0)*{};(10,0)*{};
 \endxy\right)
\quad &\maps&
  1 \otimes 1 \mapsto   \quad x_{r} \otimes 1 = 1 \otimes (x_1 -y)
\end{eqnarray*}

\begin{eqnarray*}
 \mathcal{F}'\left(
 \xy
  (0,0)*{\dblue\lined{\black n}};
  (-0.25,4)*{\dblue\txt\large{$\bullet$}};
 %(-6,0)*{ \lambda};
 (8,0)*{ \lambda};
 (-10,0)*{};(10,0)*{};
 \endxy\right)
\quad &\maps&
  1 \otimes 1 \mapsto   \quad  (x_1-y)\otimes 1 = 1 \otimes x_{r}
\end{eqnarray*}

\begin{rem}\label{conjtrick}
 It is natural to wonder how the previous images relate to Khovanov and Lauda's $2$--representation $\Gamma^G_r$. Indeed, take a generating $2$--morphism 
with $n$--strands. 
By Lemma~\ref{comsingbim}, the images of its source and target $1$--morphisms are isomorphic, up to a same shift, to conjugates of the images 
of $1$--morphisms which do not contain factors $\mathcal{E}_{\pm n}$. 
Here conjugation means conjugation by certain invertible twisted bimodules.
One can thus ask if the image of the $2$--morphism can be obtained 
by conjugating a $2$-morphism which does not contain strands of color $n$.
Before answering this question, 
let us do an example to make things more concrete (we omit the shifts here). 
Consider the $2$--morphism
$$
 \xy
  (0,0)*{\xybox{
    (0,0)*{\dblue\xybox{(-4,-4)*{};(4,4)*{} **\crv{(-4,-1) & (4,1)}?(1)*\dir{>} ;}};
    (0,0)*{\dpink\xybox{(4,-4)*{};(-4,4)*{} **\crv{(4,-1) & (-4,1)}?(1)*\dir{>};}};
    (-5,-3)*{\scs n};
     (5.1,-3)*{\scs j};
     (8,1)*{ \lambda};
     (-12,0)*{};(12,0)*{};
     }};
  \endxy
$$
with $|n - j|>1$. The image of its source and target $1$--morphisms are 
$$\mbox{Res}_{\lambda_n,1}^{\lambda_n +1}R_{\rho^{-1}}\mbox{Ind}_{\lambda_1}^{1,\lambda_1 -1} \mbox{Res}_{\lambda_j,1}^{\lambda_j +1} \mbox{Ind}_{\lambda_{j+1}}^{1,\lambda_{j+1} -1} \left(R_{\lambda_1,\ldots,\lambda_n}\right) $$
and
$$ \mbox{Res}_{\lambda_j,1}^{\lambda_j +1} \mbox{Ind}_{\lambda_{j+1}}^{1,\lambda_{j+1} -1} \mbox{Res}_{\lambda_n,1}^{\lambda_n +1}R_{\rho^{-1}}\mbox{Ind}_{\lambda_1}^{1,\lambda_1 -1}\left(R_{\lambda_1,\ldots,\lambda_n}\right). $$
These are isomorphic to  
\begin{equation}
\label{eq:twistexample}
R_{\rho^{-(\lambda_n +1)}} \mbox{Res}_{\lambda_n,1}^{\lambda_n +1}  \mbox{Ind}_{\lambda_1}^{1,\lambda_1 -1} \mbox{Res}_{\lambda_j,1}^{\lambda_j +1} \mbox{Ind}_{\lambda_{j+1}}^{1,\lambda_{j+1} -1}   R_{\rho^{\lambda_n}}\left(R_{\lambda_1,\ldots,\lambda_n}\right)
\end{equation}
and 
\begin{equation}
\label{eq:twistexample2}
R_{\rho^{-(\lambda_n +1)}} \mbox{Res}_{\lambda_j,1}^{\lambda_j +1} \mbox{Ind}_{\lambda_{j+1}}^{1,\lambda_{j+1} -1}  \mbox{Res}_{\lambda_n,1}^{\lambda_n +1}  \mbox{Ind}_{\lambda_1}^{1,\lambda_1 -1} R_{\rho^{\lambda_n}}\left(R_{\lambda_1,\ldots,\lambda_n}\right)
\end{equation}
respectively, where in both cases the isomorphism is given by 
\begin{equation}
 \label{eq:twistiso}
a \ot 1 \ot b \mapsto a \ot 1 \ot 1 \ot \rho^{\lambda_n}(b) = 1 \ot \rho^{\lambda_n+1}(a)  \ot \rho^{\lambda_n}(b)\ot 1.
\end{equation}
The inverse is given by 
\begin{equation}
 \label{eq:twistiso2}
 1 \ot a \ot b \ot 1= \rho^{-( \lambda_n +1)}(a) \ot 1 \ot \rho^{-\lambda_n}(b) = 1 \ot \rho^{- \lambda_n }(a)  \ot \rho^{-\lambda_n}(b).
\end{equation}
Note that the tensor factors
$$\mbox{Res}_{\lambda_n ,1}^{\lambda_n +1}\mbox{Ind}_{\lambda_1}^{1,\lambda_1 -1}\mbox{Res}_{\lambda_j,1}^{\lambda_j +1} \mbox{Ind}_{\lambda_{j+1}}^{1,\lambda_{j+1} -1} (R_{\lambda_n,\lambda_1\ldots,\lambda_{n-1}})$$ 
and 
$$\mbox{Res}_{\lambda_j,1}^{\lambda_j +1} \mbox{Ind}_{\lambda_{j+1}}^{1,\lambda_{j+1} -1} \mbox{Res}_{\lambda_n ,1}^{\lambda_n +1}\mbox{Ind}_{\lambda_1}^{1,\lambda_1 -1}(R_{\lambda_n,\lambda_1\ldots,\lambda_{n-1}})$$ 
in the middle of~\eqref{eq:twistexample} and~\eqref{eq:twistexample2} are, up to a same shift, the images of $\mathcal{E}_{1}\mathcal{E}_{j+1}{\mathbf 1}_{\lambda_n,\lambda_1,\ldots,\lambda_{n-1}}$ and $\mathcal{E}_{j+1} \mathcal{E}_{1} {\mathbf 1}_{\lambda_n,\lambda_1,\ldots,\lambda_{n-1}}$ under $\mathcal{F}'$. 

One can see that if one applies the isomorphism \eqref{eq:twistiso} to $x_{r}^{\alpha_1}\otimes 1 \otimes x_{k_{j}+1}^{\alpha_2}$ followed by the tensor product of the identity on the two twisted bimodules in~\eqref{eq:twistexample} and of the bimodule map on the central tensor factor given by the image of   
$$
 \xy
  (0,0)*{\xybox{
    (0,0)*{\dgreen\xybox{(-4,-4)*{};(4,4)*{} **\crv{(-4,-1) & (4,1)}?(1)*\dir{>} ;}};
    (0,0)*{\dyellow\xybox{(4,-4)*{};(-4,4)*{} **\crv{(4,-1) & (-4,1)}?(1)*\dir{>};}};
    (-5,-3)*{\scs 1};
     (7.1,-3)*{\scs j+1};
     (8,1)*{ \lambda'};
     (-12,0)*{};(12,0)*{};
     }};
  \endxy,
$$
where $\lambda'=(\lambda_n,\lambda_1,\ldots,\lambda_{n-1})$, and 
finally followed by the inverse isomorphism \eqref{eq:twistiso2}, one gets
\begin{multline}
 x_{r}^{\alpha_1}\otimes 1 \otimes x_{k_{j}+1}^{\alpha_2} \mapsto \\
1 \ot (x_{\lambda_n +1}-y)^{\alpha_1} \ot x_{k_{j}+\lambda_n+1}^{\alpha_2} \ot 1 = \\ 
\xsum{p=0}{\alpha_1} \binom{\alpha_{1}}{p} (-y)^{\alpha_1-p}\ot x_{\lambda_n +1}^{p}\otimes x_{k_{j}+\lambda_n+1}^{\alpha_2} \ot 1  
\mapsto  \\
\xsum{p=0}{\alpha_1} \binom{\alpha_{1}}{p} (-y)^{\alpha_1-p}\ot  x_{k_{j}+\lambda_n+1}^{\alpha_2} \ot x_{\lambda_n +1}^{p}\otimes 1 =\\ 
1 \ot  x_{k_{j}+\lambda_n+1}^{\alpha_2} \ot (x_{\lambda_n +1}-y)^{\alpha_1} \ot 1 \mapsto x_{k_{j}}^{\alpha_2}\otimes 1 \otimes  (x_{1}-y)^{\alpha_1}
\end{multline}
which is precisely the image under $\F'$ of our original $2$--morphism with 
the $n$--strand. So in this example one obtains indeed 
the same result using this conjugation trick turning $n$ into $1$. 

Of course one could have used a similar conjugation trick turning 
$n$ into $n-1$, i.e. writing everything as conjugates 
$R_{\rho^{\lambda_1 -1}} \ot - \ot R_{\rho^{-\lambda_1}}$ and using 
the bimodule map corresponding to 
$$
 \xy
  (0,0)*{\xybox{
    (0,0)*{\dred\xybox{(-4,-4)*{};(4,4)*{} **\crv{(-4,-1) & (4,1)}?(1)*\dir{>} ;}};
    (0,0)*{\dturq\xybox{(4,-4)*{};(-4,4)*{} **\crv{(4,-1) & (-4,1)}?(1)*\dir{>};}};
    (-7,-3)*{\scs n-1};
     (7.1,-3)*{\scs j-1};
     (8,1)*{ \lambda''};
     (-12,0)*{};(12,0)*{};
     }};
  \endxy,
$$
where $\lambda''=(\lambda_2,\ldots,\lambda_{n},\lambda_1)$. In this case one 
obtains again the image under $\F'$ of our original $2$--morphism.

In fact, the image of any generating $2$--morphism containing $n$--strands 
can be obtained by either one of the conjugation tricks, i.e. turning $n$ into 
$1$ or $n-1$, except for the dotted $n$--identities. 
The images of these two $2$--morphisms (up and downward) can be 
obtained by the conjugation trick which turns $n$ into $n-1$, but not 
by the one which turns $n$ into $1$. 
Indeed, if one writes $\F'( \mathcal{E}_{\pm n}{\mathbf 1}_{\lambda})$ as 
$ R_{\rho^{-(\lambda_n \pm 1)}} \ot \F'( \mathcal{E}_{\pm 1}{\mathbf 1}_{\lambda'}) \ot 
R_{\rho^{\lambda_n}}$, one obtains
\begin{equation*} 
  1 \otimes 1 \mapsto  (x_{r}+y) \otimes 1 = 1 \otimes x_1 \quad (\mbox{resp.} \ x_1\otimes 1 = 1 \otimes (x_{r}+y))
\end{equation*} 
which differs from the image under $\F'$ of the oriented upward 
(resp. downward) dotted $n$--identity.

As for the non-dotted generating $2$--morphisms of color $n$, 
both conjugation tricks give the same bimodule maps which we 
have used in the definition of $\F'$. Whenever only one expression was given 
in our definition of $\F'$, it is because the expressions obtained from 
both conjugation tricks were obviously equal. Whenever two expressions 
are given, let us prove that they are equal.  

We start with the image under $\F'$ of the right $n$-cup. Note that 
in the bimodule 
$\mbox{Res}_{1,\lambda_1-1}^{\lambda_1}R_{\rho} \mbox{Ind}_{\lambda_n+1}^{\lambda_n,1}
\mbox{Res}_{\lambda_n, 1}^{\lambda_n +1}R_{\rho^{-1}}\mbox{Ind}_{\lambda_1}^{1,\lambda_1 -1}\left(R_{\lambda_1\cdots \lambda_n}\right)$ 
we have $x_1\otimes 1 \ot 1 \otimes p=1\otimes (x_r+y)\ot 1 \otimes p$, 
for any polynomial $p$. Therefore, we have to show that
\begin{gather*}
\sum_{f=0}^{\lambda_n}(-1)^{\lambda_n -f} 
\ot (x_r+y)^f\otimes 1 \ot \epsilon_{\lambda_n-f}(x_{r-\lambda_n+1}+y,\ldots,x_r +y)=\\
\sum_{f=0}^{\lambda_n}(-1)^{\lambda_n -f} \ot 
x_r^f\otimes 1 \ot \epsilon_{\lambda_n-f}(x_{r-\lambda_n+1},\ldots,x_r).
\end{gather*}
For a fixed power of $x_r$, say $k$, this amounts to showing that 
\begin{gather*}
 1 \ot x_r^k\otimes 1 \ot \sum_{i=0}^{\lambda_n-k}(-1)^i \binom{k+i}{i} y^i\epsilon_{\lambda_n-k-i}
(x_{r-\lambda_n+1}+y,\ldots, x_r+y)=\\
1 \ot x_r^k\otimes 1 \ot  \epsilon_{\lambda_n-k}(x_{r-\lambda_n+1},\ldots, x_r).
\end{gather*} 
This follows from the following lemma.  
\begin{lem}
For any $0\leq k\leq n$, we have 
$$
\epsilon_{n-k}(a_1,\ldots,a_n)=\sum_{i=0}^{n-k}(-1)^i \binom{k+i}{i} y^i\epsilon_{n-k-i}(a_1+y,\ldots,a_n+y).$$
\end{lem}

\begin{proof}
By induction w.r.t. $n$. For $n=0$ there is nothing to prove. 

Suppose $n>0$. Write 
$$
\epsilon_{n-k}(a_1,\ldots,a_n)=\epsilon_{n-k}(a_1,\ldots,a_{n-1})+
a_n\epsilon_{n-1-k}(a_1,\ldots,a_{n-1}).$$
Note that $n-k=(n-1)-(k-1)$, so by induction the sum above is equal to  
\begin{gather*}
\sum_{i=0}^{n-k}(-1)^i\binom{k-1+i}{i}y^i\epsilon_{n-k-i}(a_1+y,\ldots,a_{n-1}+y) 
+ \\
a_n\sum_{i=0}^{n-1-k}(-1)^i\binom{k+i}{i}y^i\epsilon_{n-1-k-i}
(a_1+y,\ldots,a_{n-1}+y).
\end{gather*}
Write $a_n=-y+a_n+y$. Then we get 
\begin{gather}
\label{eq:firstsum}
\sum_{i=0}^{n-k}(-1)^i\binom{k-1+i}{i}y^i\epsilon_{n-k-i}(a_1+y,\ldots,a_{n-1}+y) 
-\\
\label{eq:secondsum}
\sum_{i=0}^{n-1-k}(-1)^i\binom{k+i}{i}y^{i+1}\epsilon_{n-1-k-i}
(a_1+y,\ldots,a_{n-1}+y)
+ \\
\label{eq:thirdsum}
(a_n+y)\sum_{i=0}^{n-1-k}(-1)^i\binom{k+i}{i}y^i\epsilon_{n-1-k-i}
(a_1+y,\ldots,a_{n-1}+y).
\end{gather}
After reindexing the sum in~\eqref{eq:secondsum}, the difference of the sums in~\eqref{eq:firstsum} and~\eqref{eq:secondsum} 
becomes
\begin{gather*}
\epsilon_{n-k}(a_1+y,\ldots,a_{n-1}+y)+\\
\sum_{i=1}^{n-k}(-1)^i\left(\binom{k-1+i}{i}+\binom{k-1+i}{i-1}\right)
y^i\epsilon_{n-k-i}(a_1+y,\ldots,a_{n-1}+y)=
\\
\sum_{i=0}^{n-k}(-1)^i \binom{k+i}{i}
y^i\epsilon_{n-k-i}(a_1+y,\ldots,a_{n-1}+y).
\end{gather*}
Together with the sum in~\eqref{eq:thirdsum}, we now get 
\begin{gather*}
\sum_{i=0}^{n-k}(-1)^i \binom{k+i}{i}
y^i\left(\epsilon_{n-k-i}(a_1+y,\ldots,a_{n-1}+y)+\right. \\
\left. (a_n+y)\epsilon_{n-1-k-i}
(a_1+y,\ldots,a_{n-1}+y)\right)=
\\
\sum_{i=0}^{n-k}(-1)^i \binom{k+i}{i}
y^i\epsilon_{n-k-i}(a_1+y,\ldots,a_{n-1}+y,a_n+y).
\end{gather*}
\end{proof}

The proof for the image of the left $n$--cup is similar, using $x_r\otimes 1 \ot 1 \ot 1 =1\otimes (x_1-y) \ot 1 \ot 1 $ and the lemma above with $-y$ instead of $y$.  

The result for the cups implies the result for the caps, because of the 
biadjointness relations \eqref{eq_biadjoint1} and \eqref{eq_biadjoint2}. If two maps corresponding to a cap 
both satisfy the biadjointness relations w.r.t. one fixed map 
associated to the corresponding cup, then the two maps have to be equal. 

The result for the upward oriented $n$-crossing follows from the fact that both bimodule maps satisfy (easy calculations):
\begin{enumerate}
\item $f(1\otimes 1 \ot 1 \otimes 1)=0$;
\item for any $\alpha_1,\alpha_2\in\mathbb{N}$, 
$$(x_r\otimes 1\otimes 1 \ot 1 )f(x_r^{\alpha_1}\otimes 1 \ot 1 \otimes x_1^{\alpha_2})-
f(x_r^{\alpha_1}\otimes 1 \ot 1\otimes x_1^{\alpha_2}(x_1-y))=x_r^{\alpha_1}\otimes 1 \ot  1\otimes x_1^{\alpha_2};$$
\item for any $\alpha_1,\alpha_2\in\mathbb{N}$, 
$$f(x_r^{\alpha_1+1}\otimes 1 \ot  1\otimes x_1^{\alpha_2})-(1\otimes1 \ot 1 \otimes (x_1-y))
f(x_r^{\alpha_1}\otimes 1 \ot 1\otimes x_1^{\alpha_2})=x_r^{\alpha_1}\otimes 1 \ot 1 \otimes x_1^{\alpha_2}.$$
\end{enumerate}
These three properties determine the maps completely by recursion, so they have to be equal.  

A similar argument proves the result for downward oriented $n$-crossings and the remaining cases of the crossings colored $(1,n)$ and $(n,n-1)$ are easy 
computations.

Just to summarize, all generating $2$--morphisms containing $n$-strands can be 
obtained by the conjugation trick which turns $n$ into $n-1$, 
while only the ones whose $n$--strands have no dots can also be obtained 
by the conjugation trick which turns $n$ into $1$.
\end{rem}

\begin{prop}
$\F' : \Scat(n,r)^{*}_{[y]} \to \esbim_{\hat{A}_{r-1}}^{*}$ is a well-defined degree preserving $2$--functor.
\end{prop}

\begin{proof}
All relations between $2$--morphisms which do not have $n$--colored strands are satisfied by the results in Section 6 in~\cite{KL3}.  
Since no relation in $\Scat(n,r)_{[y]}$ involves all colors at the same time, 
the proof that $\F'$ preserves a given relation can always be reduced 
to the fact that $\F'$ preserves the same relation with colors belonging to 
$\{1,\ldots, n-1\}$ by using the conjugation trick which turns $n$ into $n-1$, 
except in the case of relations~\eqref{eq_r2_ij-gen} 
and~\eqref{eq_dot_slide_ij-gen} for $\{ i,j\} = \{1,n\}$. 
The proof that $\F'$ preserves these relations is straightforward and 
is left to the reader. 
\end{proof}

\begin{rem}
Note that the twisted bimodules $B_{\rho} = R_{(1^r),\rho}$ and $B_{\rho^{-1}} = R_{(1^r),\rho^{-1}}$ are isomorphic to 
the images under $\F'$ of respectively 
$$\mathcal{E}_{-n} \dots \mathcal{E}_{-r-1} \mathcal{E}_{-1} \dots \mathcal{E}_{-r}{\mathbf 1}_r\quad\text{and}\quad 
\mathcal{E}_{r} \dots \mathcal{E}_{1} \mathcal{E}_{r+1} \dots \mathcal{E}_{n}{\mathbf 1}_r.$$ 

As a matter of fact, any twisted singular bimodule is isomorphic to the image under $\F'$ of a certain product of categorified divided powers 
(see \cite{KLMS} and \cite{MSVschur} for more details on divided powers and extended graphical calculus). 
Indeed, let $n>r$ and let $\lambda\in \Lambda(n,r)$ be arbitrary. 
At least one entry of 
$\lambda$ is equal to zero, let us assume it is $\lambda_i$. 
Then 
$$R_{\lambda_n \lambda_1 \dots \lambda_{n-1}, \rho}\cong\F'\left(\mathcal{E}_{-i-1}^{(\lambda_{i+1})} \dots \mathcal{E}_{-n}^{(\lambda_{n})} \mathcal{E}_{-1}^{(\lambda_{1})} \dots \mathcal{E}_{-i+1}^{(\lambda_{i-1})} \onel \right)$$ 
and 
$$R_{ \lambda_{2} \dots  \lambda_n \lambda_1, \rho^{-1}}\cong
\F'\left(\mathcal{E}_{i-2}^{(\lambda_{i-1})} \dots \mathcal{E}_{1}^{(\lambda_{2})} \mathcal{E}_{n}^{(\lambda_{1})} \dots \mathcal{E}_{i}^{(\lambda_{i+1})} \onel\right).$$
\end{rem}

\section{The Grothendieck group of $\Scat(n,r)_{[y]}$}
\label{GrAlg}
The following Lemma is the affine analogue of Lemma 6.6 in~\cite{MSVschur}. 
\begin{lem}
\label{lem:comm}
The following diagram commutes  
\begin{equation*}
\xymatrix{
\debim_{\hat{A}_{r-1}}^* \ar[rr]^{\F}\ar[dr]_{\Sigma_{n,r}} && \esbim_{\hat{A}_{r-1}}^*
\\
& \Scat(n,r)^*_{[y]}((1^r),(1^r))\ar[ur]_{\F'} &
}
\end{equation*}
\end{lem}

\begin{proof}
The proof is straightforward and follows from checking the definitions 
carefully.  
\end{proof}

Note that the $2$-hom spaces of $\Ucataffy$ are 
finite-dimensional $\Q$-vector spaces, because the original $2$-HOM spaces 
in $\Ucataff^*$ are finite-dimensional in each degree, their grading 
is bounded below and $\deg(y)=2>0$. Therefore, the Karoubi envelope 
(or idempotent completion) of $\Ucataffy$, denoted 
$\mathrm{Kar}(\Ucataffy)$, is Krull-Schmidt. The same holds for 
$\Uglcataffy$ and $\Scat(n,r)_{[y]}$, of course. 

By the same arguments, we see that the ideal generated by $y$ is virtually 
nilpotent (for virtually nilpotent ideals and basic facts about them, 
see Section 3.8.1 and 3.8.2 in~\cite{KL3}). This proves that 
\begin{equation}
\label{eq:virnilGrot}
K_0^{\Q(q)}(\mathrm{Kar}\mathcal{C}_{[y]})\cong K_0^{\Q(q)}(\mathrm{Kar}
\mathcal{C}),
\end{equation}
where $\mathcal{C}$ is $\Ucataff$, $\Uglcataff$ or $\Scat(n,r)$.

\begin{cor}
\label{cor:embed}
The algebra homomorphism 
$$K_0^{\Q(q)}(\Sigma_{n,r})\colon \hat{\He}_{\hat{A}_{r-1}}\to 
K_0^{\Q(q)}(\mathrm{Kar}{\Scat(n,r)_{[y]}})$$
is an embedding. 
\end{cor}
\begin{proof}
We already know that $K_0^{\Q(q)}(\F)$ is injective, by Theorem~\ref{thm:diameqbim} 
and the fact that $\ebim_{\hat{A}_{r-1}}$ is a full sub-2-category of $\esbim_{\hat{A}_{r-1}}$. The result 
now follows from the commutativity of the diagram in Lemma~\ref{lem:comm}.
\end{proof}

Theorem~\ref{thm:diameqbim} and Lemma~\ref{lem:comm} also 
imply that $\Sigma_{n,r}$ is faithful. We do not know if 
it is full, as in the finite type $A$ case (Proposition 6.9 
in~\cite{MSVschur}), but we conjecture that to be true. 
\begin{conj}
\label{prop:fulfaith}
The functor 
$$\Sigma_{n,r} :\debim_{\hat{A}_{r-1}} \to \Scat(n,r)_{[y]}((1^r),(1^r))$$ 
is an equivalence of 2-categories.
\end{conj}

\begin{thm}
The algebra homomorphism 
$$\gamma\colon \hat{\SD}(n,r)\to K_0^{\Q(q)}(\mathrm{Kar}{\Scat(n,r)_{[y]}})$$
defined by 
$$\gamma(E_{\pm i}1_{\lambda}):=[\mathcal{E}_{\pm i}\onel]$$
is an isomorphism. 
\end{thm}

\begin{proof}
Khovanov and Lauda proved surjectivity of the homomorphism 
$$\Uaff \to K_0^{\Q(q)}(\Ucataff)$$
in Theorem 1.1 in~\cite{KL3}. The same arguments which proved Lemma $7.7$ 
in~\cite{MSVschur} can thus be used to prove  
that 
$$\gamma\colon \hat{\SD}(n,r)\to K_0^{\Q(q)}(\mathrm{Kar}{\Scat(n,r)})$$
is surjective. By~\eqref{eq:virnilGrot} this implies that 
the analogous homomorphism 
$$\gamma\colon \hat{\SD}(n,r)\to K_0^{\Q(q)}(\mathrm{Kar}{\Scat(n,r)_{[y]}})$$
is surjective. 

The rest of the proof follows from Lemma~\ref{alg} and 
Corollary~\ref{cor:embed} with 
$A = K_0^{\Q(q)}(\mathrm{Kar}\Scat(n,r)_{[y]})$.
\end{proof}

\bibliographystyle{alpha}
\bibliography{biblioAFFdef3}

\def\cprime{$'$}
\begin{thebibliography}{KLMS12}

\bibitem[CF94]{CrF}
L.~Crane and I.~Frenkel.
\newblock Four-dimensional topological quantum field theory, {H}opf categories,
  and the canonical bases.
\newblock {\em J. Math. Phys.}, 35(10):5136--5154, 1994.
\newblock Topology and physics.

\bibitem[CP94]{ChP}
V.~Chari and A.~Pressley.
\newblock {\em Quantum Groups}.
\newblock Cambridge University Press, 1994.

\bibitem[DDF12]{DDF}
B.~Deng, J.~Du, and Q.~Fu.
\newblock {\em A double {H}all algebra approach to affine quantum
  {S}chur-{W}eyl theory}.
\newblock Cambridge University Press, 2012.

\bibitem[DF12]{DuFu}
J.~Du and Q.~Fu.
\newblock Small representations for affine $q$-{S}chur algebras.
\newblock math.QA/1201.3476, 2012.

\bibitem[DG07]{DG}
S.~Doty and R.~Green.
\newblock Presenting affine {$q$}-{S}chur algebras.
\newblock {\em Math. Z.}, 256(2):311--345, 2007.

\bibitem[DO08]{OpD}
P.~Delorme and E.~Opdam.
\newblock The {S}chwartz algebra of an affine {H}ecke algebra.
\newblock {\em J. Reine Angew. Math.}, 625:59--114, 2008.

\bibitem[EK10]{EKh}
B.~Elias and M.~Khovanov.
\newblock Diagrammatics for {S}oergel categories.
\newblock {\em Int. J. Math. and Math. Sci.}, 2010.
\newblock Article ID 978635.

\bibitem[EW13]{EW}
B.~Elias and G.~Williamson.
\newblock {S}oergel calculus.
\newblock math.QA/1309.0865, 2013.

\bibitem[Fuc95]{Fu}
J.~Fuchs.
\newblock {\em Affine {L}ie algebras and quantum groups}.
\newblock Cambridge Monographs on Mathematical Physics. Cambridge University
  Press, Cambridge, 1995.
\newblock An introduction, with applications in conformal field theory,
  Corrected reprint of the 1992 original.

\bibitem[GH07]{GH}
I.~Grojnowski and M.~Haiman.
\newblock Affine {H}ecke algebras and positivity of {LLT} and {M}acdonald
  polynomials.
\newblock Available on M. Haiman webpage, http://math.berkeley.edu/~mhaiman/,
  2007.

\bibitem[Gre99]{Gr}
R.~Green.
\newblock The affine {$q$}-{S}chur algebra.
\newblock {\em J. Algebra}, 215(2):379--411, 1999.

\bibitem[GV93]{GV}
V.~Ginzburg and E.~Vasserot.
\newblock Langlands reciprocity for affine quantum groups of type ${A}_n$.
\newblock {\em Int. Math. Res. Notices}, 3:67--85, 1993.

\bibitem[H{\"{a}}r99]{Har}
M.~H{\"{a}}rterich.
\newblock Kazhdan-{L}usztig-{B}asen, unzerlegbare {B}imoduln, und die
  {T}opologie der {F}ahnenmannigfaltigkeit einer {K}ac-{M}oody-{G}ruppe.
\newblock Dissertation in Freiburg, 1999.

\bibitem[Kac90]{Kac}
V.~Kac.
\newblock {\em Infinite-dimensional {L}ie algebras}.
\newblock Cambridge University Press, Cambridge, third edition, 1990.

\bibitem[Kas02]{Kash}
M.~Kashiwara.
\newblock On level zero representations of quantized affine algebras.
\newblock {\em Duke Math. J.}, 112(1):117--175, 2002.

\bibitem[KL10]{KL3}
M.~Khovanov and A.~Lauda.
\newblock A categorification of quantum {${\rm sl}(n)$}.
\newblock {\em Quantum Topol.}, 1(1):1--92, 2010.

\bibitem[KL11]{KLer}
M.~Khovanov and A.~Lauda.
\newblock Erratum to: "a categorification of quantum $\mathrm{sl}(n)$".
\newblock {\em Quantum Topol.}, 2(1):97--99, 2011.

\bibitem[KLMS12]{KLMS}
M.~Khovanov, A.~Lauda, M.~Mackaay, and M.~Sto{\v{s}}i{\'c}.
\newblock Extended graphical calculus for categorified quantum
  $\mathrm{sl}(2)$.
\newblock {\em Memoirs of the AMS}, 219(1029), 2012.

\bibitem[Lau10]{Lau}
A.~Lauda.
\newblock A categorification of quantum sl(2).
\newblock {\em Adv. Math.}, 225(6):3327--3424, 2010.

\bibitem[Lib08a]{Li1}
N.~Libedinsky.
\newblock Equivalences entre conjectures de {S}oergel.
\newblock {\em J. Algebra}, 320(7):2695--2705, 2008.

\bibitem[Lib08b]{Li}
N.~Libedinsky.
\newblock Sur la cat\'egorie des bimodules de {S}oergel.
\newblock {\em J. Algebra}, 320(7):2675--2694, 2008.

\bibitem[Lus89]{Lu}
G.~Lusztig.
\newblock Affine {H}ecke algebras and their graded version.
\newblock {\em J. Amer. Math. Soc.}, 2(3):599--635, 1989.

\bibitem[Lus93]{LuQG}
G.~Lusztig.
\newblock {\em Introduction to quantum groups}.
\newblock Birkh\"{a}user, 1993.

\bibitem[Lus99]{LuAff}
G.~Lusztig.
\newblock Aperiodicity in quantum affine $\mathfrak{gl}_n$.
\newblock {\em Asian J. Math.}, 3:147--177, 1999.

\bibitem[Maz12]{Maz}
V.~Mazorchuk.
\newblock {\em Lecture on algebraic categorification}.
\newblock The QGM master class series. European Mathematical Society, 2012.

\bibitem[MSV11]{MSVcolhom}
M.~Mackaay, M.~Sto{\v{s}}i{\'c}, and P.~Vaz.
\newblock The {$1,2$}-coloured {HOMFLY}-{PT} link homology.
\newblock {\em Trans. Amer. Math. Soc.}, 363(4):2091--2124, 2011.

\bibitem[MSV13]{MSVschur}
M.~Mackaay, M.~Sto{\v{s}}i{\'c}, and P.~Vaz.
\newblock {A diagrammatic categorification of the $q$-{S}chur algebra.}
\newblock {\em Quantum Topol.}, 4(1):1--75, 2013.

\bibitem[MT13]{MTh2}
M.~Mackaay and A.-L. Thiel.
\newblock A diagrammatic categorification of the affine $q$-{S}chur algebra
  {$\hat{\mathbf{S}}$}$(n,n)$ for $n\geq 3$.
\newblock {\em in preparation}, 2013.

\bibitem[Opd03]{Op1}
E.~Opdam.
\newblock A generating function for the trace of the {I}wahori-{H}ecke algebra.
\newblock In {\em Studies in memory of {I}ssai {S}chur ({C}hevaleret/{R}ehovot,
  2000)}, volume 210 of {\em Progr. Math.}, pages 301--323. Birkh\"auser
  Boston, Boston, MA, 2003.

\bibitem[Opd04]{Op2}
E.~Opdam.
\newblock On the spectral decomposition of affine {H}ecke algebras.
\newblock {\em J. Inst. Math. Jussieu}, 3(4):531--648, 2004.

\bibitem[OS09]{OpS}
E.~Opdam and M.~Solleveld.
\newblock Homological algebra for affine {H}ecke algebras.
\newblock {\em Adv. Math.}, 220(5):1549--1601, 2009.

\bibitem[PS86]{PS}
A.~Pressley and G.~Segal.
\newblock {\em Loop groups}.
\newblock Oxford Mathematical Monographs. The Clarendon Press Oxford University
  Press, New York, 1986.
\newblock Oxford Science Publications.

\bibitem[Rou06]{R}
R.~Rouquier.
\newblock Categorification of {${\mathfrak{sl}}\sb 2$} and braid groups.
\newblock In {\em Trends in representation theory of algebras and related
  topics}, volume 406 of {\em Contemp. Math.}, pages 137--167. Amer. Math.
  Soc., Providence, RI, 2006.

\bibitem[Rou08]{R2}
R.~Rouquier.
\newblock 2-{K}ac-{M}oody algebras.
\newblock math.RT/0812.5023, 2008.

\bibitem[Soe07]{S3}
W.~Soergel.
\newblock Kazhdan-{L}usztig-{P}olynome und unzerlegbare {B}imoduln \"uber
  {P}olynomringen.
\newblock {\em J. Inst. Math. Jussieu}, 6(3):501--525, 2007.

\bibitem[Web10]{Web}
B.~Webster.
\newblock Knot invariants and higher representation theory i: diagrammatic and
  geometric categorification of tensor products.
\newblock math.GT/1001.2020, 2010.

\bibitem[Wil11]{Wi}
G.~Williamson.
\newblock Singular {S}oergel bimodules.
\newblock {\em Int. Math. Res. Not. IMRN}, (20):4555--4632, 2011.

\end{thebibliography}
\vspace{0.1in}

\noindent M.M.: { \sl \small Center for Mathematical Analysis, Geometry, and Dynamical Systems, Departamento de Matem\'{a}tica, Instituto Superior T\'{e}cnico, 
1049-001 Lisboa, Portugal; Departamento de Matem\'{a}tica, FCT, Universidade do Algarve, Campus de Gambelas, 8005-139 Faro, Portugal} 
\newline \noindent {\tt \small email: mmackaay@ualg.pt}

\noindent A.-L.T.: { \sl \small Center for Mathematical Analysis, Geometry, and Dynamical Systems, Departamento de Matem\'{a}tica, Instituto Superior T\'{e}cnico, 
1049-001 Lisboa, Portugal; Matematiska Institutionen, Uppsala Universitet, 75106 Uppsala, Sverige} 
\newline \noindent {\tt \small email: anne-laure.thiel@math.uu.se}
\end{document}